\documentclass[a4paper,11pt]{article}
\usepackage{amsthm,amsfonts,amssymb,euscript,
mathrsfs,graphics,color,amsmath,latexsym,marginnote, dsfont}
\usepackage{todonotes}
\usepackage{mathtools, upgreek}

\theoremstyle{plain}

\usepackage{amsmath}
\usepackage{times}
\usepackage{accents}

\numberwithin{equation}{section}
\newcommand{\grad}{\nabla}
\newcommand{\tB}{\mathtt{B}}
\newcommand{\tA}{\mathtt{A}}

\newcommand{\tM}{\mathtt{M}}
\newcommand{\tC}{\mathtt{C}}

\newcommand{\tnf}{\mathtt{nf}}
\newcommand{\tSAP}{\mathtt S\mathtt A\mathtt P}
\newcommand{\im}{{\rm i}}
\newcommand{\cA}{\mathcal{A}}
\newcommand{\cB}{\mathcal{B}}
\newcommand{\cC}{\mathcal{C}}
\newcommand{\cD}{\mathcal{D}}
\newcommand{\cE}{\mathcal{E}}
\newcommand{\cF}{\mathcal{F}}
\newcommand{\cG}{\mathcal{G}}

\newcommand{\cI}{\mathcal{I}}

\newcommand{\cM}{\mathcal{M}}
\newcommand{\cN}{\mathcal{N}}

\newcommand{\cK}{\mathcal{K}}
\newcommand{\cP}{\mathcal{P}}

\newcommand{\tb}{\mathtt{b}}

\newcommand{\cR}{\mathcal{R}}
\newcommand{\cS}{\mathcal{S}}

\newcommand{\cU}{\mathcal{U}}

\newcommand{\cW}{\mathcal{W}}

\newcommand{\uno}{ {\rm Id}}
\newcommand{\la}{\langle}
\newcommand{\ra}{\rangle}
\newcommand{\tth}{\mathtt{h}}

\newcommand{\vr}{\varrho}
\newcommand{\di}{{\rm d}}

\newcommand{\td}{{\mathtt d }}
\newcommand{\tq}{{\mathtt q }}
\newcommand{\tG}{{\mathtt G }}
\newcommand{\X}{{\mathfrak X }}

\newcommand{\meas}{{\rm meas}}

\newtheorem{theorem}{Theorem}[section]
\newtheorem{proposition}[theorem]{Proposition}
\newtheorem{lemma}[theorem]{Lemma}

\newtheorem{definition}[theorem]{Definition}
\theoremstyle{definition}
\newtheorem{remark}[theorem]{Remark}
\newtheorem{remarks}[theorem]{Remark}

\newcommand{\be}{\begin{equation}}
\newcommand{\ee}{\end{equation}}
\newcommand{\bes}{\begin{equation*}}
\newcommand{\ees}{\end{equation*}}
\newcommand{\ba}{\begin{align}}
\newcommand{\ea}{\end{align}}
\newcommand{\bas}{\begin{align*}}
\newcommand{\eas}{\end{align*}}

\newcommand{\ta}{\mathtt{a}}

\newcommand{\id}{{\rm Id}}
\newcommand{\e}{\varepsilon}

\newcommand{\ov}{\overline}

\newcommand{\R}{\mathbb R}
\newcommand{\C}{\mathbb C}

\newcommand{\Z}{\mathbb Z}

\newcommand{\N}{\mathbb N}
\newcommand{\T}{\mathbb T}
\newcommand{\sign}{{\rm sign}}
\renewcommand{\a }{\alpha }

\newcommand{\s }{\sigma }
\newcommand{\ii }{{\rm i} }

\newcommand{\vect}[2]{\begin{pmatrix}#1 \\#2\end{pmatrix}}

\newcommand{\mN}{\mathcal{N}}

\newcommand{\x }{\xi }
\newcommand{\tm}{\mathtt{m}}

\newcommand{\mH}{\mathcal{H}}

\newcommand{\mL}{\mathcal{L}}

\newcommand{\mM}{\mathcal{M}}

\newcommand{\mA}{\mathcal{A}}

\newcommand{\mR}{\mathcal{R}}
\newcommand{\mG}{\mathcal{G}}

\newcommand{\mF}{\mathcal{F}}

\newcommand{\mE}{\mathcal{E}}
\newcommand{\mS}{\mathcal{S}}

\newcommand{\mP}{\mathcal{P}}
\newcommand{\pa}{\partial}

\newcommand{\opbw}{{{\rm Op}^{{\scriptscriptstyle{\mathrm BW}}}}}

\def\hat{\widehat}
\def\wt{\widetilde}
\def\bar{\overline}
\def\cal{\mathcal}
\renewcommand{\Re}{\mathrm{Re}\,}
\renewcommand{\Im}{\mathrm{Im}\,}

\newcommand{\mD}{\mathcal{D}}

\def\ba{\begin{aligned}}
\def\ea{\end{aligned}}
\def\beginm{\begin{multline}}
\def\endm{\end{multline}}

\newcommand{\mB}{\mathcal{B}}

\newcommand{\mC}{\mathcal{C}}

\newcommand{\Opbw}[1]{{{{\rm Op}^{{\scriptscriptstyle{\mathrm BW}}}}\left(#1\right)}}

\newcommand{\vOpbw}[1]{{{{\rm Op}_{\mathtt{vec}}^{{\scriptscriptstyle{\mathrm BW}}}}\left(#1\right)}}
\newcommand{\vOmega}{{\bf \Omega}}
\newcommand{\bA}{{\bf A}}
\newcommand{\bB}{{\bf B}}

\newcommand{\bF}{{\bf F}}
\newcommand{\bG}{{\bf G}}

\newcommand{\bK}{{\bf K}}

\newcommand{\bM}{{\bf M}}

\newcommand{\bR}{{\bf R}}
\newcommand{\bS}{{\bf S}}

\newcommand{\fT}{{\mathfrak{T}}}

\newcommand{\wtcS}{{\widetilde {\mathcal S}}}
\newcommand{ \wtcR}{{\widetilde {\mathcal R}}}
\newcommand{\wtcM}{{\widetilde {\mathcal M}}}

\newcommand{\rS}{{\bS}}

\newcommand{\sgn}{{\rm sign}}

\newcommand{\norm}[1]{{\| #1 \| }}
\providecommand{\vect}[2]{{\bigl[\begin{smallmatrix}#1\\#2\end{smallmatrix}\bigr]}} \providecommand{\sign}{\mathrm{sgn}\,}  

\def\l@subsection{\@tocline{2}{0pt}{2.5pc}{5pc}{}}

\usepackage[colorlinks = true,
            linkcolor = blue,
            urlcolor  = magenta,
            citecolor = magenta,
            anchorcolor = blue]{hyperref}
\begin{document}
\bibliographystyle{plain}
\title{\bf 
Hamiltonian Birkhoff normal form for \\
 gravity-capillary 
 water waves 
 with constant vorticity: \\
almost global existence}
\date{}
\author{Massimiliano Berti, 
Alberto Maspero,  
Federico Murgante
\footnote{
International School for Advanced Studies (SISSA), Via Bonomea 265, 34136, Trieste, Italy. \newline
 \textit{Emails: } \texttt{berti@sissa.it},  \texttt{alberto.maspero@sissa.it},
 \texttt{fmurgant@sissa.it}
 }
}

\maketitle

\noindent 
{\bf Abstract.}  
We prove an almost global in time existence  result of small amplitude space {\it periodic} 
solutions of the 1D gravity-capillary water waves equations with constant vorticity.
The result holds for any value of  
gravity,  vorticity and  depth
and any 
surface tension belonging to  a full measure set.  
The proof  demands
a {\it Hamiltonian} paradifferential 
Birkhoff normal form reduction 
for quasi-linear PDEs in presence of resonant wave interactions: 
the normal form may be not integrable but it preserves the Sobolev norms 
thanks to 
its Hamiltonian nature. 
A major difficulty 
is that usual paradifferential calculus 
 used 
to prove local well posedness 
(as the celebrated Alinhac good unknown) 
does not preserve the Hamiltonian structure.  A major novelty of this paper 
is to develop 
an  
algorithmic  perturbative procedure \`a la Darboux to 
correct usual paradifferential transformations to 
symplectic maps, up to an arbitrary degree of homogeneity.
The symplectic correctors turn out to be smoothing perturbations of the identity, and therefore only slightly modify the paradifferential structure of the equations.  
 The  Darboux procedure
which recovers 
the nonlinear Hamiltonian structure is written in an abstract functional setting, 
in order 
to be applicable  also in other contexts. 
\smallskip


\phantomsection
\tableofcontents

\section{Introduction and main results}

The study of the time evolution of water waves  is a classical question, 
and one of the major problems in fluid dynamics. 
The very first attempt of a theory of water waves trace its origin in the {\em Principia} of Newton in the second half of the seventeenth century, but one more century is needed for
the foundational works by giants such as  Euler,  Laplace, Lagrange, Cauchy, Poisson,  Bernoulli, 
and  then  by the British school with Russel, Robinson, Green,  Airy, Stokes among others. We refer to the historical overview
 in \cite{Craik}.

\smallskip

In this paper we consider the Euler equations of hydrodynamics for a 2-dimensional perfect and incompressible 
fluid with constant {\it vorticity}  $\gamma $, under the action of gravity and capillary forces at the free surface.
The fluid fills  the time dependent region 
\be\label{domain} 
\cD_{\eta} := \big\{ (x,y)\in \T\times \R \ : \ - \tth < y<\eta(t,x) \big\} \, , 
\quad  \T := \T_x :=\R/ (2\pi\Z) \, ,
\ee
with depth $ \tth > 0 $, possibly infinite, 
and space periodic boundary conditions. 
The unknowns 
are the  free surface  $ y = \eta (t, x)$
of $\cD_{\eta} $ and the 
divergence free  velocity field
$\begin{psmallmatrix} u(t,x,y) \\ v(t,x,y) \end{psmallmatrix} $. 
In case of a fluid with constant vorticity
$ v_x - u_y = \gamma $ (a property which is preserved along the time evolution), 
 the velocity field   is the sum of the Couette flow $\begin{psmallmatrix} - \gamma y \\ 0 \end{psmallmatrix}$, which carries
all the  vorticity $ \gamma $, 
and an irrotational field, 
expressed as the gradient of a harmonic function $\Phi $, called the generalized velocity potential. 

We study the water waves problem in the 
{\it Hamiltonian}  
Zakharov-Craig-Sulem \cite{Zak1,CrSu} formulation, extended by
Constantin, Ivanov, Prodanov \cite{CIP} and Wahl\'en \cite{Wh}
for constant vorticity fluids.  
Denoting by 
$\psi(t,x)$  the evaluation of the generalized velocity potential at the free interface
$ \psi (t,x) := \Phi (t,x, \eta(t,x)) $, one recovers
$ \Phi $ as 
the unique harmonic function 
$ \Delta \Phi = 0 $ in $ \cD_{\eta} $ with Dirichlet boundary condition
$ \Phi = \psi $ at $  y = \eta(t,x) $ and Neumann boundary condition
$\Phi_y ( t, x, y) \to  0  $ as $ y \to  - \tth $. 
Imposing  
that the fluid particles at the free surface remain on it along the evolution
(kinematic boundary condition) and that
the pressure of the fluid 
plus the capillary forces at the free surface  is equal to the constant 
atmospheric pressure (dynamic boundary condition),  the 
time evolution of the fluid is determined by the non-local quasi-linear 
equations 
\begin{equation} \label{eq:etapsi}
\left\{\begin{aligned}
 &   \partial_t \eta = G(\eta)\psi+ \gamma \eta \eta_x \\
&\partial_t\psi =  -g\eta  -\frac{1}{2} \psi_x^2 
+  \frac{1}{2}\frac{(\eta_x  \psi_x + G(\eta)\psi)^2}{1+\eta_x^2}+\kappa 
\partial_x\Big[ \frac{\eta_x}{(1+\eta_x^{2})^{\frac{1}{2}}}\Big]+ \gamma \eta \psi_x+ \gamma \partial_x^{-1} G(\eta) \psi 
\end{aligned}\right.
\end{equation}
where $ g > 0 $ is the gravity constant, $ \kappa > 0 $ is the surface tension coefficient,
$ \partial_x\big[ \frac{\eta_x}{(1+\eta_x^{2})^{1/2}}\big] $
is the curvature of the  surface
and 
$ G(\eta)$ is the  Dirichlet-Neumann operator  
\be\label{def:DN}
G(\eta)\psi  := 
(- \Phi_x \eta_x + \Phi_y)\vert_{y = \eta(x)} \, . 
\ee
The quantity $\int_\T \eta(x) \, \di x$ is a prime integral of \eqref{eq:etapsi}
(indeed $\int_\T G(\eta)\psi \di x = 0$)
and then, with no loss of generality,  we restrict to interfaces  with zero average
$ \int_\T \eta(x) \, \di x = 0 $. The component $\eta $  of the solution of \eqref{eq:etapsi}
 will lie in a Sobolev space $H_0^{s+\frac{1}{4}}(\T)$ of periodic functions with zero mean. 
Moreover the vector field on the right hand side of  \eqref{eq:etapsi} 
depends only on $ \eta $ and $  \psi - \frac{1}{2 \pi}\int_\T \psi \, \di x  $
(indeed $G(\eta)[1] = 0$) 
and therefore  $ \psi $ 
 will evolve in a 
homogeneous Sobolev space $ \dot H^{{s-\frac{1}{4}}}(\T)$ of periodic functions modulo constants. 

By \cite{Zak1,CrSu,CIP,Wh} the equations \eqref{eq:etapsi} are the Hamiltonian system 
\be\label{HamWW} 
\pa_t\vect{\eta}{\psi} = J_\gamma \vect{\nabla_{\eta} H_\gamma(\eta,\psi)}{\nabla_{\psi} H_\gamma(\eta,\psi)} \qquad
\text{where} 
\qquad 
J_\gamma := \begin{pmatrix} 0 & \uno \\ -\uno & \gamma \pa_x^{-1} \end{pmatrix} 
\ee
and  
\be\label{H.gamma}
H_\gamma(\eta, \psi)  := \frac12 \int_\T \big(\psi \, G(\eta ) \psi+ g \eta^2\big)  \, \di x + \kappa  \int_{\T}  \sqrt{1+ \eta_x^2}  \, \di x+  \frac{\gamma}{2} \int_{\T}  \big(-\psi_x \eta^2+ \frac{\gamma}{3} \eta^3\big)  \, \di x  \, .
\ee
The $L^2$-gradients $(\grad_\eta H_\gamma, \grad_\psi H_\gamma) $ in
\eqref{HamWW}  belong to (a dense subspace of)  
$ \dot L^2(\T)\times L^2_0(\T)$.

Since the bottom of $  \cD_{\eta} $  in \eqref{domain} 
is flat, 
the Hamiltonian vector field $X_\gamma $, defined by  the right hand side of \eqref{eq:etapsi}, is translation invariant, namely 
\be\label{X.tra0}
X_\gamma \circ \tau_\varsigma = \tau_\varsigma \circ X_\gamma 
\, , \quad \forall \varsigma \in \R \, ,
\qquad \text{where} \qquad 
\tau_\varsigma \colon f(x) \mapsto f(x + \varsigma) 
\ee
is the translation operator.
Equivalently 
the Hamiltonian $H_\gamma$ in  \eqref{H.gamma} satisfies 
$H_\gamma \circ \tau_\varsigma = H_\gamma$ for any 
$ \varsigma \in \R$. 
The associated conservation law induced 
by Noether theorem is the 
 momentum 
$ \int_{\T}  \psi ( x ) \eta_x (x) \, \di x $. 

\smallskip

The main result of this paper (Theorem \ref{teo1}) is that,  
for almost all surface tension coefficients $ \kappa $, 
for any integer $N  $,  
the solutions of the water waves equations \eqref{eq:etapsi} 
with initial data (smooth enough) of size $\e$  small enough,    
are defined over a time
 interval of length at least $ 
 c \e^{-N-1}$.  This is   the most general 
 almost global existence in time 
 result for the solutions of the water waves equations with periodic boundary conditions known so far.  
We present below the mathematical literature concerning  the local and global 
well posedness theory of water waves,  
focusing on the maximal time life span of the solutions. 

 In order to state precisely 
 the main theorem we define, for any $ s \in \R $,  the Sobolev spaces 
$$
 H^s_0 (\T, \C) = \Big\{ u(x) \in H^s (\T,\C) \, : \, \int_{\T} u(x) \, \di x = 0  \Big\} \, ,
\quad \dot H^s (\T, \C) = H^s (\T, \C) \slash \C \, ,
$$
equipped with the same norm 
$$
\| u \|_{H^s_0} = \| u \|_{\dot H^s} = 
\Big( \sum_{n \in \N} \| \Pi_n u \|_{L^2}^2 n^{2s} \Big)^{\frac12} =
\Big( \sum_{j \in \Z \setminus \{0\} }  |u_j|^2 | j |^{2s} \Big)^{\frac12} 
$$
where $ \Pi_n $ denote the orthogonal projectors  from $ L^2 (\T, \C) $ on the subspaces spanned by  
$ \{ e^{-\ii n x}, e^{ \ii n x} \} $ and $ u_j $ are the Fourier coefficients of $ u(x) $.  
The quotient map induces 
an isometry between $ H^s_0 $ and $ \dot H^s $ and we shall often identify 
$ H^s_0 $ with $  \dot H^s $.
Our main result is the following.

\begin{theorem}\label{teo1}
{\bf (Almost global  in time 
gravity-capillary 
  water waves with constant vorticity)}
For any value of the gravity $ g > 0 $, depth $ \tth \in (0,+\infty] $ and vorticity $ \gamma \in \R $, there is a zero measure set $ {\mathcal K} \subset  (0,+\infty)$ such that, for any 
 surface tension coefficient $ {\kappa} \in (0,+\infty) \setminus {\mathcal{K}}  $, for any
$N$ in $\N_0 $, there is $s_0>0$ and, for any $s\geq s_0$, there are
$\e_0>0, c>0, C>0$ such that, for any $ 0 < \e < \e_0 $, any initial datum     
$$
(\eta_0,\psi_0) \in H_0^{s+\frac{1}{4}}(\mathbb{T},\R)\times {\dot H}^{s-\frac{1}{4}}(\mathbb{T},\R) \quad \mbox{ with} \quad
\| \eta_0 \|_{H^{s+\frac14}_0}+  \|\psi_0\|_{{\dot H}^{s-\frac14}}<\e \, , 
$$
system \eqref{eq:etapsi} has a unique classical solution $(\eta,\psi)$ in 
\be\label{timeexi}
C^0\Bigl([-T_\e,T_\e],H_0^{s+\frac14}(\mathbb{T},\R)\times {\dot H}^{s-\frac14}(\mathbb{T},\R)\Bigr) 
\quad \mbox{with} \quad T_\e \geq c\e^{-N-1} \, , 
\ee
satisfying the initial condition $\eta |_{t=0} = \eta_0, \psi|_{t=0} = \psi_0$. Moreover
\be\label{boundsolN}
\sup_{t\in [-T_{\e},T_{\e}]}\big( 
 \|\eta\|_{H_0^{s+\frac14}}+\|\psi\|_{\dot{H}^{s-\frac14}}\big)\leq C\e\,.
\ee
\end{theorem}

Let us make some comments on the result. 
\\[1mm]
1.  {\sc Comparison with \cite{BD}.}
We first discuss the relation between Theorem \ref{teo1} 
and the result in  Berti-Delort \cite{BD}.  
Theorem \ref{teo1} extends the one in 
 \cite{BD} in two ways:
($i$) the equations \eqref{eq:etapsi}
may have a {\it non zero vorticity}, whereas the water waves in \cite{BD} are irrotational, 
i.e. $ \gamma = 0 $.  
 ($ii$) Also in the irrotational case Theorem \ref{teo1} is new since 
the almost global existence result in \cite{BD} 
holds for initial data $(\eta_0, \psi_0) $   
even in $ x $, whereas Theorem \ref{teo1} applies to 
{\it any} $(\eta_0, \psi_0) $.
We  remark that, 
in the irrotational case,  
the subspace of functions even in $ x $ -the so called standing waves- is invariant under evolution, whereas  for $ \gamma \neq 0 $ it 
is not invariant under the flow of \eqref{eq:etapsi} and the approach of \cite{BD} 
can not be applied.  
\\[1mm]
2. {\sc Periodic setting vs $\R^d$.} 
Global (and almost global) in time results 
\cite{Wu, GMS, Wu2, GMS2, IP, AlDe1, IFRT, IP3,DIPP}  
have been proved for irrotational water waves equations on $\R^d$ 
for  
sufficiently small,  localized 
and  smooth  
 initial data,  
 exploiting 
 the dispersive effects of the linear flow. 
 So far no global existence is  known for the solutions of 
 \eqref{eq:etapsi} in $\R^2$, not even for
  irrotational 
 fluids (\cite{DIPP} applies  in $\R^3 $). 
The periodic setting  is deeply different, as the linear waves  oscillate without decaying in time, and the long time dynamics of the equations strongly depends
on 
the presence of $N$-{\it wave resonant interactions} and  
the Hamiltonian and reversible nature of the equations. 
\\[1mm]
3. 
{\sc Dispersion relation and non-resonant parameters.} 
The water waves equations \eqref{eq:etapsi} 
may be regarded as a {	\it quasi-linear} 
complex PDE of the form   
$$
\pa_t u = - \im  \Omega(D) u + \cN( u, \bar u) , 
\quad u(x) = \tfrac{1}{\sqrt {2\pi}} {\mathop \sum}_{j \in \Z\setminus \{0\}} u_j \,  e^{\im j x} \, ,
$$
where $ \cN $ is a quadratic 
nonlinearity (depending on derivatives of $ u $) and $\Omega_j(\kappa) $
is the  {\it dispersion relation} 
\be \label{omegonejin}
\Omega_j(\kappa):= \omega_j(\kappa) +  \frac{ \gamma}{2} \frac{\tG(j)}{j} \, , 
\qquad 
\omega_j(\kappa):=\sqrt{ \tG(j) \Big(g+ \kappa j^2 + \frac{\gamma^2}{4} \frac{\tG(j)}{j^{2}}\Big)} \, , 
\ee
where  $ \tG(\xi) = |\xi| \tanh (\tth |\xi|) $ ($ = |\xi|$ in infinite depth)  
is the symbol of the Dirichlet-Neumann operator
$ G(0) $. The linear frequencies $  \Omega_j  (\kappa) $ actually 
depend on  $ (\kappa,g,\tth,\gamma)$.
The restriction on the  parameters  required in Theorem \ref{teo1} 
arises  to ensure the absence of $N$-wave resonant interactions  
\be\label{Nwavei}
\Omega_{j_1} (\kappa) \pm \ldots \pm \Omega_{j_N} (\kappa) \neq 0  
\ee
(with  quantitative lower bounds as \eqref{Nwavei1} below)
among integer indices $ j_1 , \ldots, j_N $ which are not super-action preserving, 
cfr. Definition \ref{passaggioalpha}. 
In Theorem  \ref{teo1} we fix arbitrary $ (g,\tth,\gamma) $ and 
require $ \kappa \notin \cK $, but other choices are possible. 
\\[1mm]
4. {\sc Energy estimates.}
The 
 life span estimate \eqref{timeexi} and the bound \eqref{boundsolN}
 for the solutions of \eqref{eq:etapsi} 
 follow by  
an energy estimate  for
$ \| ( \eta, \psi ) \|_{X^s} := \| \eta \|_{H_0^{s+\frac14}} +
\|\psi\|_{\dot{H}^{s-\frac14}} $  of the form 
\be\label{energyesint}
{\| (\eta,\psi)(t) \|}_{X^{s}}^{2} 
\lesssim_{s,N} {\|(\eta,\psi)(0)\|}^{2}_{X^{s}} +
 \int_{0}^{t} \| (\eta,\psi)(\tau)\|^{N+3}_{X^{s}} \, \di \tau 
\, . 
\ee
The fact that the 
right hand side in \eqref{energyesint}
  contains 
the same norm $ \| \ \|_{X^s} $ of the left hand side 
is  non trivial because  
the equations \eqref{eq:etapsi} are quasi-linear. 
The presence of the exponent $ N $   is 
not trivial at all 
because the 
nonlinearity in 
\eqref{eq:etapsi} vanishes only  quadratically for $ ( \eta, \psi ) = (0,0) $.
Actually it will be a major consequence of our 
Hamiltonian Birkhoff normal form reduction, as we explain below. 
\\[1mm]
5.
{\sc Long time existence of 
water waves.} We now describe the 
long time existence  results  proved in literature 
for {\it space} periodic water waves, with or without capillarity and vorticity.   
\begin{enumerate}
\item[($i$)] $ T_\e \geq c \e^{-1} $. 
The 
local  
well posedness 
theory  for free boundary Euler equations has been developed along several years in different scenarios in 
  \cite{Nali, Yosi, Crai, Wu0,Wu1,Lannes, ABZduke,ABZ1,BeGun,MZ,Sch, Lind1,Com, SZ1, SZ2,
  ABZnl,
  IFT,AMS}.
 As a whole they
  prove  the existence, for sufficiently nice initial data, of 
 classical smooth solutions on a small time interval.
 When specialized to initial data of size $ \e $ in some Sobolev space, imply 
 a time of existence  
larger than $ c \e^{-1} $ (the nonlinearity in \eqref{eq:etapsi} 
vanishes quadratically at zero). We remark that 
other large initial data can lead to breakdown in finite time, see for example the papers \cite{CCFGG,CoSh} on ``splash" singularities. 
\item[($ii$)]  $ T_\e \geq c \e^{-2} $. 
Wu \cite{Wu}, Ionescu-Pusateri \cite{IP}, Alazard-Delort \cite{AlDe1} 
for pure gravity waves, and Ifrim-Tataru \cite{ifrTat},    
Ionescu-Pusateri~\cite{IP3} for 
  $\kappa>0$, $g=0$ and $ \tth = + \infty $,  proved that  
small data
of size $ \e $ 
(periodic or on the line)
give rise to  irrotational solutions 
defined on a time interval  
at least $c \e^{-2}$. 
We quote  
\cite{IFT} 
for  $\kappa = 0$, $g>0$, infinite depth
and constant
vorticity,  \cite{HuIT} for irrotational fluids, and 
\cite{HarIT}  in finite depth.
All the  previous results hold in absence of three wave interactions. 
Exploiting the Hamiltonian nature of the water waves equations,
Berti-Feola-Franzoi \cite{BFF}  
proved, for any value of 
gravity, capillarity and depth,  
an energy estimate as \eqref{energyesint} with $ N  = 1 $, and so a 
$ c \e^{-2}$  lower bound for the time of existence. 
The interesting fact is that 
in  these cases 
three wave interactions may occur, giving rise to the well known Wilton ripples in fluid mechanics literature.
We finally mention the $ \e^{- \frac53+} $ long time existence result 
 \cite{IoP} for periodic 2D
gravity-capillary water waves (see \cite{FGI,FIM} for NLS). 
\item[($iii$)]  $T_\e \geq c \e^{-3}$.
A 
 time of existence larger than  $ c \e^{-3}$ has been recently proved for the pure 
gravity water waves equations in deep water 
in Berti-Feola-Pusateri \cite{BFP}.
In this case
four wave interactions 
may occur, 
but 
the Hamiltonian Birkhoff normal form turns out to be completely integrable by the formal computation in Zakharov-Dyachenko \cite{ZakD},
 implying 
an energy estimate as \eqref{energyesint} with $ N  = 2 $.  
 This result has been 
recently extended by S. Wu \cite{Wu3}   
for a larger class of initial data, developing 
a novel approach in configuration space, and, even more  recently, 
by Deng-Ionescu-Pusateri \cite{DIP} 
for waves with large period. 
\item[($iv$)] $ T_\e \geq c_N \e^{-N} $ {\it for any} $ N $.  
Berti-Delort \cite{BD} proved, 
for almost all the values of the surface tension 
$ \kappa \in (0,+\infty)$, 
 an almost global existence result as in Theorem \ref{teo1} 
for the solutions of 
 \eqref{eq:etapsi} 
in the case of {\it zero} vorticity $ \gamma = 0 $ 
{\it and} for initial data $(\eta_0, \psi_0) $   
even in $ x $. The restriction on the capillary parameter
arises 
to imply the absence of 
$ N $-wave interactions, for any $ N $. 
As already said, Theorem \ref{teo1} extends this result  for 
{\it any} $ \gamma $ 
and for {\it any}  periodic initial data, see comment 1. 
\end{enumerate}
The results \cite{BD,BFP,BFF} are based on paradifferential calculus.
We remark that all the transformations performed
to get energy estimates,  
 as the celebrated Alinhac good unknown \cite{AlM,ABZduke,ABZ1,AlDe1}, 
 are {\it not} symplectic. 
In  \cite{BFP,BFF}
 an a-posteriori identification argument 
allows to prove that the corresponding quadratic and cubic 
Poincar\'e-Birkhoff normal forms 
are nevertheless Hamiltonian.  This argument does not work
for any $  N $. 
We  note that also the  
local well posedness approach of S.Wu \cite{Wu0,Wu1} 
introduces coordinates which break the Hamiltonian nature of the 
equations.

The lack of preservation of the 
Hamiltonian structure 
 is a substantial difficulty in order to deduce long time existence results. 
A major novelty of this paper is to provide an effective  tool 
to 
recover, in the framework of paradifferential calculus, 
  the nonlinear Hamiltonian structure,  at any degree of homogeneity $ N $.  
 The present approach is in principle applicable to a wide range of quasi-linear PDEs. 
\\[1mm]
6.
{\sc Open problem:}
we do not know if the almost global solutions of the
Cauchy problem proved in Theorem \ref{teo1} 
are global in time or not, being \eqref{eq:etapsi} a 
 quasi-linear system of equations with 
periodic boundary conditions
(no dispersive effects of the flow nor 
conservation laws at the regularity 
level of the local well posedeness theory can be exploited). 
Nevertheless several families of time periodic/quasi-periodic 
 solutions of \eqref{eq:etapsi} 
have been constructed in the last years in \cite{Wh2,AB,BFM1,BM}
(other KAM results for pure gravity water waves are  proved in
\cite{PlTo,IPT,BBHM,BFM2,FG}).
We  point out that it could also happen that
small, 
smooth and localized initial data  
lead to solutions which blow-up in finite time (as it happen 
for quasilinear wave equations  
 \cite{John} 
 and 
 for compressible Euler equations \cite{Si}).
The following natural 
question therefore
remains still open: what happens to the solutions of \eqref{eq:etapsi} which do not start on a KAM invariant torus for times longer than the ones provided by Theorem \ref{teo1}? 

\smallskip

We now illustrate some of the main ideas
of our  approach. 
\\[1mm]
1. {\sc Paradifferential Hamiltonian Birkhoff normal form.}
 For PDEs on a compact manifold
--where dispersion is not available-- a natural tool to extend the life span of 
solutions is to implement normal form ideas.
This approach has been  developed 
for Hamiltonian 
semilinear PDEs  
starting with the seminal works  
\cite{Bam1,BG1,BDGS} 
(for very recent advances see e.g. \cite{BiMP,BeFG0,BeG,BaFM}),
and for quasi-linear ones by Delort \cite{Del2, Del3}.
These methods do not work for the quasi-linear 
equations \eqref{eq:etapsi}, as we explain   below.
The long time existence result of 
Theorem \ref{teo1} 
relies on 
a   novel 
 {\it paradifferential Hamiltonian Birkhoff normal form} reduction 
for quasi-linear PDEs in presence of resonant wave interactions.

The situation is substantially more difficult than in \cite{BD} which 
exploits only the reversible structure of the water waves,
and it is preserved by usual paradifferential calculus.
 On  the subspace of functions even in $  x $, 
 it implies that its  normal form
possesses 
the actions $ |u_n|^2 $  as prime integrals
 (on the subspace of even functions the  linear frequencies $ \omega_j (\kappa) $ in the dispersion relation
 are simple). 
 On the other hand,  without this restriction, 
the 
$ \omega_j (\kappa) $ in \eqref{omegonejin} are double and the approach in \cite{BD} fails.
We remark   that, in view of the $ 8$-wave resonant interactions 
\eqref{8waves} described below, also for $ \gamma \neq 0 $
 the approach in \cite{BD} fails. 
In order to prove 
Theorem \ref{teo1}, 
it is necessary to change strategy and {\em preserve the Hamiltonian nature of the normal form}
to show 
that the  {\it super-actions} 
\be\label{superacpin}
\| \Pi_n u \|_{L^2}^2 = |u_{-n}|^2 + |u_n|^2  \, , \quad  \forall n \in \N \, , 
\ee
are prime integrals.
This is a major difficulty 
since usual paradifferential calculus  transformations 
performed 
 to get energy estimates 
do not preserve the Hamiltonian structure.  
\\[1mm]
2.  {\sc Symplectic Darboux corrector.}
In order to preserve the Hamiltonian structure along the normal form reduction 
-it is sufficient up to homogeneity $ N $- we construct {\it symplectic correctors} 
of 
usual paradifferential
 transformations. 
 We remind that the first step to apply paradifferential calculus to PDEs
 relies on 
 a suitable para-linearization of  the equations.  For Hamiltonian PDEs, 
the paradifferential part inherits a linear Hamiltonian structure that is 
preserved by performing ``linearly symplectic" 
transformations. 
The aim of the abstract Theorem \ref{thm:almost}  is to correct paradifferential 
(more generally spectrally localized) linearly symplectic  maps 
(up to homogeneity $ N $) to {\it nonlinear} symplectic ones, 
 up to an arbitrary degree of homogeneity.
 Theorem \ref{thm:almost}  is proved 
via  Darboux-type arguments. 
The Darboux corrector turns out to be 
a  smoothing perturbation of the identity. As a consequence
it only slightly  modifies  the paradifferential structure of the  PDE. 

Symplectic corrections via  Darboux-type  arguments 
have been used 
in different contexts 
by Kuksin-Perelman \cite{KukPer}, Bambusi \cite{Bam2}, Cuccagna \cite{Cuc1,Cuc2}, Bambusi-Maspero \cite{BaM0,BaM1}.
The present case is much more delicate since    the symplectic form to be corrected might be an unbounded perturbation of the standard one (in all the above  works it is a smoothing perturbation). This requires a novel analysis that we describe below.
The present approach 
is quite efficient in PDE applications, 
since it systematically allows to symplectically correct  usual 
  paradifferential 
 transformations which lead to energy estimates. 

Paradifferential calculus 
has been also 
developed 
by Delort \cite{Del2,Del3} 
for Hamiltonian  quasi-linear Klein-Gordon
equations on spheres, with  a different approach.
Also in these works 
the Hamiltonian structure is preserved only up to homogeneity $ N $.
\\[1mm]
3.  {\sc Non-resonance conditions.}
A key 
ingredient to achieve the Hamiltonian Birkhoff normal form reduction 
which possesses the super-actions \eqref{superacpin} as prime integrals, 
are the  
non-resonance conditions \eqref{verad0} for the linear frequencies $ \Omega_j (\kappa) $ in \eqref{omegonejin} 
 proved in Theorem \ref{nonresfin0},  which 
exclude, for almost all surface tension coefficients,  
$ N $-{\it wave  interactions}, 
\be\label{Nwavei1}
|\Omega_{j_1} (\kappa) \pm \ldots \pm \Omega_{j_N} (\kappa) | \gtrsim 
 \max(|j_1|, \ldots, |j_N|)^{-\tau}  \, ,
\ee
for all integer indices $ j_1 , \ldots, j_N $ which are {\it not super-action preserving}.
Their proof is based on the Delort-Szeftel 
Theorem 5.1 in \cite{DelS} about measure estimates 
for sublevels of subanalytic functions. 
\\[1mm]
4.   {\sc $\tSAP$-Hamiltonians.}
Thanks to the non-resonance conditions \eqref{verad0}  we eliminate 
the Hamiltonian monomials which do not Poisson commute with the super-actions \eqref{superacpin}, cfr. Lemma \ref{lemma:Poissonbra}. The remaining monomials, 
 which we call
{\it super-action-preserving} ($\tSAP$) (Definition \ref{sapham}),   
have 
either the 
{	\it integrable}  form $ |z_{j_1}|^2 \ldots |z_{j_m}|^2 $ 
or the form 
\be\label{nonintm}
z_{j_1}\overline{z_{-j_1}} \ldots z_{j_m}\overline{z_{-j_m}} \times \text{integrable monomial} \, 
\ee
(with not necessarily distinct  indexes $j_1, \ldots, j_m$).
The 
not integrable monomials \eqref{nonintm} allow
an exchange of energy between the Fourier modes $ \{ z_{j_a}, $ $ z_{-j_a} \} $, $ a = 1, \ldots, m $, 
but, thanks to the Hamiltonian structure, 
each super-action $ |z_{j_a}|^2 + |z_{-j_a}|^2  $ remains constant in time.

We may {\it not} expect to get an  integrable  Hamiltonian Birkhoff 
 normal form for the water waves equations \eqref{eq:etapsi} starting 
from the  degree of homogeneity $ 8 $. 
Actually, using the 
conservation of momentum, 
the fourth order Hamiltonian Birkhoff normal form 
is  integrable, see Remark \ref{rem:Nwave}.  
The same holds if $ \gamma \neq 0 $ also at degree $ 6 $. 
But 
 there are
$ 8$-wave resonant interactions corresponding to $\tSAP$ {\it not} integrable
monomials 
$$ 
z_{n_1}\overline{z_{-n_1}} z_{n_2}\overline{z_{-n_2}} z_{-n_3}\overline{z_{n_3}} z_{-n_4}\overline{z_{n_4}} 
$$ 
(which are momentum preserving if   $ n_1 + n_2 = n_3 + n_4 $)
for {\it any} $ \kappa > 0 $ and  {\it any} $ \gamma  \in \R $. Indeed,  
for any positive 
integer $ n_1, n_2, n_3, n_4 $ we have, if  $ \tth = + \infty $,  
\be\label{8waves}
\begin{aligned}
& \Omega_{n_1}(\kappa)- \Omega_{-n_1}(\kappa) + 
\Omega_{n_2}(\kappa)- \Omega_{-n_2}(\kappa)
+ \Omega_{-n_3}(\kappa)- \Omega_{n_3}(\kappa) + \Omega_{-n_4}(\kappa)- \Omega_{n_4} (\kappa) \\
& \stackrel{\eqref{omegonejin}} = 
\gamma ( \sign (n_1) +  \sign (n_2) - \sign (n_3) - \sign (n_4)) 
\equiv 0 \, .
\end{aligned}
\ee
The analytical difficulties  of the {\it loss of derivatives} 
caused 
by the quasi-linearity of the equations and the small divisors 
in \eqref{Nwavei} 
along the Birkhoff normal form reduction is overcome 
by preserving  the paradifferential structure of the equations.
The final outcome is that the water waves 
system in Hamiltonian Birkhoff normal form satisfies an  {\it energy estimate} of the form 
\be\label{energy.intro}
  {\| z(t)\|}^{2}_{\dot{H}^{s}}   \leq  {\| z(0) \|}^{2}_{\dot{H}^s} 
+  C(s)  \int_0^t {\| z(\tau)\|}^{N+3}_{\dot{H}^s} \, \di \tau \, . 
\ee
5.  {\sc Comparison with the approach in \cite{Del2,Del3} and \cite{BD}.}
The Hamiltonian approach to paradifferential calculus in 
 \cite{Del2,Del3} is developed for quasi-linear 
 Klein-Gordon equations and  can not be applied to prove Theorem \ref{teo1}. 
Indeed,  since the Klein-Gordon dispersion relation  is asymptotically linear, 
it is not required a reduction 
to $ x $-independent paradifferential operators 
up to smoothing remainders:  
since the commutator between first order paradifferential operators is 
still  a first order paradifferential operator,
it is possible to implement a 
Hamiltonian Birkhoff normal form  reduction 
in degrees of homogeneity, in the same spirit of  semilinear PDEs.
This approach can not be applied 
for \eqref{eq:etapsi} since the dispersion relation \eqref{omegonejin}  is super-linear. 
It is for  this reason that we first reduce in Proposition \ref{prop:Z} 
 the paralinearized water waves equations to $ x$-independent 
symbols  up to smoothing remainders.
This was done in \cite{BD} for $ \gamma = 0 $ (in a different way)
but breaking the Hamiltonian structure (see  \cite{FI2} for NLS).
Incidentally we mention that the paradifferential normal form 
in \cite{BD} is not a Birkhoff normal form: 
for standing waves 
it is not needed to reduce  the $ x $-independent symbols 
to deduce that the actions $ |u_n|^2 $ are prime integrals.

Summarizing, the proof of Theorem \ref{teo1} demands 
\begin{itemize}
\item a reduction of the water waves equations  \eqref{eq:etapsi}
to paradifferential $ x$-independent symbols up to smoothing remainders,
done in \cite{BD} for $ \gamma = 0 $ (in a different way) 
losing the Hamiltonian structure, and, additionally, 
reduce  the $ x $-independent symbols to super-action preserving Birkhoff 
normal form; 
\item  
preserve the Hamiltonian structure of the Birkhoff normal form,   goal 
achieved in \cite{Del2,Del3} but only for 
 Klein-Gordon equations. 
\end{itemize}
The resolution of these 
requirements is a
main achievement of this paper.  

\smallskip

 Before presenting   further
 ideas of the proof of 
 Theorem \ref{teo1}  
we  state the following byproduct of 
the  Darboux-type  Theorem \ref{thm:almost} concerning a symplectic 
version of the Alinhac good unknown. Such result  
may be of separate interest and use 
for water waves results in other contexts. 
\\[1mm]
{\bf Symplectic good unknown up to homogeneity $ N $.} 
The  celebrated Alazard, Burq, Zuily approach \cite{ABZduke,ABZ1,AlM}
to local well posedness extends Lannes \cite{Lannes}   
introducing  the nonlinear, {\it not} symplectic, Alinach good unknown 
$$
\omega := \psi - \Opbw{B(\eta,\psi)} \eta \qquad 
\text{where} \qquad B (\eta,\psi) := (\Phi_y)(x,y)_{|y = \eta (x)}
$$
and  $ \Phi  $ is the generalized harmonic velocity potential in \eqref{def:DN} (the notation $ \Opbw{ \cdot } $ refers to a paradifferential operator in the Weyl quantization, according to Definition \ref{quantizationtotale}).
The nonlinear map
\be\label{simgud}
 { {\cal G}_{A} \begin{pmatrix}
  \eta \\
\psi 
  \end{pmatrix} = 
\opbw{  \begin{pmatrix}
  1 & 0 \\
  -B(\eta, \psi) & 1 
  \end{pmatrix} }   \begin{pmatrix}
  \eta \\
\psi 
  \end{pmatrix} \, ,}
\ee 
although  not symplectic,   is linearly symplectic, namely
\be\label{guls}
{ \opbw{  \begin{pmatrix}
  1 & 0 \\
  -B(\eta, \psi) & 1 
  \end{pmatrix} }^\top \,   E_0  \,
  \opbw{  \begin{pmatrix}
  1 & 0 \\
 -  B(\eta, \psi) & 1 
  \end{pmatrix} }
  }
   = E_0 \qquad \text{where} \qquad  
E_0 := \begin{pmatrix}
0 & -{\rm Id} \\
{\rm 
Id} & 0
\end{pmatrix} \, . 
\ee
A direct corollary of  Theorem \ref{thm:almost}  is the following result, 
proved at the end of Section \ref{sec:Darboux}.
We refer to Definition \ref{omosmoothing}
for the precise definition of smoothing operators.

\begin{theorem} {\bf (Symplectic good unknown up to homogeneity $ N $)}\label{thm:symplgu}
Let $ N \in \N $. 
There exists a
 pluri-homogeneous matrix of  real smoothing operators 
 $ R_{\leq N}(\cdot) $ in $ \Sigma_1^N \wt \cR_q^{-\varrho} 
 \otimes \mathcal{M}_2(\mathbb{\C}) $ for any $ \varrho \geq 0 $ such that 
\be\label{smooGU}
\big( \uno + R_{\leq N}(\cdot) \big) \circ    { {\cal G}_{A} } (\eta, \psi) 
\ee
 is symplectic up to homogeneity $ N $, according to \eqref{real:sym:boh}.
 \end{theorem}

Let us make some comments about Theorem \ref{thm:symplgu} and the 
more general Theorem \ref{thm:almost}. 
\\[1mm]
1.
The pluri-homogeneous smoothing correcting operators 
$ R_{\leq N} (\cdot ) $ in \eqref{smooGU} 
are constructed in Proposition \ref{DarB} by a  Darboux 
deformation argument \`a la Moser.
More precisely the $ R_{\leq N} (\cdot ) $ are defined 
as 
approximate inverses 
of approximate flows, up to homogeneity $ N $, generated by 
smoothing vector fields, which are algorithmically 
determined 
 by the Darboux mechanism 
and depend only on the
pluri-homogeneous components of $ \cG  $ up to degree $ N $
(more in general of  $ \bB_{\leq N} $ in \eqref{PsiN}).  
%
\\[1mm]
2. The Alinhac good unknown map \eqref{simgud} is bounded, but
Theorem \ref{thm:almost}
also holds for a (spectrally localized)  map $ \bB_{\leq N}(U) $ in \eqref{PsiN} which is 
{\it unbounded}.
This is the case for example 
when $ \bB_{\leq N} (U) $ is the Taylor expansion of a linear flow 
generated by an unbounded operator, as we discuss later in \eqref{linflowin}. 
\\[1mm]
3.
We do not expect to find in Theorem \ref{thm:almost} a corrector which produces  
a completely symplectic transformation of the  phase space,
but only up to an  arbitrary  degree of homogeneity $ N $. 
In the  Darboux  approach of Section \ref{sec:Darboux} 
this is because 
the equation \eqref{fin} 
for the smoothing vector field $ Y^\tau $, 
whose flow defines the symplectic corrector, 
can be solved only in homogeneity
having  the form $ E_c Y^\tau = R(V, Y^\tau, \di_V Y^\tau )$ 
and so losing derivatives, see Remark \ref{rem:dY}.
We remark that also 
the transformations 
in \cite{Del2,Del3}  are symplectic  at  degree of homogeneity $ \leq N $. 
A similar problem appears in  \cite{FI3}.
\\[1mm]
4.
Darboux perturbative methods for Hamiltonian 
PDEs have been developed in different contexts in \cite{Bam2,BaM0,BaM1,Cuc1,Cuc2,KukPer}.
 In all these cases, the perturbed symplectic form is a smoothing perturbation of the standard one and thus Darboux correctors are symplectic maps.
On the other hand, 
 in this work 
the perturbed   symplectic tensor  is a (possibly) {\it unbounded
 perturbation} of the standard one, 
\be\label{eq:ecnos}
 E_{{\leq N}}(V) 
 =  E_c + \text{(possibly) unbounded 
 operator} \, . 
\ee
A key tool to overcome this difficulty is the structural Lemma  \ref{diadg}. 
\\[1mm]
5.
A symplectic map up to homogeneity $ N $, 
transforms a Hamiltonian system up to homogeneity $N$ into 
another Hamiltonian system up to homogeneity $N$, see Lemma 
\ref{conj.ham.N}. 

\smallskip

\subsection*{Further ideas of proof and plan of the paper}
The paper is divided in 
\begin{enumerate}
\item 
Part \ref{part:I}) containing  the abstract functional setting used along the paper 
and the 
Darboux result; 
\item Part \ref{part:II}) with the 
proof of  the almost global in time Theorem \ref{teo1}. 
 \end{enumerate}
We first illustrate 
the way we proceed to preserve the Hamiltonian structure, up to homogeneity $ N $,   
in a generic transformation step along the proof of Theorem \ref{teo1}. 
\\[1mm]
{\bf Symplectic conjugation step up to homogeneity $ N $.} Consider a real-to-real system in paradifferential form 
\be\label{equaz1}
\pa_t U = X(U)= \Opbw{A(U; t,x,\xi)}[U] + R(U; t)[U] \ , \quad U = \begin{bmatrix}u \\ \bar u \end{bmatrix} \ , 
\ee
where $ A(U; t,x,\xi) $ is a matrix of symbols and $ R(U; t) $ are $ \varrho$-smoothing operators, 
which admit a homogeneous expansion up to homogeneity $N$, whereas the terms with
  homogeneity $>N$  are dealt,  as in \cite{BD}, as time dependent symbols and remainders, see Section \ref{sec:para}.
 This is quite convenient from a 
technical point of view because it does not demand much information about the higher degree terms. 
Moreover this 
enables to directly use the  paralinearization of the 
Dirichlet-Neumann operator 
proved in \cite{BD}.
System  \eqref{equaz1} is Hamiltonian up to homogeneity $N$, namely 
 the homogeneous components  of the vector field $X(U)$ of degree $ \leq N+1 $ 
 have the Hamiltonian  form 
 \be\label{JcH.intro}
J_c \nabla H(U)  \qquad \text{where} \qquad  J_c := \begin{bmatrix}0 & - \im \\ \im  & 0  \end{bmatrix}  
\ee
is the Poisson tensor and  $H(U)$ is a real valued 
 pluri-homogeneous Hamiltonian of degree $\leq N+2$. 
Moreover  the paradifferential operator 
$ \Opbw{ A(U)} $  in \eqref{equaz1} is a linear  Hamiltonian operator, up to homogeneity 
$ N $, namely of the form $  \Opbw{ A(U)} = J_c  \Opbw{ B(U)} $ where 
$ B(U) $ is a symmetric operator up to homogeneity 
$ N $, see Definition \ref{def:LHN}.  
\smallskip

In order to prove energy estimates for \eqref{equaz1} 
we transform it under  several changes of variables. 
Actually we do not really perform changes of variables of the phase space, 
but we proceed in the time dependent setting 
due to the high homogeneity terms. 
Let us discuss a typical transformation step. 
Let  ${\cal G}(U;t) := {\cal G}^{\tau} (U;t){|_{\tau = 1}} $ be the 
time $ 1$-flow 
\be\label{linflowin}
\pa_\tau\mG^\tau(U;t)=
J_c  \, \opbw\big(B(U; \tau, t, x, \xi)\big)
  \mG^\tau(U;t) \, , \quad 
\mG^0(U;t)={\rm Id} \, , 
\ee
generated by a linearly Hamiltonian operator  $J_c  \, \opbw\big(B(U; \tau, t, x, \xi)\big)
 $ up to homogeneity $ N $. 
  The transformation $\cG(U;t)$ is invertible and bounded 
  on $\dot  H^s(\T)  \times \dot H^s(\T)$ 
  for any $ s \in \R $
 and   it admits a pluri-homogeneous expansion $ {\cal G}_{\leq N}(U) $, which is 
 an unbounded operator if the generator
 $ J_c  \, \opbw\big(B\big) $ is unbounded, see Section \ref{sec:SLF}. 
  If $ U $ solves \eqref{equaz1} then the variable 
\be\label{intro.G}
W:= {\cal G} (U;t)U  
\ee
solves 
a new system in paradifferential form 
\be\label{equaz2}
 \pa_t W = X_+ (W) = \Opbw{A_+ (W; t,x,\xi)}[W] + R_+ (W; t)[W] 
 \ee
 (actually 
 the symbols and remainders of  homogeneity $ > N $ in \eqref{equaz2} 
 are still expressed in terms of $ U $, but for simplicity we skip to discuss  this issue here). 
In 
Section \ref{sec:redu}  we perform several transformations of this kind,
choosing suitable generators 
$ J_c  \, \opbw\big(B\big) $ (either bounded or unbounded) in order to  obtain a 
diagonal matrix  $ A_+ $
with $x$-independent  symbols.   

We remark that, with this procedure, 
since the time one flow map $ {\cal G}(U;t) $ of the linear Hamiltonian system 
\eqref{linflowin} is  only {\em  linearly} symplectic up to homogeneity $ N $, namely
$$
{\cal G(U;t)}^\top E_c {\cal G}(U;t) = E_c + E_{> N} (U;t) \, , \quad E_{> N} (U;t) = O(\|U\|^{N+1}) \, , 
$$
where $E_c := J_c^{-1}$ is the standard symplectic tensor,  
the new system \eqref{equaz2} {\em  is not} Hamiltonian anymore, 
not even its pluri-homogeneous components of degree $ \leq N + 1 $. 
The new system \eqref{equaz2} is only linearly symplectic, up to homogeneity $ N $, 
see  Lemma \ref{coniugazionehamiltoniana}.  
In order to obtain a new 
Hamiltonian system
up to homogeneity $ N $,  we use the Darboux  results of Section \ref{sec:Darboux}
to construct perturbatively a ``symplectic corrector" of the transformation 
\eqref{intro.G}. 

Let us say some words about the 
construction of the  symplectic corrector.  
We remark that the 
 perturbed  symplectic tensor  
$ E_{{\leq N}}(V) $ induced by the non-symplectic transformation $ {\cal G}_{\leq N} (U) $ is 
{\it not} 
a smoothing perturbation of the standard Poisson tensor $ E_c $, cfr. \eqref{eq:ecnos}. 
However, Lemmata \ref{Lem:ZB} and \ref{diadg}
prove that, for any pluri-homogeneous vector field    $X(V)$, we have 
$$
E_{{\leq N}}(V) [X(V)] =  E_c X(V)+ \grad \cW(V) +  \text{ smoothing vector fields} 
+  \text{high homogeneity terms}  
$$
where $ \cW (V) $ is a scalar function.
This algebraic structural property enables to prove the Darboux 
Proposition \ref{DarB}, thus Theorem \ref{thm:almost}, 
via a deformation argument \`a la  Moser. We also remark that 
the operators $ R_{\leq N} (\cdot ) $ of Theorem \ref{thm:almost} 
are smoothing for arbitrary  $ \varrho \geq 0 $, since  
they have  $ 2 $ equivalent frequencies, namely 
$ {\rm \max}_2( n_1,\ldots, n_{p+1}) \sim \max( n_1,\ldots, n_{p+1}) $ 
in \eqref{hom:rest}, 
arising by applications of Lemma \ref{aggiunto}. This
property compensates  the presence 
of  unbounded operators in 
$ {\cal G}_{\leq N} (U) $.

In conclusion, Theorem \ref{thm:almost} provides a nonlinear 
map 
$ W + R_{\leq N}(W)W $,
where $R_{\leq N}(W) $ are pluri-homogeneous
 $ \varrho $-smoothing 
operators for arbitrary  $ \varrho > 0 $, 
such that the pluri-homogeneous map 
$$
\cD_N (U) := (\id + R_{\leq N}(\cdot ))\circ   \cG_{\leq N}(U) U
$$
is symplectic up to homogeneity $N$, i.e. 
 \be\label{sympl-approx}
\big[ \di_U \cD_{N} (U) \big]^\top E_c \, \big[ \di_U \cD_{N} (U) \big] = E_c + E_{> N} (U)  
\ee
where $ E_{> N} (U) $ is an operator of homogeneity degree $ \geq N +1 $.
As a consequence, since \eqref{equaz1} is Hamiltonian up to homogeneity $ N $, 
the variable
$$
Z(t) := \cD(U(t) ;t) := W(t)  + R_{\leq N}(W(t) )  = (\id + R_{\leq N}(\cdot ))\circ   \cG(U(t) ;t) U(t) 
$$
satisfies a system which  is {\em Hamiltonian up to homogeneity $N$} as well, 
and which has, since $ R_{\leq N} ( \cdot ) $ are smoothing operators, 
the same paradifferential form as in \eqref{equaz2}, 
\be\label{equaz3}
 \pa_t Z = X_{++} (Z)  = \Opbw{A_{++} (Z; t,x,\xi)}[Z] + R_{++} (Z; t)[Z] \, .  
 \ee 
This is the content of Theorem \ref{conjham}. 
Note that the matrix of symbols $ A_{++}(Z; t,x,\xi) $ in \eqref{equaz3}
is obtained by substituting in  $ A_+ (W; t,x,\xi)  $ 
the relation $ W = Z - R_{\leq N}(Z) + \ldots $ obtained  inverting 
$ Z = W + R_{\leq N}(W) $ approximately up to homogeneity $ N $. 
This procedure is rigorously justified 
in 
Lemmata \ref{lem:conj.fou} and \ref{NormFormMap0}.

\smallskip

\noindent
{\bf Scheme of proof of Theorem \ref{teo1}.} In part \ref{part:II} 
we apply the abstract formalism developed in part \ref{part:I} to prove Theorem \ref{teo1}.
We proceed as follows.  
\\[1mm]
{\it Section \ref{sec:paraWW}: paralinearization of 
 the water waves equations}.  
In Section \ref{sec:paraWW}  
we first paralinearize the water waves equations 
\eqref{eq:etapsi}, we introduce the Wahl\'en variables 
$ (\eta, \zeta) $ in \eqref{Whalen} and 
the complex variable $ U $ in 
\eqref{defMM-1} 
which diagonalizes the linearized equations at 
zero.   The resulting paralinearized equations \eqref{complexo} are a  Hamiltonian system of the form 
\be\label{JcH.intro2}
\pa_t U = J_c \nabla H_\gamma (U)  
\ee
where $ H_\gamma $ is the 
Hamiltonian  in \eqref{H.gamma} written in the variable $ U $. 
Our goal is to perform several changes of variable to prove energy estimates for 
\eqref{complexo}, i.e. \eqref{JcH.intro2}, valid up to times of  order $ \e^{-N-1} $. 
We split the proof in two major steps. 
\\[1mm]
{\it Section \ref{sec:redu}: Hamiltonian paradifferential normal form.}
In Section \ref{secAl1} we introduce  the good unknown 
of Alinhac $ {\cal G(U)} $   (written in complex coordinates), obtaining a system which has  energy estimates for times of order $\varepsilon^{-1}$. 
The Alinhac good unknown is 
{\it not} symplectic and therefore the transformed 
system \eqref{sisV0} is not Hamiltonian anymore.  
Next 
we transform \eqref{sisV0} 
into a diagonal matrix of 
$ x$-independent 
 symbols  up to smoothing remainders, 
in order to compensate along the Birkhoff normal form reduction process 
the loss of derivatives due to the small divisors and  the 
quasi-linear nature  of the water waves equations, see Proposition \ref{teoredu1}.   
The resulting system 
\be\label{st1nfr}
\pa_tW = \vOpbw{\ii \tm_{\frac32}(U;t,\xi)}W 
 + R(U;t)W 
\ee 
 is {\em no longer Hamiltonian}. 
 In \eqref{st1nfr}  the imaginary part of the symbol $ \tm_{\frac32} $ has order zero and 
 homogeneity larger  than $ N $, whereas    
$ R(U;t)  $ is a  smoothing remainder vanishing linearly in $ U $. 
\\[1mm]
{\it Section \ref{sec:birk}: Hamiltonian Birkhoff normal form.}
In order to recover 
the  Hamiltonian structure 
we apply the symplectic corrector given by
Theorem \ref{thm:almost}: using Theorem \ref{conjham} and Lemmata \ref{lem:conj.fou} and \ref{NormFormMap0},
we obtain in Proposition \ref{prop:Z}  system \eqref{teo62} which is 
Hamiltonian  up to homogeneity $ N $. 
We perform the Hamiltonian Birkhoff normal form reduction
for any value of the
 surface tension 
 $ \kappa $ outside the set $ {\cal K}$ defined in Theorem 
\ref{nonresfin0}.
 Iteratively we first  reduce 
the $p$--homogeneous $ x$-independent paradifferential symbol
to its super-action-preserving component, via the linear flow generated by 
an unbounded Fourier multiplier, see \eqref{W.bir}. 
Since such transformation is only linearly symplectic,
we apply again  
Theorem \ref{conjham} to recover  
a Hamiltonian system up to homogeneity $ N $, see system \eqref{eq:W:Bir4}. 
Finally we reduce the  $ (p+1)$-homogeneous component of the Hamiltonian 
smoothing vector field to its super-action preserving part, see 
\eqref{eq:W:Bir7}. The key property is that 
a super-action preserving Hamiltonian Poisson commutes with
the super-actions 
defined 
in \eqref{superacpin}. 
After $ N $ iterations, the final outcome is  the 
Hamiltonian Birkhoff normal form system  \eqref{final:eq}, which has
 the form 
\be\label{finalZin}
\pa_tZ  = J_c \nabla H^{(\tSAP)}(Z) 
+ \vOpbw{- \ii (\tm_{\frac32})_{>N}(U;t,\xi)}Z+ R_{>N}(U;t)U 
\ee
where $ H^{(\tSAP)}(Z) $ is a super-action preserving Hamiltonian (Definition \ref{sapham})
and the higher order homogeneity paradifferential and smoothing terms 
admit energy estimates in Sobolev spaces
(the imaginary part of 
the symbol $ (\tm_{\frac32})_{>N} $ has  order zero). 
\\[1mm]
{\it Section \ref{sec:EE}: energy estimates.}
The  Hamiltonian Birkhoff normal form 
equation $\pa_tZ  = J_c \nabla H^{(\tSAP)}(Z) $
obtained  neglecting the terms of homogeneity larger than $ N $ in 
\eqref{finalZin} 
possesses 
the super--actions  $  |z_{-n}|^2 + |z_n|^2 $,  
for any $  n\in \N $, 
as  prime integrals. Thus it preserves the Sobolev norms 
and the solutions of 
\eqref{finalZin} with  initial data of size $ \e $ 
have energy estimates up to times 
of order $  \e^{-N-1}$.
In conclusion, since the Sobolev norms of  $ U $ in  \eqref{JcH.intro2}
and $ Z $ in \eqref{finalZin} are equivalent,  
we deduce 
 energy estimates for \eqref{JcH.intro2},
 $$
{\|U(t)\|}_{\dot{H}^{s}}^{2} 
\lesssim_{s,K} {\|U(0)\|}^{2}_{\dot{H}^{s}} +
 \int_{0}^{t} \|U(\tau)\|^{N+3}_{\dot{H}^{s}} \, \di \tau 
 $$
  valid up to times of  order $  \e^{-N-1} $. A standard bootstrap argument concludes the proof of Theorem \ref{teo1}.


\smallskip

 \noindent {\bf Notation:} The notation $ A \lesssim B $ means that there exists  a constant $ C \geq 0 $ such that  $ A\leq C B $. We  denote $ \N = \{1, 2, \ldots \} $ and 
 $\N_0 := \N \cup \{0\}$.

 \part{Abstract setting and Darboux symplectic corrector}\label{part:I}

\section{Functional Setting}

This section contains the abstract 
 functional setting 
used along the paper. 
In  Section \ref{sec:para} we present  definitions and results about para-differential calculus following Berti-Delort \cite{BD}, but   defining 
the different  notion of $ m$-operators.  
Using the same classes of symbols and smoothing operators of 
\cite{BD} 
has  the advantage to directly rely on the result in \cite{BD} concerning the 
paralinearization of the 
Dirichlet-Neumann operator 
with multilinear expansions. 
In Section \ref{sec:SL} we introduce the notion of spectrally localized maps which includes, 
as a particular case, 
 paradifferential operators of any order.   
Then we prove several properties of spectrally localized maps among which 
that the transpose of the differential of a homogeneous spectrally localized map is smoothing, 
see  Lemma \ref{aggiunto}. 
This result generalizes a lemma which 
has been  proved in Feola-Iandoli \cite{FI3} for paradifferential operators, and it 
is relevant for producing the Hamiltonian corrections to the homogeneous components of the vector field in Section \ref{sec:Darboux}, by means of a Darboux approximate procedure. 
In Section \ref{sec:app} we  construct approximate inverses of non--linear maps and approximate
 flows  up to an arbitrary degree of homogeneity.
In Section \ref{SGeo} we introduce the 
formalism of  pluri-homogeneous  $k$-forms, Lie derivatives  and Cartan's magic formula. 
Let us first fix some notation  used along the paper. 

\paragraph{Function spaces.} Along the paper we deal with real parameters 
\be\label{varipar}
s\geq s_0 \gg K \gg \vr \gg N  
\ee
where $ N \in \N_0 $ is the constant in Theorem \ref{teo1}. 

Given an interval $ I\subset \R$ symmetric with respect to $ t = 0 $ 
and $s\in \R$, we define the space 
$$
C_*^K(I,\dot{H}^s\left(\mathbb{T},\mathbb{C}^2)\right) := 
\bigcap_{k=0}^K C^k\big(I, \dot{H}^{s- \frac32 k}(\mathbb{T},\mathbb{C}^2)\big)
$$ 
endowed with the norm 
\be \label{Knorm}
\sup_{t\in I} \| U(t, \cdot)\|_{K,s} \qquad 
{\rm where} \qquad 
\| U(t, \cdot)\|_{K,s}:= \sum_{k=0}^K \| \partial_t^kU(t, \cdot)\|_{\dot{H}^{s- \frac32k}} \, ,
\ee
and we also consider  its subspace 
$$
C_{*\R}^K(I,\dot{H}^s\left(\mathbb{T},\mathbb{C}^2)\right)
:= \Big\{ U\in C_*^K(I,\dot{H}^s\left(\mathbb{T},\mathbb{C}^2)\right) \ : \ 
U=\begin{pmatrix} u \\ \bar{u}\end{pmatrix}\Big\} \, .
$$  
Given $r>0$ we set $B^K_s(I;r)$ the ball of radius $r$ in $C_*^K(I,\dot{H}^s\left(\mathbb{T},\mathbb{C}^2)\right)$ and by 
 $B^K_{s,\R}(I;r)$ the ball of radius $r$ in $C_{*\R}^K(I,\dot{H}^s\left(\mathbb{T},\mathbb{C}^2)\right)$. 

The parameter $ s $ in \eqref{Knorm} denotes the spatial Sobolev regularity of the solution $ U(t, \cdot) $
and $ K $  its regularity in the time variable. 
The gravity-capillary
water waves vector field  
 loses $ 3/2$-derivatives, 
and therefore,  differentiating the solution $ U(t) $ for  $k$-times in the time variable,
there is a loss of $ \tfrac32 k$-spatial derivatives.   
The parameter $\vr$ in \eqref{varipar} denotes the order where we decide to stop our regularization of the system and depends on  the number $ N$ of steps of Birkhoff normal form  that we will perform and the smallness of the  small divisors due to the resonances. 

We  denote $\dot L^2(\T,\C):= \dot H^0(\T,\C)$ and $ \dot L^2_r:=\dot L^2(\T,\R)= \dot H^0(\T,\R)$ the subspace of $\dot L^2(\T,\C)$ made by real valued functions.
Given $u, v \in \dot L^2(\T,\C)$ we define
\be\label{scpr12hom}
\la  u,  v \ra_{\dot L^2_r}:= \int_\T \Pi_0^\bot u(x)\, \Pi_0^\perp v(x) \, \di x  \, , \quad \text{respectively} \quad 
\la u, v \ra_{\dot L^2} := \int_\T \Pi_0^\bot u(x)\, \bar { \Pi_0^\bot v(x)} \, \di x \, ,
\ee
where $\Pi_0^\perp u := u - \frac{1}{2 \pi} \int u(x) \, \di x $ is the  projector onto the 
zero mean functions. 

We also consider the non-degenerate bilinear form on $\dot L^2(\T;\C^2)$
\be\label{real.bil.form.intro}
\Big\la  \vect{v_1^+}{  v_1^-}, \vect{v_2^+}{ v_2^-} \Big\ra_r:= \la v_1^+, v_2^+ \ra_{\dot L^2_r} +\la v_1^-, v_2^- \ra_{\dot L^2_r} \, .
\ee
{\bf Fourier expansions.} 
Given a $2\pi$-periodic function $u(x)$ in the homogeneous  space $ \dot L^2 (\T,\C)$, we identify 
$ u(x) $ with its zero average  representative and we expand it in Fourier series as 
\begin{equation}\label{fourierseries}
u(x)= \sum_{j \in \mathbb{Z}\setminus\{0\}} \hat{u}(j) \frac{e^{\im j x}}{\sqrt{2\pi}}, \quad \hat{u}(j):= \frac{1}{\sqrt{2\pi}}\int_{\mathbb{T}}u(x) e^{- \im j x }\,\di x \, .
\end{equation}
We shall expand a function $ \vect{u^+}{u^-}$ as 
\be \label{u+-}
\vect{u^+}{u^-}= \sum_{\sigma\in \pm} \sum_{j \in \Z\setminus\{0\}} \tq^\sigma u_j^\sigma \frac{e^{\ii\sigma j x}}{\sqrt{2\pi}}, \quad   u^\sigma_j:= \hat{u^\sigma}(\sigma j)=\frac{1}{\sqrt{2\pi}}\int_{\mathbb{T}}u^\sigma(x) e^{- \im \sigma j x }\,\di x 
\ee
 where
 \be \label{tiqu}
 \tq^+:= \vect{1}{0}, \quad \tq^-:= \vect{0}{1} \, .
 \ee
For $n\in \mathbb{N}$ we denote by $\Pi_n$ the orthogonal projector from $L^2(\mathbb{T},\mathbb{C})$ to the linear subspace spanned by
 $\{ e^{\im nx}, e^{-\im nx}\}$, 
 \be\label{Pin2}
  (\Pi_n u)(x) := 
  \hat{u}(n) \frac{e^{\im nx}}{\sqrt{2\pi}}+ \hat{u}(-n) \frac{e^{-\im nx}}{\sqrt{2\pi}} \, , 
\ee
 and we denote by $ \Pi_n $ also the corresponding projector in $ L^2 (\T, \C^2 ) $. 
 
 If $ \cU=(U_1, \dots , U_p)$ is a $p$-tuple of functions and $\vec{n}=(n_1,\dots,n_p)\in \mathbb{N}^p $, we set 
 $$
 \Pi_{\vec{n}} \cU := \big( \Pi_{n_1}U_1,\dots,\Pi_{n_p}U_p \big) \, , 
 \qquad
 \tau_\varsigma \cU := \big(  \tau_\varsigma U_1,\dots,  \tau_\varsigma U_p \big)  \, . 
 $$
For $ \vec{\jmath}_p = (j_1,\dots,j_p) \in (\Z\setminus\{0\})^p$ 
and $\vec{\sigma}_p = (\sigma_1,\dots, \sigma_p)\in \{\pm\}^p$ we  denote 
$ |\vec{\jmath}_p | := \max(|j_1|, \ldots, |j_p| ) $ and 
\be\label{notationuvecjvecs}
u_{\vec{\jmath}_p}^{\vec{\sigma}_p}:= u_{j_1}^{\sigma_1}\dots u_{j_p}^{\sigma_p} \, , 
\qquad \vec{\sigma}_p \cdot \vec{\jmath}_p := \sigma_1 j_1 + \ldots + \sigma_p j_p \, .
\ee 
 \noindent
 Note that, under the translation operator $\tau_\varsigma$ defined in \eqref{X.tra0}, the Fourier coefficients of $\tau_\varsigma u$ transform as
 $$ 
 (\tau_\varsigma u)^{\sigma}_j  = e^{\im \sigma j \varsigma} u^\sigma_j \, . 
 $$
 We finally denote
\begin{equation}
\label{fTset}
\fT_p := \Big\{ (\vec \jmath_p, \vec \sigma_p) \in (\Z\setminus \{0\})^{p} \times \{\pm \}^{p} \colon \ \ 
\vec \sigma_p\cdot \vec \jmath_p = 0
\Big\} \, .
\end{equation}
\noindent{\bf Real-to-real operators and vector fields.} 
 Given a linear operator  $ R(U) [ \cdot ]$ acting on $\dot L^2(\T;\C)$ 
we associate the linear  operator  defined by the relation 
\begin{equation}\label{opeBarrato}
\ov{R}(U)[v] := \ov{R(U)[\ov{v}]} \, ,   \quad \forall v: \T \rightarrow \C \, .
\end{equation}
An operator $R(U)$ is {\em real } if $R(U) = \bar R (U)$. 
We say that a matrix of operators acting on $\dot L^2(\T;\C^2)$  is \emph{real-to-real}, if it has the form 
\begin{equation}\label{vinello}
R(U) =
\left(\begin{matrix} R_{1}(U) & R_{2}(U) \\
\ov{R_{2}}(U) & \ov{R_{1}}(U)
\end{matrix}
\right) \, , 
\end{equation}
for any $   U $ in 
\be\label{L2rC2}
\dot L^2_\R (\T, \C^2) := \Big\{ V \in \dot L^2 (\T, \C^2) \, : \, V=\vect{v}{\ov{v}}  \Big\}  \, .
\ee
We define similarly $ \dot H^s_\R (\T, \C^2) $. 
A real-to-real matrix of operators $R(U)$  acts in the subspace   
$ \dot L^2_\R (\T, \C^2) $.

If $R_1(U)$ and $R_2(U)$ are  real-to-real operators then also $R_1(U)\circ R_2(U)$ is real-to-real.

Similarly we will say that a vector field  
\be\label{rtr}
X(U): =  \vect{X(U)^+}{X(U)^-}  \quad \text{is real-to-real if} \quad
\bar{X(U)^+}=X(U)^- \, , \quad \forall  U \in \dot L^2_\R (\T, \C^2) \, .  
\ee

\subsection{Paradifferential calculus}\label{sec:para}
 
We first introduce 
the paradifferential  operators (Definition \ref{quantizationtotale}) following \cite{BD}.
Then we define the new class of 
$ m $-Operators (Definition \ref{Def:Maps}) that, for $ m \leq 0  $, are  the 
smoothing ones (Definition \ref{omosmoothing}), and we prove properties of 
$ m$-operators under transposition and composition.

\paragraph{Classes of symbols.}
We give the definition of the classes of symbols that we use. Roughly speaking the class $\wt{\Gamma}_p^m$ contains symbols of order $m$ and homogeneity $p$ in $U$, whereas the class $\Gamma_{K,K',p}^m$ contains non-homogeneous symbols of order $m$ that vanishes at degree at least $p$ in $U$ and that are $(K-K')$-times differentiable in $t$; we can think the parameter $K'$ like the number of time derivatives of $U$ that are contained in the symbols. 
In the following we denote $ \dot{H}^{\infty}(\mathbb{T};\mathbb{C}^2)
:= \bigcap_{s \in \R} \dot{H}^{s}(\mathbb{T};\mathbb{C}^2)$.

\begin{definition}
Let $m\in \R$, $p,N\in \N_0 $, 
$ K, K' \in \N_0 $ with $ K' \leq K  $, and $ r>0$.

($i$) $p$-{\bf homogeneous symbols.} We denote by $\wt{\Gamma}^m_p$ the space of symmetric $p$-linear maps from $ (\dot{H}^{\infty}(\mathbb{T};\mathbb{C}^2))^p$ to the space of $ C^\infty $ functions from $\mathbb{T}\times \R$ to $\mathbb{C}$, 
$ (x, \xi) \mapsto a(\cU;x,\xi)$,  satisfying the following: there exist $\mu>0$ and, for any $\alpha, \beta\in \N_0$, 
there is a constant $C>0$ such that 
\begin{equation}\label{homosymbo}
|\partial_x^{\alpha}\partial_{\xi}^{\beta}  
a( \Pi_{\vec n} \cU;x,\xi)|
\leq 
C |\vec{n}|^{\mu+\alpha} \langle \xi \rangle^{m-\beta} 
\prod_{j=1}^p \| \Pi_{n_j} U\|_{L^2}
\end{equation}
for any $ \cU = (U_1,\dots,U_p)\in ( \dot{H}^{\infty}(\mathbb{T};\mathbb{C}^2) )^p$ and $\vec{n}=(n_1,\dots,n_p)\in \mathbb{N}^p$. 
Moreover we assume that, if for some $(n_0,\dots, n_p)\in \N_0\times \N^p$, $\Pi_{n_0}a\left( \Pi_{n_1} U_1,\dots \Pi_{n_p}U_p;\cdot\right)\not=0$, then there is a choice of signs $\sigma_0,\dots,\sigma_p\in \{ -1,1\}$ such that $\sum_{j=0}^p \sigma_j n_j=0$.  In addition we require the translation invariance property
\begin{equation} \label{mome}
a\left( \tau_{\varsigma} \cU; x,\xi\right)= a\left( \cU; x+\varsigma, \xi\right),\quad \forall 
\varsigma\in \R \, , 
\end{equation}
where $\tau_\varsigma$ is the translation operator in \eqref{X.tra0}.

For $ p = 0 $ we denote by $\wt{\Gamma}^m_0 $ the space of constant coefficients symbols $ \xi \mapsto a(\xi) $ which satisfy \eqref{homosymbo} with $ \a = 0 $ and the right hand side replaced by $ C \la \xi \ra^{m - \beta} $. 

We denote by $\Sigma_p^N \widetilde \Gamma^{m}_q$ the class of pluri-homogeneous symbols $\sum_{q=p}^{N}a_{q}$ with $a_q \in  \widetilde{\Gamma}_{q}^m$.
For $ p \geq N + 1 $ we mean that the sum is empty. 

($ii$) {\bf Non-homogeneous symbols. }   We denote by $\Gamma_{K,K',p}^m[r]$ the space of functions  $ (U;t,x,\xi)\mapsto a(U;t,x,\xi) $, 
defined for $U\in B_{s_0}^{K'}(I;r)$ for some $s_0$ large enough, with complex values, such that for any $0\leq k\leq K-K'$, any $\sigma\geq s_0$, there are $C>0$, $0<r(\sigma)<r$ and for any $U\in B_{s_0}^K(I;r(\sigma))\cap C_{*}^{k+K'}(I, \dot{H}^{\sigma}(\mathbb{T};\mathbb{C}^2))$ and any $\alpha,\beta \in \N_0$, with $\alpha\leq \sigma-s_0$ one has the estimate
\begin{equation}\label{nonhomosymbo}
| \partial_t^k\partial_x^\alpha\partial_\xi^\beta a(U;t,x,\xi)| \leq C \langle \xi \rangle^{m-\beta} \| U\|_{k+K',s_0}^{p-1}\|U\|_{k+K',\sigma} \, .
\end{equation} 
If $ p = 0 $ the right hand side has to be replaced by $ C \langle \xi \rangle^{m-\beta} $. 

($iii$) {\bf Symbols.} We denote by $\Sigma \Gamma_{K,K',p}^m[r,N]$ the space of functions 
$ (U;t,x,\xi) \mapsto a(U;t,x,\xi), $
with complex values such that there are homogeneous symbols $a_q\in \wt{\Gamma}_q^m$, $q=p,\dots, N$ and a non-homogeneous symbol $a_{>N}\in \Gamma_{K,K',N+1}^m$ such that 
\be\label{espsymbol}
a(U;t,x,\xi)= \sum_{q=p}^{N} a_q(U,\dots,U;x,\xi) + a_{>N}(U;t,x,\xi) \, .
\ee
We denote by  $\Sigma \Gamma_{K,K',p}^m[r,N]\otimes \mathcal{M}_2(\mathbb{C})$ the space of $2\times 2$ matrices with entries in  $\Sigma \Gamma_{K,K',p}^m[r,N]$.

We say that a symbol  $a(U;t,x,\xi) $ is \emph{real} if it is real valued for any 
$ U \in B^{K'}_{s_0,\R}(I;r)$.
\end{definition}

\noindent
$ \bullet $  If $ a ( \cU; \cdot  )$ is a homogeneous 
symbol in $ \widetilde \Gamma_p^m $ then 
$ a (U, \ldots, U; \cdot ) $ belongs to  $\Gamma^m_{K,0,p} [r] $, for any $ r >0 $.

\noindent
$ \bullet $  If $a $ is a symbol in $ \Sigma \Gamma^m_{K,K',p}[r,N] $ 
then $ \partial_x a  \in \Sigma \Gamma^{m}_{K,K',p}[r,N]   $ and  
$ \partial_\xi a \in  \Sigma \Gamma^{m-1}_{K,K',p}[r,N]   $.
If in addition $
b $ is a symbol in $ \Sigma \Gamma^{m'}_{K,K',p'}[r,N]  $ then 
$a b \in \Sigma \Gamma^{m+m'}_{K,K',p+p'}[r,N]   $.

\begin{remark}\label{rem:symbol}
{\bf (Fourier representation of symbols)}  The translation invariance property \eqref{mome}
 means that the dependence with respect to the variable $x$ of a symbol 
$a(\cU;x,\xi)$  enters only through the functions $\cU(x)$, 
implying that a symbol 
$ a_q(U;x,\xi)$ in  $\wt{\Gamma}_q^m$, $ m\in \mathbb{R} $, has the form (recall notation \eqref{notationuvecjvecs})
\be\label{sviFou}
a_q(U;x,\xi)=  \!\!\!\!
\sum_{\vec \jmath \in {(\Z \setminus
 \{ 0 \}})^{q},\vec{\sigma}\in \{\pm 1\}^{^q}}  \! \! \! \! \!\!\!\!\!  \left( a_q\right)_{\vec \jmath}^{\vec{\sigma}}(\xi) 
 u_{\vec \jmath}^{\vec \sigma}
 e^{\im \vec{\sigma} \cdot \vec{\jmath} x}  
\ee
where $ (a_q)_{\vec \jmath}^{\vec{\sigma}}(\xi) \in \C $ are  Fourier multipliers of order $m$ satisfying: there exists $ \mu>0 $, and  
for any $ \beta \in \N_0 $, there is $ C_\beta > 0 $ such that 
\be\label{rem:symbol.1}
  | \pa_\xi^\beta\left( a_q\right)_{\vec \jmath}^{\vec{\sigma}}(\xi) |\leq C_\beta 
|   \vec \jmath \, |^\mu \langle \xi \rangle^{m-\beta} , 
\quad
\forall  (\vec \jmath, \vec \sigma) \in (\Z \setminus \{0\})^q \times \{ \pm \}^q  \, . 
\ee
A symbol 
$ a_q(U;x,\xi) $ as in \eqref{sviFou} is  real if 
\be\label{realsim}
\overline{\left( a_q\right)_{\vec \jmath}^{\vec{\sigma}}(\xi)} = 
\left( a_q\right)_{\vec \jmath}^{-\vec{\sigma}}(\xi) \, . 
\ee
By \eqref{sviFou} 
a symbol
$ a_{1} $ in  $\widetilde{\Gamma}_{1}^{m}$ 
can be written as
$ a_{1}(U;x,\x)= 
\sum_{ \substack{j \in  \Z \setminus \{0\}, \s =\pm } }
(a_{1})^{\s}_{j}(\x)
u_{j}^{\s}e^{\ii \s j x} $, and therefore, 
if  $ a_1 $ is 
\text{independent} of $x$, it  is actually $ a_1\equiv0$.
\end{remark}

We also define classes of functions in analogy with our classes of symbols.

\begin{definition}{\bf (Functions)} \label{apeape} Let $p, N \in \N_0 $,  
 $K,K'\in \N_0$ with $K'\leq K$, $r>0$.
We denote by $\widetilde{\mathcal{F}}_{p}$, resp. $\mathcal{F}_{K,K',p}[r]$,  $\Sigma\mathcal{F}_{K,K',p}[r,N]$, 
the subspace of $\widetilde{\Gamma}^{0}_{p}$, resp. $\Gamma^0_{K,K',p}[r]$, 
resp. $\Sigma\Gamma^{0}_{K,K',p}[r,N]$, 
made of those symbols which are independent of $\xi $.
We write $\widetilde{\mathcal{F}}^{\R}_{p}$, resp. $\mathcal{F}_{K,K',p}^{\R}[r]$, 
$\Sigma\mathcal{F}_{K,K',p}^{\R}[r,N]$,  to denote functions in $\widetilde{\mathcal{F}}_{p}$, 
resp. $\mathcal{F}_{K,K',p}[r]$,  $\Sigma\mathcal{F}_{K,K',p}[r,N]$, 
which are real valued for any $ U \in B^{K'}_{s_0,\R}(I;r)$.
\end{definition}

\paragraph{Paradifferential quantization.}
Given $p\in \N_0$ we consider   functions
  $\chi_{p}\in C^{\infty}(\R^{p}\times \R;\R)$ and $\chi\in C^{\infty}(\R\times\R;\R)$, 
  even with respect to each of their arguments, satisfying, for $0<\delta\ll 1$,
\begin{align*}
&{\rm{supp}}\, \chi_{p} \subset\{(\xi',\xi)\in\R^{p}\times\R; |\xi'|\leq\delta \langle\xi\rangle\} \, ,\qquad \chi_p (\xi',\xi)\equiv 1\,\,\, \rm{ for } \,\,\, |\xi'|\leq \delta \langle\xi\rangle / 2 \, ,
\\
&\rm{supp}\, \chi \subset\{(\xi',\xi)\in\R\times\R; |\xi'|\leq\delta \langle\xi\rangle\} \, ,\qquad \quad
 \chi(\xi',\xi) \equiv 1\,\,\, \rm{ for } \,\,\, |\xi'|\leq \delta   \langle\xi\rangle / 2 \, . 
\end{align*}
For $p=0$ we set $\chi_0\equiv1$. 
We assume moreover that 
$$ 
|\partial_{\xi}^{\alpha}\partial_{\xi'}^{\beta}\chi_p(\xi',\xi)|\leq C_{\alpha,\beta}\langle\xi\rangle^{-\alpha-|\beta|} \, , \  \forall \alpha\in \N_0, \,\beta\in\N_0^{p} \, ,  
\ \ 
|\partial_{\xi}^{\alpha}\partial_{\xi'}^{\beta}\chi(\xi',\xi)|\leq C_{\alpha,\beta}\langle\xi\rangle^{-\alpha-\beta}, \  \forall \alpha, \,\beta\in\N_0 \, .
$$ 
If $ a (x, \xi) $ is a smooth symbol 
we define its Weyl quantization  as the operator
acting on a
$ 2 \pi $-periodic function
$u(x)$ (written as in \eqref{fourierseries})
 as
$$
{\rm Op}^{W}(a)u=\frac{1}{\sqrt{2\pi}}\sum_{k\in \Z}
\Big(\sum_{j\in\Z}\hat{a}\big(k-j, \frac{k+j}{2}\big)\hat{u}(j) \Big)\frac{e^{\im k x}}{\sqrt{2\pi}}
$$
where $\hat{a}(k,\xi)$ is the $k^{th}-$Fourier coefficient of the $2\pi-$periodic function $x\mapsto a(x,\xi)$.

\begin{definition}{\bf (Bony-Weyl quantization)}\label{quantizationtotale}
If a is a symbol in $\widetilde{\Gamma}^{m}_{p}$, 
respectively in $\Gamma^{m}_{K,K',p}[r]$,
we set
\be\label{regula12}
\begin{aligned}
& a_{\chi_{p}}(\mathcal{U};x,\xi) := \sum_{\vec{n}\in \N^{p}}\chi_{p}\left(\vec{n},\xi \right)a(\Pi_{\vec{n}}\mathcal{U};x,\xi) \, , \\ 
& a_{\chi}(U;t,x,\xi) :=\frac{1}{2\pi}\int_{\mathbb{R}}  
\chi (\xi',\xi )\hat{a}(U;t,\xi',\xi)e^{\im \xi' x}\di \xi'  \, ,
\end{aligned} 
\ee
where in the last equality $  \hat a $ stands for the Fourier transform with respect to the $ x $ variable, and 
we define the \emph{Bony-Weyl} quantization of $ a $ as 
\be\label{BW}
\opbw(a(\mathcal{U};\cdot))= {\rm Op}^{W} (a_{\chi_{p}}(\mathcal{U};\cdot)) \, ,\qquad
\opbw(a(U;t,\cdot))= {\rm Op}^{W} (a_{\chi}(U;t,\cdot)) \, .
\ee
If  $a$ is a symbol in  $\Sigma\Gamma^{m}_{K,K',p}[r,N]$, 
we define its \emph{Bony-Weyl} quantization 
$$
\opbw(a(U;t,\cdot))=\sum_{q=p}^{N}\opbw(a_{q}(U,\ldots,U;\cdot))+\opbw(a_{>N}(U;t,\cdot)) \, . 
$$
We will use also the notation 
\be \label{vecop}
\vOpbw{a(U;t,x,\xi)}:= \Opbw{ \begin{bmatrix} a(U;t,x,\xi)&0\\0& \bar{ a^{\vee}(U;t,x,\xi)}\end{bmatrix}} \, , \quad a^{\vee}(x,\xi):= a(x,-\xi) \, .
\ee
\end{definition}

\noindent
$ \bullet $ The operator 
$ \opbw (a) $ acts on homogeneous spaces of functions, see Proposition 3.8  of \cite{BD}.

\noindent
$ \bullet $  If $ a$ is a homogeneous  symbol, the two definitions  of quantization in \eqref{BW} differ by a  smoothing operator according to 
Definition \ref{omosmoothing} below. 
With the first regularization 
in \eqref{regula12} we guarantee the important property 
that $ \opbw(a) $ is  a spectrally localized map according to Definition \ref{smoothoperatormaps} below.

\noindent
$\bullet$ 
The action of
$\opbw(a)$ on  homogeneous spaces only depends
on the values of the symbol $ a = a(U;t,x,\xi)$ (or
$a(\mathcal{U};t,x,\xi)$) for $|\xi|\geq 1$.
Therefore, we may identify two symbols $ a(U;t,x,\xi)$ and
$ b(U;t,x,\xi)$ if they agree for $|\xi| \geq 1/2$.
In particular, whenever we encounter a symbol that is not smooth at $\xi=0 $,
such as, for example, $a = g(x)|\x|^{m}$ for $m\in \R\setminus\{0\}$, or $ \sign (\xi) $, 
we will consider its smoothed out version
$\chi(\xi)a$, where
$\chi\in  C^{\infty}(\R;\R)$ is an even and positive cut-off function satisfying  
\begin{equation*}
\chi(\x) =  0 \;\; {\rm if}\;\; |\x|\leq \tfrac{1}{8}\, , \quad 
\chi (x) = 1 \;\; {\rm if}\;\; |\x|>\tfrac{1}{4}, 
\quad  \pa_{\x}\chi(\x)>0\quad\forall  \x\in \big(\tfrac{1}{8},\tfrac{1}{4} \big) \, .
\end{equation*}

\noindent
$\bullet$
Definition \ref{quantizationtotale} 
is  independent of the cut-off functions $\chi_{p}$, $\chi$,  
up to smoothing operators that we define below (see Definition \ref{omosmoothing}), see the remark at pag. 50  of  \cite{BD}. 

\noindent
$\bullet$
If for some $(n_0, \dots, n_{p+1}) \in \N^{p+2}$, 
$\Pi_{n_0} \opbw(a(\Pi_{\vec{n}} \cU;\cdot))\Pi_{n_{p+1}} U_{p+1} \not=0 \, ,$
then there exist  signs $ \epsilon_j \in \{\pm\}$, $j=0, \dots, p+1$, such that $ \sum_0^{p+1} \epsilon_j n_j=0$ and the indices satisfy (see Proposition 3.8 in \cite{BD})
\be \label{paraspecloc}
n_0\sim n_{p+1}, \quad n_j \leq C\delta n_{p+1} \, , \quad n_j\leq C\delta  n_0 \, , \quad j=1, \dots, p \, .
\ee
\noindent
$\bullet$
Given  a paradifferential  operator
$ A = \Opbw{a(x,\xi)} $ it results
\be\label{A1b}
\bar A = \Opbw{\bar{a(x, - \xi)}} \, , \quad 
A^\top = \Opbw{a(x, - \xi)} \, , \quad
A^*= \Opbw{\bar{a(x,  \xi)}} \, , 
\ee
where $ A^\top $  and $ A^* $ denote respectively the transposed and  adjoint operator with respect to the complex, respectively real,  scalar product 
of $ \dot L^2 $ in \eqref{scpr12hom}. It results $ A^* = \bar A^\top $. 

\noindent
$\bullet$
 A paradifferential operator $A= \opbw(a(x,\xi))$ is {\it real} (i.e. $A = \bar A$) if 
\be \label{realetoreale}
 \bar{a(x,\xi)}= a^\vee(x,\xi) \quad \text{ where} \quad a^{\vee}(x,\xi) := a(x,-\xi) \, .
 \ee
\noindent
$ \bullet $
A matrix of paradifferential operators $ \opbw(A(U;t, x,\x))$ is real-to-real, i.e. \eqref{vinello} holds, if and only if 
the  matrix of symbols $A(U;t, x,\x)$ has the form 
\begin{equation}\label{prodotto}
A(U;x,\x) =
\left(\begin{matrix} {a}(U;t, x,\x) & {b}(U;t, x,\x)\\
{\ov{b^\vee(U;t, x,\x)}} & {\ov{a^\vee(U;t, x,\x)}}
\end{matrix} 
\right) \,   \, . 
\end{equation}

\paragraph{Classes of $m$-Operators and smoothing Operators.} 

Given integers $(n_1,\ldots,n_{p+1})\in \N^{p+1}$, we denote by $\max_{2}(n_1 ,\ldots, n_{p+1})$ 
the second largest among  $ n_1,\ldots, n_{p+1}$.
We shall often use that $\max_2$ is monotone in each component, i.e. if $n'_j \geq n_j$ for some $j$, then
\be
{\rm max}_2 (n_1 , \ldots, n_j , \ldots, n_p) \leq 
{\rm max}_2(n_1, \ldots, n'_j, \ldots, n_p) \ . 
\ee
In addition $\max_2$ is non decreasing by adding  elements, namely 
\be \label{mono:max2}
{\rm max}_2 (n_1 , \ldots, n_p) \leq 
{\rm max}_2(n_1, \ldots,  n_p, n_{p+1}) \, .
\ee
We now define  the  $ m $-operators. 
The class $\widetilde{\mathcal{M}}^{m}_{p}$ denotes multilinear
 operators that lose $m$ derivatives
 and are $p$-homogeneous in $U$, 
while the class $\mathcal{M}_{K,K',p}^{m}$ contains non-homogeneous
operators  which lose $m$ derivatives, 
vanish at degree at least $ p $ in $ U $, satisfy tame estimates
 and are $(K-K')$-times differentiable in $ t $. 
The  constant $ \mu $ in \eqref{omomap} takes into account possible loss of derivatives in 
the ``low" frequencies.

\begin{definition}{\bf (Classes of $m$-operators)}\label{Def:Maps}
Let  $ m \in \R $,  $p,N\in \N_0 $ 
$K,K'\in\N_0$ with $K'\leq K$, and $ r > 0 $.

(i) {\bf $p$-homogeneous $m$-operators.} 
We denote by $\widetilde{\mathcal{M}}^{m}_{p}$
 the space of $(p+1)$-linear operators $M$ 
 from $(\dot{H}^{\infty}(\T;\C^{2}))^{p}\times \dot{H}^{\infty}(\T;\C)$ to 
 $\dot{H}^{\infty}(\T;\C)$ which are symmetric
 in $(U_{1},\ldots,U_{p})$, of the form
$$ 
(U_{1},\ldots,U_{p+1})\to M(U_1,\ldots, U_p)U_{p+1} 
$$
that satisfy the following. There are $\mu\geq0$, $C>0$ such that 
 \be \label{omomap}
 \|\Pi_{n_0}M(\Pi_{\vec{n}}\mathcal{U})\Pi_{n_{p+1}}U_{p+1}\|_{L^{2}}\leq
 C{\rm \max}_2( n_1,\ldots, n_{p+1})^{\mu}\max( n_1,\ldots, n_{p+1})^{m} \prod_{j=1}^{p+1}\|\Pi_{n_{j}}U_{j}\|_{L^{2}}
\ee
  for any  $ \mathcal{U}=(U_1,\ldots,U_{p})\in (\dot{H}^{\infty}(\T;\C^{2}))^{p}$, any 
 $ U_{p+1}\in \dot{H}^{\infty}(\T;\C) $,
 $ \vec{n} = (n_1,\ldots,n_p) $ in $  \N^{p}$, any $ n_0,n_{p+1}\in \N$.
  Moreover, if 
 \begin{equation}\label{omoresti2}
 \Pi_{n_0}M(\Pi_{n_1}U_1,\ldots,\Pi_{n_{p}}U_{p})\Pi_{n_{p+1}}U_{p+1}\neq 0 \, ,
 \end{equation}
 then there is a choice of signs $\s_0,\ldots,\s_{p+1}\in\{\pm 1\}$ such that 
 $\sum_{j=0}^{p+1}\s_j n_{j}=0 $. 
 In addition we require the translation invariance property
\begin{equation}\label{def:R-trin}
M( \tau_\varsigma {\cal U}) [\tau_\varsigma U_{p+1}]  =  
\tau_\varsigma \big( M( {\cal U})U_{p+1} \big) \, , \quad \forall \varsigma \in \R \, . 
\end{equation}
 We denote $ \widetilde{\mathcal{M}}_{p}:=\cup_{m\geq0}\widetilde{\mathcal{M}}_{p}^{m}$
 and 
  $\Sigma_p^N \widetilde \mM^{m}_q$ the class of pluri-homogeneous operators 
  $\sum_{q=p}^{N}M_{q}$ with $M_q  $ in $  \widetilde{\mM}^{m}_{q}$. 
For $ p \geq N + 1 $ we mean that the sum is empty.   
  We set $\Sigma_p \widetilde \mM^{m}_q:= \bigcup_{N\in \N} \Sigma_p^N \widetilde \mM^{m}_q$.

(ii) {\bf Non-homogeneous $m$-operators.} 
  We denote by  $\mathcal{M}^{m}_{K,K',p}[r]$ 
  the space of operators $(U,t,V)\mapsto M(U;t) V $ defined on $B^{K'}_{s_0}(I;r)\times I \times C^0_{*}(I,\dot{H}^{s_0}(\T,\C))$ for some $ s_0 >0  $, 
  which are linear in the variable $ V $ and such that the following holds true. 
  For any $s\geq s_0$ there are $C>0$ and 
  $r(s)\in]0,r[$ such that for any 
  $U\in B^K_{s_0}(I;r(s))\cap C^K_{*}(I,\dot{H}^{s}(\T,\C^2))$, 
  any $ V \in C^{K-K'}_{*}(I,\dot{H}^{s}(\T,\C))$, any $0\leq k\leq K-K'$, $t\in I$, we have that
\begin{equation}
\label{piove}
\|{\partial_t^k\left(M(U;t)V\right)}\|_{\dot{H}^{s- \frac32 k-m}}
 \leq C \!\!\!\! \sum_{k'+k''=k} \!\!\!\! \|{V}\|_{k'',s}\|{U}\|_{k'+K',s_0}^{p} 
 +\|{V}\|_{k'',s_0}\|U\|_{k'+K',s_0}^{p-1}\|{U}\|_{k'+K',s} \, .
\end{equation}
In case $ p = 0$ we require  the estimate
$ \|{\partial_t^k\left(M(U;t)V\right)}\|_{\dot{H}^{s- \frac32 k-m}}
 \leq C  \|{V}\|_{k,s}$.
 
(iii) {\bf $m$-Operators.}
We denote by $\Sigma\mathcal{M}^{m}_{K,K',p}[r,N]$, 
the space of operators $(U,t,V)\to M(U;t)V$ such that there are homogeneous $m$-operators  $M_{q} $ in $ \widetilde{\mathcal{M}}^{m}_{q}$, $q=p,\ldots, N$  and a non--homogeneous $m$-operator $M_{>N}$  in  
$\mathcal{M}^{m}_{K,K',N+1}[r]$ such that 
\begin{equation}
\label{maps}
M(U;t)V=\sum_{q=p}^{N}M_{q}(U,\ldots,U)V+M_{>N}(U;t)V \, .
\end{equation}
We denote 
$$
\widetilde \mM_p:= \bigcup_{m\geq0} \widetilde \mM_p^m \, , \ \
 \mathcal{M}_{K,K',p}[r]:= \bigcup_{m\geq 0} \mathcal{M}_{K,K',p}^m[r] \, , \ \
 \Sigma\mathcal{M}_{K,K',p}[r,N]:=\bigcup_{m\geq0}
\Sigma\mathcal{M}^{m}_{K,K',p}[r,N] \, , 
$$
and   $\Sigma\mathcal{M}_{K,K',p}^{m}[r,N]\otimes\mathcal{M}_2(\C)$
 the space of $2\times 2$ matrices whose entries are operators in
 $\Sigma\mathcal{M}^{m}_{K,K',p}[r,N]$.
\end{definition}

 \noindent
$ \bullet $  If $M(U, \dots ,U)$ is a  $p$--homogeneous $ m$-operator  in $  \wtcM_p^m $ then the  differential of the non--linear map $ M(U, \dots, U) U $, $
\di_U \big(M(U,\dots, U) U\big) V   = p M(V,U,\dots, U)U + M(U, \dots, U)V 
$ is a $p$--homogeneous  $ m$-operator  in  $ \wtcM_p^m $. 
This follows because the right hand side of \eqref{omomap}
is  symmetric in $(n_1, \ldots, n_{p+1}) $.

 \noindent
$\bullet$
 If $m_1 \leq m_2$ then $\Sigma\mathcal{M}^{m_1}_{K,K',p}[r,N] \subseteq \Sigma\mathcal{M}^{m_2}_{K,K',p}[r,N]$.

\smallskip

 \noindent
$ \bullet $
\noindent{\bf Notation for $p$-homogeneous $m$-operators:}
 if $M(U_1, \ldots, U_p)$ is a $p$-homogeneous $m$-operator, we shall often denote by
$M(U):= M(U, \ldots, U)$ the corresponding polynomial and say that $ M(U) $ is in
$ \wtcM_p^m  $. 
Viceversa, a polynomial can be represented 
by a $(p+1)$-linear form $M(U_1, \ldots, U_p)U_{p+1}$ not necessarily symmetric  in the internal variables. If it fulfills the symmetric estimate \eqref{omomap}, 
the polynomial is generated by the $m$-operator in $ \wtcM_p^m  $ obtained by  symmetrization of the internal variables. We will do this consistently without mentioning it further.

 \noindent
$ \bullet $
\noindent {\bf Notation for projection on homogeneous components:} given an operator  $M(U;t)$ in $ \Sigma \mM_{K,K',p}^m[r, N]$   of the form \eqref{maps} 
we denote by  
\be \label{pienne}
\mP_{\leq N}[ M(U;t)] :=  \sum_{q=p}^{N} M_q(U) \, ,
 \quad \text{resp.} \quad 
 \mP_{q}[ M(U;t)] := M_q(U) \, , 
 \ee
the projections on the pluri-homogeneous, resp. homogeneous, operators in 
 $\Sigma_p^N \widetilde \mM^m_q $ , resp. in $\wt\cM_q^m$. Given an integer $ p\leq p'\leq N$ we also denote 
$$
 \mP_{\geq p'}[ M(U;t)] :=  \sum_{q=p'}^{N} M_q(U), \quad \mP_{\leq  p'}[ M(U;t)] :=  \sum_{q=p}^{p'} M_q(U) \, .
$$
 The same notation will be also used to denote 
 pluri-homogeneous/homogeneous components of symbols. 

\begin{remark}\label{mapsBD}
Definition \ref{Def:Maps} of homogeneous $m$-operators 
is different than the one in Definition 3.9 in \cite{BD}, due to the different 
bound \eqref{omomap}.
However  for $m \geq 0$ the class of homogeneous $m$-operators 
contains the class of homogeneous maps of order $m$ in Definition 3.9 of \cite{BD}, and in view of \eqref{omomap} is contained in the class of  maps of order $m+\mu$ of \cite{BD}.
On the other hand the class of non-homogeneous $m$-operators coincides with the
class of non-homogeneous maps in 
 Definition 3.9 of \cite{BD}.
\end{remark}

If $m  \leq 0 $ the  operators in $ \Sigma \mM^{m}_{K,K',p}[r,N]$ are referred to as smoothing operators.
 \begin{definition}{\bf (Smoothing operators)} \label{omosmoothing}
Let $ \vr\geq0$. A $ (-\vr)$-operator $R(U)$ belonging to $ \Sigma \mM^{-\vr}_{K,K',p}[r,N]$ is called  a smoothing operator. Along the paper will use also the notation 
\be\label{smoothingoper}
\begin{aligned}
& \widetilde{\mathcal{R}}^{-\vr}_{p}:= \widetilde{\mathcal{M}}^{-\vr}_{p} \, ,
\quad
\Sigma_{p}^N \widetilde{\mathcal{R}}^{-\vr}_{q}:= \Sigma_{p}^N\widetilde{\mathcal{M}}^{-\vr}_{q} \, ,\\
&   \mathcal{R}^{-\vr}_{K,K',p}[r]:=\mathcal{M}^{-\vr}_{K,K',p}[r] \, , \quad \Sigma\mathcal{R}^{-\vr}_{K,K',p}[r,N]:=\Sigma\mathcal{M}^{-\vr}_{K,K',p}[r,N] \, . 
 \end{aligned}
\ee
\end{definition}

\noindent
$\bullet$ Given $\varrho \geq 0$, an operator $R(U)$ belongs to $ \wt \cR^{-\varrho}_p$ if and only if there is $\mu  = \mu(\varrho) \geq 0$ and $C >0$ such that
 \be \label{hom:rest} 
 \|\Pi_{n_0}R(\Pi_{\vec{n}}\mathcal{U})\Pi_{n_{p+1}}U_{p+1}\|_{L^{2}}\leq
 C\frac{{\rm \max}_2( n_1,\ldots, n_{p+1})^{\mu}}{\max( n_1,\ldots, n_{p+1})^{\varrho}} \prod_{j=1}^{p+1}\|\Pi_{n_{j}}U_{j}\|_{L^{2}} \ . 
\ee
We remark that  Definition \ref{omosmoothing} of smoothing operators coincides with Definition $3.7$ in \cite{BD}.

\noindent
$ \bullet $ In view of \eqref{omomap} and  \eqref{hom:rest}  a homogeneous $ m $-operator in $ \widetilde{\mathcal{M}}^{m}_{p} $ 
with the property that, on its support,  
$ \max_2 (n_1, \ldots, n_{p+1}) \sim  \max (n_1, \ldots, n_{p+1}) $  is actually a 
{\it smoothing} operator in $ \widetilde{\mathcal{R}}^{-\vr}_{p} $ for {\it any} $ \varrho \geq 0 $. 

\noindent
$ \bullet $ The Definition \ref{omosmoothing} 
of smoothing operators is modeled to gather remainders which 
satisfy  either the property $ \max_2 (n_1, \ldots, n_{p+1}) \sim  \max (n_1, \ldots, n_{p+1}) $
or arise as remainders of  compositions of paradifferential operators, 
see Proposition \ref{teoremadicomposizione} below, and thus have a fixed order $ \varrho $
of regularization.

\begin{lemma}\label{mappabonetta}
If $M( U)$ is a $p$--homogeneous 
 $m$-operator in $  \widetilde \mM_p ^m$ then 
  for any $K\in \N_0$ and  $0\leq k\leq K$ there exists $s_0 >0$ such that for any $s \geq s_0$,
  for any   $U \in C^K_{*}(I,\dot{H}^{s}(\T,\C^2))$,  any $ v \in C^K_{*}(I,\dot{H}^{s}(\T,\C)) $, 
 one has 
 \be
\begin{aligned} 
 \label{mappabonetta1}
  \| \pa_t^{k}(M(U_1,\dots, U_p) v)\|_{\dot H^{s-m-\frac32k}}
  \lesssim_{K}
   \sum_{k_1+\dots+k_{p+1}=k}
     \Big(
 &
  \norm{v}_{k_{p+1},s } \prod_{a =1}^p \norm{U_a}_{k_a, s_0} \\
 & +   \norm{v}_{k_{p+1}, s_0 }\sum_{\bar a =1}^p \norm{U_{\bar a}}_{k_{\bar a}, s}
  \prod_{a =1\atop a\neq \bar a}^p \norm{U_a}_{k_a, s_0}\Big) \, . 
  \end{aligned}
 \ee
 In particular  $M(U) $ is a  non-homogeneous $ m$-operator in 
 $  {\cal M}^m_{K,0,p}[r] $ for any $r>0$ and $K\in \N_0$.
   \end{lemma}
   
   \begin{proof}
   For any $0\leq k\leq K$ we estimate
 $$\| \pa_t^{k}(M(U_1,\dots, U_p) v)\|_{\dot H^{s-m-\frac32k}}\lesssim_{K} \sum_{k_1+\dots+k_{p+1}=k}
  \| M(\pa_t^{k_1}U_1,\dots, \pa_t^{k_p}U_p) \pa_t^{k_{p+1}}v\|_{\dot H^{s-m-\frac32k}}.$$
  We now estimate each term in the above sum.  
We denote   $ \vec n_{p+1} := (n_1, \ldots, n_{p+1} ) \in \N^{p+1}$, $\vec \sigma_{p+1} := (\sigma_1, \ldots, \sigma_{p+1})\in \{\pm \}^{p+1}$  and
  $
  \cI(n_0, \vec \sigma_{p+1}):= \left\lbrace \vec n_{p+1} \in \N^{p+1}  \colon n_ 0 = \sigma_1 n_1 + \ldots +\sigma_{p+1} n_{p+1} \right\rbrace$.  
We get
\begin{align}
 & \| M(\pa_t^{k_1}U_1,\dots, \pa_t^{k_p}U_p) \pa_t^{k_{p+1}}v 
 \|_{\dot H^{s-m-\frac32k}}
  = 	
 \big\| n_0 ^{s-m-\frac32k} 
  \| \Pi_{n_0} (M(\pa_t^{k_1}U_1,\dots,\pa_t^{k_p}U_p) \pa_t^{k_{p+1}}v)\|_{L^2}\big\|_{\ell^2_{n_0} }\notag\\
 &
 \stackrel{ \eqref{omoresti2}}{\leq} \Big\| n_0^{s-m-\frac32k} \!\!\!\!\!\!\!\! 
 \sum_{
 \vec \sigma_{p+1} \in \{\pm\}^{p+1} \atop
 \vec n_{p+1} \in \cI(n_0, \vec \sigma_{p+1})}  \!\!\!\!\!\!\!\!\!\!\!\!\! \big\| \Pi_{n_0} (M(\Pi_{n_1}\pa_t^{k_1}U_1,\dots, \Pi_{n_p}\pa_t^{k_p}U_p) \Pi_{n_{p+1}}\pa_t^{k_{p+1}}v)\|_{L^2} \Big\|_{\ell^2_{n_0} } \notag \\
 & 
  \stackrel{ \eqref{omomap}}{\lesssim}
 \sum_{\vec \sigma_{p+1} }
  \Big\| 
\sum_{ \vec n_{p+1} \in \cI(n_0, \vec \sigma_{p+1}) } \!\!\!\!\!\!\!\!\!\!\!\!\!\!
{\rm \max}_2( n_1,\ldots, n_{p+1})^{\mu} {\rm \max}( n_1,\ldots, n_{p+1})^{s-\frac32k}\prod_{a=1}^{p} \| \Pi_{n_a} \pa_t^{k_a}U_a\|_{L^2} \| \Pi_{n_{p+1}} \pa_t^{k_{p+1}}v\|_{L^2} \Big\|_{\ell^2_{n_0}} \notag 
 \end{align}
 where in the last inequality we also used that  $n_0\lesssim \max\{n_1, \dots, n_{p+1}\}$ and  $s-m-\frac32 k\geq 0$ to bound $n_0^{s-m-\frac32k}$.
For any choice of
$\vec \sigma_{p+1} \in \{ \pm \}^{p+1}$, 
 we  split the internal sum in $p+1$ components
$$
\sum_{\bar a=1}^{p+1} \Sigma_{\vec \sigma_{p+1}}^{(\bar a)}  , \qquad 
\Sigma_{\vec \sigma_{p+1}}^{(\bar a)}:= \sum_{\substack{\vec n_{p+1} \in \cI(n_0, \vec \sigma_{p+1}) \\   \max(n_1, \ldots, n_{p+1}) = n_{\bar a}  }}   \!\! \!\! . 
$$
We first deal with the term $\Sigma_{\vec \sigma_{p+1}}^{(p+1)}$.
In this case we bound 
\begin{align*}
& \|\Pi_{n_{p+1}} \pa_t^{k_{p+1}}v\|_{L^2}  \leq \frac{\wt c_{n_{p+1}}}{n_{p+1}^{s-\frac32 k_{p+1}} } \, \| \pa_t^{k_{p+1}}v\|_{\dot H^{s-\frac32 k_{p+1}}}  
\leq \frac{\wt c_{n_{p+1}}}{n_{p+1}^{s-\frac32 k_{p+1}} } \, \| v\|_{ k_{p+1},s}  
\\
& \|\Pi_{n_a} \pa_t^{k_{a}}U_a\|_{L^2}  \leq \frac{c_{n_{a}}^{(a)}}{n_{a}^{\mu +1-{k_a}} } \, \| \pa_t^{k_{a}}U_a\|_{\dot H^{\mu +1- k_a}} \leq
\frac{c_{n_{a}}^{(a)}}{n_{a}^{\mu +1-{k_a}} } \, \| U_a\|_{k_a, \mu +1 }, \quad a = 1, \ldots, p 
\end{align*}
for some sequences $(\wt c_n)_{n \in \N}$, $( c_n^{(a)})_{n \in \N}$ in $\ell^2(\N)$.
With these bounds, and using 
$\max(n_1, \ldots, n_{p+1})^{s-\frac32 k} = 
n_{p+1}^{s-\frac32 k} 
 \leq n_{p+1}^{s-\frac32 k_{p+1}} \, n_1^{-\frac32 k_1} \ldots n_p^{-\frac32 k_p}$, 
we get 
\begin{align*}
\Big\| 
\Sigma_{\vec \sigma_{p+1}}^{(p+1)}
\Big\|_{\ell^2_{n_0}} 
\lesssim \Big\|
\!\!\!\!\!\!\!\!\!
\sum_{\substack{(\vec{n}_{p+1}) \in \N^{p+1}  \\ n_{0} = \sigma_1 n_1+\dots+ \sigma_{p+1} n_{p+1} }} \!\!\!\!\!\!\!\!\!\!\!\!\!\!\!\!\!\!\
 \wt c_{n_{p+1}} \, 
\frac{c_{n_1}^{(1)}}{ n_1}
\times \dots  \times 
\frac{c_{n_p}^{(p)}}{n_p} 
 \Big\|_{\ell^2_{n_0}} \, 
 \| v\|_{ k_{p+1},s}   \prod_{a = 1}^p \| U_a\|_{k_a, \mu +1} \ . 
\end{align*}
Applying   Young inequality for convolution of sequences and  using that  
$ \big( c_{n}^{(a)} n^{-1} \big)_{n \in \N} \in \ell^1(\N)$, we finally arrive at
$$
\Big\| 
\Sigma_{\vec \sigma_{p+1}}^{(p+1)}
\Big\|_{\ell^2_{n_0}} 
\lesssim \norm{v}_{k_{p+1},s} \, \prod_{a=1}^p \norm{U_a}_{k_a, s_0}  \ , \ \   \ \ \  k_1 + \ldots + k_{p+1} = k  \, , 
$$
which is the first term of inequality \eqref{mappabonetta1} with $s_0 = \max(m+ \frac32 K, \mu +1)$. 
Proceeding similarly we obtain, for any $ \bar a = 1, \ldots, p $,  
$$
\| \Sigma_{\vec \sigma_{p+1}}^{(\bar a)} \|_{\ell^2_{n_0}}  \lesssim   \| v \|_{k_{p+1}, \mu +1} \, 
\norm{U_{\bar a}}_{k_{\bar a}, s} 
\prod_{a \neq \bar a}\norm{U_a}_{k_a, \mu+1} \, 
$$
which are terms in the sum in the second line of \eqref{mappabonetta1}. 
If $U_a = U $ for any $a$, we deduce by  \eqref{mappabonetta1} 
and the estimate
$\norm{U}_{k_{a}, \sigma} \leq \norm{U}_{k - k_{p+1}, \sigma}$  that $M(U)$ fulfills  \eqref{piove} with $K' = 0$, $k'':= k_{p+1}$, $ k' = k - k_{p+1} $. Hence $M(U)$ belongs to $\cM^m_{K,0,p}[r]$  for any $r>0$ and $K\in \N_0 $.
   \end{proof}
   
   \noindent 
$\bullet$ 
A pluri-homogeneous nonlinear map $ Z +R_{\leq N}(Z)Z $ where $ R_{\leq N}(Z)$ is  in $\Sigma_1^N \wtcR^{-\vr}_q\otimes \mM_2(\C)$ satisfies the following bound:  for any $K\in \N_0$ there is $s_0 >0$ such that for any $s\geq s_0$, $0<r< r_0(s,K)$ small enough and  any $ Z \in B_{s_0}^K(I;r)\cap C^{K}_{*\R}(I;\dot H^{s}(\T, \C^2))$  one has 
\be\label{mappaBflow}
2^{-1} \| Z \|_{k,s}\leq \|Z +R_{\leq N}(Z)Z  \|_{k,s} \leq 
2\| Z \|_{k,s} \, , \quad \forall 0 \leq k \leq K \, .
\ee
\noindent
{\bf Fourier representation of $ m $-operators.} We may also represent a matrix of operators 
\be\label{Mupm}
M(U)= \begin{pmatrix}M^+_+(U)&M^-_+(U)\\M^+_-(U)&M^-_-(U)\end{pmatrix} \in \widetilde{\mathcal{M}}_{p} \otimes \cM_2(\C)
\ee
  through  their Fourier matrix elements, see \eqref{u+-},  writing
\be \label{smoocara0}
\begin{aligned}
M(U)V= \vect{(M(U)V)^+}{(M(U)V)^-}, \quad (M(U)V)^\sigma= \sum_{\substack{(\vec \jmath_{p},j,k) \in (\Z\setminus\{0\})^{p+2}\\  (\vec\sigma_{p},\sigma')\in \{\pm\}^{p+1} \\ \sigma k = \vec \sigma_{p} \cdot \vec \jmath_{p}+ \sigma' j} } M_{\vec \jmath_{p}, j,k}^{\vec \sigma_{p}, \sigma',\sigma} u_{\vec \jmath_{p}}^{\vec \sigma_{p}} v^{\sigma'}_{j}  \frac{e^{\ii \sigma k x}}{\sqrt{2\pi}} \, ,
\end{aligned}
 \ee 
where \footnote{ Given $u=\vect{u_1}{u_2},\, v=\vect{v_1}{v_2} \in \C^2$ we set $u\cdot v:= u_1v_1+u_2v_2$.}
\be\label{M.coeff} 
\begin{aligned}
M_{\vec \jmath_{p}, j,k}^{\vec \sigma_{p}, \sigma', \sigma}:=&
\int_{\T} M\Big( \mathtt{q}^{\sigma_1} \frac{e^{\ii \sigma_1 j_1 x}}{\sqrt{2\pi}}, \dots,\mathtt{q}^{\sigma_p} \frac{e^{\ii \sigma_{p} j_{p} x}}{\sqrt{2\pi}}\Big)
\Big[ \frac{\mathtt{q}^{\sigma'} e^{\ii \sigma' j x}}{\sqrt{2\pi}} \Big]  \cdot  \mathtt{q}^\sigma  \frac{e^{-\ii \sigma k x}}{\sqrt{2\pi}}\, \di x\\
 =&\int_{\T} M^{\sigma'}_\sigma\Big( \mathtt{q}^{\sigma_1} \frac{e^{\ii \sigma_1 j_1 x}}{\sqrt{2\pi}}, \dots,\mathtt{q}^{\sigma_p} \frac{e^{\ii \sigma_{p} j_{p} x}}{\sqrt{2\pi}}\Big)
\Big[ \frac{ e^{\ii \sigma' j x}}{\sqrt{2\pi}} \Big] \,  \frac{e^{-\ii \sigma k x}}{\sqrt{2\pi}}\, \di x \,  \in \C  \ , 
\end{aligned}
 \ee
and $\tq^\pm$ are defined in \eqref{tiqu}. 
In \eqref{smoocara0} we have exploited the translation invariance property 
 \eqref{def:R-trin} which implies that if $M_{\vec \jmath_{p}, j,k}^{\vec \sigma_{p}, \sigma', \sigma}\not=0$ then 
 \be \label{momentoj}
  \sigma k = \vec \sigma_{p} \cdot \vec \jmath_{p}+ \sigma' j 	\, .
 \ee
Note also that since $M$ is symmetric in the internal entries, the coefficients $M_{\vec \jmath_{p}, j,k}^{\vec \sigma_{p}, \sigma', \sigma}$ in 
\eqref{M.coeff} satisfy the following symmetric property:
  for any permutation $ \pi $ of $ \{1, \ldots, p \} $, it results
\begin{equation}
\label{M.coeff.p}
M_{j_{\pi(1)}, \ldots,j_{\pi(p)}, j,k}^{ \sigma_{\pi(1)}, \ldots, \sigma_{\pi(p)},\sigma',\sigma} 
=  
M_{j_{1}, \ldots, j_{p}, j,k}^{ \sigma_{1}, \ldots, \sigma_{p},\sigma',\sigma} \, . 
\end{equation}
The operator $M(U)$ is real-to-real, according to definition  \eqref{vinello}, if and only if its coefficients fulfill
\be\label{M.realtoreal}
\bar{M_{\vec \jmath_{p}, j,k}^{\vec \sigma_{p}, \sigma', \sigma}} = M_{\vec \jmath_{p}, j,k}^{-\vec \sigma_{p}, -\sigma', -\sigma} \ .
\ee
The matrix entries of the transpose operator $M(U)^\top$ 
with respect to the non-degenerate bilinear form  \eqref{real.bil.form.intro} are  
\begin{align}
(M^\top)_{\vec \jmath_{p}, j,k}^{\vec \sigma_{p}, \sigma', \sigma}= &  \int_{\T} M^\top\Big( \mathtt{q}^{\sigma_1} \frac{e^{\ii \sigma_1 j_1 x}}{\sqrt{2\pi}}, \dots,\mathtt{q}^{\sigma_p} \frac{e^{\ii \sigma_{p} j_{p} x}}{\sqrt{2\pi}}\Big)
\Big[ \frac{\mathtt{q}^{\sigma'} e^{\ii \sigma' j x}}{\sqrt{2\pi}} \Big]  \cdot  
\frac{\mathtt{q}^\sigma e^{-\ii \sigma k x}}{\sqrt{2\pi}}\, \di x\label{emmetop}\\
 &= \int_{\T} M\Big( \mathtt{q}^{\sigma_1} \frac{e^{\ii \sigma_1 j_1 x}}{\sqrt{2\pi}}, \dots,\mathtt{q}^{\sigma_p} \frac{e^{\ii \sigma_{p} j_{p} x}}{\sqrt{2\pi}}\Big)
\Big[ \frac{\mathtt{q}^{\sigma} e^{-\ii \sigma k x}}{\sqrt{2\pi}} \Big]  \cdot  
\frac{\mathtt{q}^{\sigma'} e^{\ii \sigma' j x}}{\sqrt{2\pi}}\, \di x= M_{\vec \jmath_{p},- k,-j}^{\vec \sigma_{p}, \sigma, \sigma'}\notag \, .
\end{align}
One can directly verify that $(M^\top)_{\vec \jmath_{p}, j,k}^{\vec \sigma_{p}, \sigma', \sigma}$ fulfill \eqref{momentoj}, \eqref{M.coeff.p}. If $  M(U) $ is real-to-real (i.e.  \eqref{M.realtoreal} holds) then $  M(U)^\top $ 
is real-to-real as well.

\begin{lemma}{\bf (Characterization of $m$-operators in Fourier basis)}\label{carasmoofou}
Let $m\in \R$. A real-to-real linear operator $ M(U) $  as in \eqref{smoocara0}-\eqref{M.coeff}
is a matrix of $p$-homogeneous $m$-operators in $ \widetilde{\mathcal{M}}_{p}^{m} \otimes \cM_2(\C)$ if and only if its coefficients $M_{\vec \jmath_p, j, k}^{\vec \sigma_p, \sigma', \sigma}$ 
 satisfy \eqref{momentoj}, \eqref{M.coeff.p}, \eqref{M.realtoreal} and  
 there exist $ \mu>0 $ and $ C > 0 $  such that for any $(\vec \jmath_{p},j) \in (\Z \setminus \{0\})^{p+1} $,   $ (\vec \sigma_p,\sigma) \in \{ \pm \}^{p+1} $,  
\be\label{smoocara}
 |M_{\vec \jmath_p, j, k}^{\vec \sigma_p, \sigma', \sigma} |\leq C 
{\rm max}_2\{ |j_1|,\dots, | j_p|,|j|\}^{\mu}\max\{ |j_1|,\dots, | j_p|,|j|\}^{m}   \, . 
 \ee
\end{lemma}

\begin{proof}
Let $M(U)$ be a matrix of $m$-operators in $ \widetilde{\mathcal{M}}_{p}^{m} \otimes \cM_2(\C)$. 
Then  by \eqref{M.coeff},  applying Cauchy-Schwartz inequality and recalling  \eqref{Pin2} we get 
\begin{align*}
 |M_{\vec \jmath_p, j, k}^{\vec \sigma_p, \sigma', \sigma} |
 & \leq  \Big\| {M^{\sigma'}_{\sigma}}\Big( \mathtt{q}^{\sigma_1}\Pi_{|j_1|} \frac{e^{\ii \sigma_1 j_1 x}}{\sqrt{2\pi}}, \dots,\mathtt{q}^{\sigma_p} \Pi_{|j_p|} \frac{e^{\ii \sigma_{p} j_{p} x}}{\sqrt{2\pi}}\Big) \Big[ \Pi_{|j|}\frac{e^{\ii \sigma' j x}}{\sqrt{2\pi}} \Big] \Big\|_{L^2}\\
\stackrel{\eqref{omoresti2}}{\leq} & \sum_{\substack{\epsilon_1, \dots, \epsilon_p, \epsilon \in \{\pm\}\\ n_0=\epsilon_1|j_1|+ \dots + \epsilon_p|j_p|+ \epsilon|j|}} \Big\|
\Pi_{n_0} {M^{\sigma'}_{\sigma}}\Big( \mathtt{q}^{\sigma_1}\Pi_{|j_1|} \frac{e^{\ii \sigma_1 j_1 x}}{\sqrt{2\pi}}, \dots,\mathtt{q}^{\sigma_p} \Pi_{|j_p|} \frac{e^{\ii \sigma_{p} j_{p} x}}{\sqrt{2\pi}}\Big) \Big[ \Pi_{|j|}\frac{e^{\ii \sigma' j x}}{\sqrt{2\pi}} \Big] 
\Big\|_{L^2}\\
\stackrel{\eqref{omomap}}{\leq} &C 2^{p+1} {\rm max}_2\{ |j_1|,\dots, | j_p|,|j|\}^{\mu}\max\{ |j_1|,\dots, | j_p|,|j|\}^{m}
\end{align*}
proving \eqref{smoocara}. 
Viceversa suppose that $M(U)$ is an operator as in 
\eqref{Mupm}-
\eqref{M.coeff}  with coefficients  satisfying \eqref{smoocara}.  Then, for any 
$ \sigma, \sigma' \in \{\pm \} $, 
\begin{align*}
&\| \Pi_{n_0}{M^{\sigma'}_{\sigma}}(\Pi_{n_1} U_1,\dots, \Pi_{n_p} U_p) \Pi_{n_{p+1}} v^{\sigma'}\|_{L^2} 
= \Big\| \!\! \sum_{\substack{j_1=\pm n_1,\dots, j_p=\pm n_p\\  j= \pm n_{p+1},\, k= \pm n_0}}  \!\!\!\!\!\!\!\!M_{\vec \jmath_{p}, j,k}^{\vec \sigma_{p}, \sigma',\sigma}(u_1)_{j_1}^{\sigma_1}\cdots (u_p)_{j_p}^{\sigma_p} v^{\sigma'}_{j}\frac{e^{\ii \sigma k x}}{\sqrt{2\pi}} \Big\|_{L^2}\\
& \stackrel{\eqref{smoocara}} 
\leq C \!\! \!\! \sum_{\substack{j_1=\pm n_1,\dots, j_p=\pm n_p\\  j= \pm n_{p+1},\, k= \pm n_0}} \!\!\!\! {\rm max}_2\{ |j_1|,\dots, | j_p|,|j|\}^{\mu}\max\{ |j_1|,\dots, | j_p|,|j|\}^{m}
 |(u_1)_{j_1}^{\sigma_1}| \cdots |(u_p)_{j_p}^{\sigma_p}| |v^{\sigma'}_{j}| 
\\
& \leq C 2^{p+2} {\rm max}_2\{ n_1,\dots,  n_p,n_{p+1}\}^{\mu}\max\{n_1,\dots,  n_p,n_{p+1}\}^{m}\prod_{\ell=1}^p \| \Pi_{n_\ell} U_\ell \|_{L^2} \| \Pi_{n_{p+1}} v \|_{L^2}
\end{align*}
proving \eqref{omomap}.
\end{proof}

The transpose of a matrix of $m$-operators is a $m'$-operator.
 
\begin{lemma}\label{primecosette}
{\bf (Transpose of $ m $-Operators)}
Let $p\in \N_0$,  $m\in \R$. 
 If $M(U)$ is a matrix of $p$--homogeneous 
 $m$-operators in $  \widetilde \mM_p^m \otimes \mM_2(\C)$ then 
  $M(U)^\top$ (where the transpose is computed with respect to the non-degenerate bilinear form  \eqref{real.bil.form.intro}) 
 is a matrix of $p$-homogeneous operators in $  \wtcM_p^{m'}\otimes \cM_2(\C) $ for some   $m' \geq \max(m,0)$.
 
 If in addition 
 there exists $C >1$ such that 
 \be\label{spec.cond.t}
 M_{\vec \jmath_{p},k,j}^{\vec \sigma_{p}, \sigma, \sigma'} \neq 0 \quad \Rightarrow \quad C^{-1} |k|\leq |j| \leq C  |k|
 \ee
  then $M(U)^\top \in   \wtcM_p^{m}\otimes \cM_2(\C)$.
\end{lemma} 

\begin{proof}
By \eqref{emmetop},  
\eqref{smoocara} (applied to $M_{\vec \jmath_{p},- k,-j}^{\vec \sigma_{p}, \sigma, \sigma'}$) and since, by \eqref{momentoj},
 $ | k| \leq (p+1) \max\{ |j|, |j_1|, \dots, |j_p|\}$, we deduce that 
\begin{align}
\label{MtopMm}
|(M^\top)_{\vec \jmath_{p}, j,k}^{\vec \sigma_{p}, \sigma', \sigma}| & 
\lesssim {\rm max}_2\{ |j_1|,\dots, | j_p|,|k|\}^{\mu}\max\{ |j_1|,\dots, | j_p|,|k|\}^{m}\\
\notag
& \lesssim
  \max\{ |j_1|,\dots, | j_p|,|j|\}^{ \max(m,0) + \mu }
\end{align}
proving \eqref{smoocara} for $M(U)^\top$ with  $m\leadsto m':=  \max(m,0) +\mu $.

If in addition \eqref{spec.cond.t} holds true then 
$\max\{ |j_1|,\dots, | j_p|,|k|\}^{ m } \sim \max\{ |j_1|,\dots, | j_p|,|j|\}^{ m }$ for any $m \in \R$
and similarly 
$\max_2\{ |j_1|,\dots, | j_p|,|k|\}^{ \mu } \sim \max_2\{ |j_1|,\dots, | j_p|,|j|\}^{ \mu }$ for any
 $\mu \geq 0$. 
 We deduce by \eqref{MtopMm}  that $M(U)^\top  $ is in $ \wt\cM^{m}_p \otimes \cM_2(\C)$.
\end{proof}

\begin{remark}\label{rem:top.smoo}
{\bf (Transpose of Smoothing operators)}
If $R(U) $ is a matrix of smoothing operators in $ \wt \cR^{-\varrho}_p \otimes \cM_2(\C)$ 
and
the spectral condition \eqref{spec.cond.t} holds true, then
$R(U)^\top$ is a smoothing operator in  the same class. Without the spectral condition
\eqref{spec.cond.t}  this might fail: for example consider $ R(U) $ such that its transpose is 
\be\label{rucontroex}
R(U_1, \ldots, U_p)^\top V := \Opbw{A(V, U_1, \ldots, U_{p-1})} U_p \ , \ \ A(\cdot ) \in \wt \Gamma^\mu_p \otimes \cM_2(\C) \ .  
\ee
As a consequence of Lemma \ref{aggiunto} below, we have that  
 $ R(U) = [R(U)^\top]^\top $ is in $ \wt \cR^{-\varrho}_p \otimes \cM_2(\C)$ for any $\vr\geq 0$, but $R(U)^\top$ in \eqref{rucontroex} is a $\mu$-operator.
\end{remark}

We conclude this subsection with 
the  paralinearization of the product (see e.g. Lemma 7.2 in \cite{BD}).
\begin{lemma}{\bf (Bony paraproduct decomposition)}
\label{bony}
Let  $u_1, u_2$ be functions in $H^\s(\T;\C)$   with  
$\s >\frac12$. Then 
\begin{equation}\label{bonyeq}
u_1 u_2  = \Opbw{u_1 }u_2 + \Opbw{u_2 }u_1 + 
R_1(u_1)u_2 +  R_2(u_2)u_1
\end{equation}
where for $j=1, 2 $, $R_j$ is a homogeneous smoothing operator in $ \widetilde \mR^{-\vr}_{1}$ for any $ \vr \geq 0$. 
\end{lemma}

\noindent
{\bf Composition theorems.}
Let 
$ \s(D_{x},D_{\x},D_{y},D_{\eta}) := D_{\x}D_{y}-D_{x}D_{\eta}  $
where $D_{x}:=\frac{1}{\ii}\pa_{x}$ and $D_{\x},D_{y},D_{\eta}$ are similarly defined. 
The following is   Definition 3.11 in \cite{BD}. 

\begin{definition}{\bf (Asymptotic expansion of composition symbol)}
\label{def:as.ex}
Let $ p $, $ p' $ in $\N_0 $, $ K, K' \in \N_0 $ with $K'\leq K$,  $ \varrho  \geq 0 $, $m,m'\in \R$, $r>0$.
Consider symbols $a \in \Sigma\Gamma_{K,K',p}^{m}[r,N]$ and $b\in \Sigma \Gamma^{m'}_{K,K',p'}[r,N]$. For $U$ in $B_{\s}^{K}(I;r)$ 
we define, for $\varrho< \s- s_0$, the symbol
\begin{equation}\label{espansione2}
(a\#_{\varrho} b)(U;t,x,\x):=\sum_{k=0}^{\varrho}\frac{1}{k!}
\left(
\frac{\ii}{2}\s(D_{x},D_{\x},D_{y},D_{\eta})\right)^{k}
\Big[a(U;t,x,\x)b(U;t,y,\eta)\Big]_{|_{\substack{x=y, \x=\eta}}}
\end{equation}
modulo symbols in $ \Sigma \Gamma^{m+m'-\varrho}_{K,K',p+p'}[r,N] $. 
\end{definition}

\noindent
$ \bullet $
The symbol $ a\#_{\varrho} b $ belongs  to $\Sigma\Gamma^{m+m'}_{K,K',p+p'}[r,N]$.

\noindent
$ \bullet $
We have  that 
$ a\#_{\varrho}b =ab+\frac{1}{2 \ii }\{a,b\} $ 
up to a symbol in $\Sigma\Gamma^{m+m'-2}_{K,K',p+p'}[r,N]$,     
where 
$$ 
\{a,b\}  :=  \pa_{\xi}a\pa_{x}b -\pa_{x}a\pa_{\xi}b  
$$
denotes the Poisson bracket. 

\noindent
$ \bullet $
If $c $ is a symbol in $ \Sigma\Gamma^{m''}_{K,K',p''} [r,N]$ then 
$ a\#_{\varrho}b\#_{\varrho}c +
c\#_{\varrho}b\#_{\varrho}a - 2  a bc$ is a symbol in  $\Sigma\Gamma^{m+m'+m''-2}_{K,K',p+p'+ p''}[r,N]$. 

\noindent
$ \bullet $
$ \overline{a^\vee} \#_{\varrho} 
\overline{b^\vee}  = \overline{a \#_{\varrho} b}^\vee  $ where $ a^\vee $ is defined in \eqref{realetoreale}.

\smallskip
The following result is proved in Proposition $3.12$ in \cite{BD}. 

\begin{proposition}{\bf (Composition of Bony-Weyl operators)} \label{teoremadicomposizione}
Let $p,q,N, K, K'  \in \N_0 $ with $ K' \leq K $,  $\varrho \geq 0 $, $m,m'\in \R$, $r>0$. 
Consider  symbols
$a\in \Sigma {\Gamma}^{m}_{K,K',p}[r,N] $ and $b\in \Sigma {\Gamma}^{m'}_{K,K',q}[r, N]$. 
Then
\be\label{smoospec}
\Opbw{a(U;t,x,\x)}\circ\Opbw {b(U;t,x,\x)} - \Opbw{(a\#_{\varrho} b)(U;t,x,\x)}
\ee
is a smoothing operator in $ \Sigma {\mathcal{R}}^{-\varrho+m+m'}_{K,K',p+q}[r,N]$.
\end{proposition}

We now prove other composition results concerning $m$-operators.

\begin{proposition}{\bf (Compositions of $m$-operators)} \label{composizioniTOTALI}
Let $p, p', N, K, K' \in \N_0$ with $K'\leq K$ and $r>0$. Let $m,m' \in \R$.
Then 
\begin{itemize}
\item[$(i)$]
 If  $M(U;t) $ is  in 
$ \Sigma\mathcal{M}^{m}_{K,K',p}[r,N]$ and $M'(U;t) $ is  in
$ \Sigma\mathcal{M}^{m'}_{K,K',p'}[r,N] $ then the composition 
$ M(U;t)\circ M'(U;t) $  
is  in $\Sigma\mathcal{M}^{m+\max(m',0)}_{K,K',p+p'}[r,N]$.

\item[$(ii)$]
If  $M (U) $ is a homogeneous $ m$-operator in $  \widetilde{\mathcal{M}}_{p}^{m}$  
and $M^{(\ell)}(U;t)$, $\ell=1,\dots,p+1$, are matrices of  $ m_\ell $-operators  in 
$ \Sigma \mM^{m_\ell}_{K,K',q_\ell}[r,N]\otimes \mM_2(\C) $ with  $m_\ell \in \R$, 
$q_\ell\in \N_0$,   
then  
$$ 
M(M^{(1)}(U;t)U, \ldots,  M^{(p)}(U;t)U)M^{(p+1)}(U;t)
$$ 
belongs to $ \Sigma\mM_{K,K',p+\bar q}^{m+ \bar m}[r,N]$ with $ \bar m:=\sum_{\ell=1}^{p+1} \max(m_\ell,0)$ and $\bar q:= \sum_{\ell=1}^{p+1}q_\ell$.

\item[$(iii)$] 
If   $M(U;t) $ is in $ {\mathcal{M}}_{K,0,p}^{m}[\breve r]$ for any $\breve r\in \R^+$ and $ \bM_0(U;t) $ belongs to  $ \mM^0_{K,K',0}[r]\otimes \mM_2(\C)$, 
then $M(\bM_0(U;t)U;t)$ is in  $  \mM^{m}_{K,K',p}[r]$.

\item[(iv)] Let $c$ be a homogeneous symbol in $\widetilde{\Gamma}_p^{m}$  and $M^{(\ell)} (U;t) $, $\ell=1,\dots, p $, be operators in 
$ \Sigma\mM_{K,K',q_\ell}[r,N] $ with $ q_\ell \in \N_0 $. 
Then 
$$
U \rightarrow  b(U;t,x,\x):= c(M^{(1)}(U;t)U,\ldots,M^{(p)}(U;t)U;t,x,\x)
$$
is a symbol in $\Sigma\Gamma^{m}_{K,K',p+\bar q}[r,N]$ with $\bar q:=q_1+\dots+q_p$ and 
$$ 
\opbw( c(W_1,\ldots,W_p;t,x,\x))_{|W_\ell=M^{(\ell)}(U;t)U}=\opbw(b(U;t,x,\x))+R(U;t) 
$$
where $R(U; t) $ is a smoothing operator in 
$\Sigma \mathcal{R}^{-\varrho}_{K,K', p + \bar q}[r,N] $  
for any $\vr\geq0$.
\end{itemize}
\end{proposition}

\begin{proof}
{\sc Proof of $(ii)$:} 
 It is sufficient to prove the thesis for $ m_\ell\geq 0$ otherwise we regard $M^{(\ell)}(U;t)$ as a $0$-operator in $ \Sigma\mathcal{M}^{0}_{K,K',p}[r,N]$.
We decompose, for any $ \ell=1, \dots,p+1$, $M^{(\ell)} = \sum_{a=q_\ell}^N M_a^{(\ell)}+ M^{(\ell)}_{>N}$  as in \eqref{maps}.
Given integers  $ a_\ell \in [q_\ell, N]$,  $\ell=1, \dots p+1$, we use the notation $ \bar a_\ell:= a_1+\dots+ a_\ell+\ell$ and $ \bar a_0:=0$. 
Note that
$ \bar a_{p+1} - 1 = a_1 + \ldots + a_{p+1} + p \geq \bar q + p $. 
We also denote the vector with $a_\ell$ elements
$$ 
\cU_\ell := (U_{\bar{a}_{\ell -1}+1}, \ldots, U_{\bar a_\ell - 1})  \ , \quad \ell = 1, \ldots, p+1 \, . 
$$
By multi-linearity we have to show on the one hand that, 
 if $ \bar a_{p+1}-1  \leq N$ then 
\be \label{omoomo35}
M\Big(M^{(1)}_{a_1}(\cU_1)
U_{a_1+1}, \dots, M^{(p)}_{a_p}(\cU_p)U_{\bar a_{p}}
\Big)M^{(p+1)}_{a_{p+1}} (\cU_{p+1} )
\ee
is a homogeneous operator in $ \wtcM^{m+\bar m}_{\bar a_{p+1}-1}$. On the other hand,  if  $\bar a_{p+1}-1\geq N+1$ then  
\be \label{nonomo35}
 M\Big(M^{(1)}_{a_1}(U,\dots, U)U, \dots , M^{(p)}_{a_p}(U,\dots, U)U\Big)M^{(p+1)}_{a_{p+1}}( U,\dots, U) \quad 
\ee
is a non--homogeneous operator in $ \mM^{m+\bar m}_{K,K',N+1}[r]$, having included for notational convenience the inhomogeneous term as  $M^{(\ell)}_{N+1}(U,\dots,U):= M^{(\ell)}_{>N}(U;t)$ which belongs to
$\mM^{m_\ell}_{K,K',N+1}[r]$. 

We first study \eqref{omoomo35}. 
First of all, using the notation 
$\vec n_\ell := (n_{\bar{a}_{\ell -1}+1}, \ldots, n_{\bar a_\ell -1})$
 one has 
\begin{align}
&\|
\Pi_{n_0}
M\Big(
M^{(1)}_{a_1}(\Pi_{\vec n_1}\cU_1)
\Pi_{n_{\bar a_1}} U_{\bar a_1},
 \dots, 
M^{(p)}_{a_p}(\Pi_{\vec n_p}\cU_p)
\Pi_{n_{\bar a_p}} U_{\bar a_{p}}
\Big)
M^{(p+1)}_{a_{p+1}}
(\Pi_{\vec n_{p+1}}\cU_{p+1} )
\Pi_{n_{\bar a_{p+1}}} U_{\bar a_{p+1}}
\|_{L^2}\notag\\
&\leq \!\!\!\!\!\sum_{n'_1,\dots, n_{p+1}'} \!\!\!\!\!\|
\Pi_{n_0}
M\Big(
\Pi_{n'_1}
M^{(1)}_{a_1}(\Pi_{\vec n_1}\cU_1)
\Pi_{n_{\bar a_1}} U_{\bar a_1},
 \dots, 
 \Pi_{n'_p}
M^{(p)}_{a_p}(\Pi_{\vec n_p}\cU_p)
\Pi_{n_{\bar a_p}} U_{\bar a_{p}}
\Big)
\notag
\\
& \qquad \qquad 
\circ
\Pi_{n'_{p+1}}
M^{(p+1)}_{a_{p+1}}
(\Pi_{\vec n_{p+1}}\cU_{p+1} )
\Pi_{n_{\bar a_{p+1}}} U_{\bar a_{p+1}}
\|_{L^2} \, .\label{inininni35}
\end{align}
Thanks to the conditions \eqref{omoresti2} for $M$ and $ M^{(\ell)} $ 
the indices in the above sum satisfy, for some choice of signs $\sigma_b,\epsilon_j$, $b=1,\dots, \bar a_{p+1}$, $ j=1,\dots,p+1$, 
the restrictions 
\be\label{scece35}
n_0= \sum_{j=1}^{p+1} \epsilon_j n'_j \, , 
\quad n'_\ell=\sum_{b=\bar a_{\ell-1}+1}^{\bar a_{\ell}} \sigma_b n_b \, , 
\quad \ell=1,\dots, p+1 \, .
\ee
As a consequence \eqref{omoomo35} satisfies the corresponding condition \eqref{omoresti2} and 
\be\label{denaco34}
 n'_\ell \lesssim \max \{ \vec n_\ell, n_{\bar a_\ell} \}  \ , 
 \qquad
\forall
\ell=1,\dots, p+1 \ , 
 \ee
 then we have 
 \begin{align}
 \label{denaco35}
& {\rm max}\{ n'_1,\dots, n'_{p+1}\}\lesssim  {\rm max}\{ n_1,\dots, n_{\bar a_{p+1}}\} \\
\notag
&  {\rm max}_2\{ n'_1,\dots, n'_{p+1}\}
 \stackrel{\eqref{denaco34}}{\lesssim} 
 {\rm max}_2 \left\lbrace 
\max \{ \vec n_1, n_{\bar a_1} \}, \ldots , 
\max \{ \vec n_{p+1}, n_{\bar a_{p+1}} \} 
 \right\rbrace 
\lesssim 
  {\rm max}_2\{ n_1,\dots, n_{\bar a_{p+1}}\} 
  \end{align}
where in the last inequality we used that 
$\{ n_1, \ldots, n_{\bar a_{p+1}}\}$ is the disjoint union of the sets
$ \{ \vec n_\ell, n_{\bar a_\ell} \}_{\ell = 1, \ldots, p+1}$.\\
 Using  \eqref{scece35}, \eqref{omomap} for $M$ and $M^{(\ell)}$  we get, with $\vec n':= ( n_1',\dots, n'_{p+1})$
\begin{align}
&\eqref{inininni35}\lesssim {\rm max}_2\{ \vec n'\}^{\mu}
 {\rm max}\{ \vec n'\}^{m} 
\prod_{\ell=1}^{p+1}
{\rm max}_2\{ \vec n_{\ell}, n_{\bar a_\ell}\}^{\mu_\ell}
{\rm max}    \{  \vec n_{\ell}, n_{\bar a_\ell}\}^{m_\ell}\notag
 \prod_{b=1}^{\bar a_{p+1}} \| \Pi_{n_b} U_b\|_{L^2}\notag\\ 
&\stackrel{\eqref{denaco35},\eqref{mono:max2}, m_\ell\geq0}{\lesssim} 
{\rm max}_2\{ n_1,\dots, n_{\bar a_{p+1}}\}^{\bar \mu}
\max\{ n_{1},\dots, n_{\bar a_{p+1}}\}^{\bar m}
{\rm max}\{ \vec n'\}^{m} 
\prod_{b=1}^{\bar a_{p+1}} \| \Pi_{n_b} U_b\|_{L^2}\label{danonprovare35},
\end{align}
where $\bar \mu:=\mu + \mu_1+\dots+\mu_{p+1}$ and recall $\bar m= m_1+\dots+m_{p+1}$. 
We claim that 
\be \label{daprovare35}
{\rm max}\{\vec n'\}^{m}\lesssim {\rm max}_2\{ n_1,\dots, n_{\bar a_{p+1}}\}^{ \mu''}\max\{ n_{1},\dots, n_{\bar a_{p+1}}\}^{m} 
\ee
for some $\mu''\geq0 $. 
Then \eqref{omomap}, \eqref{danonprovare35} and \eqref{daprovare35} imply that the operator in \eqref{omoomo35} belongs to  $\wtcM_{\bar a_{p+1}-1}^{m+\bar m}$.

We now prove  \eqref{daprovare35}. 
If $m\geq 0$ it  follows by 
\eqref{denaco35} with $\mu''=0$. 
So from now on we consider  $m<0$. 
We fix $\bar \ell$ such that 
\be\label{scelgouguali}
\max\{ n_{1},\dots, n_{\bar a_{p+1}}\}=
{\rm max}\{ \vec n_{\bar \ell}, \ n_{\bar a_{\bar \ell}}\}
\ee
 and   we distinguish two cases: 

\noindent {\bf Case $1$:} $n'_{\bar \ell}\geq \frac12 {\rm max}\{ \vec n_{\bar \ell}, \ n_{\bar a_{\bar \ell}}\}$. In this case 
$$
{\rm max}\{ \vec n '\}^{m}
\stackrel{m<0}{\leq}
 (n'_{\bar \ell})^m
\stackrel{{\rm Case }\, 1}{\leq} 
2^{|m|} \, {\rm max}\{ \vec n_{\bar \ell},  n_{\bar a_{\bar \ell}}\}^m
\stackrel{\eqref{scelgouguali}}{=}
2^{|m|} \max\{ n_{1},\dots, n_{\bar a_{p+1}}\}^m
$$
which proves \eqref{daprovare35} with $\mu''=0$.

\noindent {\bf Case $2$:} $n'_{\bar \ell}< \frac12 {\rm max}\{ \vec n_{\bar \ell}, \ n_{\bar a_{\bar \ell}}\}$. 
In view of the momentum condition \eqref{scece35}, 
\begin{align*} 
\max \{ \vec n_{\bar \ell}, \ n_{\bar a_{\bar \ell}}\}
& \leq
a_{\bar \ell}{\rm max}_2 \{ \vec n_{\bar \ell}, \ n_{\bar a_{\bar \ell}} \} 
+ n'_{\bar \ell}\\
& \stackrel{{\rm Case } \, 2}{\leq}
a_{\bar \ell} {\rm max}_2 \{ \vec n_{\bar \ell}, \ n_{\bar a_{\bar \ell}} \}
+  \frac12 {\rm max}\{ \vec n_{\bar \ell}, \ n_{\bar a_{\bar \ell}}\}
  \end{align*} 
and consequently
\be\label{smoopercaso35}
\max \{ \vec n_{\bar \ell}, \ n_{\bar a_{\bar \ell}} \} 
\leq 2 a_{\bar \ell}{\rm max}_2\{ \vec n_{\bar \ell}, \ n_{\bar a_{\bar \ell}}\}  \, . 
\ee
Then, since $ m < 0 $, 
\begin{align*}
\max\{ \vec n'\}^{m}
&
\leq 
1 
\stackrel{\eqref{smoopercaso35}}{\leq}
(2 a_{\bar \ell})^{|m|} \ 
{\rm max}_2
\{  \vec n_{\bar \ell}, \ n_{\bar a_{\bar \ell}}\}^{|m|} 
\ 
\max\{  \vec n_{\bar \ell}, \ n_{\bar a_{\bar \ell}}\}^{m}
\\
\stackrel{\eqref{scelgouguali}, \eqref{mono:max2}}{\leq} 
&
(2 a_{\bar \ell})^{|m|}
{\rm max}_2\{ n_1,\dots,  n_{\bar a_{p+1}}\}^{|m|}
\max\{ n_1,\dots,  n_{\bar a_{p+1}}\}^{m}
\end{align*}
which proves \eqref{daprovare35} with $\mu''=|m|$. 
This concludes the proof of \eqref{daprovare35}
and then that the operator in \eqref{omoomo35} is in  $\wtcM_{\bar a_{p+1}-1}^{m+\bar m}$.

Now we  prove that the operator in \eqref{nonomo35}  
is in $ \mM^{m+\bar m}_{K,K',N+1}[r]$. 
We have to verify \eqref{piove} with $m\leadsto m+\bar m$ and $p \leadsto N+1$.
For simplicity we denote $ M^{(\ell)}_{a_\ell}(U) = M^{(\ell)}_{a_\ell}(U, \ldots, U ) $.
First  we apply \eqref{mappabonetta1}
to  $M(U) \in \wt\cM^m_p$  
  (with 
 $U_\ell \leadsto M^{(\ell)}_{a_\ell}(U)U$, 
 $ v \leadsto M^{(p+1)}_{a_{p+1}}(U)v$ and 
 $ s \leadsto s-\bar m $) getting, for any $ k = 0, \ldots, K- K' $, 
 \begin{align}
  \| \pa_t^{k} \big( \eqref{nonomo35}v \big) \|_{\dot H^{s-m-\bar m-\frac32k}} & \label{due1}\lesssim_{K} 
  \!\!\!\!\!\!\!\!\! \sum_{k_1+\dots+k_{p+1}=k}\!\!\!\!\!
     \Big(
  \norm{M^{(p+1)}_{a_{p+1}}(U)v}_{k_{p+1},s-\bar m }
   \prod_{\ell =1}^p \
   \norm{ M^{(\ell)}_{a_\ell}(U)U}_{k_\ell, s_0} \\
   \label{due2}
&   \!\!\!\!\!\!\!\!\! + 
   \norm{M^{(p+1)}_{a_{p+1}}(U)v}_{k_{p+1}, s_0 }\sum_{\bar \ell =1}^p 
   \norm{ M^{(\bar \ell)}_{a_{\bar \ell}}(U)U}_{k_{\bar \ell}, s-\bar m}
  \prod_{\ell =1\atop \ell \neq \bar \ell }^p 
  \norm{ M^{(\ell)}_{a_\ell}(U)U}_{k_\ell, s_0}\Big)\, . 
\end{align}
Then we estimate  line \eqref{due1} where, 
by Lemma \ref{mappabonetta}, 
 each  $ M^{(\ell)}_{a_\ell}(U) $ is in $ \mM^{m_\ell}_{K,0,a_\ell}[r]$
 and
 $M^{(\ell)}_{N+1}(U)$  belongs to 
$\mM^{m_\ell}_{K,K',N+1}[r]$.   For any $ \ell = 1, \ldots, p $, using 
 \eqref{piove} (with $ m \leadsto m_\ell $ and $ p \leadsto a_\ell $)
 we bound (since $k_\ell  \leq k - k_{p+1} $ and $ 0 \leq k \leq K - K' $)
 \begin{align}
 &\norm{ M^{(\ell)}_{a_\ell}(U)U}_{k_\ell, s_0}
\leq C \norm{U}_{k_\ell + K',s_0+\bar m}^{a_\ell+1} 
\leq  C \norm{U}_{k - k_{p+1} + K',s_0+\bar m}^{a_\ell+1} \, , \label{boundsp215} \\
 & \norm{M^{(p+1)}_{a_{p+1}}(U)v}_{k_{p+1},s-\bar m }
 \leq
 C \sum_{k' + k''\leq k_{p+1}} 
 \norm{v}_{k'',s} 
 \norm{U}_{k'+K',s_0}^{a_{p+1}} + 
\norm{v}_{k'',s_0} 
\norm{U}_{k'+K' ,s_0}^{a_{p+1}-1}
 \norm{U}_{k'+K',s} \notag \\
&
\qquad \qquad\qquad \qquad   \leq  C  \sum_{k''\leq k_{p+1}} 
 \norm{v}_{k'',s} 
 \norm{U}_{k-k''+K',s_0}^{a_{p+1}}
  + 
\norm{v}_{k'',s_0} 
\norm{U}_{k-k''+K' ,s_0}^{a_{p+1}-1} 
\norm{U}_{k-k''+K',s} \notag
   \end{align} 
since $ k' \leq k-k'' $ (being $ k_{p+1} \leq k $).  
By \eqref{boundsp215} and since
  $k - k_{p+1} + K' \leq k - k'' + K'$, we get 
  \begin{align}
 \notag
\eqref{due1} & 
\lesssim_{K,s}  \sum_{k''\leq  k} 
 \norm{v}_{k'',s}
  \norm{U}_{k-k''+K',s_0+ \bar m}^{a_1 + \dots +a_{p+1}+p} + 
\norm{v}_{k'',s_0} 
\norm{U}_{k-k''+K' ,s_0+ \bar m}^{a_1+\dots + a_{p+1}+p-1} 
\norm{U}_{k-k''+K',s} \\
\label{due10}
& \lesssim_{K,s} \sum_{k''  \leq k} 
 \norm{v}_{k'',s} \norm{U}_{k-k''+K',s_0+ \bar m}^{N+1} + 
\norm{v}_{k'',s_0} 
\norm{U}_{k-k''+K' ,s_0+ \bar m}^{N}
 \norm{U}_{k-k''+K',s}  
\end{align}   
because  $a_1+\dots + a_{p+1}+p = \bar a_{p+1} - 1\geq N+1$
(for $ \| U\|_{K,s_0+\bar m} < 1 $).
Regarding \eqref{due2}, we first bound  by 
\eqref{piove} (and \eqref{boundsp215} with $ s = s_0 + \bar m $)
\begin{align*}
&\norm{ M^{(\bar \ell)}_{a_{\bar \ell}}(U)U}_{k_{\bar \ell}, s - \bar m}
 \leq
 C \norm{U}_{k_{\bar \ell}+K', s_0}^{a_{\bar \ell}} \norm{U}_{k_{\bar \ell}+K',s}  \leq
 C \norm{U}_{k-k_{p+1}+K', s_0}^{a_{\bar \ell}} \norm{U}_{k-k_{p+1}+K',s} \\
 & \norm{M^{(p+1)}_{a_{p+1}}(U)v}_{k_{p+1},s_0}
\leq
 C \sum_{k''\leq k_{p+1}} 
 \norm{v}_{k'',s_0 + \bar m } \ \norm{U}_{k-k''+K',s_0+\bar m }^{a_{p+1}} 
\end{align*}
and, proceeding similarly to the previous computation, we deduce 
\be \label{due20}
\eqref{due2} 
\lesssim_{K,s}  \sum_{k''\leq  k} 
\norm{v}_{k'',s_0+ \bar m} \norm{U}_{k-k''+K' ,s_0+ \bar m}^{N} 
\norm{U}_{k-k''+K',s} \, .
\ee
Hence \eqref{due10}, \eqref{due20} imply that  for any $k = 0, \ldots, K- K'$
\be\label{due30}
\| \pa_t^{k} \big( \eqref{nonomo35}v \big) \|_{\dot H^{s-m-\bar m-\frac32k}}
\lesssim_{K,s} \!\!\!\!\!\!
\sum_{k'+k''  = k} \!\!\!\!
 \norm{v}_{k'',s} \norm{U}_{k'+K',s_0+ \bar m}^{N+1} + 
\norm{v}_{k'',s_0+ \bar m} \norm{U}_{k'+K' ,s_0+\bar m}^{N} \norm{U}_{k'+K',s}  
\ee
proving that the operator \eqref{nonomo35} 
satisfies \eqref{piove} with $m\leadsto m+\bar m$, $p \leadsto N+1$
and $ s_0 \leadsto s_0 + \bar m $.
\\
\noindent{\sc Proof of $(i)$:} 
Decomposing $M= \sum_{q=p}^N M_q+ M_{>N}$  and $M'= \sum_{q=p'}^N M'_q+ M'_{>N}$ as in \eqref{maps}, by item $(ii)$ we deduce that 
$ M_{q_1}(U) M'_{q_2}( U) $
is  in $ \wtcM^{m+ \max{\{m',0\}}}_{q_1+q_2}$ if $ q_1+q_2\leq N$,
and $ M_{q_1}(U) M'_{q_2}( U) $, $ q_1+q_2\geq N+1 $
and $ M_{q_1}(U) M'_{>N}( U;t) $,  $ q_1=p,\dots, N, $ are  
 in $ \mM^{m+\max\{m',0\}}_{K,K',N+1}[r]$.
Furthermore, proceeding as in the proof of \eqref{due30}, one shows that 
$ M_{>N}(U;t) M'_{q_2}( U) $, $ q_2=p',\dots, N $, and 
$ M_{>N}(U;t) M'_{>N}( U;t) $
are operators in $ \mM^{m+\max\{m',0\}}_{K,K',N+1}[r]$.

\noindent {\sc Proof of $(iii)$:} 
As $\bM_0(U;t) \in  \mM^0_{K,K',0}[r]\otimes \mM_2(\C)$, for any 
$ U \in B^{K'}_{s_0}(I;r)$,  
the function  $\bM_0(U;t) U \in B^{0}_{s_0}(I;C r)$, 
for some $C\geq 1$.  Hence the composition $M (\bM_0(U;t)U)$ is a well defined operator.  Moreover, for any $ U \in B_{s_0}^K(I;r(\sigma))$, we have  the quantitative estimate
$\norm{\bM_0(U;t)U}_{k, \sigma} \lesssim \norm{U}_{k,\sigma}$,  for 
$k = 0, \ldots, K- K'$ and $\sigma \geq s_0$. 
To bound the time derivatives of $M (\bM_0(U;t)U)$  we use  \eqref{piove} and the previous estimate to get, 
for any $k = 0, \ldots, K-K'$ 
\begin{align*}
& \norm{\pa_t^k   \left(M (\bM_0(U;t)U) V \right)}_{s-\frac32k - m} 
\\
& \quad \lesssim
\sum_{k'+k'' = k} \norm{\bM_0(U;t) U}_{k',s_0}^{p} \norm{V}_{k'',s} +
\norm{\bM_0(U;t) U}_{k',s}  \norm{\bM_0(U;t) U}_{k', s_0}^{p-1} \norm{V}_{k'',s_0}
 \\
& \quad  \lesssim 
\sum_{k'+k''=k} \!\!\!\! \|V\|_{k'',s}\|{U}\|_{k'+K',s_0}^{p} 
 +\|{V}\|_{k'',s_0}\|U\|_{k'+K',s_0}^{p-1}\|{U}\|_{k'+K',s} 
\end{align*}
proving \eqref{piove}. 

 \noindent {\sc Proof of $(iv)$:} It follows, in view of Remark \ref{mapsBD}, 
 by Proposition 3.17-$(i)$ and Proposition 3.18 in \cite{BD}.
\end{proof} 

We shall use the following facts which follow by 
 Proposition \ref{composizioniTOTALI}. 

 \noindent
$ \bullet $ 
If $ a(U;t,x,\xi) $ is a symbol in $ \Sigma \Gamma_{K,K',p}^m[r,N] $ and $ U $ is a solution of $ \pa_t U = M(U;t) U $ for some $ M(U;t) $ in  $\Sigma  \mM_{K,0,0}[r,N] \otimes \mM_2(\C) $, then $ \pa_t a(U;t,x,\xi) $ is a symbol in $ \Sigma \Gamma_{K,K'+1,p}^m[r,N] $. 

 \noindent
$ \bullet $ If $ R(U;t) $ is a smoothing operator 
in $ \Sigma  \mR^{-\varrho}_{K,K',p}[r,N]  $ and $ U $ is a solution of $ \pa_t U = M(U;t) U $ for some $ M(U;t) $ in  $\Sigma  \mM_{K,0,0}[r,N] \otimes \mM_2(\C) $, then $ \pa_t R(U;t) $ is  a smoothing operator $ \Sigma  \mR^{-\varrho}_{K,K'+1,p}[r,N] $.

\noindent
$\bullet$ 
\label{McR} If  $M(U;t) $ is  in 
$ \Sigma\mathcal{M}^{m}_{K,K',p}[r,N]$ and $R(U;t) $ is  in
$ \Sigma\mathcal{R}^{-\varrho}_{K,K',p'}[r,N] $ then the composition 
$ M(U;t)\circ R(U;t) $  
is  in $\Sigma\mathcal{M}^{m}_{K,K',p+p'}[r,N]$, and so it is not a smoothing map.

\subsection{Spectrally localized maps}\label{sec:SL}

We introduce the notion of  a ``spectrally localized" map.
The class $\widetilde{\mathcal{S}}^{m}_{p}$ denotes $p$-linear
$m$-operators 
with the spectral support similar to a para--differential operator (compare \eqref{paraspecloc} and \eqref{specres}), i.e. the `` internal frequencies" are controlled by the ``external" ones which are equivalent. 
On the other hand the class $\mathcal{S}_{K,K',p}^{m}$ contains non-homogeneous
$m$-operators  
which vanish at degree at least $ p $ in $ U $, 
 are $(K-K')$-times differentiable in $ t $, and  satisfy estimates similar to 
para--differential operators, see \eqref{piovespect}. 
These maps include para--differential operators, smoothing remainders which come from compositions of para--differential operators (see \eqref{smoospec}) and also  linear flows generated by para--differential operators. 
The  class of spectrally localized maps do  not enjoy a symbolic calculus 
and it is reminiscent of 
the maps 
introduced in Definition 1.2.1 in \cite{Del3}.

The class of  spectrally localized  maps is closed under transposition (Lemma \ref{omonon})
and under ``external" and ``internal" compositions, 
see Proposition \ref{composizioniTOTALIs}. A key property is that 
the  transpose of the internal differential of a spectrally localized map is a smoothing operator, see Lemma \ref{aggiunto}.

\begin{definition}{\bf (Spectrally localized maps)} \label{smoothoperatormaps}
Let $ m \in \R $, $p,N\in \N_0 $, $K,K'\in\N_0$ with $K'\leq K$ and $r >0$. 

(i) {\bf Spectrally localized $p$-homogeneous maps.}
We denote by 
 $\widetilde{\mathcal{S}}^{m}_{p}$
 the subspace of $m$-operators $S(\mathcal{U})$ in $\widetilde{\mathcal{M}}^{m}_{p}$ 
 satisfying the following spectral condition: there exist $\delta >0$, $C>1$ such that   for any 
 $(U_1,\ldots,U_{p})\in (\dot{H}^{\infty}(\T;\C^{2}))^{p}$,  for any  
 $U_{p+1}\in \dot{H}^{\infty}(\T;\C)$ and for any   $n_0,\ldots,n_{p+1} \in \N $ such that 
 $$
 \Pi_{n_0}S(\Pi_{n_1}U_1,\ldots,\Pi_{n_{p}}U_{p})\Pi_{n_{p+1}}U_{p+1}\neq 0 \, , 
 $$
it results
\be \label{specres}
\max\{ n_1, \dots , n_p\}\leq \delta  n_{p+1} \, , \qquad
 C^{-1} \, n_0 \leq n_{p+1} \leq C \, n_0 \, .  
\ee
We denote   $ \widetilde{\mS}_p := \bigcup_{m}\widetilde{\mS}^m_p $
and by $\Sigma_p^N \widetilde \mS^{m}_q$ the class of pluri-homogeneous spectrally localized maps of the form 
$\sum_{q=p}^{N} S_{q}$ with $ S_q \in  \widetilde{\mS}^{m}_{q}$
and 
$ \Sigma_p \widetilde \mS^{m}_q:= \bigcup_{N\in \N} \Sigma_p^N \widetilde \mS^{m}_q $.
For $ p \geq N + 1 $ we mean that the sum is empty. 

(ii) {\bf Non-homogeneous spectrally localized maps.}
 We denote   $\mathcal{S}^{m}_{K,K',p}[r]$ 
  the space of maps $(U,t,V)\mapsto S(U;t) V $ defined on $B^{K'}_{s_0}(I;r)\times I\times  C^0_{*}(I,\dot{H}^{s_0}(\T,\C))$ for some $ s_0 >0  $, 
  which are linear in the variable $ V $ and such that the following holds true. 
For any $s\in \R$ there are $C>0$ and 
  $r(s)\in]0,r[$ such that for any 
  $U\in B^K_{s_0}(I;r(s))\cap C^K_{*}(I,\dot{H}^{s}(\T,\C^2))$, 
  any $ V \in C^{K-K'}_{*}(I,\dot{H}^{s}(\T,\C))$, any $0\leq k\leq K-K'$, $t\in I$, we have that  
\begin{equation} \label{piovespect}
\|{\partial_t^k\left(S(U;t)V\right)(t,\cdot)}\|_{\dot{H}^{s- \frac32 k- m}}  \leq C \sum_{k'+k''=k} \|{U}\|_{k'+K',s_0}^{p} \|{V}\|_{k'',s}  \, .
\end{equation}  
In case $ p = 0$ we require  the estimate
$ \|{\partial_t^k\left(S(U;t)V\right)}\|_{\dot{H}^{s- \frac32 k-m}}
 \leq C  \|{V}\|_{k,s}$.
 
 We denote  $ \mS_{K,K',N}[r]= \bigcup_{m}\mS^{m}_{K,K',N}[r]  $.

(iii) {\bf Spectrally localized Maps.}
We denote by $\Sigma\mS^{m}_{K,K',p}[r,N]$, 
the space of maps $(U,t,V)\to S(U;t)V$ of the form 
\begin{equation}
\label{mapsS}
S(U;t)V=\sum_{q=p}^{N}S_{q}(U,\ldots,U)V+S_{>N}(U;t)V
\end{equation}
where $S_{q} $ are spectrally localized homogeneous maps in $ \widetilde{\mS}^{m}_{q}$, $q=p,\ldots, N $
and $S_{>N}$ is a non--homogeneous spectrally localized map   in  
$\mS^{m}_{K,K',N+1}[r] $.
We denote 
by  $\Sigma\mS_{K,K',p}^{m}[r,N]\otimes\mathcal{M}_2(\C)$
 the space of $2\times 2$ matrices whose entries are spectrally localized maps in
 $\Sigma\mS^{m}_{K,K',p}[r,N]$.
We will use also the notation
$ \Sigma \mS_{K,K',p}[r,N]:=\bigcup_{m\geq0}\Sigma \mS^{m}_{K,K',p}[r,N] $. 
\end{definition}

 \noindent
$\bullet$
Note that \eqref{specres} implies that 
$ \max \{n_1, \ldots, n_{p+1} \} \sim n_{p+1}$ and 
$ \max_2 \{n_1, \ldots, n_{p+1} \} \sim \max \{n_1, \ldots, n_{p} \}$
and therefore, by  
\eqref{omomap}, 
 $S \in \wt\cS^m_p$ if and only if there are $ \mu \geq 0 $, $ C > 0 $ such that  
\be\label{hom:sest}
 \|\Pi_{n_0}S(\Pi_{n_1}U_1,\ldots,\Pi_{n_{p}}U_{p})\Pi_{n_{p+1}}U_{p+1}\|_{L^{2}}\leq
 C{\rm \max}( n_1,\ldots, n_{p})^{\mu}  \, n_{p+1}^{m}
 \prod_{j=1}^{p+1}\|\Pi_{n_{j}}U_{j}\|_{L^{2}} 
\ee 
for any $ n_0, \ldots, n_{p+1} \in \N $.
 
\noindent
$\bullet$
 In view of Lemma \ref{carasmoofou} a matrix of spectrally localized maps $S(U)$  is characterized in terms of its Fourier coefficients  as follows: $S(U)$ is in $ \widetilde{\mathcal{S}}_{p}^{m} \otimes \cM_2(\C)$ if and only if   $S_{\vec \jmath_p, j, k}^{\vec \sigma_p, \sigma', \sigma}$ (defined as in \eqref{M.coeff}) satisfy,
for some $ \mu \geq 0 $, $ C > 0 $,   
\be\label{smoocaraspec}
 |S_{\vec \jmath_p, j, k}^{\vec \sigma_p, \sigma', \sigma} |\leq C 
{\rm max}\{ |j_1|,\dots, | j_p|\}^{\mu}|j|^{m} \, , \quad \forall (\vec \jmath_p ,  j,k) \in 
(\Z \setminus \{0\})^{p+2} \, , \
(\vec \sigma_p , \sigma',\sigma) \in  \{\pm \}^{p+2} \, , 
 \ee
 and if $S_{\vec \jmath_p, j, k}^{\vec \sigma_p, \sigma', \sigma}\not=0$ then  \eqref{momentoj}, \eqref{M.coeff.p} hold and, for some $ \delta > 0 $, 
 \be \label{smoocaraspecres}
\max\{ |j_1|, \dots , |j_p|\}\leq \delta  |j| \, , \quad
 C^{-1} \, |k| \leq |j|\leq C \, |k| \, .
 \ee
 
 \noindent
 $\bullet$
If $ S $ is a   $p$--homogeneous 
spectrally localized map in $ \widetilde \mS^m_p$ then
$ S(U)[V] $ defines a non--homogeneous spectrally localized map  in $ \mS_{K,0,p}^{m}[r]$ for any $r>0$. The proof is similar to the one of Lemma \ref{mappabonetta} noting  that the estimate \eqref{piovespect} holds for any $ s\in \R$ because of the equivalence $ n_0\sim n_{p+1}\sim \max\{ n_1,\dots, n_{p+1}\}$ in \eqref{specres}.

 \noindent
$\bullet$ 
{\bf (Paradifferential operators as spectrally localized maps)}
If $a(U;t,x,\xi)$ is a symbol in $ \Sigma \Gamma^{m}_{K,K',p}[r,N]$  
then the paradifferential operator  $ \opbw(a(U;t,x,\xi))$ is a spectrally localized map in  
$ \Sigma  \mS_{K,K',p}^m[r,N]$. This is a consequence of Proposition 3.8 in \cite{BD}. 
We remark that for a homogeneous symbol this is a consequence of the 
choice of the first quantization in 
\eqref{BW}. 

 \noindent
$\bullet$ 
If $S(U, \dots ,U)$ is a  $p$--homogeneous spectrally localized map in $ \wtcS_p$, then  the differential of the non--linear map $ S(U, \dots, U) U $, $
\di_U \big(S(U,\dots, U) U\big) V   = p S(V,U,\dots, U)U + S(U, \dots, U)V 
$ is a $p$--homogeneous  operator in $ \wtcM_p$, \emph{not} necessarily  spectrally localized. Indeed the operator $S(V,U,\dots, U)U $  is \emph{not}, in general, spectrally localized.

 \noindent
$\bullet$
 If $m_1 \leq m_2$ then $\Sigma\mathcal{S}^{m_1}_{K,K',p}[r,N] \subseteq \Sigma\mathcal{S}^{m_2}_{K,K',p}[r,N]$.
 If $K_1' \leq K'_2$ then $\Sigma\mathcal{S}^{m}_{K,K'_1,p}[r,N] \subseteq \Sigma\mathcal{S}^{m}_{K,K'_2,p}[r,N]$.

\begin{remark}
The constant $\delta>0$ in the spectral condition \eqref{specres} is not assumed to be small (unlike the one in \eqref{paraspecloc} for  paradifferential operators).
\end{remark}

The class of matrices of  spectrally localized  homogeneous maps is closed under transposition.

\begin{lemma} \label{omonon}
Let $p \in \N_0$, $m \in \R$.
If $ S (U) $ is a  matrix of  $p$--homogeneous 
spectrally localized maps in $ \widetilde \mS^m_p \otimes\mM_2(\C) $ then
 $ S(U)^\top $ is in $ \widetilde \mS^m_p\otimes\mM_2(\C) $, where the transpose is computed with respect to the non-degenerate real bilinear form \eqref{real.bil.form.intro}.
\end{lemma}

\begin{proof}
It results 
$$
|(S^\top)_{\vec \jmath_{p}, j,k}^{\vec \sigma_{p}, \sigma', \sigma}|\stackrel{\eqref{emmetop}}{=}|S_{\vec \jmath_{p},- k,-j}^{\vec \sigma_{p}, \sigma, \sigma'}|\stackrel{\eqref{smoocaraspec}}{\leq} C 
{\rm max}\{ |j_1|,\dots, | j_p|\}^{\mu}|k|^{m}\stackrel{\eqref{smoocaraspecres}}{\leq} C' 
{\rm max}\{ |j_1|,\dots, | j_p|\}^{\mu}|j|^{m}
$$
which means that $(S^\top)_{\vec \jmath_{p}, j,k}^{\vec \sigma_{p}, \sigma', \sigma}$ satisfies \eqref{smoocaraspec}. Since $ |j|\sim |k|$ then $\max\{ |j_1|,\dots, |j_p|\}\leq \delta C |k|$ hence \eqref{smoocaraspecres} holds for $(S^\top)_{\vec \jmath_{p}, j,k}^{\vec \sigma_{p}, \sigma', \sigma}$.
\end{proof}

We now prove some further composition results for spectrally localized maps. 

\begin{proposition}{\bf (Compositions of spectrally localized maps)}\label{composizioniTOTALIs}
Let $p, p', p'',  N, K, K' \in \N_0$ with $K' \leq K$ and $r>0$. 
Let $m,m'\in \R$ and $S(U;t)$ be  a spectrally localized map in $ \Sigma\mathcal{S}^{m}_{K,K',p}[r,N]$. Then 
\begin{itemize}
\item[$(i)$]
if $ R(U;t)$ is a smoothing operator in $ \Sigma\mathcal{R}^{-\varrho}_{K,K',p''}[r,N]$ for some $\varrho \geq 0$,  then  $S(U;t) \circ R(U;t)  $ and $R(U;t) \circ S(U;t)$ 
are  smoothing operators in $\Sigma\mathcal{R}^{-\varrho+\max(m,0)}_{K,K',p+p''}[r,N] $.
\item[$(ii)$] If  $S'(U;t)$ is in  $ \Sigma\mathcal{S}^{m'}_{K,K',p'}[r,N]$ then the composition  $S(U;t)\circ S'(U;t)$ is  in $ \Sigma\mathcal{S}^{m+m'}_{K,K',p+p'}[r,N]$.

\item[$(iii)$] 
If $S $ is in $ \wt\cS^m_p$ and  $S^{(a)}(U;t)$ are spectrally  localized maps   in $\Sigma \cS_{K,K',q_a}^{\ell_a}[r,N]$ for some $\ell_a \in \R $ and $ q_a\in \N$,   $a=1, \dots, p $,  then also the internal composition  
\be\label{compomap} 
S(S^{(1)}(U;t)U, \dots, S^{(p)}(U;t)U)
\ee
is a spectrally localized map  in $ \Sigma \mS^{m}_{K,K',p+ \bar q}[r,N]$ 
with  $ \bar q:= \sum_{a=1}^{p} q_a$. 
\item[$(iv)$] 
If $S(U;t) $ is in $ {\mathcal{S}}_{K,0,p}^{m}[\breve r]$ for any $ \breve r \in \R^+ $ 
and $ \bM_0(U;t) $ belongs to  $ \mM^0_{K,K',0}[r]\otimes \mM_2(\C)$, 
then $S(\bM_0(U;t)U;t)$ is in  $  \mS^{m}_{K,K',p}[r]$.

\item[$(v)$] If $S(U)$ is a matrix of 
$p$-homogeneous spectrally localized maps in $\wt \cS_p^m
\otimes \cM_2(\C) $ and $S_a(U;t)$ are  
in $ \mS^{m_a}_{K,K',q_a}[r] \otimes \cM_2(\C) $, $a = 1,2$, $m_a \in \R$,  then 
$$
\di_U \left. (S(U)U)\right|_{S_1(U;t)U} \, [S_2(U;t)U] 
$$ 
is in 
$\cS^{m +\max(m_1, m_2)}_{K, K',p+pq_1 + q_2}[r] \otimes \cM_2(\C) $.
\end{itemize}
\end{proposition}

\begin{proof}
 {\sc Proof of $(i)$.} The operator $R(U;t) \circ S(U;t)$ 
is in $\Sigma\mathcal{R}^{-\varrho+\max(m,0)}_{K,K',p+p''}[r,N] $
by Proposition \ref{composizioniTOTALI}-($i$) since 
$S(U;t)$ is in $ \Sigma\mathcal{M}^{m}_{K,K',p}[r,N]$. 
Then we prove that $ S(U;t)  \circ R(U;t) $  is in $\Sigma\mathcal{R}^{-\varrho+\max(m,0)}_{K,K',p+p''}[r,N] $.  It is sufficient to consider the case 
$ m\geq 0$ since, if $ m<0$, we regard $S(U;t)$ as a spectrally localized map in $ \Sigma\mathcal{S}^{0}_{K,K',p}[r,N]$. Decomposing $S= \sum_{q=p}^N S_q+ S_{>N}$ as in \eqref{mapsS} and $R= \sum_{q=p''}^N R_q+ R_{>N}$ as in \eqref{maps}, we have to show, on the one hand that 
\be \label{omoomo}
S_{q_1}(U_1,\dots, U_{q_{1}}) R_{q_2}( U_{q_1+1},\dots, U_{q_1+q_2})
\ee
is a homogeneous smoothing operator in $ \wtcR^{-\vr+ m}_{q_1+q_2}$ if $ q_1+q_2\leq N$ and, on the other hand, that 
\be \label{nonomo}
\begin{aligned}
& S_{q_1}(U,\dots, U) R_{q_2}( U,\dots, U), \ q_1+q_2\geq N+1, \\
& S_{>N}(U;t) R_{q_2}( U,\dots, U), \ q_2=p'',\dots, N, \\ 
&S_{q_1}(U,\dots, U) R_{>N}( U;t), \ q_1=p,\dots, N, \\ 
&S_{>N}(U;t) R_{>N}( U;t) 
\end{aligned}
\ee
are non--homogeneous smoothing operators in $ \mR^{-\vr +m}_{K,K',N+1}[r]$. 
We first study \eqref{omoomo}. First of all one has 
\begin{align}
&\|\Pi_{n_0}S_{q_1}(\Pi_{n_1}U_1,\dots, \Pi_{n_{q_1}}U_{q_{1}}) R_{q_2}( \Pi_{n_{q_1+1}}U_{q_1+1},\dots, \Pi_{n_{q_1+q_2}}U_{q_1+q_2})\Pi_{n_{q_1+q_2+1}}U_{q_1+q_2+1}\|_{L^2} \label{inininni}\\
&\leq \sum_{n'} \|\Pi_{n_0}S_{q_1}(\Pi_{n_1}U_1,\dots, \Pi_{n_{q_1}}U_{q_{1}}) \Pi_{n'}R_{q_2}( \Pi_{n_{q_1+1}}U_{q_1+1},\dots, \Pi_{n_{q_1+q_2}}U_{q_1+q_2})\Pi_{n_{q_1+q_2+1}}U_{q_1+q_2+1}\|_{L^2}.\notag
\end{align}
Thanks to the conditions \eqref{omoresti2}, \eqref{specres} for $S$ 
and \eqref{omoresti2} for $R$ 
the indices in the above sum satisfy, 
for some choice of signs $\sigma_a$, $a=1,\dots, q_1+q_2+1$,  the restrictions 
\be\label{scece}
n'= \sum_{a=0}^{q_1} \sigma_a n_a, \quad n'= \sum_{a=q_1+1}^{q_1+q_2+1} \sigma_a n_a, \quad \max\{ n_1, \dots, n_{q_1}\}\lesssim n', \quad n_0\sim n' \, . 
\ee
As a consequence \eqref{omoomo} satisfies the corresponding condition \eqref{omoresti2} and 
\be \label{denaco}
 n' \lesssim \max\{ n_{q_1+1}, \dots, n_{q_1+q_2+1}\}, \quad \max\{ n_{1}, \dots, n_{q_1+q_2+1}\}\sim\max\{ n_{q_1+1}, \dots, n_{q_1+q_2+1}\}  . \ee
  We also claim that 
 \be \label{claimo}
  \max\{ n_1, \dots, n_{q_1}\}\lesssim{\rm max}_2\{ n_{1}, \dots, n_{q_1+q_2+1}\}.
 \ee
Indeed, by  \eqref{scece} and \eqref{denaco},  we have 
 \be \label{igerla}
 \max\{ n_1, \dots, n_{q_1}\}\lesssim \max\{ n_{q_1+1}, \dots, n_{q_1+q_2+1}\},
 \ee 
and we distinguish two cases. If $\max\{ n_{1}, \dots, n_{q_1+q_2+1}\}= \max\{ n_{q_1+1}, \dots, n_{q_1+q_2+1}\}$ then we directly obtain   $\max\{ n_1, \dots, n_{q_1}\}\leq {\rm max}_2\{ n_{1}, \dots, n_{q_1+q_2+1}\}$ which gives \eqref{claimo}.
If $\max\{ n_{1}, \dots, n_{q_1+q_2+1}\}= \max\{ n_{1}, \dots, n_{q_1}\}$ then $ \max\{ n_{q_1+1}, \dots, n_{q_1+q_2+1}\}\leq {\rm max}_2\{ n_{1}, \dots, n_{q_1+q_2+1}\}$ which together with \eqref{igerla} proves \eqref{claimo}.
 Using \eqref{inininni}, \eqref{scece}, \eqref{hom:sest}, \eqref{omomap} (with $m=-\vr$ and $ \mu \leadsto \mu'$)  we get  
\begin{align*}
\|\Pi_{n_0}&S_{q_1}(\Pi_{n_1}U_1,\dots, \Pi_{n_{q_1}}U_{q_{1}}) R_{q_2}( \Pi_{n_{q_1+1}}U_{q_1+1},\dots, \Pi_{n_{q_1+q_2}}U_{q_1+q_2})\Pi_{n_{q_1+q_2+1}}U_{q_1+q_2+1}\|_{L^2}\\
&\leq C {\rm max}\{ n_1,\dots, n_{q_1}\}^{\mu} (n')^{m}\frac{\max_2\{ n_{q_1+1},\dots, n_{q_1+q_2+1}\}^{\mu'}}{ \max\{ n_{q_1+1},\dots, n_{q_1+q_2+1}\}^{\vr}} \prod_{a=1}^{q_1+q_2+1} \| \Pi_{n_a} U_a\|_{L^2}\\ 
&\stackrel{\eqref{claimo}, \eqref{denaco}, m\geq 0}{\leq} C \frac{\max_2\{ n_{1},\dots, n_{q_1+q_2+1}\}^{\mu'+\mu }}{ \max\{ n_{1},\dots, n_{q_1+q_2+1}\}^{\vr-m}}\prod_{a=1}^{q_1+q_2+1} \| \Pi_{n_a} U_a\|_{L^2}.
\end{align*}
This proves that \eqref{omoomo} is in $ \wtcR^{-\vr+ m}_{q_1+q_2} $. 

In order to prove that the operators in \eqref{nonomo} satisfy \eqref{piove} (with $m \leadsto -\varrho +m$ and $ p \leadsto N $) we recall that  $S_{q_1}(U,\dots, U)$ defines a spectrally localized map  in $ \mS^{m}_{K,0,q_1}[r]$ (as remarked below Definition \ref{smoothoperatormaps}) and $ R_{q_2}(U,\dots,U)$ defines a smoothing operator  in $\mR^{-\vr}_{K,0,q_2}[r]$ thanks to Lemma \ref{mappabonetta}. 
Then the thesis follows by estimates \eqref{piovespect} and \eqref{piove}. For instance consider the last term in \eqref{nonomo}.
For any
 $ k= 0, \ldots,  K-K'$
$$
\begin{aligned}
\|{\partial_t^k\left(S_{>N}(U;t) R_{>N}(U;t)V\right)}\|_{\dot{H}^{s- \frac32 k- m +\varrho}}  
 \leq C \sum_{k'+k''=k} \|U\|_{k'+K',s_0}^{N+1} \|R_{>N}(U;t)V\|_{k'',s+\varrho}  \\
\leq C \sum_{k'+k'' = k} \sum_{0 \leq j \leq k'' \atop j'+j''  = j}
\|U\|_{k'+K',s_0}^{N+1} \, 
\left( 
\|{V}\|_{j'',s} \| U \|_{ j'+K', s_0}^{N+1}  
+   \|{V}\|_{j'',s_0} \| U \|_{ j'+K', s_0}^{N} \| U \|_{ j'+K', s} 
\right)
\\
 \leq  C \sum_{j''=0}^k  \|{V}\|_{j'',s}   \| U\|_{k-j''+K',s_0}^{2N+2}  
+ 
 \|{V}\|_{j'',s_0}   \| U \|_{k- j''+K', s_0}^{2N+1} \| U \|_{ k-j''+K', s}
\end{aligned}
$$
using  that  $k' \leq k - j''$ and $ j' \leq k - j''$. This proves that \eqref{omoomo} 
is in $ \mR^{-\vr +m}_{K,K',N+1}[r]$ (as $ \| U\|_{K,s_0} < 1 $).

\noindent {\sc Proof of $(ii)$. }  
Decomposing $S= \sum_{q=p}^N S_q+ S_{>N}$ and $S'= \sum_{q=p'}^N S'_q+ S'_{>N}$ as in \eqref{mapsS}, we have to show, on the one hand that 
\be \label{omoomo3}
S_{q_1}(U_1,\dots, U_{q_{1}}) S'_{q_2}( U_{q_1+1},\dots, U_{q_1+q_2})
\ee
is a homogeneous spectrally localized map in $ \wtcS^{ m+m'}_{q_1+q_2}$ if $ q_1+q_2\leq N$ and, on the other hand, that 
\be \label{nonomo3}
\begin{aligned}
& S_{q_1}(U,\dots, U) S'_{q_2}( U,\dots, U), \ q_1+q_2\geq N+1, \\
& S_{>N}(U;t) S'_{q_2}( U,\dots, U), \ q_2=p',\dots, N, \\ 
&S_{q_1}(U,\dots, U) S'_{>N}( U;t), \ q_1=p,\dots, N, \\ 
&S_{>N}(U;t) S'_{>N}( U;t) 
\end{aligned}
\ee
are non--homogeneous spectrally localized map in $ \mS^{m +m'}_{K,K',N+1}[r]$. 
We first prove that \eqref{omoomo3} is in $ \wtcS^{ m+m'}_{q_1+q_2}$. First of all one has 
\begin{align}\label{triangolo3}
&\|\Pi_{n_0}S_{q_1}(\Pi_{n_1}U_1,\dots, \Pi_{n_{q_1}}U_{q_{1}}) S'_{q_2}( \Pi_{n_{q_1+1}}U_{q_1+1},\dots, \Pi_{n_{q_1+q_2}}U_{q_1+q_2})\Pi_{n_{q_1+q_2+1}}U_{q_1+q_2+1}\|_{L^2}\\
&\leq \sum_{n'} \|\Pi_{n_0}S_{q_1}(\Pi_{n_1}U_1,\dots, \Pi_{n_{q_1}}U_{q_{1}}) \Pi_{n'}S'_{q_2}( \Pi_{n_{q_1+1}}U_{q_1+1},\dots, \Pi_{n_{q_1+q_2}}U_{q_1+q_2})\Pi_{n_{q_1+q_2+1}}U_{q_1+q_2+1}\|_{L^2}.\notag
\end{align}
Thanks to the conditions \eqref{omoresti2}, \eqref{specres} for $S$ and $S'$ the indices in the above sum satisfy, for some choice of signs $\sigma_a$, $a=1,\dots, q_1+q_2+1$,  the restriction 
\be 
\begin{aligned}
\label{denaco3}
&n'= \sum_{a=0}^{q_1} \sigma_a n_a, \quad n'= \sum_{a=q_1+1}^{q_1+q_2+1} \sigma_a n_a, \\
& \begin{cases}
\max\{ n_1,\dots, n_{q_1}\} \leq \delta n'\\
 \frac{n_0}{C} \leq  n'\leq C n_0 \, , 
 \end{cases} \quad \begin{cases}\max\{ n_{q_1+1},\dots, n_{q_1+q_2}\} \leq \delta' n_{q_1+q_2+1} \\ 
 \frac{ n'}{C'} \leq  n_{q_1+q_2+1}\leq C' n' \, , 
 \end{cases} 
\end{aligned}
\ee
for some $ C,C'>1$ and $\delta, \delta'>0$.
Therefore
$$ 
\begin{cases}
\max\{ n_1,\dots, n_{q_1+q_2}\} \leq  \max\{ \delta' , \delta C'\}  n_{q_1+q_2+1}\\
\frac{n_0}{C C'}\leq n_{q_1+q_2+1} \leq C C' n_0 \, . 
\end{cases} 
$$
This proves that $ S_{q_1}\circ S'_{q_2} $ fulfills the localization property of Definition \ref{smoothoperatormaps} $(i)$ (see \eqref{specres}). 
In addition \eqref{omoomo3} satisfies the corresponding condition \eqref{omoresti2}. Using  \eqref{triangolo3}, \eqref{hom:sest}  we get  
\begin{align*}
&\|\Pi_{n_0}S_{q_1}(\Pi_{n_1}U_1,\dots, \Pi_{n_{q_1}}U_{q_{1}}) S'_{q_2}( \Pi_{n_{q_1+1}}U_{q_1+1},\dots, \Pi_{n_{q_1+q_2}}U_{q_1+q_2})\Pi_{n_{q_1+q_2+1}}U_{q_1+q_2+1}\|_{L^2}\\
&\leq C {\rm max}\{ n_1,\dots, n_{q_1}\}^{\mu} (n')^{m}{\rm max}\{ n_{q_1+1},\dots, n_{q_1+q_2}\}^{\mu'}  n_{q_1+q_2+1}^{m'} \prod_{a=1}^{q_1+q_2+1} \| \Pi_{n_a} U_a\|_{L^2}\\ 
&\stackrel{\eqref{denaco3}}{\leq} C \max\{ n_{1},\dots, n_{q_1+q_2}\}^{\mu'+\mu}(n_{q_1+q_2+1})^{m+m'}\prod_{a=1}^{q_1+q_2+1} \| \Pi_{n_a} U_a\|_{L^2}
\end{align*}
which proves that  $ S_{q_1}\circ S'_{q_2}$ satisfies \eqref{hom:sest}.
In order to prove that the terms in \eqref{nonomo} satisfy \eqref{piovespect} we first note that, thanks to $(i)$ of Lemma \ref{omonon}, we have that $S_{q_1}(U,\dots, U) \in \mS^{m}_{K,0,q_1}[r]$ and $ S'_{q_2}(U,\dots,U) \in \mS^{m'}_{K,0,q_2}[r]$ and 
then the thesis follows using  \eqref{piovespect}.
For instance consider the first term in \eqref{nonomo3}. 
Using twice \eqref{piovespect} (with $K' = 0$) we get, for any $k= 0, \ldots,  K $, 
$$
\begin{aligned}
\|{\partial_t^k\left(S_{q_1}(U;t)S'_{q_2}(U;t)V\right)}\|_{\dot{H}^{s- \frac32 k- m_1-m_2}}  
& \leq C \sum_{k'+k''=k} \|U\|_{k',s_0}^{q_1} \|S_{q_2}'(U;t)V\|_{k'',s-m_2}  \\
&
\leq C \sum_{k'+k'' = k} \sum_{0 \leq j \leq k'' \atop j'+j''  = j}
\|U\|_{k',s_0}^{q_1} \, \| U \|_{ j', s_0}^{q_2} \, \|{V}\|_{j'',s}
\\
& \leq  C \sum_{j''=0}^k \| U\|_{k-j'',s_0}^{q_1+ q_2}   \|{V}\|_{j'',s}\\
\end{aligned}
$$
where in the last step we used that 
$ j' \leq k - j''$ and $k' \leq k - j''$.
The last line is $\eqref{piovespect}$ with  $ p $ replaced by $q_1+q_2 \geq N+1 $.

\noindent {\sc Proof of $(iii)$:} We now consider the internal composition \eqref{compomap} in the homogeneous case. 
For simplicity of notation we consider the case  $q_2=\dots= q_{p}=0$,  $ S^{(2)}=\dots= S^{(p)}=\uno$ and $ q_1= \bar q =: q$, the general case follows in the same way. 
So we need to show that 
$S(S^{(1)}(U)U,U, \ldots, U)$
is a spectrally localized map in $ \widetilde{\mS}^{m}_{p+ q}$. We first estimate
\begin{align}
&\| 
\Pi_{n_{0}} S\big(S^{(1)}(\Pi_{n_{1}}U_{1},\dots, \Pi_{n_{q}}U_{q}) \Pi_{n_{q+1}}U_{q+1} , \Pi_{n_{q+2}}U_{q+2} ,\dots, \Pi_{n_{q+p}}U_{q+p}\big)\Pi_{n_{p+q+1}}U_{p+q+1}
\|_{L^2}\label{triangolo4} \\
&\leq 
\sum_{n' }
\| 
\Pi_{n_{0}} S\big(\Pi_{n'}S^{(1)} (\Pi_{n_{1}}U_{1},\dots, \Pi_{n_{q}}U_{q}) \Pi_{n_{q+1}}U_{q+1} , \Pi_{n_{q+2}} U_{q+2},\dots, \Pi_{n_{q+p}}U_{q+p}\big)\Pi_{n_{p+q+1}}U_{p+q+1}
 \|_{L^2}. \notag
\end{align}
Thanks to the conditions \eqref{omoresti2} for $S$ and $S^{(1)}$ the indices in the above sum satisfy, for some choice of signs $\sigma_a$, $a=1,\dots, p+q+1$,  the restrictions 
$ n'= \sum_{a=1}^{q+1} \sigma_a n_a $ and $  n'= \sigma_0 n_0 + \sum_{a=q+2}^{p+q+1} \sigma_an_a $, 
proving that  $S(S^{(1)}(U)U,U, \ldots, U)$ fulfills \eqref{omoresti2}.
Moreover the condition \eqref{specres} for $S$ and $ S^{(1)}$ imply  the existence of $\delta, \delta_1>0$, $C,C_1 >1$ such that 
\be\label{doppiacondizione}
 \begin{cases}\max\{ n',n_{q+2},\dots, n_{q+p}\} \leq \delta  n_{q+p+1}\\ 
\frac{n_{0}}{C}\leq n_{p+q+1} \leq C n_0 \, , 
\end{cases}
\quad 
\begin{cases}
\max\{ n_1,\dots, n_q\} \leq \delta_1  n_{q+1}\\
\frac{n'}{C_1}\leq n_{q+1} \leq C_1 n' \, , \end{cases}
\ee
and therefore 
$ \max\{ n_1,\dots, n_{p+q}\}\leq \max( \delta \delta_1  C_1 , \delta C_1)  n_{q+p+1} 
$ and $  \frac{n_{0}}{C}\leq n_{p+q+1} \leq C n_0  $, 
proving that $S(S^{(1)}(U)U,U, \ldots, U)$ fulfills  \eqref{specres}. 
We now prove it fulfills also \eqref{hom:sest}.
 We get 
\begin{align*} 
&\| 
\Pi_{n_{0}} S(S^{(1)}(\Pi_{n_{1}}U_{1},\dots, \Pi_{n_{q}}U_{q}) \Pi_{n_{q+1}}U_{q+1} , \Pi_{n_{q+2}}U_{q+2} ,\dots, \Pi_{n_{q+p}}U_{q+p})\Pi_{n_{p+q+1}}U_{p+q+1}
\|_{L^2}\\
& 
\stackrel{\eqref{triangolo4}, \eqref{hom:sest}}{\lesssim} \sum_{n'} 
 \max(n',n_{q+2},  \ldots, n_{q+p})^\mu \, n_{p+q+1}^{m}  \, \max(n_{1}, \ldots, n_{q})^{\mu'} \, n_{q+1}^{\ell_1} \, \prod_{a=1}^{p+q+1} \| \Pi_{n_a}U_a \|_{L^2} \\
  &\stackrel{\eqref{doppiacondizione}, n_{q+1} \sim n' }{\lesssim}  
 \sum_{n' } 
 \max(n_{q+1}, n_{q+2},  \ldots, n_{q+p})^{\mu + \max(0,\ell_1)} \, n_{p+q+1}^{m}  \, \max(n_{1}, \ldots, n_{q})^{\mu'} \,  \prod_{a=1}^{p+q+1} \| \Pi_{n_a}U_a \|_{L^2} \\
 &\lesssim    \max(n_1, \ldots, n_{p+q})^{\mu+\mu'+\max(0,\ell_1)} \, n_{p+q+1}^{m} \, \prod_{a=1}^{p+q+1} \| \Pi_{n_a}U_a \|_{L^2} \ ,
\end{align*}
proving that 
$S(S^{(1)}(U)U,U, \ldots, U)$
is a spectrally localized map in $ \widetilde{\mS}^{m}_{p+q}$.
Finally $S(S^{(1)}(U)U,U, \ldots, U)$ satisfies also  \eqref{def:R-trin}, concluding the proof that it
is a spectrally localized map in $ \widetilde{\mS}^{m}_{p+q}$.

{\noindent \sc Proof of $(iv)$:}  By the estimate below \eqref{piove} for $\bM_0(U;t)$, for any $U \in B_{s_0}^K(I;r)$ and any $k=0,\dots , K-K'$, $\| \bM_0(U;t)U\|_{k,s_0} \lesssim \|U \|_{k,s_0}$. Then estimate \eqref{piovespect} for $S(\bM_0(U;t)U;t)$ for any $ 0\leq k \leq K-K'$ follows from the ones for $ S(U;t)$ arguing as in $(iii)$ of Proposition \ref{composizioniTOTALI}.

\noindent{\sc Proof of $(v)$:} It follows computing explicitly the differential $\di_U (S(U)U)[V]$, evaluating it at $ U \leadsto  S_1(U;t)U$ and $V \leadsto  S_2(U;t)U$ and using item $(ii)$ and $(iii)$ of the proposition.
\end{proof}

The following lemma proves that  the internal composition 
of a spectrally localized map with a  map, is a spectrally localized map
plus a smoothing operator whose transpose is another smoothing operator.

\begin{lemma} \label{decomposta}
Let $p \in \N$, $q \in \N_0$, $m \in \R$ and $m' \geq 0$.
Let $S(U) $ be a matrix of spectrally localized homogeneous maps  in 
$ \widetilde\mS_p^m \otimes \cM_2(\C)$ and $ M(U) $ be a matrix of 
homogeneous $m'$-operators in $  \widetilde \mM_q^{m'}\otimes \cM_2(\C)$. Then  
\be \label{decompo}
S(M(U)U, U, \dots , U )= S'(U) + R(U) 
\ee
where 
\begin{itemize}
\item 
$S'(U)$  is a matrix of spectrally localized homogeneous maps   in 
$\widetilde\mS_{p+q}^m \otimes \cM_2(\C)$;
\item  $R(U)$ is a matrix of homogeneous smoothing operators in $ \widetilde \mR^{-\vr}_{p+q} \otimes \cM_2(\C)$ for any $ \vr \geq 0 $, 
and $ R(U)^\top$ is  a matrix of homogeneous smoothing operators 
  in $ \widetilde \mR^{-\vr}_{p+q}\otimes \cM_2(\C)$ for any $\varrho \geq 0$ as well. 
  \end{itemize} 
\end{lemma}

\begin{proof}
By multilinearity    we expand 
\begin{align}
&S\big(M(U_1,\dots,  U_q) U_{q+1},  U_{q+2}, \dots , U_{p+q} \big) \label{SMUU} \\
&= \!\!\!\!\!\! \sum_{(n',\, n_0,\,  \dots, n_{p+q+1}) 
\in \mN} \!\!\!\!\!\!
\Pi_{n_0}S\big(\Pi_{n'}M(\Pi_{n_1} U_1,\dots, \Pi_{n_q} U_q)\Pi_{n_{q+1}} U_{q+1}, \Pi_{n_{q+2}} U_{q+2}, \dots , \Pi_{n_{p+q}}U_{p+q} \big) \Pi_{n_{p+q+1}}  \notag
\end{align}
where  $ \mN$ is a subset of $\N^{ p+q+3}$ made by indexes fulfilling, by  restrictions \eqref{specres} and \eqref{omoresti2}, the following conditions: there exist 
signs $\{ \sigma_j, \epsilon_j, \epsilon' \} \subset \{ \pm  \} $
and constants  $\delta >0$, $ C> 1$ such that 
\be \label{enne}
\begin{cases} 
 \max(n',n_{q+2}, \dots,  n_{p+q})\leq \delta n_{p+q+1} &(a)\\
  \quad C^{-1} n_0 \leq  n_{p+q+1}\leq C n_0  &(b)\\ 
 n'= \sum_{j=1}^{q+1} \sigma_j n_j, \quad n_0= \epsilon' n'+  \sum_{j=q+2}^{p+q+1} \epsilon_j n_j \, &(c) \ . 
 \end{cases}  
\ee 
We fix $ \delta'>\delta$ and we denote by $\mN'$ the subset of $\mN$ made by indices which satisfy the  additional restriction
\be \label{additional}
\max(n_1, \dots,  n_{q+1})\leq \delta' n_{p+q+1} \, . 
\ee
Then we define 
\begin{align}
&S'(U_1, \dots, U_{p+q}) \label{DEFSUP}\\
&:= \!\!\!\!\!\!\!\! \sum_{(n',\, n_0,\,  \dots, n_{p+q+1})\in \mN'}
\!\!\!\!\!\!\!\! \Pi_{n_0}
S\big(\Pi_{n'}M(\Pi_{n_1} U_1,\dots, \Pi_{n_q} U_q)\Pi_{n_{q+1}} U_{q+1}, \Pi_{n_{q+2}} U_{q+2}, \dots , \Pi_{n_{p+q}}U_{p+q} \big) \Pi_{n_{p+q+1}}. \notag
\end{align}
By \eqref{enne} and \eqref{additional} one has that $S'(U)$ fulfills 
the spectral condition 
$n_0 \sim n_{p+q+1}$, $ \max(n_1, \ldots, n_{p+q}) \leq \delta' n_{p+q+1}$
which is the condition \eqref{specres}.
Moreover using \eqref{hom:sest} and \eqref{omomap} we  bound
\begin{align}
\notag
\Big\| 
\Pi_{n_0}
S\big(\Pi_{n'}M(\Pi_{n_1} U_1,\dots, \Pi_{n_q} U_q)\Pi_{n_{q+1}} U_{q+1}, \Pi_{n_{q+2}} U_{q+2}, \dots , \Pi_{n_{p+q}}U_{p+q} \big) \Pi_{n_{p+q+1}}U_{p+q+q} \Big\|_{L^2}\\
\label{Senne3}
\lesssim 
\max(n_1, \ldots, n_{p+q})^{\mu+ m'}  \, n_{p+q+1}^m \, \prod_{a=1}^{p+q+1} \norm{\Pi_{n_a} U_a}_{L^2}
\end{align}
for some $\mu >0$.
Finally by \eqref{enne}$(c)$ one has also that \eqref{omoresti2} holds.
One checks that also  \eqref{def:R-trin} holds true. 
We have proved that $S'(U)$ is a matrix of spectrally localized maps in 
$\wt \cS^m_{p+q}\otimes \cM_2(\C)$.

Then, recalling \eqref{SMUU} and \eqref{DEFSUP}, we define 
\begin{align}
&R(U_1,\dots, U_{p+q}) \label{defRMA}\\
&:= \!\!\!\! \!\!\!\! \sum_{(n',\, n_0,\,  \dots, n_{p+q+1})\in \mN\setminus \mN'}
\!\!\!\!\!\!\!\!\!\! \! \! 
\Pi_{n_0}S\big(\Pi_{n'}M(\Pi_{n_1} U_1,\dots, \Pi_{n_q} U_q)\Pi_{n_{q+1}} U_{q+1}, \Pi_{n_{q+2}} U_{q+2}, \dots , \Pi_{n_{p+q}}U_{p+q} \big) \Pi_{n_{p+q+1}}. \notag 
\end{align}
We claim that  there is $ C'>0$ such that  
 if $ (n',\, n_0,\,  \dots, n_{p+q+1})\in \mN\setminus \mN'$ then 
\be \label{smotrasmo}
\begin{cases}
\max(n_{q+2}, \dots,  n_{p+q})\leq \delta' n_{p+q+1}& (a)\\
 C^{-1} n_0\leq n_{p+q+1} \leq C n_0&(b)\\
n_{p+q+1} \leq C' \max_2( n_1, \dots, n_{q+1})&(c)\\ 
 \max( n_1, \dots, n_{q+1})\leq C' \max_2( n_1, \dots, n_{q+1})  &(d)
 \, . 
\end{cases} 
\ee
Before proving  \eqref{smotrasmo} we note that it implies 
$$
\begin{aligned}
\max(n_{1}, \dots,  n_{p+q+1})&\stackrel{(a)}{\leq } (1+\delta') \max(n_1, \dots, n_{q+1}, n_{p+q+1})\\
&\stackrel{(c)+(d)}{\leq} (1+\delta')C'  {\rm max}_2( n_1, \dots, n_{q+1})
\leq (1+\delta') C' {\rm max}_2( n_1, \dots, n_{p+q+1})  
\end{aligned}
$$
proving that 
$ \max(n_{1}, \dots,  n_{p+q+1}) \sim {\rm max}_2( n_1, \dots, n_{p+q+1})$. 
Then by \eqref{defRMA}  and \eqref{Senne3} we obtain
\begin{align*}
& \norm{\Pi_{n_0} R(\Pi_{n_1}U_1, \ldots, \Pi_{n_{p+q}} U_{p+q}) 
\Pi_{n_{p+q+1}} U_{p+q+1}}_{L^2}
\lesssim
\max(n_1, \ldots, n_{p+q+1})^{\wt \mu} \, \prod_{a=1}^{p+q+1} \norm{\Pi_{n_a} U_a}_{L^2}\\
& \lesssim
\frac{\max_2(n_1, \ldots, n_{p+q+1})^{\wt \mu + \varrho}}{\max(n_1, \ldots, n_{p+q+1})^{ \varrho}} \prod_{a=1}^{p+q+1}\norm{\Pi_{n_a} U_a}_{L^2} \qquad
\text{with} \qquad \wt \mu = \mu + m' + \max (m,0) \, ,  
\end{align*}
showing that 
$ R(U)$ is a $(p+q)$--homogeneous smoothing operator in $ \wtcR^{-\vr}_{p+q}\otimes \cM_2(\C)$ for any $\vr \geq 0$.

We now prove \eqref{smotrasmo}. Note that  $ (a)$ and $ (b)$ of \eqref{smotrasmo} follow by $(a)$ and $(b)$ of \eqref{enne} and $ \delta'>\delta$. Then note that if $ (n',\, n_0,\,  \dots, n_{p+q+1})\in \mN\setminus \mN'$ then 
\begin{equation}
\label{enne2}
\max(n_1, \dots,  n_{q+1}) > \delta' n_{p+q+1} \, .
\end{equation}
Then, by $(c)$  of \eqref{enne}, one has 
\begin{equation}
\label{enne3}
\max(n_1, \ldots, n_{q+1}) 
= \epsilon' n' + \sum_{n_j \leq {\rm \max}_2(n_1, \ldots, n_{q+1}) }\epsilon_j n_j 
\leq  n'+ q {\rm \max}_2(n_1, \ldots, n_{q+1}) 
\end{equation}
so that 
${\rm max}_2(n_1, \ldots, n_{q+1}) \geq \frac{1}{q} \big( \max(n_1, \ldots, n_{q+1}) - n' \big)$.
We deduce 
using \eqref{enne2} and \eqref{enne}$(a)$  that 
\begin{equation}
\label{enne4}
{\rm \max}_2(n_1, \ldots, n_{q+1}) \geq \frac{\delta'-\delta}{q} \ n_{p+q+1}
\end{equation}
thus proving \eqref{smotrasmo} $(c)$.
Finally using   \eqref{enne3}, \eqref{enne} (a) and \eqref{enne4}  we get 
$$
 \max(n_1, \dots,  n_{q+1})\leq q \big(\frac{\delta}{\delta'-\delta}+1\big) {\rm max}_2(n_1, \dots,  n_{q+1})
$$
which proves $(d)$ of \eqref{smotrasmo}.

We finally prove that $ R(U)^\top$ is a smoothing operator. 
First note that if $ \Pi_{n_0} R(\Pi_{\vec n} \cU)^\top \Pi_{n_{p+q+1}}\neq 0$ then $ \Pi_{n_{p+q+1}} R(\Pi_{\vec n} \cU) \Pi_{n_0}\neq 0$, which implies that 
$ C^{-1} n_{p+q+1}\leq n_{0} \leq C n_{p+q+1}$ by \eqref{smotrasmo} $(b)$. 
Then $R(U)^\top$ is a smoothing operator in $ \wtcR^{-\vr}_{p+q}\otimes \cM_2(\C)$ 
for any $\varrho \geq 0$ by Lemma \ref{primecosette} and Remark \ref{rem:top.smoo}.
\end{proof}

We finally prove  the following  lemma 
which generalizes a result in  \cite{FI3}  for paradifferential operators: 
the  transpose of the internal differential of a spectrally 
localized map is a smoothing operator (with two equivalent frequencies).

\begin{lemma}\label{aggiunto}
Let $p \in \N $ and $m \in \R$.
Given  a matrix of spectrally localized $p$--homogeneous maps 
$ S(U)  \in  \wtcS_p^m \otimes \mM_2 (\C) $,  consider 
\be\label{effonetra}
V \mapsto 
L(U) V := \frac{1}{p} \di_U S(U)[V] U =  S(V, U,\dots, U) U \, \, .
\ee
 Then the transposed map $ L(U)^\top$ is a matrix of smoothing operators in 
 $ \widetilde{\mR}^{-\vr}_p \otimes \mM_2 (\C)$ for any $\varrho \geq 0$. 
\end{lemma}
\begin{proof}
We first note that the matrix entries $L(U)^\top$ defined as in \eqref{M.coeff} are 
\begin{align}
[L^\top]^{\vec \sigma_p,\sigma',\sigma}_{\vec \jmath_p,j,k}\stackrel{\eqref{emmetop}}{=} &L^{\vec \sigma_p,\sigma,\sigma'}_{\vec \jmath_p,-k,-j}\notag\\
\stackrel{\eqref{M.coeff}}{=}& \int_{\T} L\Big( \mathtt{q}^{\sigma_1} \frac{e^{\ii \sigma_1 j_1 x}}{\sqrt{2\pi}}, \dots,\mathtt{q}^{\sigma_p} \frac{e^{\ii \sigma_{p} j_{p} x}}{\sqrt{2\pi}}\Big)
\Big[ \frac{\mathtt{q}^{\sigma} e^{-\ii \sigma k x}}{\sqrt{2\pi}} \Big]  \cdot  \mathtt{q}^{\sigma'}  \frac{e^{\ii \sigma'j x}}{\sqrt{2\pi}}\, \di x\notag\\
\stackrel{\eqref{effonetra}}{=}&\int_{\T} S\Big( \mathtt{q}^{\sigma} \frac{e^{-\ii \sigma k x}}{\sqrt{2\pi}},\mathtt{q}^{\sigma_1} \frac{e^{\ii \sigma_1 j_1 x}}{\sqrt{2\pi}} \dots,\mathtt{q}^{\sigma_{p-1}} \frac{e^{\ii \sigma_{p-1} j_{p-1} x}}{\sqrt{2\pi}}\Big)
\Big[ \frac{\mathtt{q}^{\sigma_p} e^{\ii \sigma_p j_p x}}{\sqrt{2\pi}} \Big]  \cdot  \mathtt{q}^{\sigma'}  \frac{e^{\ii \sigma'j x}}{\sqrt{2\pi}}\, \di x\notag\\
 \stackrel{\eqref{M.coeff}}{=}& S^{\sigma,\sigma_1,\dots, \sigma_{p-1},\sigma_p,\sigma'}_{-k,j_1,\dots, j_{p-1},j_p,-j}\,\label{Essonefourier}
\end{align}
which, by \eqref{smoocaraspecres}, are different from zero only if  
\be\label{altrospec}
{\rm max}\{ |k|,|j_1|,\dots, | j_{p-1}|\}\leq \delta |j_p| \, ,
\quad C^{-1} |j|\leq |j_p|\leq C |j| \, .
\ee
The restriction \eqref{altrospec} implies that 
\be\label{segunda}
\max(|j_1|, \dots,|j_p|,| j|)\sim  {\rm max}_2(|j_1|, \dots,|j_p|,| j|)
\ee 
because 
$$
\begin{aligned}
 \max(|j_1|, \dots,|j_p|,| j|) 
 & \leq (1+\delta)\max(|j_p|,|j|) \\
 &   \leq C(1+\delta)  {\rm max}_2(|j_p|,|j|)
  \leq C(1+\delta){\rm max}_{2} (|j_1|, \dots,|j_p|,| j|) \, .
 \end{aligned}
$$
Finally by \eqref{Essonefourier}, we estimate
\begin{align*}
\big| [L^\top]^{\vec \sigma_p,\sigma',\sigma}_{\vec \jmath_p,j,k} \big|= 
\big| S^{\sigma,\sigma_1,\dots, \sigma_{p-1},\sigma_p,\sigma'}_{-k,j_1,\dots, j_{p-1},j_p,-j}
\big|
&\stackrel{\eqref{smoocaraspec}}{\leq} C 
{\rm max}\{ |k|,|j_1|,\dots, | j_{p-1}|\}^{\mu}|j_p|^{m}\\
& \stackrel{\eqref{altrospec}}{\leq} C' {\rm max}\{|j_1|,\dots, | j_{p}|,|j|\}^{\mu+\max(m,0)}\\
& \stackrel{\eqref{segunda}}{\leq} C'' \frac{{\rm max}_2\{|j_1|,\dots, | j_{p}|,|j|\}^{\mu+\max(m,0)+\vr}}{{\rm max}\{|j_1|,\dots, | j_{p}|,|j|\}^{\vr}}
\end{align*}
implying, in view of Lemma \ref{carasmoofou} (with $m\leadsto -\vr$ and $\mu \leadsto \mu+\max(m,0)+\vr$), that $L(U)^\top$ is a matrix of smoothing operators in $\wtcR^{-\vr}_p\otimes \mM_2(\C)$ for any $\varrho \geq 0$.
\end{proof}

We conclude this section with a lemma which shall be used in Section \ref{sec:LHam}.

\begin{lemma}\label{lem:ide} 
Let $p\in \N$ and $\varrho \geq 0$. 
Let $S(U)$  be a matrix of $ p$-homogeneous 
spectrally localized maps in $  \wtcS_{p}\otimes \mM_2(\C) $
  of the form 
\be \label{decospesmo}
S(U)= L(U)+ R(U) \, , 
\ee
where $ L (U)^\top$  and   $R(U)$ are  matrices 
of $p$-homogeneous smoothing 
operators  in $ \widetilde 
{\cal R}^{-\vr}_{p}\otimes \mM_2(\C)$. 
Then  $ S(U)$ is a matrix of $p$--homogeneous smoothing 
operators in $ \widetilde 
{\cal R}^{-\vr}_{p}\otimes \mM_2(\C)$.
\end{lemma}

\begin{proof}
In view of 
\eqref{emmetop} 
and \eqref{smoocara} (for $L(U)^\top \in \wt \cR^{-\varrho}_p \otimes \cM_2(\C)$), 
\be\label{L.t.s}
|  L_{\vec \jmath_p, j, k}^{\vec \sigma_p, \sigma', \sigma} |
= 
|  (L^\top)_{\vec \jmath_p, -k, -j}^{\vec \sigma_p, \sigma, \sigma'} | 
\leq C 
{\rm max}_2\{ |j_1|,\dots, | j_p|,|k|\}^{\mu}\max\{ |j_1|,\dots, | j_p|,|k|\}^{-\varrho}   \ .
\ee
Since $ S(U) $ is spectrally localized, 
its non zero Fourier coefficients 
$ S_{\vec \jmath_p, j, k}^{\vec \sigma_p, \sigma', \sigma}$ (defined as in \eqref{M.coeff}) 
satisfy  $ |j| \sim | k | $ (see \eqref{smoocaraspecres})
and,
in view of \eqref{decospesmo},
\begin{align*}
|  S_{\vec \jmath_p, j, k}^{\vec \sigma_p, \sigma', \sigma} | 
\leq |  L_{\vec \jmath_p, j, k}^{\vec \sigma_p, \sigma', \sigma} |  + 
|  R_{\vec \jmath_p, j, k}^{\vec \sigma_p, \sigma', \sigma} | 
& \stackrel{\eqref{L.t.s}, \eqref{smoocara}}{\lesssim}  
\frac{{\rm max}_2\{ |j_1|,\dots, | j_p|,|k|\}^{\mu}}{\max\{ |j_1|,\dots, | j_p|,|k|\}^{\varrho}}  
+
\frac{{\rm max}_2\{ |j_1|,\dots, | j_p|,|j|\}^{\mu}}{\max\{ |j_1|,\dots, | j_p|,|j|\}^{\varrho}}  \\
& \stackrel{|j| \sim |k|}{\lesssim}
\frac{{\rm max}_2\{ |j_1|,\dots, | j_p|,|j|\}^{\mu}}{\max\{ |j_1|,\dots, | j_p|,|j|\}^{\varrho}}
\end{align*}
proving that $S(U)$ is a smoothing operator in $ \widetilde 
{\cal R}^{-\vr}_{p}\otimes \mM_2(\C)$.
\end{proof}

\subsection{Approximate inverse of  non--linear maps 
and flows}\label{sec:app}

In this section we construct an approximate version of two fundamental non--linear operations that we will need: the inverse of a non--linear map and the flow generated by a non--linear vector field. We first provide 
the definition of an approximate inverse of a map, up to homogeneity $N$.

\begin{definition}\label{inversoapp:def}
{\bf (Approximate inverse up to homogeneity $ N $)}
Let $ p, N \in \N $ with  $ p \leq N $. Consider 
\be\label{phidicuinv}
\Psi_{\leq N}(U)= U+ M_{\leq N}(U)U  \quad \text{where} \quad M_{\leq N}(U)  \in  \Sigma_p^N \wtcM_q \otimes \mM_2 (\C) 
\ee
is a matrix of  pluri--homogeneous   operators. We say that 
\be\label{LTCPhi}
\Phi_{\leq N}(V)=  V+ \breve M_{\leq N}(V)V  \quad \text{where} \quad
\breve M_{\leq N}(V) \in \Sigma_p^N \wtcM_q\otimes \mM_2 (\C)  \, , 
\ee
is an approximate inverse of $\Psi_{\leq N} (U) $
 up to homogeneity $ N $ if 
 \be \label{omogeneaaaa2}
\begin{aligned}
&\big(\uno + M_{\leq N}(\Phi_{{\leq N}}(V))\big)\big( \uno+ \breve M_{\leq N}(V)\big)=\uno +M'_{>N}(V)\\
&\big(\uno + \breve M_{\leq N}(\Psi_{{\leq N}}(U))\big)\big( \uno+ M_{\leq N}(U)\big)=\uno +M''_{>N}(U)\\
\end{aligned}
\ee
where $ M'_{>N} (V)$ and $ M''_{>N} (U)$   are pluri--homogeneous matrices of
operators  in $ \Sigma_{N+1} \wtcM_q\otimes \mM_2 (\C) $.
\end{definition}

Note that, if $ \Phi_{\leq N}(V)$ is an approximate inverse  up to homogeneity $N$ of $\Psi_{\leq N}(U)$ then 
\be\label{invmappe}
\Psi_{\leq N} \circ \Phi_{{\leq N}}(V) = V +  M'_{>N}(V) V
\, , \quad  \Phi_{\leq N} \circ \Psi_{{\leq N}}(U)= U+ M''_{>N}(U) U \, , 
\ee
and, by differentiation and taking the transpose
\be \label{diinv2}
\begin{aligned}
&\di_U \Psi_{\leq N}( \Phi_{{\leq N}}(V))\di_V \Phi_{{\leq N}}(V)= \uno + M^1_{>N}(V)\\
&\di_V \Phi_{{\leq N}}( \Psi_{\leq N}(U)) \di_U \Psi_{\leq N}(U)= \uno+ M^2_{>N}(U)
\, \\
& \di_V \Phi_{\leq N}(V)^\top \, \di_U \Psi_{\leq N}(\Phi_{\leq N}(V))^\top =
\uno + M^3_{>N}(V) \\
& \di_U \Psi_{\leq N}(\Phi_{\leq N}(V))^\top \, \di_V \Phi_{\leq N}(V)^\top =
\uno + M^4_{>N}(V) 
\end{aligned}
\ee
where $ M^a_{>N}(V)$, $a = 1, \ldots, 4$
 are other 
pluri-homogeneous operators  in $ \Sigma_{N+1}  \wtcM_q\otimes \mM_2 (\C) $ (the differential of a homogeneous $ m$-operator is a homogeneous $m$-operator by the first remark after  Definition \ref{Def:Maps}, so is its transpose by Lemma \ref{primecosette}).

The following lemma ensures the existence of an approximate inverse.

\begin{lemma}
\label{inversoapp}
{\bf (Approximate inverse)}
Let $p,N\in \N $ with   $ p \leq N $.  Consider $\Psi_{\leq N}(U)= U+ M_{\leq N}(U)U$ as in 
\eqref{phidicuinv}. 
Then there exists an approximate inverse of $\Psi_{\leq N}(U)$ up to homogeneity $ N $ 
(according to Definition \ref{inversoapp:def}) of the form \eqref{LTCPhi} with 
\be \label{menoinverso}
\cP_{p}[\breve M_{\leq N}(V)]= - \cP_p[ M_{\leq N}(V)] \, . 
\ee
Moreover,  if  $M_{\leq N}(U)$ in \eqref{phidicuinv} is  a matrix of pluri--homogeneous 
\begin{itemize}
\item[(i)]   
spectrally localized maps in $ \Sigma_p^N \wtcS_q^m\otimes \mM_2 (\C) $, $ m \geq 0 $,  then   $ \breve M_{\leq N}$ in \eqref{LTCPhi} is a matrix of pluri-homogeneous spectrally localized maps in $  \Sigma_p^N \wtcS_q^{m(N-p+1)}\otimes \mM_2 (\C) $ 
and $M'_{>N}$, $M''_{>N}$ in \eqref{omogeneaaaa2}  belong to  
$ \Sigma_{N+1}  \wtcS_q^{m(N-p+2)}\otimes \mM_2 (\C) $;
\item[(ii)] 
 smoothing operators  in $ \Sigma_p^N \wtcR^{-\vr}_q\otimes \mM_2 (\C) $ for some $\varrho \geq 0$,   then 
 $ \breve M_{\leq N} $  in \eqref{LTCPhi} is a matrix of pluri-homogeneous  smoothing operators  
  in $  \Sigma_p^N \wtcR^{-\vr}_q\otimes \mM_2 (\C) $ 
 and $M'_{>N}$, $M''_{>N}$ in \eqref{omogeneaaaa2}  belong to $ \Sigma_{N+1}  \wtcR^{-\vr}_q\otimes \mM_2 (\C) $.
\end{itemize}
 \end{lemma}

\begin{proof}
We expand  in homogeneous components 
\be\label{BNU2}
M (U) := \uno+ M_{\leq N}(U) = \uno+ \sum_{q=p}^N M_{q}(U)  \quad \text{with} \quad 
 M_{q}(U) \in \wtcM_{q}\otimes \mM_2 (\C)  \, .
\ee
In order to solve 
the first equation in \eqref{omogeneaaaa2} 
we look for a  pluri--homogeneous operator
\be\label{expAN2}
\breve M(V)= \uno + \breve M_{\leq N}(V) = \breve M_0(V) + \sum_{a=p}^N \breve M_{a}(V) 
   \quad \text{with} \quad \breve M_0(V):= \uno \, , \
\breve M_{a}(V) \in \wtcM_q \otimes \mM_2 (\C)  \, , 
\ee
such that 
\be \label{dadim2}
M\big( \breve M(V)V\big)\breve M(V) =\uno+ M_{>N}(V) 
 \quad \text{with} \quad M_{>N}(V) \in \Sigma_{N+1} \wtcM_q\otimes \mM_2 (\C)  \, .
\ee
By  \eqref{BNU2}, \eqref{expAN2}  we get 
$$ 
\begin{aligned}
& M\big( \breve M(V)V\big) \breve M(V) \\
& = \uno+ \sum_{\ell=p}^N \breve M_{\ell}(V)+ \sum_{q=p}^N \sum_{0\leq a_1, \dots a_{q+1}\leq N} M_{q}(\breve M_{a_1}(V)V, \dots, \breve M_{a_{q}}(V)V) 
\breve M_{a_{q+1}}(V)\\
& = \uno 
+ \sum_{\ell=p}^N \Big[ \breve M_{\ell}(V)+ \sum_{\substack{ p\leq q\leq \ell\\q+a_1+ \dots + a_{q+1}=\ell} } M_{q}(\breve M_{a_1}(V)V, \dots, \breve M_{a_{q}}(V)V)
\breve M_{a_{q+1}}(V)\Big] 
+ M_{>N}(V) 
\end{aligned}
$$
and therefore equation \eqref{dadim2} is recursively solved 
by defining, for any 
$ \ell=p, \dots, N, $ 
\be \label{recu}
\breve M_{\ell}(V, \dots ,V) :=-  
\sum_{q+a_1+ \dots + a_{q+1}=\ell} M_{q}(\breve M_{a_1}(V)V, \dots,
\breve M_{a_{q}}(V)V) \breve M_{a_{q+1}}(V) \, .
\ee
Note that each $ \breve M_{\ell}(V, \dots ,V) $ is a matrix of homogeneous operators by $(ii)$ of Proposition \ref{composizioniTOTALI}. 
We proved the first identity in \eqref{omogeneaaaa2}.  

We now prove the second identity in \eqref{omogeneaaaa2}. Using the same recursive procedure 
we  find a matrix of pluri--homogeneous operators of the form 
\be\label{defM2}
M''(U)= \uno+ \sum_{n=p}^N M''_n(U) \quad \text{with}  \quad 
 M''_n(U)  \in \wtcM_n\otimes \mM_2 (\C) 
\ee
 such that $M''(U)$ is an approximate right inverse of $ \breve M(V)$, i.e.
\be\label{defM3}
\breve M \big( M''(U)U\big) M''(U)= \uno + M_{>N}(U) \quad
\text{with} \quad 
M_{>N} (U) \in  \Sigma_{N+1}\wtcM_{q}\otimes \mM_2 (\C) \, . 
\ee
 Applying \eqref{dadim2} with $ V= M''(U)U$ and right--composing it with $M''(U)$ defined in 
 \eqref{defM2}, we obtain by \eqref{defM3} 
 and Proposition \ref{composizioniTOTALI}  that 
\be\label{Mquasifin}
 M \big( U + M_{>N}(U)U\big) = M''(U)+ M'_{>N}(U)
\ee
where 
$M_{>N} (U)$ and $ M'_{>N}(U)  $ are   
operators  in 
$  \Sigma_{N+1} \wtcM_q\otimes \mM_2 (\C) $. Then, expanding 
the left hand side of 
\eqref{Mquasifin}
by multilinearity, we get 
$$
M \big( U \big) -M''(U)= \sum_{q=p}^N \sum_{\ell=1}^q 
\binom{q}{\ell} M_q(\underbrace{ M_{>N}(U)U, \dots, M_{>N}(U)U}_{\ell-times}, U,\dots, U) + M'_{>N}(U)= \tilde M_{>N}(U)
$$
where $  \tilde M_{>N}(U) $ is in $ \Sigma_{N+1} \wtcM_{q}\otimes \mM_2 (\C) $ thanks to $(ii)$ of Proposition \ref{composizioniTOTALI}. This implies, since both $ M(U)$ and $M''(U)$ are pluri--homogeneous operators up to  homogeneity $ N $ (cfr. \eqref{BNU2}, \eqref{defM2}), 
 that $ M \big( U \big) = M''(U)$ and, by \eqref{defM3}, we conclude that 
\be \label{inversionefine2}
 \breve M \big( M(U)U\big) M(U)= \uno + M_{>N}(U) \, .
 \ee
 This proves, recalling the notation \eqref{BNU2}, \eqref{expAN2}, \eqref{phidicuinv}, 
the second identity in \eqref{omogeneaaaa2}. 
Moreover  for $\ell=p$ the sum in \eqref{recu} reduces to  the unique element with $ q=p$, $a_1=\dots=a_{q+1}=0$ and $ \breve M_p(V)= - M_p(V)$, proving \eqref{menoinverso}.

$(i)$ If $M_{\leq N}(U)$ is a spectrally localized map in $\Sigma_{p}^N \wtcS^m_q\otimes \mM_2 (\C)$ we claim that $ \breve M_\ell(V)$ is a spectrally localized map in $ \wtcS^{m+m(\ell-p)}_\ell\otimes \mM_2 (\C)$ for $\ell=p, \dots, N$. For $\ell=p$ 
by \eqref{menoinverso} we have $ \breve M_p(V)= - M_p(V)$ 
which is in $ \wtcS^{m}_\ell\otimes \mM_2 (\C)$.
 Then supposing inductively that $ \breve M_a(V)$ is in $\wtcS^{m+m(a-p)}_a\otimes \mM_2 (\C)$ for $ a=p,\dots \ell-1$, we deduce 
by $(ii)$ and $(iii)$ of Proposition  \ref{composizioniTOTALIs}, that each term in the sum in \eqref{recu} is a 
  spectrally localized map in $ \wtcS^{m+(m+m(a_{q+1}-p))}_\ell\otimes \mM_2 (\C)$ which is included in $ \wtcS^{m+m(\ell-p)}_\ell\otimes \mM_2 (\C)$ using that $ a_{q+1}\leq \ell-1$. 
 
$(ii)$ In the same way, if $ M_{\leq N}(U)$ is a smoothing operator in $ \Sigma_p^N \wtcR_q^{-\vr}\otimes \mM_2 (\C) $, thanks to $(ii)$ of Proposition \ref{composizioniTOTALI} one proves recursively that $ \breve M_\ell(V)$ are smoothing 
operators  in $\wtcR_\ell^{-\vr}\otimes \mM_2 (\C) $.
\end{proof}

We now define the approximate flow of a smoothing 
$\tau$--dependent vector field. 

\begin{definition} \label{def:approflow}
{\bf (Approximate flow of a smoothing vector field up to homogeneity $N$)}
Let  $p, N \in \N $ with $p\leq N$.
An \emph{approximate flow up to homogeneity $N$} of a $ \tau$--dependent 
pluri-homogeneous smoothing vector field  $X^\tau (Z) $  in $ \Sigma_{p+1}
\wt \X_q^{-\varrho} $, defined for $ \tau \in [0,1] $ and some $\varrho \geq 0$, 
 is 
 a non--linear map,  
\be \label{flussoappdef}
\mF_{\leq N}^\tau(Z)= Z+ F_{\leq N}^\tau(Z)Z \, , \qquad \tau\in [0,1] \, ,  
\ee
where $ F_{\leq N}^\tau(Z) $ is 
 a matrix of pluri--homogeneous, $\tau$--dependent, smoothing operators in $\Sigma_p^N \wtcR^{-\vr}_q \otimes \mM_2 (\C) $, with estimates uniform in $ \tau \in [0,1] $,
 solving
\be \label{appX}
\pa_\tau \mF^\tau_{\leq N}(Z) = X^\tau ( \mF^\tau_{\leq N}(Z)) +  R^\tau_{>N}(Z)Z \, , 
\quad \mF^0_{\leq N}(Z)=Z \, , 
\ee
where $ R^\tau_{>N}(Z)$ is a matrix of 
$\tau$--dependent, smoothing operators 
 in $ \Sigma_{N+1} \wtcR^{-\vr}_q \otimes \mM_2 (\C) $, with estimates uniform in $ \tau \in [0,1] $. 
\end{definition}

The following lemma ensures the existence of an approximate  flow.

\begin{lemma}\label{extistencetruflow}
{\bf (Approximate flow)} 
Let  $p, N \in \N $ with  $p \leq N  $.
Consider a pluri--homogeneous  $ \tau$--dependent smoothing vector field
$ X^\tau (Z)  $ in $ \Sigma_{p+1} \wt \X^{-\vr}_q $, defined for $ \tau \in [0,1 ]$ and some $\varrho \geq 0$. 
 Then 
 \begin{itemize}
 \item 
 there exists an 
 approximate flow 
$\mF^\tau_{\leq N}$  according to Definition \ref{def:approflow};
\item 
denoting by $ G_{p}^{\tau} (Z) Z $ with 
$ G^\tau_{p}(Z) \in  \wtcR^{-\vr}_p \otimes \mM_2 (\C) $  the (p+1)-homogeneous component of 
$ X^\tau (Z) $, then  the $p$-homogeneous component of the  smoothing 
operator  $ F_{\leq N}^\tau(Z) $ in  
\eqref{flussoappdef} is 
\be\label{flowexp}
\mP_p(F_{\leq N}^\tau(Z)) =   \int_0^\tau G_{p}^{\tau'} (Z) \di \tau'  	\, .
\ee
\end{itemize}
\end{lemma}

\begin{proof}
We write 
$ X^\tau (Z) := G^\tau(Z)Z  $ with $ G^\tau(Z)  = \sum_{q=p}^M G^\tau_{q}(Z) $ and 
$ G^\tau_{q}(Z) \in  \wtcR^{-\vr}_q \otimes \mM_2 (\C) $, 
and we look for an approximate flow  solution of  \eqref{appX}  
of the form 
$$
\mF_{\leq N}^\tau(Z)= Z+ \sum_{\ell=p}^N F^\tau_{\ell}(Z, \dots, Z)Z 
\qquad 
\text{with} 
\qquad 
F_\ell^{\tau}(Z) \in \wt \cR^{-\varrho}_\ell \otimes \cM_2(\C) \, .
$$
Since $G^\tau(Z)= \sum_{q=p}^M G^\tau_{q}(Z)$ then, using the notation $ F_0^\tau(Z):=\uno$,   we expand by multilinearity 
$$
X^\tau (\mF_{\leq N}^\tau(Z)) =\sum_{q=p}^M G_{q}^\tau( \mF_{\leq N}^{\tau}(Z), \dots , \mF_{\leq N}^{\tau}(Z))\mF_{\leq N}^{\tau}(Z)
= \sum_{\ta=p}^{N+ (N+1)M} \breve X_\ta^\tau(Z) Z
$$
where
\be\label{Xa}
 \breve X_\ta^\tau(Z) :=  \sum_{\substack{q = p, \ldots, M \\ \ell_1, \dots, \ell_{q+1} \in \{0,p,\dots, N\} \\
\ell_1+ \dots + \ell_{q+1} +q= \ta}}   G_{q}^\tau\big( F_{\ell_1}^{\tau}(Z)Z, \dots , F_{\ell_q}^{\tau}(Z)Z\big)F_{\ell_{q+1}}^{\tau}(Z) \ . 
\ee
Then we solve \eqref{appX} defining recursively  for $\ta=p,\dots, N$, 
$$
F_{\ta}^\tau(Z):=\int_0^\tau \breve X^{\tau'}_\ta(Z)\,\di \tau',
\qquad
R_{>N}^\tau (Z) :=\sum_{\ta=N+1}^{N+(N+1)M} X^{\tau}_\ta(Z) \ . 
$$
Using recursively formula \eqref{Xa} and  Proposition \ref{composizioniTOTALI} 
one verifies that each  $ \breve X_\ta^\tau(Z)$ is a  $\ta$-homogeneous smoothing operator in  $\wt \cR_\ta^{-\varrho} \otimes \cM_2(\C)$, so is 
$F_{\ta}^\tau(Z)$, and 
$R_{>N}^\tau (Z)$ is a pluri-homogeneous smoothing operator in $\Sigma_{N+1} \wt \cR_\ta^{-\varrho} \otimes \cM_2(\C)$.
Note that for $ \ta=p$ the sum in  \eqref{Xa} reduces to the indices $ q=p$, $ \ell_1=\dots=\ell_{q+1}=0$. As a consequence 
$ F_{p}^\tau(Z)= \int_0^\tau G^{\tau'}_{p}(Z)\,\di \tau'$
proving  \eqref{flowexp}.
\end{proof}

\subsection{Pluri-homogeneous differential  geometry}\label{SGeo}

In this section we introduce pluri-homogeneous  $k$-forms. 
We revisit the classical identities of differential geometry ($\di^2 = 0$, Cartan's magic formula) for $k= 0, 1, 2 $ which are the only cases needed in this paper.

\begin{definition} {\bf ($r$-homogeneous $k$-form)}
Let $ p  \in \N_0$, 
$ k=0,1,2 $  and set $ r := p + 2 - k $.  
A $r$-homogeneous $k$-form is a $(r+k)$-linear map from $ (\dot H^{\infty}(\T;\C^2))^r\times (\dot H^{\infty}(\T;\C^2))^k$ to $ \C$ of the form $\Lambda(U_1,\dots,U_r)[ V_1,\dots, V_k]$, symmetric in the variables 
$ {\cal U} := (U_1,\dots, U_r)$ and  antisymmetric in the entries $
{\cal V} := (V_1,\dots, V_k)$, satisfying the following:  
there are constants  $ C > 0 $  and $m\geq 0$ such that 
\begin{align}
|\Lambda&(\Pi_{n_{1}}U_1,\dots,\Pi_{n_r}U_r)[\Pi_{n_{r+1}} V_1,\dots, \Pi_{n_{r+k}}V_k]|\notag\\
&\leq C  {\rm max}\{ n_1,\dots ,n_{r+k}\}^m \prod_{j=1}^r \| \Pi_{n_j} U_j\|_{L^2}\prod_{\ell=1}^k \| \Pi_{n_{r+\ell }}V_\ell \|_{L^2}\label{prop.form} 
\end{align}
for any   $ \cal U 
\in (\dot{H}^{\infty}(\T;\C^{2}))^{r}$, any
$ \cal V 
\in (\dot{H}^{\infty}(\T;\C^{2}))^k $, and any 
 $ (n_1,\ldots,n_{r+k}) $ in $  \N^{r+k} $.
 Moreover, if 
 $$
  \Lambda(\Pi_{n_{1}}U_1,\dots,\Pi_{n_r}U_r)[\Pi_{n_{r+1}} V_1,\dots, \Pi_{n_{r+k}}V_k] \neq 0 \, ,
$$
 then there is a choice of signs $\s_1,\ldots,\s_{r+k}\in\{\pm 1\}$ such that 
 $\sum_{j=1}^{r+k}\s_j n_{j}=0$. 
  In addition we require the translation invariant property:
 \be\label{Lambda.inv.tra}
 \Lambda(\tau_\varsigma \cU)[ \tau_\varsigma {\cal V} ]  = \Lambda(\cU)[  {\cal V}] \, , \quad \forall \varsigma \in \R  \, .  
 \ee
  We also require that $ \Lambda({\cal U})[{\cal V}] $ is real valued for any 
  $ { \cal U} \in (L^2_\R (\T, \C^2))^r  $  and $ {\cal V } \in (L^2_\R (\T, \C^2))^k $, cfr. \eqref{L2rC2}. 
 We denote  by $\wt \Uplambda^k_r$  the space of $r$-homogeneous $k$-forms 
 and by $ \Sigma_r^N \wt \Uplambda^k_q$ the space of pluri--homogeneous $k$-forms.
 We set $ \Sigma_{r} \wt \Uplambda^k_q:= \cup_{N\geq r}\Sigma_r^N \wt \Uplambda^k_q$.
\end{definition} 

\noindent
$\bullet$ A $r$-homogeneous  $ 0 $-form is  also called a homogeneous {\it Hamiltonian}.

\noindent 
$\bullet$ {\bf Fourier representation of $(p+2)$-homogeneous $ 0 $-forms:}  
Let $ p  \in \N_0 $.  A $ (p+2)$-homogeneous $0$-form can be expressed in Fourier as 
(recall \eqref{fTset})
\be\label{polham}
H(U)= \frac{1}{p+2}\sum_{
(\vec \jmath_{p+2}, \vec \sigma_{p+2}) \in \fT_{p+2}
 } 
H_{\vec \jmath_{p+2}}^{\vec \sigma_{p+2}} u_{\vec \jmath_{p+2}}^{\vec \sigma_{p+2}} \, .
\ee
The reality condition $  H(U)\in \R$ for  any $U\in L^2_\R(\T,\C^2)$ amounts to 
\be \label{realham}
\bar{H_{\vec \jmath_{p+2}}^{\vec \sigma_{p+2}}}= H_{\vec \jmath_{p+2}}^{-\vec \sigma_{p+2}} \, . 
\ee
Moreover the scalar coefficients $H_{\vec{ \jmath}_{p+2}}^{\vec{ \sigma}_{p+2}}\equiv H_{\ j_1,\dots, j_{p+2}}^{ \sigma_1,\dots, \sigma_{p+2} } \in \C  $ 
 satisfy the symmetric condition: for any permutation $\pi$ of $\{ 1, \ldots, p+2\}$ 
\be \label{symham}
H_{\ j_{\pi(1)},\dots, j_{\pi(p+2)}}^{ \sigma_{\pi(1)},\dots, \sigma_{\pi(p+2)} }=
H_{\ j_1,\dots, j_{p+2}}^{ \sigma_1,\dots, \sigma_{p+2} } \, 
\ee
and, for some $m\geq 0$, the bound
\be\label{upperbounH}
| H_{\vec \jmath_{p+2}}^{\vec \sigma_{p+2}}|\lesssim\max( | j_1|, \ldots, |j_{p+2}|)^m \, . 
\ee
 \noindent
$\bullet$ {\bf Fourier representation of $ (p+1)$-homogeneous $1$-forms:}
A $ (p+1)$-homogeneous $1$-form can be expressed in Fourier as 
\be\label{thetone}
\theta(U)[V]= \sum_{
(\vec \jmath_{p+1}, k, \vec \sigma_{p+1}, \sigma) \in \fT_{p+2}
 } \Theta_{\vec \jmath_{p+1},k}^{\vec \sigma_{p+1},\sigma} u_{\vec \jmath_{p+1}}^{\vec \sigma_{p+1}} v_k^{\sigma} \, . 
\ee
The reality condition $  \theta(U)[V]\in \R$ for any $U,V \in L^2_\R(\T,\C^2)$ amounts to 
\be \label{reaFouT}
\bar{\Theta_{\vec \jmath_{p+1},k}^{\vec \sigma_{p+1},\sigma}}= \Theta_{\vec \jmath_{p+1}, k}^{-\vec \sigma_{p+1}, -\sigma} \, .
\ee
Moreover the coefficients satisfy, for some $m\geq 0$, 
\be\label{upperbounT}
| \Theta_{\vec \jmath_{p+1},k}^{\vec \sigma_{p+1},\sigma}|\lesssim 
\max( | j_1|, \ldots, |j_{p+1}|)^m 
 \, . 
\ee


\noindent $\bullet$ {\bf Fourier representation of $p$-homogeneous $2$-forms:}
A $p$-homogeneous $2$-form can be expressed in Fourier as 
\be \label{lambdone}
\Lambda(U)[V_1,V_2]=\sum_{
(\vec \jmath_{p},j,k,  \vec\sigma_{p},\sigma',\sigma)\in \fT_{p+2}
} \Lambda_{\vec \jmath_p, j,k}^{\vec \sigma_p, \sigma',\sigma} u_{\vec \jmath_p}^{\vec \sigma_p} (v_1)_j^{\sigma'}(v_2)_k^\sigma \, . 
\ee
The antisymmetry condition $ \Lambda(U)[V_1,V_2]= - \Lambda(U)[V_2,V_1]$ amounts to 
\be \label{antiFou}
\Lambda_{\vec \jmath_p, j,k}^{\vec \sigma_p, \sigma',\sigma}= - \Lambda_{\vec \jmath_p, k,j}^{\vec \sigma_p, \sigma,\sigma'} \, .
\ee
The reality condition $  \Lambda(U)[V_1,V_2]\in \R$ for  any $U,V_1,V_2 \in L^2_\R(\T,\C^2)$ amounts to 
\be \label{reaFouL}
\bar{\Lambda_{\vec \jmath_p, j,k}^{\vec \sigma_p, \sigma',\sigma}}=  \Lambda_{\vec \jmath_p, j,k}^{-\vec \sigma_p, -\sigma',-\sigma} \, .
\ee
Moreover the coefficients satisfy, for some $m\geq 0$,  
\be\label{upperbounL} 
| \Lambda_{\vec \jmath_p, j,k}^{\vec \sigma_p, \sigma',\sigma}|\lesssim \max( |\vec \jmath_p|, | j |)^m \, . 
\ee 
\begin{definition}[{\bf Homogeneous vector fields}]
\label{def:X}
Let $m \in \R$ and $ p, N \in \N_0$. We denote by $\wt \X_{p+1}^m$ the space of $ (p+1)$-homogeneous vector fields of the form $ X(U)=M(U)U$ where $M(U)$ is a matrix of $p$-homogeneous $m$-operators in $ \wtcM_{p}^m\otimes \mM_2(\C)$.  
We denote $ \wt \X_{p+1}:=\cup_{m\geq0}\wt \X_{p+1}^{m}$
 and 
 $ \Sigma_{p+1}^{N+1}\wt\X^{m}_q$ the class of pluri-homogeneous vector fields. 
  We also set $\Sigma_{p+1}  \wt \X_{q}:= \bigcup_{N\in \N} \Sigma_{p+1}^{N+1}  \wt \X_{q}$.
 The vector fields in  $\wt \X_{p+1}^{- \varrho}$, $ \varrho \geq 0 $, are called smoothing. 
\end{definition}

Note that  $X(U)=M(U)U$ is real-to-real in the sense of \eqref{rtr}  if and only if  the operator $M(U)$ is  real-to-real in the sense of \eqref{vinello}.

 \noindent
$\bullet$ {\bf Fourier representation of $ (p+1)$-homogeneous vector fields:} A $ (p+1)$-homogeneous vector field can be expressed in Fourier as: for any $\sigma = \pm$,  
\be\label{polvect}
X(U)^\sigma
= \sum_{k\in \Z\setminus\{0\}} X(U)^\sigma_k 	\, \frac{e^{\ii\sigma k x}}{\sqrt{2\pi}} , 
\qquad X(U)^\sigma_k= 
\!\!\!\!\!\!\!\!\!\!\!\!\! \!\!\!\!\! \sum_{
( \vec{\jmath}_{p+1},k,  \vec{\sigma}_{p+1}, - \sigma) \in \fT_{p+2}
}\!\!\!\!\!\!\!\!\!\!\!\!\!\!\!
 X_{\vec{\jmath}_{p+1}, k}^{ \vec{\sigma}_{p+1}, \sigma}u^{\vec{\sigma}_{p+1}}_{\vec{\jmath}_{p+1}} \,,
\ee
the last sum being in $(\vec \jmath_{p+1}, \vec \sigma_{p+1})$, 
with coefficients $ X_{\vec{\jmath}_{p+1}, k}^{ \vec{\sigma}_{p+1}, \sigma}\in \C$ given by 
\begin{equation}
 \label{simmetrizzata}
 X_{\ j_1, \ldots, j_{p}, j_{p+1}, k}^{\sigma_1, \ldots, \sigma_{p}, \sigma_{p+1},  \sigma} = 
 \frac{1}{p+1} \left(
 M_{\ j_1, \ldots, j_{p}, j_{p+1}, k}^{\sigma_1, \ldots, \sigma_{p}, \sigma_{p+1},  \sigma} + 
 M_{\ j_{p+1}, \ldots, j_{p}, j_1, k}^{\sigma_{p+1}, \ldots, \sigma_{p}, \sigma_1,  \sigma} +
 \cdots
  +
  M_{\ j_1, \ldots, j_{p+1}, j_{p}, k}^{\sigma_1, \ldots, \sigma_{p+1}, \sigma_{p},  \sigma} 
  \right) \, , 
 \end{equation} 
 namely they are obtained symmetrizing with respect to the last index the coefficients  $M_{\ \vec \jmath_{p}, j, k}^{\vec \sigma_{p}, \sigma',\sigma}$ of $M(U)$.
 
 In particular they satisfy 
 the symmetry condition: for any permutation $ \pi $ of $ \{1, \ldots, p+1 \} $, 
 \be\label{symmetric}
X_{\ j_{\pi(1)}, \ldots,j_{\pi(p+1)}, k}^{ \sigma_{\pi(1)}, \ldots, \sigma_{\pi(p+1)},\sigma} 
=  
X_{\ j_{1}, \ldots, j_{p+1}, k}^{ \sigma_{1}, \ldots, \sigma_{p+1},\sigma} \, . 
\ee
In addition, if $X(U)$ is real-to-real, see \eqref{rtr}, then one has 
\be\label{X.real}
 \bar{X(U)^+_k}=X(U)^-_k\quad {\mathrm i.e.}  \quad\bar{X_{\vec{\jmath}_{p+1}, k}^{ \vec{ \sigma}_{p+1}, +}}=X_{\vec{\jmath}_{p+1}, k}^{- \vec{\sigma}_{p+1}, -} \, . 
\ee
By Lemma \ref{carasmoofou} we obtain the following  characterization of  vector fields.

 \begin{lemma}\label{lem:X.R}
 {\bf (Characterization of  vector fields in Fourier basis)} 
 Let $ m \in \R $. 
 A  real-to-real vector field
 $X(U) $ belongs to $ \wt \X_{p+1}^{m}$ 
 if and only if its coefficients 
$ X_{\vec{\jmath}_{p+1}, k}^{ \vec{\sigma}_{p+1}, \sigma}$ (defined as in 
\eqref{polvect})
 fulfill the symmetric and real-to-real conditions \eqref{symmetric}, \eqref{X.real} and:
 there exist $ \mu\geq 0 $ and $C>0$  such that 
\be\label{smoocara20}
 |X_{\vec{\jmath}_{p+1}, k}^{ \vec{\sigma}_{p+1}, \sigma} |\leq C 
{\rm max}_2
\{ |j_1|,\dots, | j_{p+1}|\}^{\mu} \
\max\{ |j_1|,\dots, | j_{p+1}|\}^{m} 
 \ee
 for any 
 $ (\vec \jmath_{p+1}, k,  \vec \sigma_{p+1},-\sigma) \in \fT_{p+2}  $ (cfr. \eqref{fTset}).
 \end{lemma}

The following lemma characterizes $0$, $1$ and $2$ forms. 

\begin{lemma} \label{lem:ident.forms}
{\bf (Operatorial characterization of Hamiltonians and $1$- $2$ forms)}
Let $p \in \N_0$.  Then
\begin{itemize}
\item[$(i)$] A  $0$-form $H(U) $ belongs to 
$ \wt \Uplambda_{p+2}^0$ if and only if there exists a matrix of  $p$-homogeneous 
real-to-real  operators $ M(U) $ in $ \widetilde \mM_p \otimes \cM_2(\C)$, 
 such that, 
\be\label{1form.rep0}
H(U)=  \langle M(U)U,U\rangle_r, \quad \forall U \in \dot{H}^{\infty}(\T;\C^{2}) 
  \,  . 
\ee 
\item[$(ii)$] A $1$-form $ \theta(U) $ belongs to 
$  \wt \Uplambda_{p+1}^1$ if and only if there exists a matrix of  pluri--homogeneous 
real-to-real  operators $ M(U) $ in $  \widetilde \mM_p \otimes \cM_2(\C)$, 
 such that,  
\be\label{1form.rep}
 \theta(U)[V]:=  \langle M(U)U,V\rangle_r \, , \quad \forall V \in \dot{H}^{\infty}(\T;\C^{2}) 
  \,  . 
\ee 
\item[$(iii)$]  A $ 2 $-form
 $\Lambda(U) $ belongs to $   \wt \Uplambda_p^2 $
  if and only if there exists a matrix of pluri--homogeneous real-to-real operators
  $ M(U) $ in $   \widetilde \mM_p \otimes \cM_2(\C)$,
    satisfying $M(U)^\top=-M(U) $,   such that 
\be\label{2form.rep}
\Lambda(U)[V_1,V_2]:= \langle M(U)V_1,V_2\rangle_r \, , \quad 
\forall (V_1, V_2)  \in (\dot{H}^{\infty}(\T;\C^{2}))^2  
\, . 
\ee
\end{itemize}
\end{lemma}

\begin{proof}

\noindent {\sc Proof of $(i)$:} Identity \eqref{1form.rep0} follows with an operator $M(U)$ which has Fourier entries
\be\label{mappaemmeacca}
M_{\vec \jmath_p,j,k}^{\vec \sigma_p,\sigma',\sigma}=\frac{1}{p+2}H_{\vec \jmath_{p},j, -k}^{\vec \sigma_{p},\s',\sigma}
\ee 
where the Fourier coefficients of $H$ are defined in  \eqref{polham}. 
By  \eqref{upperbounH} and Lemma \ref{carasmoofou} the operator $M(U)$ defined by \eqref{mappaemmeacca} is  a matrix of $m$-operators in $\wt \mM_p^m \otimes \cM_2(\C) $.
Note that, in view of \eqref{mappaemmeacca}, the entries of the operator $M(U)$ satisfy the corresponding momentum condition $ \vec \sigma_{p} \cdot \vec \jmath_{p}+ \sigma' j=\sigma k$ thanks to the restriction in \eqref{polham}.  The reality condition \eqref{realham} is equivalent to \eqref{M.realtoreal}.

\noindent {\sc Proof of $(ii)$:} Identity \eqref{1form.rep} follows with an operator $M(U)$ which has Fourier entries
\be\label{mappaemmetheta}
M_{\vec \jmath_p,j,k}^{\vec \sigma_p,\sigma',\sigma}=
\Theta_{\vec \jmath_{p},j, -k}^{\vec \sigma_{p},\s',\sigma} 
\ee 
where 
 the Fourier coefficients of $\Theta$ are defined in  \eqref{thetone}.
By  \eqref{upperbounT} and Lemma \ref{carasmoofou} the operator $M(U)$ defined by \eqref{mappaemmetheta} is  a matrix of $m$-operators in $\wt \mM_p^m \otimes \cM_2(\C) $.
Note that, in view of \eqref{mappaemmelambda}, the entries of the operator $M(U)$ satisfy the corresponding momentum condition $ \vec \sigma_{p} \cdot \vec \jmath_{p}+ \sigma' j=\sigma k$ thanks to the restriction in \eqref{realham}.  The reality condition \eqref{reaFouT} comes from \eqref{M.realtoreal}.

\noindent {\sc Proof of $(iii)$:} Identity \eqref{2form.rep} follows with an operator  $M(U)$ which has Fourier entries (cfr. \eqref{M.coeff}) 
\be \label{mappaemmelambda}
M_{\vec \jmath_p, j,k}^{\vec \sigma_p, \sigma',\sigma}= \Lambda_{\vec \jmath_p, j,-k}^{\vec \sigma_p, \sigma',\sigma} 
\ee 
where   the Fourier coefficients of $\Lambda$ are defined in  \eqref{lambdone}.
By  \eqref{upperbounL} and Lemma \ref{carasmoofou} the operator $M(U)$ defined by \eqref{mappaemmelambda} is  a matrix of $m$-operators in $\wt \mM_p^m \otimes \cM_2(\C) $.
Note that, in view of \eqref{mappaemmelambda}, the entries of the operator $M(U)$ satisfy the corresponding momentum condition $ \vec \sigma_{p} \cdot \vec \jmath_{p}+ \sigma' j=\sigma k$ thanks to the restriction in \eqref{lambdone}.
The antisymmetry of $\Lambda(U)$ amounts to $M(U)^\top=-M(U)$ and the reality condition \eqref{reaFouL} comes from \eqref{M.realtoreal}.
\end{proof}

We now extend to pluri-homogeneous $k$-forms the typical ``operations'' of differential geometry.

\begin{definition}\label{def:forms}
{\bf (Exterior derivative)}
We define the exterior derivative 
of  a $r$-homogeneous $ k$-form $ \Lambda(U) $  in $   \wt \Uplambda_r^k $ as 
\be\label{dLambda}
\di  \Lambda(U)[V_1,\dots, V_{k+1}]:= \sum_{j=1}^{k+1}(-1)^{j-1} \di_U  \big(\Lambda(U)[V_1,\dots, \hat V_j,\dots  , V_{k+1}] \big)[V_j]  
\ee
where the notation $ [V_1,\dots, \hat V_j,\dots,  V_{k+1}]$ denotes the $k$-tuple obtained excluding the $j$-th component.
\end{definition}

\noindent 
$\bullet$
If $H(U)$ is a $p+2$-homogeneous  $0$-form in $\wt \Uplambda^0_{p+2}$ then 
its exterior differential coincides 
with the usual differential of functions, namely $\di H(U)[ V] = \di_U H(U)[ V]$. Moreover
$ \di H (U)$ is a  $1$-form in $ \wt \Uplambda^{1}_{p+1} $ 
and we define the gradient $ \nabla H(U):= \nabla_U H(U) $ as the vector field in $ \wt \X_{p+1} $
 such that,  cfr. \eqref{1form.rep}, 
\be\label{def:grad}
 \di H (U)[V]:=  \langle \nabla H (U),
 V\rangle_r \, , \quad \forall V \in \dot{H}^{\infty}(\T;\C^{2}) 
  \,  . 
\ee
$\bullet$
If $\theta(U) $ is a $ (p+1)$-homogeneous  $ 1$-form in $ \wt \Uplambda^1_{p+1} $ 
written as in \eqref{1form.rep}
 then its exterior differential is 
\be \label{der:uan}
\di \theta(U)[V_1,V_2] =  \Big\langle \big(\di_U X(U) - \di_U X(U)^\top\big) V_1, V_2 \Big\rangle_r  , \quad X(U):= M(U)U
\ee 
where  $\di_U X(U)$ and $\di_U X(U)^\top$ are, by the first remark below Definition \ref{Def:Maps} and Lemma \ref{primecosette},
 matrices of operators in $ \wtcM_p\otimes \mM_2(\C)$.  
 Moreover 
 $\di \theta(U)$ 
 belongs to  $ \wt \Uplambda^{2}_p $.  
 
\begin{definition}\label{def:dform}
Let $ p, p' \in \N_0 $ and set $r:=p+2-k$.  
Given  a $r$-homogeneous $ k$-form $ \Lambda(U) $  in $   \wt \Uplambda_r^k $  
and a matrix of homogeneous operators $ M(U)$  in 
$  \wtcM_{p'} \otimes \cM_2(\C) $ we define the 
\\[1mm]
$ \bullet $
{\bf Pull back} of $\Lambda(U)$ via the map $\varphi(U) := M(U)U$ as 
\be \label{pullback}
(\varphi^* \Lambda)(U)[V_1,\dots,V_k]:= \Lambda(\varphi(U))[\di_U \varphi(U)V_1,\dots,\di_U \varphi(U)V_k] \, . 
\ee
$ \bullet $  {\bf Lie derivative} of $\Lambda(U)$ along the vector field  $X(U) := M(U)U$ as
\be\label{Lie}
(\mL_X \Lambda)\,(U)[V_1,\dots,V_k]:= \di_U \Lambda(U)[X(U)][V_1,\dots,V_k]+ \sum_{j=1}^k \Lambda(U)[V_1,\dots, \di_U X(U)[V_j],\dots,V_k] .
\ee
$ \bullet $  {\bf Contraction} of $\Lambda(U)$ with the vector field  $X(U) = M(U)U$ as
\be\label{contraction}
(i_X  \Lambda)(U)[V_1,\dots,V_{k-1}]:=   \Lambda(U)[X(U),V_1,\dots,V_{k-1}] \, . 
\ee
\end{definition}

Let  $ \Lambda(U)$ is a $r$-homogeneous $k$-form in $ \wt\Uplambda_{r}^k$,
$ k = 0,1,2 $. Thanks to the first bullet below Definition \ref{Def:Maps}, 
Lemma \ref{primecosette} and $(i)$ and 
$(ii)$ of Proposition \ref{composizioniTOTALI} (see also
\eqref{pullback.form1}-\eqref{pullback.form2}), one has the following:

 \noindent
$\bullet$
 if $\varphi(U)= M(U)U$ is a map where $  M(U)$ is a pluri-homogeneous operator in $ \Sigma_{0}^N \wtcM_{q}  \otimes \cM_2(\C) $
  then  $(\varphi^*\Lambda)(U)$ defined in \eqref{pullback} belongs to  $\Sigma_{r} \wt \Uplambda^{k}_q $;

 \noindent
$\bullet$
if  $ X(U) $ is a homogeneous vector field 
in $ \Sigma_{p'+1} \mathfrak{ \wt X}_q $ for some $p' \in \N_0$,  then
$(\mL_X \Lambda)\,(U)$ defined in \eqref{Lie} belongs to 
$\Sigma_{r + p'  } \wt \Uplambda^{k}_q $; 

 \noindent
$\bullet$
if $k=1,2$ then $(i_X  \Lambda)(U)$ defined in \eqref{contraction} belongs to $\Sigma_{r+p'+1 } \wt \Uplambda^{k-1}_q $;

 \noindent
$\bullet$
the basic identities of differential geometry are directly verified 
for pluri-homogeneous $k$-forms:
Let $p\in \N_0$, $k=0,1,2$. 
Then for any $ \Lambda $ in $  \wt \Uplambda_{p+2}^k $ it results
\be\label{d2=0}
\di^2 \Lambda = 0 \, .
\ee
Given $\varphi(U):= M(U)U$ with
$ M(U) $ in $  \Sigma_0 \wtcM_q \otimes \cM_2(\C)$, it results 
\be\label{commu}
 \di (\varphi^* \Lambda)\,  = (\varphi^* \di \Lambda)\, \, . 
 \ee
 Given also $\phi(U) := M'(U)U$ with $M'(U) \in \Sigma_0 \wt \cM_q \otimes \cM(\C)$, it results
 \be\label{pullback2}
 \phi^* \varphi^* \Lambda = (\varphi \circ \phi)^* \Lambda   \ . 
 \ee
Given $ X(U) $ in $ \Sigma_{1} \mathfrak{ \wt X}_q $ then 
 \be\label{cartan}
 \mL_X\Lambda   = \di \circ i_X \Lambda + i_X \circ \di \Lambda  \, ;
\ee
 \noindent
$\bullet$
if $\varphi(U)=  M'(U)U$ is a map where $  M'(U)$ is a pluri-homogeneous operator in $ \Sigma_{0} \wtcM_{q}  \otimes \cM_2(\C) $ and $\theta(U) \in \wt\Uplambda^1_{p+1}$ and
$\Lambda(U) \in \wt \Uplambda^2_p$ are represented as in 
\eqref{1form.rep}, \eqref{2form.rep} then
  \begin{align}\label{pullback.form1}
  & (\varphi^*\theta)(U)[V] 
  =
  \la \di_U \varphi(U)^\top \, \, M(\varphi(U)) \, \varphi(U) , V \ra  \, ; \\
  \label{pullback.form2}
&  (\varphi^*\Lambda)(U)[V_1, V_2] 
  =
  \la \di_U \varphi(U)^\top \, \, M(\varphi(U)) \, \, \di_U \varphi(U) V_1, V_2 \ra  \, . 
  \end{align}

In Section \ref{sec:Darboux} we shall use the following 
result about Lie derivatives and approximate flows. 

\begin{lemma}\label{lem:pullback}
Let $p, N \in \N$ with  $ p\leq N$. 
Let  $ \theta^\tau $ be a $\tau$-dependent family of $1$-forms in $ \Sigma_{1} \wt \Uplambda_q^1 $ defined for   $\tau \in [0,1]$.  Let $ \mF_{\leq N}^\tau$ be the approximate flow  generated  
by a pluri--homogeneous, $\tau$--dependent 
 smoothing vector field $Y^\tau(U) $ in $ \Sigma_{p+1}
\wt \X_q^{-\varrho} $, defined for $ \tau \in [0,1] $  and some $\varrho \geq 0$  
(cfr. Lemma \ref{extistencetruflow}). Then 
\be
 \label{deriv:pullback}
 \frac{\di}{\di \tau} ( \mF_{\leq N}^\tau)^*\theta^\tau  = ( \mF_{\leq N}^\tau)^*\big[\mL_{Y^\tau}\theta^\tau + \pa_\tau \theta^\tau\big] +\theta_{>N+1}^\tau 
\ee
where $ \theta_{>N+1}^\tau $ 
is a pluri--homogeneous $1$--form in $ \Sigma_{N+2} \wt \Uplambda_q^1 $, with estimates uniform in $ \tau \in [0,1]$.  
\end{lemma}

\begin{proof}
 Recalling  the definition of pullback \eqref{pullback} and using that $ \mF_{\leq N}^\tau$ fulfills the approximate equation \eqref{appX} (with $Y^\tau$ replacing $X^\tau $) we get 
$$
\begin{aligned}
&\frac{\di}{\di \tau} ( \mF_{\leq N}^\tau)^*\theta^\tau \, (U)[ \hat U]  =  \frac{\di}{\di \tau} \big(\theta^\tau(\mF_{\leq N}^\tau(U))[ \di_U \mF_{\leq N}^\tau(U) \hat U] \big)\\
& =  \pa_\tau \theta^\tau(\mF_{\leq N}^\tau(U))[ \di_U \mF_{\leq N}^\tau(U) \hat U]+\di_V \theta^\tau(\mF_{\leq N}^\tau(U))[ Y^\tau( \mF_{\leq N}^\tau(U)) + R_{>N}^\tau(U)U] [ \di_U \mF_{\leq N}^\tau (U) \hat U]  \\
& \quad + \theta^\tau(\mF_{\leq N}^\tau(U))[ \di_V Y^\tau( \mF^\tau_{\leq N}(U))\di_U \mF^\tau_{\leq N}(U) \hat U+ \di_U ( R^\tau_{>N}(U) U )  \hat U]\\
&\!\!\!\!\!\!\!\!\!\!\stackrel{\eqref{Lie},\eqref{pullback}}{=}  ( \mF_{\leq N}^\tau)^* \big[\mL_{Y^\tau}\theta^\tau   +\pa_\tau \theta^\tau \big](U) [ \hat U]+ \theta_{>N+1}^\tau(U)[\hat U] 
\end{aligned}
$$
where 
\be\label{theta>N.lie}
 \theta_{>N+1}^\tau (U)[\hat U]:= \di_V \theta^\tau(\mF_{\leq N}^\tau(U))[  R_{>N}^\tau(U)U] [ \di_U \mF_{\leq N}^\tau (U) \hat U]+\theta^\tau(\mF_{\leq N}^\tau(U))[ \di_U ( R^\tau_{>N}(U) U )  \hat U]
 \, . 
\ee
We now verify that
$ \theta_{>N+1}^\tau  $ is a $ 1$-form in $  \Sigma_{N+2} \wt \Uplambda_q^1$.
Representing, by  Lemma \ref{lem:ident.forms},   
$$
\theta^\tau (U)[\hat U] = \langle M^\tau (U)U, \hat U\rangle_r \quad \text{with} \quad 
M^\tau (U) \in  \Sigma_{0} \wtcM_q \otimes \cM_2(\C) \, , 
$$ 
 and since 
 $\cF^\tau_{\leq N}(U) = (\id + R^{\tau}_{\leq N}(U))U$ with 
 $R^{\tau}_{\leq N}(U) $ in  
$ \Sigma_{p}^N \wt \cR^{-\varrho}_q \otimes \cM_2(\C)$,  
the $ 1$-form  $\theta_{>N+1}^\tau $ in \eqref{theta>N.lie} reads
\begin{align*}
\theta^\tau_{>N+1}(U)[\hat U] = \langle M'(U)U, \hat U \rangle_r   \quad \text{where} \quad 
M'(U) := &
\left[  \di_U \mF_{\leq N}^\tau(U) \right]^\top \, \left.\di_V (M^\tau (V)V)\right|_{V = \mF_{\leq N}^\tau(U)} \, R_{>N}^\tau(U) \\
&  \! \! \! + 
 [\di_U ( R^\tau_{>N}(U)U)]^\top M^\tau (\mF_{\leq N}^\tau(U))(\id + R^{\tau}_{\leq N}(U)) 
\end{align*}
is a $ \tau$-dependent 
matrix of operators in $ \Sigma_{N+1} \wtcM_q \otimes \cM_2(\C)$, with estimates uniform 
in $ \tau \in [0,1] $,  
because $R^\tau_{>N}(U) $ are smoothing operators in $  \Sigma_{N+1} \wt \cR^{-\varrho}_q \otimes \cM_2(\C)$,  and using $(i)$, $(ii)$ of  Proposition \ref{composizioniTOTALI}  and Lemma \ref{primecosette}. 
Thus $\theta_{>N+1}^\tau  $ is a $ 1$-form in $  \Sigma_{N+2} \wt \Uplambda_q^1$,
with estimates uniform 
in $ \tau \in [0,1] $.
\end{proof}

\section{Hamiltonian formalism}

Along the paper we consider real Hamiltonian systems and their symplectic structures 
in real,  complex and Fourier coordinates, that we describe in Section \ref{sec:HAMbasic}. 
In Section \ref{sec:uptoN}
we introduce the notion of vector fields which are Hamiltonian up to homogeneity $ N $ and we prove 
that the classical 
Hamiltonian theory is 
preserved ``up to homogeneity $ N $". 
In Section \ref{sec:SLF} we present results about linear symplectic flows. In Section \ref{sec:LHam}
discuss Hamiltonian systems with a paradifferential structure. 

\subsection{Hamiltonian and  symplectic structures}\label{sec:HAMbasic}

\paragraph{\bf Real Hamiltonian systems.}
We equip  the real phase 
 space $\dot  L^2_{r} \times \dot L^2_r $ 
 with the scalar product in \eqref{real.bil.form.intro} and the symplectic form
\begin{equation}
\label{sympl.form}
\Omega_0\left (\vect{\eta_1}{ \zeta_1}, \vect{\eta_2}{ \zeta_2} \right) := 
\Big\la E_0
\vect{\eta_1}{\zeta_1}, \vect{\eta_2}{\zeta_2} \Big\ra_r 
= - \la \zeta_1, \eta_2 \ra_{\dot L_r} + \la \eta_1, \zeta_2 \ra_{\dot L_r}
\end{equation}
where $ E_0 $ is the symplectic operator acting on $\dot L^2_r\times \dot L^2_r$ defined by 
\be \label{Ezzero}
 E_0 := \begin{pmatrix}
0 & -{\rm Id} \\
{\rm 
Id} & 0
\end{pmatrix} \, . 
\ee
The Hamiltonian vector field $ X_H $ associated to a (densely defined) 
Hamiltonian function 
$ H : \dot L^2_{r} \times \dot L^2_r \to \R $ is characterized as  the unique 
vector field satisfying 
\begin{equation}\label{defHS}
\Omega_0\left( X_H(\eta, \zeta), \vect{\breve \eta}{ \breve \zeta}\right) = \di H(\eta, \zeta)\left[\vect{\breve \eta}{ \breve \zeta}\right] , \quad \forall \vect{\breve \eta}{\breve \zeta} \in \dot  L^2_{r} \times \dot L^2_r \, . 
\end{equation}
As
\begin{align}
\di H(\eta, \zeta)\left[\vect{\breve \eta}{\breve \zeta}\right] & = \di_\eta H(\eta, \zeta)[\breve \eta] + \di_\zeta H(\eta, \zeta)[\breve \zeta] 
 = \Big\la \vect{\grad_\eta H}{\nabla_\zeta H}, \vect{\breve \eta}{\breve \zeta} \Big\ra_r \label{id1f}
\end{align}
where $ (\nabla_{\eta}H, \nabla_\zeta H) \in \dot L^2_r \times \dot L^2_{r}   $ 
denote the $\dot  L^2_r $- gradients, 
the  Hamiltonian vector field is given by 
\be\label{realHS}
X_H = J \vect{\grad_\eta H}{\grad_\zeta H}  = \vect{\grad_\zeta H}{-\grad_\eta H} 
\quad \text{where} \quad J := E_0^{-1} = \begin{pmatrix}
0 & \uno \\
-\uno & 0
\end{pmatrix}.
\ee 
We also denote by 
\be\label{thetezero}
\theta_0(\eta,\zeta)\Big[\vect{\breve \eta}{\breve \zeta}\Big]:= \frac12 \Big\la E_0 \vect{\eta}{\zeta}, \vect{\breve \eta}{\breve \zeta} \Big\ra_r
\ee
the Liouville $1$--form. Note that $ \di \theta_0= \Omega_0$, where the  
exterior differential is recalled in Section \ref{SGeo}.

\paragraph{\bf Real linear Hamiltonian systems.}
We now consider the most general quadratic real Hamiltonian
\be\label{defbA}
H(\eta, \zeta)  = \frac12 
\Big\la \bA  \vect{\eta}{\zeta}, \vect{\eta}{\zeta} \Big\ra_r 
 , \quad
\bA := \begin{pmatrix}
A & B \\
B^\top & D 
\end{pmatrix} \, , \quad \bA^\top = \bA \, ,
\ee
where $A, B, D$ are linear real  operators acting on $ \dot L^2_r $
and the operators $ A, D $ are symmetric, i.e. $A^\top = A$, $D^\top = D$,  where $ A^\top $ denotes the transpose operator with respect to the 
real scalar product $\la \ , \ \ra_{\dot L^2_r} $.
 $ \bA^\top $ is the transpose  
with respect to  the scalar product $\la \cdot, \cdot \ra_r $ in \eqref{real.bil.form.intro}. 
\begin{definition}{\bf (Linear Hamiltonian operator)} 
A linear  operator  $ \cA $ acting on (a dense subspace) of $ \dot L^2_r \times \dot L^2_r $
is Hamiltonian if it 
has the form
\begin{equation}\label{cA}
\cA = J \bA = J  \begin{pmatrix}
A & B \\
B^\top & D
\end{pmatrix} , \qquad \bA = \bA^\top \, , 
\end{equation}
with $A, B, D$ real operators satisfying $ A = A^\top $ and $ D = D^\top $; equivalently if 
 $ E_0 \cA = \bA $ is symmetric with respect to the 
 real scalar product $ \la \cdot, \cdot \ra_r$ defined in \eqref{real.bil.form.intro}. 
 \end{definition}

We now provide the  characterization of a real linear Hamiltonian 
paradifferential operator.
In view of \eqref{cA} and \eqref{A1b} a
 matrix of paradifferential operators is Hamiltonian if it has the form 
\be\label{HSrealpara}
J \Opbw{ \begin{bmatrix}
a(x, \xi) & b(x, \xi) \\
b(x, - \xi) & d(x, \xi)
\end{bmatrix} } =
\Opbw {\begin{bmatrix}
b(x, - \xi) &d(x, \xi)  \\
- a(x, \xi)  & - b(x, \xi)
\end{bmatrix} }
\ee
with 
 \begin{equation}\label{HSrealpara1}
 \begin{aligned} 
& a(x, \xi) \in \R , \ a(x, \xi) = a(x, - \xi) \, , 
\quad d(x, \xi) \in \R , \ d(x, \xi) = d(x, - \xi) \, , \quad  b(x, -\xi)  = \bar{b(x, \xi)} \, . 
 \end{aligned}
 \end{equation}

\paragraph{\bf Real Hamiltonian systems in complex coordinates.}
We now describe the above real Hamiltonian systems
in the  complex coordinates defined by the change of variables 
\begin{equation}
\label{cc}
\vect{\eta}{\zeta} = \cC \vect{u}{\bar u} \, ,
\quad 
\cC := \frac{1}{\sqrt 2} \left(\begin{matrix}
\uno & \uno \\ -\im & \im
\end{matrix}\right) \ ,\quad \cC^{-1}:= \frac{1}{\sqrt 2}\left(\begin{matrix}
\uno & \im \\ \uno & -\im
\end{matrix}\right)\,,\quad \cC^\top := \frac{1}{\sqrt 2} \left(\begin{matrix}
\uno &-\im  \\ \uno & \im
\end{matrix}\right) \,.
\end{equation}
 Note that $ \cC $ is a map between the real subspace of vector functions 
$\dot L^2_\R (\T, \C^2)$ into 
$\dot  L^2_r \times \dot L^2_r $. 
In the sequel to save space we denote $ \vect{u}{\bar u} $ also as $ (u, \bar u) $.

The pull-back  $\Omega_c := \cC^* \Omega_0$ 
of the symplectic form $\Omega_0$ in \eqref{sympl.form} is 
\be\label{sfc}
\Omega_c\left(\vect{u_1}{ \bar u_1}, \vect{u_2}{ \bar u_2}\right) 
=  \Big\la \cC^\top E_0 \cC \vect{u_1}{ \bar u_1}, \vect{u_2}{ \bar u_2} \Big\ra_r 
 =  \Big\la E_c \vect{u_1}{ \bar u_1}, \vect{u_2}{ \bar u_2} \Big\ra_r\
\ee
where $ E_c $ is the symplectic operator acting on $\dot L^2_\R (\T, \C^2)$
\be \label{Ecci}
E_c := \begin{pmatrix}
0 &  - \im \\
\im & 0 
\end{pmatrix} = \im E_0 \, .
\ee
Remark that
$ E_c^\top = - E_c $ and $ E_c^2 = \uno $. Similarly the Liouville $ 1$-form $\theta_0$ in \eqref{thetezero} is transformed into the symplectic form $\theta_c := \cC^* \theta_0$ given by
\be\label{defthetac}
\theta_c(U) \left[ V\right] = \frac{1}{2} \la E_c U, V \ra_r \, , \quad
\forall \ U = \vect{u}{\bar u} \, , \ V = \vect{v}{\bar v} \, 
\ee
and it results 
\be\label{exactc}
\di \theta_c = \Omega_c \, .
\ee
Next we show how the differential and gradient of a Hamiltonian transform under the complex change of coordinates. 
The pull-back under $ \cC $ of the 
$1$-form (cfr.  \eqref{id1f})
$\di H(\eta, \zeta)[\cdot] = \Big\la \begin{pmatrix}
\grad_\eta H \\
\grad_\zeta H
\end{pmatrix}, \cdot  \Big\ra_r  
$
is 
\begin{align} 
(\cC^*\di H)(u, \bar u)[(v,\bar v)]
 = 
 \Big\la  \begin{pmatrix}
\grad_u H \\
\grad_{\bar u} H
\end{pmatrix} \Big\vert_{\cC(u, \bar u)}, \   
 \begin{pmatrix}
v \\
\bar v 
\end{pmatrix}  \Big\ra_r  \label{ultimac}
\end{align} 
where 
\begin{equation}\label{gradcsi}
\grad_u H :=\frac{1}{\sqrt 2}\left( \grad_\eta H  - \im \grad_\zeta H \right)\vert_{\cC(u, \bar u)} 
\, , \quad 
\grad_{\bar u} H := \frac{1}{\sqrt 2} \left(\grad_\eta H  + \im \grad_\zeta H\right)\vert_{\cC(u, \bar u)} \, . 
\end{equation}
Furthermore, by  \eqref{cc}, 
\begin{align*}
(\cC^*\di H)(u, \bar u)[(v,\bar v)]
  = \di_u H \vert_{\cC(u, \bar u)}[v] + \di_{\bar u} H \vert_{\cC(u, \bar u)}[ \bar v] 
\end{align*}
having  defined
$$
\di_u H :=\frac{1}{\sqrt 2}\left( \di_\eta H  - \im \di_\zeta H \right)\vert_{\cC(u, \bar u)} \, ,
\quad
\di_{\bar u} H := \frac{1}{\sqrt 2} \left(\di_\eta H  + \im \di_\zeta H\right)\vert_{\cC(u, \bar u)} \, .
$$
In the sequel we also use the compact notation, 
given $ U = (u, \bar u) $, 
$$
\di_U H(U)[\widehat U] := \di_uH(U)[\widehat u] + \di_{\bar u}H(U)[\bar{\widehat u}] \, ,
\quad \forall \widehat U = \vect{\widehat u}{\bar  {\widehat u}} \, . 
$$
\paragraph{\bf Real Hamiltonian vector fields in complex coordinates.}\label{realHam.com}
Given a real valued Hamiltonian $ H(\eta, \zeta)$, consider the Hamiltonian in complex coordinates $H_c:= H \circ {\cal C} $ which is a function of $(u,\bar u)$.
Recalling  the characterization \eqref{defHS} of the Hamiltonian vector field  and \eqref{ultimac},   
the associated Hamiltonian vector field is   
\be\label{complexHS}
X_{H_c}(U):= \cC^{-1} (X_{H})_{|\cC (U)} = J_c \vect{\grad_u H_c }{\grad_{\bar u} H_c}  = J_c \nabla H_c(U) = 
\vect{- \im \grad_{\bar u} H_c}{ \im \grad_u H_c}
\ee
where  $ J_c := E_c^{-1} = E_c $ is the Poisson tensor in \eqref{JcH.intro}.
One has also the characterization 
\be \label{complexvecham}
\Omega_c( X_{H_c}, \cdot )= \di_U H_c(U)[\cdot ] \, .
\ee
 In  case $H$ is the quadratic form \eqref{defbA}, 
the transformed Hamiltonian $H_c = H \circ \cC$ is given by 
\begin{align}
 H_c(u, \bar u)  &  =  \frac12 \Big\la \bR \vect{u}{\bar u}, \vect{u}{\bar u} \Big\ra_r 
 , 
 \quad
\bR:= \cC^\top \bA \cC := \begin{pmatrix}
R_1 & R_2 \\
\bar R_2 & \bar R_1
\end{pmatrix} 
 \label{Hc}
\end{align}
where  $R_1 := (A-D) - \im (B+B^\top)$,  $R_2 := (A+D) + \im (B-B^\top)$.
The operator $\bR$ is  real-to-real  according to \eqref{vinello}. 
In addition, since $ \bA $ is symmetric, cfr. \eqref{defbA},
the operator $\bR$ is symmetric with respect to 
the real non-degenerate bilinear form $\la  \cdot, \cdot \ra_r$, namely 
\begin{equation}
\label{R.c}
\bR = \bR^\top \, , \quad \text{i.e.} \ \ 
R_1 = R_1^\top \, , \ \ R_2^* =  R_2 \, . 
\end{equation}

\begin{definition} {\bf (Linear Hamiltonian operator in complex coordinates)} \label{def:LH}
A real-to-real linear operator $ J_c \bM $ is 
{\em linearly Hamiltonian} if 
 $ \bM = \bM^\top $ is symmetric with respect to 
the  non-degenerate bilinear form $\la  \cdot, \cdot \ra_r$, cfr. \eqref{R.c}. 
\end{definition}

In view of \eqref{A1b} and \eqref{R.c}
a matrix of paradifferential real-to-real complex operators is  linearly 
 Hamiltonian if 
\be\label{LHS-c}
J_c 
\Opbw{\begin{matrix}
b_1(U;t, x,\xi)  & b_2(U; t, x, \xi) \\
 \bar{b_2(U;t, x, -\xi)} &  \bar{b_1(U; t, x, -\xi)}
\end{matrix}} \, , 
\quad 
\begin{cases}
b_1(U;t, x, - \xi) = b_1(U;t, x, \xi) \, ,  \\
b_2(U;t, x, \xi) \in \R \, , 
\end{cases}
\ee
namely $ b_1 $ is  even in $ \xi $ and $ b_2 $ is  real valued.

\begin{definition} {\bf  (Linearly symplectic map)} \label{LS}
 A real-to-real linear transformation
  $ \mA$ is linearly symplectic if $ \cA^* \Omega_c = \Omega_c $
where $\Omega_c$ is defined in \eqref{sfc}, namely 
$ \mA^\top \, E_c \, \mA = E_c $, 
where $E_c$ is the symplectic operator defined in \eqref{Ecci}. 
\end{definition}

\noindent{\bf Hamiltonian systems in Fourier basis.}
Given a Hamiltonian $H(U)$ expanded as in \eqref{polham} we characterize its Hamiltonian vector field. 
 We decompose each Fourier coefficients as 
$u_j = \frac{x_j + \ii y_j}{\sqrt{2}}$, $ \ov{u_j} = \frac{x_j - \ii y_j}{\sqrt{2}}$,  
where $ x_j := \sqrt{2} \Re (u_j) $ and  $ y_j := \sqrt{2} \Im (u_j) $ and we define
\be\label{defpaupaubar} 
\pa_{u_j}:= \frac{1}{\sqrt{2}}\big( \pa_{x_j} - \ii \pa_{y_j} \big) \, , 
 \quad  \pa_{\bar{u_j}} := \frac{1}{\sqrt{2}}\big( \pa_{x_j} +\ii \pa_{y_j} \big) \, , 
\ee
so that 
$ \pa_{u_j^{\sigma} } u_j^\sigma = 1 $, for any $  \sigma = \pm $, and 
$ \pa_{u_j^{\sigma}} u_j^{\sigma'} = 0 $, for any $ \sigma \sigma' = -1 $.
For a real valued Hamiltonian $ H $ it results 
\be\label{readdH}
\overline{\pa_{u_j} H} = \pa_{\overline{u_j}} H   \, . 
\ee
We now write a Hamiltonian vector field \eqref{complexHS} in the coordinates 
$ (u_j)_{j \in \Z \setminus \{0\}} $. 
For notational simplicity we also denote $ u \equiv (u_j)_{j \in \Z 
\setminus \{0\}} $.  
We first note that, by \eqref{sfc} and \eqref{fourierseries}, 
the symplectic form \eqref{sfc} reads, for any
$ U  = (u, \bar u)$, $ V = (v, \bar v)$, 
\be\label{Fex}
\Omega_c\big( U, V\big) = 
\sum_{j \in \Z \setminus \{0\} } - \ii \ov{u_j} v_j + \ii u_j \ov{v_j} =
 - \im   \sum_{j \in \Z \setminus \{0\} \atop \sigma \in \{\pm\} } \sigma u_j^{-\sigma} v_j^\sigma  \, . 
\ee

\begin{lemma}
{\bf (Fourier expansion of a Hamiltonian vector field)}
The Fourier 
components of the Hamiltonian vector field  associated to a real Hamiltonian $ H (U)$
are, for any  $ \sigma = \pm $, $ k \in \Z \setminus \{0\} $,  
\be \label{hamvecco0}
\big( J_c\nabla H(U)\big)^\sigma_k = -\ii \sigma \pa_{u_k^{-\sigma}} H(U) \, . 
\ee
In particular, if the Hamiltonian $ H $ is expanded as in \eqref{polham}, then 
\be \label{hamvecco}
\big( J_c\nabla H(U)\big)^\sigma_k = 
 -\ii \sigma 
 \!\!\!\!\!\!\!\!\!\!\! \sum_{
(\vec{\jmath}_{p+1},k, \vec{\sigma}_{p+1}, -\sigma)\in 
\fT_{p+2}
 } \!\!\!\!\!\!\!\!\!\!\!\!\!\!\!\!\!
   H_{\vec{\jmath}_{p+1},k}^{\vec{\sigma}_{p+1},-\sigma } u_{\vec{\jmath}_{p+1}}^{\vec{\sigma}_{p+1}} \, . 
\ee
\end{lemma}

\begin{proof}
The expression \eqref{hamvecco0} is a consequence of \eqref{Fex}
and the definition \eqref{complexvecham} of a Hamiltonian vector field, using 
\eqref{readdH}. 
Finally  \eqref{hamvecco} is a consequence of  \eqref{hamvecco0}, \eqref{realham}, \eqref{symham}. 
\end{proof}

 By \eqref{hamvecco} 
and  \eqref{polvect} we deduce the following characterization of Hamiltonian vector fields:

\begin{lemma}  {\bf (Characterization of $(p+1)$--homogeneous Hamiltonian vector fields)}\label{carhamvec}
A $(p+1)$--homogeneous 
real-to-real  vector field $X(U) \in \wt \X_{p+1} $ of the form \eqref{polvect} 
 is Hamiltonian if and only if  the coefficients 
\be \label{hamvec}
H_{\vec \jmath_{p+1},k}^{ \vec{\sigma}_{p+1},\sigma }:=
- \ii \sigma  X_{\vec{\jmath}_{p+1},k}^{\vec{\sigma}_{p+1}, -\sigma} , 
\qquad 
\forall  ( \vec \jmath_{p+1},k, \vec \sigma_{p+1}, \sigma) \in \fT_{p+2} \, ,  
\ee
satisfy \eqref{symham}, \eqref{realham} and \eqref{upperbounH}.
In such a case $ X(U) $ is the Hamiltonian vector field generated by 
$$
H(U) = \frac{1}{p+2}\sum_{
(\vec \jmath_{p+2}, \vec \sigma_{p+2}) \in \fT_{p+2}
 } 
H_{\vec \jmath_{p+2}}^{\vec \sigma_{p+2} } 
u_{\vec \jmath_{p+2}}^{\vec \sigma_{p+2}} \, , \quad 
H_{\vec \jmath_{p+2}}^{\vec \sigma_{p+2}}  := 
- \ii \sigma_{p+2}  X_{\vec{\jmath}_{p+1},j_{p+2}}^{\vec{\sigma}_{p+1}, -\sigma_{p+2}} \,  .
$$
\end{lemma}

The Poisson bracket between two real functions $ F, G $ is, by 
\eqref{complexvecham}, \eqref{Fex}, \eqref{hamvecco0}, \eqref{readdH}, 
\be\label{Pois}
\{F,G\} := \di F (X_G) = \Omega_c (X_F,X_G) = 
\sum_{j \in \Z \setminus \{0\}} \ii \Big( \pa_{\ov{u_j}} F  \pa_{u_j} G -  
\pa_{u_j} F  \pa_{\ov{u_j}} G    \Big)  \, .
\ee
Note that the right hand side of \eqref{Pois} is well--defined also for
complex valued functions $F$ and $ G $ and, 
with a small abuse of notation, we shall still refer to it as the Poisson bracket 
between $ F $ and $ G $. 

\subsection{Hamiltonian systems up to homogeneity $ N $}\label{sec:uptoN}

Along the paper we encounter vector fields which are Hamiltonian up to homogeneity $N$. We distinguish between linear and nonlinear ones.

\paragraph{Linear Hamiltonian operators.} 
In the sequel let $ p, N, K, K'  \in \N_0 $ and $ K' \leq K $, $ r > 0 $. 

\begin{definition}
{\bf (Linearly Hamiltonian operator up to homogeneity  $N$)}
\label{def:LHN}
A real-to-real  matrix of spectrally localized maps 
$J_c {\bf B} (U;t) $ in $ \Sigma {\cal S}_{K,K',p}[r,N] \otimes \mM_2 (\C) $
 is {\em linearly Hamiltonian up to homogeneity $N$} if 
 the  pluri-homogeneous component 
 $ \mP_{\leq N} (\bB (U;t))$ (defined in \eqref{pienne}) is symmetric,  namely 
\be\label{ipoBnBtN}
\mP_{\leq N} (\bB(U;t)) =   \mP_{\leq N} (\bB(U;t)^\top)  \, . 
\ee
\end{definition}

In particular, a matrix of paradifferential real-to-real complex operators is 
 linearly Hamiltonian up to homogeneity $N$ if it has the form    (cfr. \eqref{LHS-c})
\be\label{LHS-cN}
\!\!\! J_c 
\Opbw{\begin{matrix}
b_1(U;t, x,\xi)  & b_2(U; t, x, \xi) \\
 \bar{b_2(U;t, x, -\xi)} &  \bar{b_1(U; t, x, -\xi)}
\end{matrix}} \, ,  \ \begin{cases} b_1(U;t, x, - \xi) - b_1(U;t, x, \xi) \in  \Gamma^m_{K,K',N+1}[r]  \\ \Im  
b_2(U;t, x,\x) \in \Gamma^{m'}_{K,K',N+1}[r] 
 \end{cases} 
\ee
for some $m, m' $ in $ \R $.

\begin{definition} {\bf (Linearly symplectic map up to homogeneity $N$)} 
\label{LSUTHN}
A real-to-real matrix of spectrally localized maps $\bS(U;t) $ in $ \Sigma \mS_{K,K',0}[r,N] \otimes \mM_2 (\C) $ 
 is {\em  linearly symplectic up to homogeneity $N$} if   
\be \label{A:sym23}
\bS(U;t)^\top  \, E_c \, \bS(U;t)= E_c+ S_{>N}(U;t)  
\ee
where $ E_c $ is the symplectic operator defined in \eqref{Ecci} and 
$S_{>N}(U;t) $ is a 
matrix of spectrally localized maps 
in $ \mS_{K,K',N+1}[r] \otimes \cM_2(\C)$. 
\end{definition}

The approximate inverse up to homogeneity $N$ of  a linearly symplectic map up to homogeneity $N$ is still 
linearly symplectic up to homogeneity $N$.
\begin{lemma}\label{inv.lin.simp}
Let $p, N \in \N$ with $p \leq N$.
Let $\Phi_{\leq N}(U) := \bB_{\leq N}(U)U$  be  such that
$\bB_{\leq N}(U) - \uno \in \Sigma_p^N \wt \cS_q \otimes \cM_2(\C)$
and 
$ \bB_{\leq N}(U)$  is  linearly symplectic up to homogeneity $N$  (Definition \ref{LSUTHN}). 
Then its  approximate inverse $\Psi_{\leq N}(V)$, constructed in Lemma \ref{inversoapp}, has the form 
$ \Psi_{\leq N}(V)= \bA_{\leq N}(V)V $ 
where 
 $ {\bf A}_{{\leq N}} (V) - \uno $ is  in  $\Sigma_p^N \wtcS_q \otimes \mM_2 (\C) $ and 
$ {\bf A}_{{\leq N}} (V) $  is  linearly symplectic up to homogeneity $N$, 
more precisely 
\be\label{ATEA}
{\bf A}_{{\leq N}} (V)^\top \, E_c \, {\bf A}_{{\leq N}} (V)  = E_c + S_{>N}(V) 
\ee
where $S_{>N}(V)$ is a matrix of pluri--homogeneous spectrally localized maps in 
 $ \Sigma_{N+1} \wtcS_q \otimes \mM_2 (\C) $.
\end{lemma}
\begin{proof}
As $\bB_{\leq N}(U)$  is symplectic up to homogeneity $N$, one has 
\be\label{B.T}
\bB_{\leq N}(U)^\top \, E_c \, \bB_{\leq N}(U) = E_c + S_{>N}(U)
\ee
where $S_{>N}(U)$ is a pluri-homogeneous operator in $\Sigma_{N+1} \wtcS_q \otimes \mM_2 (\C) $, being the left hand side above  a pluri-homogeneous operator).
Then we  evaluate \eqref{B.T}  at $U = \Psi_{\leq N}(V)$, apply $\bA_{\leq N}(V)$ to the right and $\bA_{\leq N}(V)^\top$ to the left and use \eqref{omogeneaaaa2} and the composition properties in Proposition \ref{composizioniTOTALIs}.
The operator $S_{>N}(V)$ is pluri-homogeneous as the left-hand side of \eqref{ATEA}.
\end{proof}

The class of linearly Hamiltonian operators up to homogeneity $N$ is closed under conjugation 
under a linearly symplectic up to homogeneity $N$ map.

\begin{lemma}\label{coniugazionehamiltoniana}
Let $ J_c \bB(U;t) $ be   a linearly Hamiltonian operator up to homogeneity $N $ 
(Definition \ref{def:LHN})
and $ \mG(U;t) $ be an invertible map,   linearly symplectic to homogeneity $N$ (Definition \ref{LSUTHN}). 
Then  the operators 
$ \mG(U;t) J_c \bB(U;t) \mG(U;t)^{-1} $ and 
$ (\pa_t\mG(U;t)) \mG^{-1}(U;t) $
 are linearly Hamiltonian up to homogeneity $N$. 
\end{lemma}

\begin{proof}
Set $ \bB := \bB(U;t) $  and $ \mG := \mG(U;t) $ for brevity. 
As $\cG$ is invertible, we deduce from \eqref{A:sym23} that $ \cP_{\leq N}\left( \cG J_c  \right) = \cP_{\leq N} \left(J_c [\cG^{-1}]^\top \right)$.
Then 
$$
\begin{aligned}
\mP_{\leq N} \left(\mG J_c \bB \mG^{-1}\right) & =  \mP_{\leq N} \left(\mG J_c\mP_{\leq N}[ \bB] \mG^{-1}\right) \stackrel{\eqref{A:sym23}} =   \mP_{\leq N} \left( J_c[\mG^{-1}]^\top\mP_{\leq N}[ \bB] \mG^{-1}\right) = J_c {\bf M} 
\end{aligned}
$$
where 
$ {\bf M} := 
\mP_{\leq N} \left( [\mG^{-1}]^\top \mP_{\leq N}[ \bB]  \mG^{-1}\right) $ 
is symmetric  since  $\mP_{\leq N}[ \bB^\top] = \mP_{\leq N}[\bB]$.  
This proves that $\mG J_c \bB \mG^{-1}$ is linearly Hamiltonian up to homogeneity $N$.  

Next, differentiating 
\eqref{A:sym23} (with $\mG(U;t)$ replacing $\bS(U;t)$), we get 
$$
\cP_{\leq N}\left[ E_c (\pa_t \cG) \cG^{-1} \right] 
 = - \cP_{\leq N} \left[ (\cG^{-1})^\top (\pa_t \cG)^\top E_c \right] 
 = \cP_{\leq N}\left[ \big(E_c  (\pa_t \cG) \cG^{-1} \big)^\top \right]
$$
showing  that $(\pa_t \cG) \cG^{-1}$ is linearly Hamiltonian up to homogeneity $N$.
\end{proof}

\paragraph{Nonlinear Hamiltonian systems up to homogeneity $ N $.} 
Let $K, K' \in \N_0$ with $ K'\leq K$, $r>0$  and $U \in B_{s_0}^K(I;r)$.
Let 
\be\label{Z.M0}
Z := \bM_0(U;t)U \quad \text{ with }\quad  \bM_0(U;t) \in \cM^0_{K,K',0}[r]\otimes \mM_2(\C)\, . 
\ee
\begin{definition}{\bf (Hamiltonian system up to homogeneity  $N$)} 
\label{def:ham.N}
Let $N, K, K' \in \N_0$  with $K \geq K'+1$ and
assume \eqref{Z.M0}.
A $U$--dependent system 
\be \label{U.Ham}
\pa_t Z = J_c \nabla H(Z) + M_{> N}(U;t)[U] 
\ee
is {\em Hamiltonian up to homogeneity  $N$} if 

\noindent
$\bullet $ $H(Z) $ is a  pluri-homogeneous Hamiltonian  
in $ \Sigma_2^{N+2} \wt \Uplambda_{q}^0 $; 

\noindent
$\bullet $  $M_{> N}(U;t)$ is a matrix of 
non-homogeneous operators in $ \cM_{K,K'+1,N+1}[r] \otimes \mM_2 (\C) $. 
\end{definition}

In view of the first bullet after Definition \ref{def:forms}
the Hamiltonian vector field $ J_c \nabla H(Z) $ is  in $ \Sigma_1^{N+1} \wt \X_{q} $. 

\smallskip

We shall perform nonlinear changes of variables which are symplectic up to homogeneity $N$ according to the following definition.

\begin{definition}\label{def:LSMN}
{\bf (Symplectic map up to homogeneity $N$)} 
Let $ p, N \in \N $ with $p \leq N$. We say that  
\be\label{simplupN}
\mD (Z;t) = M(Z;t)Z  \qquad \text{with} \qquad M(Z;t) - \uno \in \Sigma \mM_{K,K',p}[r,N] \otimes \mM_2 (\C) \, ,  
\ee
   is {\em symplectic up to  homogeneity  $N$}, if its pluri-homogeneous component $\mD_{\leq N}(Z):= \big( \mP_{\leq N}  M(Z;t) \big) Z  $ satisfies 
 \be \label{quasisimplettica}
(\mD_{\leq N})^* \Omega_c \,  = \Omega_c+ \Omega_{>N}
\ee   
where $ \Omega_{>N} $ is a pluri-homogeneous 2-form in 
$ \Sigma_{N+1} \wt \Uplambda^2_q $. 
\end{definition}

Equivalently, 
 by \eqref{sfc} and 
the operatorial representation  \eqref{2form.rep} of $2$-forms,  
the nonlinear map  
$ \mD (Z;t) $ is symplectic up to homogeneity $N $, 
 if  
\be\label{milan}
\left[\di_Z \mD_{\leq N}(Z)\right]^\top \, E_c \, \di_Z \mD_{\leq N}(Z)
= E_c+ E_{>N}(Z) \quad \text{with} \quad 
E_{>N}(Z) \in  \Sigma_{N+1} \wtcM_q \otimes \mM_2(\C)  \, .
\ee
\begin{remark}
In the real setting we say that a map $\mD(\eta,\zeta)$ is symplectic up to homogeneity if its pluri-homogeneous component $\mD_{\leq N} (\eta,\zeta)$ satisfies 
\be\label{real:sym:boh}
\big[ \di_{(\eta,\zeta)} \mD_{\leq N}(\eta,\zeta) \big]^\top \, E_0 \, 
\big[ \di_{(\eta,\zeta)} \mD_{\leq N}(\eta,\zeta) \big]
= E_0+ E_{>N}(\eta,\zeta)
\ee
where $E_0$ is the real symplectic tensor defined in \eqref{Ezzero} and $E_{>N}$ 
is matrix of real operators in $\Sigma_{N+1} \wtcM_q \otimes \mM_2(\C)$.
\end{remark}
\noindent 

We now show that  the usual properties 
of symplectic maps still hold, up to homogeneity $ N $.  
For example the approximate inverse of a symplectic up to homogeneity $ N $
 map is symplectic up to homogeneity $ N $ as well. 

\begin{lemma}\label{lem:ails}
Let $p,N\in \N $ with  $ p \leq N $.  Let  $\mD_{\leq N}(Z)= Z+ M_{\leq N}(Z)Z$ as in 
\eqref{phidicuinv} be symplectic up to homogeneity $ N $.
Then its approximate inverse $\mE_{\leq N}(V) = V + \breve M_{\leq N}(V)V $ 
up to homogeneity $ N $ as in \eqref{LTCPhi}
(provided by Lemma \ref{inversoapp})
is  symplectic up to homogeneity $ N $
as well. 
Moreover
\be\label{lasolita}
\big[\di_Z \mD_{\leq N} (Z)\big] J_c  \big[\di_Z \mD_{\leq N} (Z)\big]^\top  = J_c + J_{> N}(Z) , \quad 
J_{> N}(Z) \in \Sigma_{N+1}\wtcM_{q} \otimes \mM_2 (\C) \ . 
\ee
\end{lemma}

\begin{proof}
As $\mD_{\leq N}(Z)$ is symplectic up to homogeneity $N$, we get that,
 using also the first bullet after Definition \ref{def:dform}, 
\be\label{Phi.sympN}
(\mE_{\leq N})^* (\mD_{\leq N})^* \Omega_c = (\mE_{\leq N})^* \left[ \Omega_c + \Omega_{>N} \right] = (\mE_{\leq N})^*\Omega_c  + \wt \Omega_{>N}
\ee
for some  pluri-homogeneous 2-forms $ \Omega_{>N}, \wt  \Omega_{>N} $ in 
$ \Sigma_{N+1} \wt \Uplambda^2_q $. 
Now recall that,  being $\mE_{\leq N}(V)$ the approximate inverse of $\mD_{\leq N}(Z)$ 
up to homogeneity $ N $,   by  \eqref{invmappe} one has   $\mD_{\leq N}\circ \mE_{\leq N} = \id + F_{>N}$ for some  $F_{>N}(V) = M_{>N}(V)V$ with 
$M_{>N}(V) $  in $\Sigma_{N+1} \wt \cM_q\otimes \cM_2(\C)$.
Thus we can also write 
\be\label{Phi.sympN1}
(\mE_{\leq N})^* (\mD_{\leq N})^* \Omega_c \stackrel{\eqref{pullback2}} = 
\left( \mD_{\leq N}\circ \mE_{\leq N} \right)^* \Omega_c = \left(\id + F_{>N} \right)^* \Omega_c 
\stackrel{\eqref{pullback.form2}}{=} \Omega_c +  \Omega_{>N}'(V)
\ee
for some  pluri-homogeneous 2-form $ \Omega_{>N}'$ in 
$ \Sigma_{N+1} \wt \Uplambda^2_q $
(by the first bullet below Definition \ref{Def:Maps}, 
Lemma \ref{primecosette} and Proposition \ref{composizioniTOTALI}). 
Then \eqref{Phi.sympN}-\eqref{Phi.sympN1}  prove that $\cE_{\leq N}$ is symplectic 
up to homogeneity $N$.

Next we prove \eqref{lasolita}. We start from \eqref{milan} for $\cE_{\leq N}$ evaluated at $\cD_{\leq N}(Z)$, i.e. 
 \be\label{milan2}
\big[ \di_V \mE_{\leq N}(\cD_{\leq N}(Z))\big]^\top \, E_c \, \di_V \mE_{\leq N}(\cD_{\leq N}(Z))
= E_c+ E_{>N}(Z)
\ee
with
$E_{>N}(Z) $ in $ \Sigma_{N+1} \wtcM_q \otimes \mM_2(\C)$, by Proposition \ref{composizioniTOTALI} and Lemma \ref{primecosette}.
Then apply $J_c$ to the left of \eqref{milan2} and 
$\big[\di_Z \mD_{\leq N} (Z)\big] J_c  \big[\di_Z \mD_{\leq N} (Z)\big]^\top$ to the right of it 
and use the first and last of \eqref{diinv2}, 
and Proposition \ref{composizioniTOTALI}
 to deduce \eqref{lasolita}.
\end{proof}

\noindent 
The approximate  flow of a 
Hamiltonian smoothing vector field 
is  symplectic  up to homogeneity $ N $. 

\begin{lemma}\label{lem:app.flow.ham}
Let $ p, N \in \N $ with $  p \leq N $. 
Let $ Y(U)  $ be a  homogeneous 
Hamiltonian  smoothing vector field in 
$ \wt \X_{p+1}^{-\varrho} $ for some $ \varrho \geq 0 $. Then its approximate  flow
$  \mF_{\leq N}^\tau $  (provided by  Lemma \ref{extistencetruflow})
is symplectic up to homogeneity $ N $ (Definition \ref{def:LSMN}).  
\end{lemma}

\begin{proof}
Recalling that  $\Omega_c = \di \theta_c $   we have 
\begin{align}
\frac{\di}{\di\tau} (\mF_{\leq N}^\tau)^* \Omega_c\, 
&
\stackrel{\eqref{commu}}{=} \di \frac{\di}{\di\tau} (\mF_{\leq N}^\tau)^* \theta_c\,  \notag \\
& 
\stackrel{\eqref{deriv:pullback}} 
= \di (\mF_{\leq N}^\tau)^* \mL_Y \theta_c\, + \di \theta_{>N+1}^\tau\,  \notag \\
& \stackrel{\eqref{cartan}, \eqref{commu}, \eqref{d2=0}}{=} (\mF_{\leq N}^\tau)^*\di ( i_Y \Omega_c ) \, + \di \theta_{>N+1}^\tau\,  \notag \\
& \stackrel{\eqref{complexvecham}}{=} (\mF_{\leq N}^\tau)^* \di^2 H_{p+2}\, + \di \theta_{>N+1}^\tau\, 
\stackrel{\eqref{d2=0}}{=} \di \theta_{>N+1}^\tau\,  \label{2forma-appr}
\end{align}
where $ H_{p+2}$ is the Hamiltonian of $ Y(U) $ and 
$\theta_{>N+1}^\tau  $ is a pluri-homogeneous $ 1$-form 
in $  \Sigma_{N+2} \wt \Uplambda^1_q $. 
Integrating \eqref{2forma-appr} from $0$ to $ \tau$, 
and using that $\cF_{\leq N}^0 = \id$,  we get 
$$
(\mF_{\leq N}^\tau)^* \Omega_c \,  = \Omega_c+ \Omega_{>N}^\tau \, , \quad
\Omega_{>N}^\tau := \int_0^\tau  \di \theta_{>N+1}^t \, \di t   	
$$
where $ \Omega_{>N} $ is in $ \Sigma_{N+1} \wt \Uplambda^2_q $.
This proves that $ \mF_{\leq N}^\tau  $ is symplectic up to homogeneity $ N $. 
\end{proof}

A symplectic map up to homogeneity $ N $
transforms a Hamiltonian system up to homogeneity $N$ into 
another Hamiltonian system up to homogeneity $N$.

\begin{lemma}\label{conj.ham.N}
Let $p, N \in \N$ with $p \leq N$,  $K, K' \in \N_0$ with $K \geq K'+1$.
Let $ Z := \bM_0(U;t)U  $ as in \eqref{Z.M0}.
Assume $\mD(Z;t)  = M(Z;t)Z   $ is a symplectic map up to homogeneity $N$ (Definition \ref{def:LSMN}) such that
\be \label{simdico}
M(Z;t) - \uno  \in \begin{cases}
 \Sigma\cM_{K,K',p}[r,N] \otimes \cM_2(\C) \qquad \qquad 
 \text{if} \quad  \bM_0(U;t)  = \id \, , \\
  \Sigma\cM_{K,0,p}[\breve r,N] \otimes \cM_2(\C) \, , \, \forall \breve r > 0 \quad \text{otherwise} \, . 
\end{cases}
\ee 
If  $Z(t)$ 
solves a $U$-dependent Hamiltonian system up to homogeneity $N$
(Definition \ref{def:ham.N}), 
then the variable
$ W := \mD(Z;t)  $
solves another $U$-dependent Hamiltonian system up to homogeneity $N$ 
(generated by the transformed Hamiltonian). 
\end{lemma}

\begin{proof}
Decompose $\cD(Z;t) = \cD_{\leq N}(Z) + M_{>N}^{\cD}(Z;t)Z$
where $ \cD_{\leq N}(Z) := \cP_{\leq N} [ M(Z;t)] Z $ is 
its pluri-homogeneous component and 
$$
M_{>N}^\cD(Z;t)
\in 
\begin{cases} 
\mM_{K, {K'},N+1}[r] \otimes \mM_2 (\C)  
\qquad \qquad  \text{if} \quad  \bM_0(U;t)  = \id \, , \\
\cM_{K,0,N+1}[\breve r] \otimes \cM_2(\C)  \, , \, \forall \breve r > 0 \quad \text{otherwise} \, . 
\end{cases}
$$
By Definition  \ref{def:LSMN} 
the map $\cD_{\leq N}(Z)$ satisfies  \eqref{quasisimplettica}.  
If  $Z(t)$ solves  \eqref{U.Ham} then $ W= \mD(Z;t) $ solves 
\begin{align}
\notag
\pa_t W  &=\big( \di_Z \mD_{\leq N} (Z) +  M_{>N}^\cD(Z;t) \big) \left[ J_c \nabla_Z H (Z) +
 M_{>N}(U;t)U \right]
+ (\pa_t M_{>N}^\cD(Z;t)) Z  \\
  \label{patW}
  & =  \di_Z \mD_{\leq N} (Z) J_c \nabla_Z H (Z) + M'_{>N}(U;t)U
\end{align}
where, by the first bullet below Definition \ref{Def:Maps} and Proposition \ref{composizioniTOTALI}, 
\be\label{mappapeggiore}
M'_{>N}(U;t) 
\in 
\mM_{K, {K'+1},N+1}[r] \otimes \mM_2 (\C)   \ . 
\ee
Denote by $  \breve \mD_{\leq N}(W) $ the approximate inverse  up 
to homogeneity $ N $ of $\mD_{\leq N}(Z)$ (see Lemma \ref{inversoapp}).  
Then
   $$
     \breve\cD_{\leq N}(W)     
     =  \breve \cD_{\leq N}\big(\cD_{\leq N}(Z) +M_{>N}^\cD(Z;t)Z \big) = Z + \breve M'_{>N}(Z;t)Z
$$
     where, by \eqref{invmappe} and  
    Proposition \ref{composizioniTOTALI}, 
     $$
\breve M'_{>N}(Z;t)\in 
\begin{cases} 
\mM_{K, {K'},N+1}[r] \otimes \mM_2 (\C)  
\qquad \qquad  \text{if} \quad  \bM_0(U;t)  = \id \, , \\
\cM_{K,0,N+1}[\breve r] \otimes \cM_2(\C)  \, , \, \forall \breve r > 0 \quad \text{otherwise} \, . 
\end{cases}
$$   
Finally we substitute $Z= \bM_0(U;t)U$, cfr. \eqref{Z.M0}, 
in the non--homogeneous term  $ \breve M'_{>N}(Z;t)Z$ 
 and using $(iii)$ and $(i)$ Proposition \ref{composizioniTOTALI}  we 
get  
\be\label{pizza}
Z = \breve\cD_{\leq N}(W) + \breve{M}_{>N}(U;t)U
\ee
with
$\breve M_{>N}(U;t) \in 
\mM_{K,K',N+1}[r] \otimes \mM_2 (\C)$. 
We substitute \eqref{pizza} in the term $\grad_Z H(Z)$ in \eqref{patW} to obtain
\begin{align}
\pa_t W & =  
\di_Z \mD_{\leq N} (Z) J_c \nabla_Z H(\breve \cD_{\leq N}(W)) + M''_{>N}(U;t)U \notag \\
& \stackrel{\eqref{diinv2}, \eqref{Z.M0}}{=}  \di_Z \mD_{\leq N} (Z) J_c [\di_Z\cD_{\leq N}(Z)]^\top
[\di_W\breve \cD_{\leq N}(\mD_{\leq N}(Z))]^\top 
\nabla_Z H (\breve \mD_{\leq N}(W))+ M'''_{>N}(U;t)U \notag \\
& \stackrel{\eqref{lasolita}, \eqref{Z.M0}}{=} J_c [\di_W\breve \cD_{\leq N}(W)]^\top
\nabla_Z H(\breve \mD_{\leq N}(W))+ M''''_{>N}(U;t)U \notag \\
& = J_c \nabla_W( H \circ \breve \mD_{\leq N})(W) + M''''_{>N}(U;t)U \label{sistWole}
\end{align}
where $ M''_{>N}(U;t) $, $ M'''_{>N}(U;t), $ $ M''''_{>N}(U;t) $ are matrices 
of operators as in \eqref{mappapeggiore}. 
Note that in the very last passage we also substituted $\cD_{\leq N}(Z) = W + M_{>N}(U;t)U$ where $ M_{>N}(U;t) $ is a matrix of operators as in 
\eqref{mappapeggiore}. 
This proves that system \eqref{sistWole} is  Hamiltonian up to homogeneity $ N $.  
\end{proof}

\subsection{Linear symplectic  flows}\label{sec:SLF}

We consider  the flow of a linearly 
Hamiltonian up to homogeneity $ N $ paradifferential operator.  

\begin{lemma} \label{flow}
{\bf (Linear symplectic flow)}
Let  $p \in \N$,  $N,  K,K'  \in \N_0 $ with $  K'\leq K $, $ m \leq 1 $, $r>0$.  Let  
$J_c \Opbw{B}$ be a linearly Hamiltonian operator up to homogeneity $N$ 
(Definition \ref{def:LHN})
where  $B(\tau,U; t, x,  \xi )$  is a matrix of symbols
\be \label{def.B0}
B(\tau,U; t, x,  \xi ):=
\begin{pmatrix}
 b_1(\tau, U; t, x, \xi) &  b_2(\tau,U; t, x, \xi) \\  \overline{b_2(\tau,U; t, x, -\xi)} &  \overline{b_1(\tau,U; t , x, -\xi)}\end{pmatrix}, 
 \quad
 \begin{cases}
  b_1  \in \Sigma \Gamma^{0}_{K,K',p}[r,N] \\
  b_2\in  \Sigma \Gamma^{m}_{K,K',p}[r,N] \, ,  
  \end{cases}
\ee
with 
$b_1^\vee - b_1 $ in $  \Gamma^0_{K,K',N+1}[r]$ and the imaginary part 
$\textup{ Im } b_2 $ in $ \Gamma^0_{K,K', N+1}[r]$ (cfr. \eqref{LHS-cN}) uniformly in $ |\tau| \leq 1  $.
Then there exists $ s_0 >0$ such that, for any $ U\in B_{s_0,\R}^{K}(I;r) $, 
the system 
 \be \label{LinFlow0}
\begin{cases}
\pa_\tau\mG^\tau_B(U;t)=
J_c  \, \opbw\big(B(\tau,U; t, x, \xi)\big)
  \mG^\tau_B(U;t)\\
\mG_B^0(U;t)={\rm Id} \, , 
\end{cases} 
\ee
has a unique solution $\mG^{\tau}_B(U)$ defined for all $ | \tau | \leq 1  $, 
satisfying  the following properties:
\begin{itemize}
\item[(i)] {\bf Boundedness:} For any $s \in \R$ the linear map $\mG^{\tau}_B(U;t)$ is invertible and there is $r(s) \in ]0,r[$ such that for any $U \in B_{s_0,\R}^{K}(I;r(s))$  for any $0 \leq k \leq K-K'$, $V \in C_{*}^{K-K'}(I;\dot{H}^{s}(\T, \C^2))$, 
\be\label{invero2B}
\|\pa_{t}^{k} ( \mG^{\tau}_B(U;t)V) \|_{\dot{H}^{s-\frac{3}{2}k}}+\|\pa_{t}^{k}
(\mG^{\tau}_B(U;t)^{-1}V) \|_{\dot{H}^{s-\frac{3}{2}k}} \leq 
\big(1+C_{s,k}\|U\|_{k+K',s_0}\big)\|V\|_{k,s}  
\ee
uniformly in $ |\tau| \leq 1  $.

In particular $\mG_B^\tau(U;t)$ and $\mG_B^\tau(U;t)^{-1}$ 
 are  non--homogeneous spectrally localized maps in $  \mS_{K,K',0}^0 [r]\otimes \mM_2 (\C) $ according to Definition \ref{smoothoperatormaps}. 
\item[(ii)] {\bf Linear symplecticity:} The map  $ \mG^\tau_{B}(U;t) $ is linearly symplectic up to homogeneity $N$ (Definition \ref{LSUTHN}).
If $J_c \Opbw{B}$ is  linearly Hamiltonian (Definition \ref{def:LH}), then $\mG^{\tau}_B(U;t)$ is linearly symplectic (Definition \ref{LS}).
\item[(iii)] {\bf Homogeneous expansion:}  $\mG_B^\tau(U;t)$ and its inverse are spectrally localized maps and  $ \mG_B^\tau(U;t)^{\pm}- \id $ belong to 
$ \Sigma  \mS_{K,K',p}^{(N+1) m_0}[r, N] \otimes\mathcal{M}_2(\C) $ with $m_0 := \max(m,0)$, uniformly in $ |\tau| \leq 1  $.
\end{itemize}
\end{lemma}

\begin{proof}
Since the symbols 
$ b_1 $ and $\textup{ Im } b_2 $ have order $ 0 $ and $\textup{ Re } b_2 $ has order 
$ m \leq 1 $, the 
existence of the flow $\mG^{\tau}_B(U;t)$ and the estimates \eqref{invero2B} 
(actually with loss of $ k $ derivatives instead of $ \tfrac32 k $) 
are classical and follow as in  Lemma 3.22 of \cite{BD}. In view of \eqref{piovespect}, the bounds \eqref{invero2B} imply that $\mG_B^\tau(U;t)$ 
 is in $  \mS_{K,K',0}^0 [r]\otimes \mM_2 (\C) $. 
 The inverse $\mG_B^\tau(U;t)^{-1}$  satisfies the same estimates regarding  it 
 as the time $\tau $-flow 
$\mG^{\tau'}_{B^-}(U;t)\vert_{\tau'=\tau}$ of the system
\be\label{Gtaupr}
\pa_{\tau'} \mG_{B^-}^{\tau'}(U;t)= J_c\opbw( B^-(\tau,\tau', U;t,x,\xi))\mG_{B^-}^{\tau'}(U;t) \, , \quad  \mG_{B^-}^{0}(U;t)=\uno \, , 
\ee
where  $ B^-(\tau,U;t,x, \xi):= -B(\tau-\tau',U;t,x, \xi) $.

Let us prove item $(ii)$.  Set $ \bB := \opbw{(B(\tau, U;t,x, \xi))}$  and $ \mG^\tau := \mG^\tau_B(U;t) $ for brevity. 
By \eqref{LinFlow0} we get, for any $ \tau $, 
\begin{align*}
\pa_\tau ( \cG^\tau )^\top \, E_c \, \cG^\tau  \notag 
& = 
- (\cG^\tau)^\top \bB^\top J_c E_c \cG^\tau + (\cG^\tau)^\top E_c J_c \bB \cG^\tau \\
 &  \stackrel{J_c E_c = \uno, \eqref{Ecci}} =  (\cG^\tau)^\top( \bB -\bB^\top) \cG^\tau  \stackrel{\eqref{ipoBnBtN}}
 =(\cG^\tau)^\top (\bB_{>N} -\bB_{>N}^\top) \cG^\tau  \, . 
\end{align*}
Therefore 
$$ 
 ( \cG^\tau )^\top \, E_c \, \cG^\tau  =  E_c+ S_{>N}  
 \quad  \text{where} \quad
 S_{>N} := 
\int_0^\tau  (\cG^{\tau'})^\top( \bB_{>N} -\bB_{>N}^\top) \cG^{\tau'}\, \di \tau'
$$ 
is 
a matrix of spectrally localized maps in $ \mS_{K,K',N+1}[r] \otimes\mM_2 (\C) $ because
$\mG_B^\tau(U;t)$ and $\mG_B^\tau(U;t)^{-1}$ 
are  in $  \mS_{K,K',0}^0 [r]\otimes \mM_2 (\C) $,  
the paradifferential operator 
$ \bB$  belongs to $ \Sigma\mS^{m_0}_{K,K',p}[r] \otimes \mM_2 (\C) $  
(see the fourth bullet after Definition \ref{smoothoperatormaps}), 
and $(ii)$ of Proposition \ref{composizioniTOTALIs}. 
This proves that the  $ \mG^\tau_{B}(U;t) $ is linearly symplectic up to homogeneity $N$ according to Definition \ref{LSUTHN}. 
The same proof shows that, if $J_c \Opbw{B}$ is  linearly Hamiltonian, then  
$\mG^{\tau}_B(U;t)$ is linearly symplectic.

Let us prove item $(iii)$. By  \eqref{LinFlow0}, 
iterating $ N$--times the fundamental theorem of calculus 
we get  the  expansion
\be\label{GtB-I}
\mG_B^\tau(U;t) = \id + \sum_{j=1}^N S_j^\tau(U)  + S_{>(pN)}^\tau(U;t)
\ee
where
$$
S_j^\tau(U) := \int_0^{\tau} \! \int_0^{\tau_1}\!\! \! \cdots \!\! \int_{0}^{\tau_{j-1}}    J_c \opbw(B(\tau_1,U;t,x,\xi))\cdots J_c \opbw(B(\tau_{j} ,U;t,x,\xi)) \,\di\tau_1\cdots \di\tau_{j}
$$
and, writing for brevity $ \opbw(B(\tau_j ,U)) := \opbw(B(\tau_j ,U;t,x,\xi)) $, 
$$
S_{> (pN)}^\tau (U;t)= \int_0^{\tau} \! \int_0^{\tau_1}\!\! \! \cdots \!\! \int_{0}^{\tau_{N}}    J_c \opbw(B(\tau_1,U))\cdots J_c \opbw(B(\tau_{N+1} ,U))\mG_B^{\tau_{N+1}}(U;t)\,\di\tau_1\cdots \di\tau_{N+1}.
$$
 Since each $ \opbw(B(\tau_j,U))$ 
 belongs to $ \Sigma\mS^{m_0}_{K,K',p}[r] \otimes \mM_2 (\C) $ 
and  $\mG_B^{\tau_{N+1}}(U)$ is in $ \mS_{K,K',0}^0 \otimes \mM_2 (\C) $ 
we deduce,  by $(ii)$ of Proposition \ref{composizioniTOTALIs},  that 
$\mG_B^\tau(U;t)- \id$ in \eqref{GtB-I} is a matrix of spectrally localized maps in $ \Sigma \mS_{K,K',p}^{(N+1) m_0}[r,N] \otimes \mM_2 (\C) $, uniformly in $ |\tau| \leq 1  $. 
The analogous statement for $\mG_B^\tau(U;t)^{-1} - \id$ follows by  \eqref{Gtaupr}. 
\end{proof}

The flow generated by a Fourier multiplier satisfies similar properties. 
 
\begin{lemma} {\bf (Flow of a Fourier multiplier)} \label{flussoconst}
Let $ p \in \N$ and $ g_p(Z;\xi) $ be a $p$--homogeneous, 
 $ x $-independent,  real symbol in $ \wt \Gamma^{\frac32}_p $. Then  the flow 
 $\cG_{g_p}^\tau (Z)$ defined by 
\be \label{FouFlow}
\pa_\tau \cG_{g_p}^\tau(Z)= \vOpbw{\ii g_p(Z;\xi)}\cG_{g_p}^\tau(Z) \, , 
\quad \cG_{g_p}^0(Z)=\uno \, ,
\ee
is well defined for any $ |\tau| \leq 1 $ and satisfies the following properties:
\begin{itemize}
\item[$(i)$] {\bf Boundedness:} 
For any $K\in \N$ and $r>0$  the flow $\cG^{\tau}_{g_p}(Z)
 $ is a real-to-real diagonal matrix of spectrally localized 
 maps in $ \mS_{K,0,0}^0[r] \otimes \mM_2(\C) $. Moreover there is $s_0>0 $ such that for  any $s\in \R$, 
there is $r(s) \in (0, r)$ such that for any 
functions $Z \in B^K_{s_0,\R} (I;r(s)) $
and 
$W\in C^{K}_{*}(I; \dot H^s(\T, \C^2))$, it results, for any $0 \leq k \leq K$, 
\be\label{inveroGood}
\|\pa_{t}^{k}(\cG_{g_p}^\tau (Z)W) \|_{\dot{H}^{s-\frac{3}{2}k}}+\|\pa_{t}^{k}(\cG_{g_p}^\tau (Z)^{-1}W) \|_{\dot{H}^{s-\frac{3}{2}k}} \leq 
\big(1+C_{s,k} \|Z\|_{k,s_0}\big)\| W \|_{k,s}  
\ee
uniformly in $ |\tau| \leq 1  $.
\item[$(ii)$]{\bf Linear symplecticity:} 
The flow map $\cG^\tau_{g_p}(Z)$    is linearly symplectic  (Definition \ref{LS}). 
\item[$(iii)$] {\bf Homogeneous expansion:} The flow map  $ \cG^{\tau}_{g_p}(Z) $ and its inverse $ \cG^{-\tau}_{g_p}(Z) $ are matrices of  spectrally localized maps 
such that $ \cG^{\pm \tau}_{g_p}(Z) -\uno 
$ belong to $ \Sigma \mS^{\frac32 (N+1)}_{K,0,p}[r,N] \otimes \mM_2(\C) $,
uniformly in $ |\tau| \leq 1  $.
\end{itemize}
\end{lemma}

\begin{proof}
Since $ g_p (Z;\xi) $ is real and independent of $ x$, then the flow $ \cG^{\tau}_{g_p}(Z)$ is well defined in $ \dot H^s $ and it is unitary,  namely $ \|\cG^{\tau}_{g_p}(Z)W\|_{\dot H^s}=\|W\|_{\dot H^s}$. Moreover, since $g_p$ is a Fourier multiplier of order $ \frac32 $,  
we have
 $$ 
 \begin{aligned}
   \|\pa_t(\cG^{\tau}_{g_p}(Z)W)\|_{\dot H^{s-\frac32}}
   & =  \|\cG^{\tau}_{g_p}(Z)\pa_t W\|_{\dot H^{s-\frac32}}+  \|\vOpbw{ \ii \pa_t g_p(Z;\x)}\cG^{\tau}_{g_p}(Z)W\|_{\dot H^{s-\frac32}}\\
   & \leq \| W \|_{1,s} +C  \| Z\|_{1,s_0}^{p}\| W \|_{0,s} \, .
   \end{aligned}
   $$
   The estimates for the $ k$--th derivative follow similarly using also that $\cG^\tau_{g_p}(Z)^{-1}=\cG^{-\tau}_{g_p}(Z)$. 
   
   To prove $(ii)$  we use that, in view of \eqref{vecop}, \eqref{LHS-c}
   and since $g_p(Z;\x)  $ is real valued, 
   the operator $\vOpbw{ \ii  g_p(Z;\x)}$ is linearly Hamiltonian, according to Definition \ref{def:LH}. 
 Then, as 
   for item ($ii$) of Lemma \ref{flow},  
the flow    $\cG^{\tau}_{g_p}(Z)$ is linearly symplectic.
Finally also item $(iii)$    follows as for item ($iii$) of Lemma \ref{flow},
since $\vOpbw{ \ii  g_p(Z;\x)}^k $ is in 
$  \wt \mS_{kp}^{\frac32 k}\otimes \mM_2(\C)$ and 
$\cG^{\tau}_{g_p}(Z) $ is in $ \mS_{K,0,0}^0[r] \otimes \mM_2(\C) $, uniformly in $ |\tau| \leq 1  $. 
  \end{proof}

\subsection{Paradifferential  Hamiltonian structure}\label{sec:LHam}

In order to compute the Hamiltonian vector field  associated to a paradifferential Hamiltonian we provide the following result. 

\begin{lemma}\label{H.S.U}
Let $p \in \N$, $m \in \R$. 
Let $S(U)$ be a real-to-real symmetric matrix of $p$-homogeneous spectrally localized maps in $\wt \cS_p^m \otimes \cM_2(\C)$  and define the Hamiltonian function
\be\label{H.S}
H(U) := \frac12 \big\la S(U) U, U \big\ra_r \ . 
\ee
Then its gradient
\be\label{gradH.S}
\grad H(U) = S(U)U + R(U)U
\ee
where $R(U)$ is a real-to-real  matrix of homogeneous smoothing operators in 
$ \wt\cR^{-\varrho}_p \otimes \cM_2(\C)$ for any $\varrho \geq 0$.
\end{lemma}
\begin{proof}
By the definition \eqref{def:grad}, the gradient $\grad H(U)$ is the vector field 
\be\label{ham.sap10}
 \grad H (U)  = S(U)U +  L(U)^\top U
\quad \text{where} \quad  L(U) W :=\frac12  \di_U S(U)[W] U \, . 
\ee
As $S(U)$ is  a spectrally localized map in $\wt \cS_p^m \otimes \cM_2(\C)$,  by Lemma \ref{aggiunto}   the transposed of its internal differential, namely $L(U)^\top $, 
is a smoothing operator in $\wt \cR_p^{-\varrho} \otimes \cM_2(\C)$ for any $ \varrho \geq  0 $. Then \eqref{gradH.S} follows from \eqref{ham.sap10}.
\end{proof}

As a corollary we  obtain the Hamiltonian vector field  associated to a paradifferential Hamiltonian. 

\begin{lemma}\label{lem:hamsym} 
Let  $ p \in \N $, $ m \in \R $ and 
$  a (U;x, \xi) $  a real valued homogeneous 
symbol  in $ \wt\Gamma^m_p $.
Then the Hamiltonian vector field generated 
by the Hamiltonian 
$$
H (U):=\Re \langle  A(U)u , \bar u\rangle_{\dot L^2_r}= \frac12 \Big\langle\begin{pmatrix} 0 & \bar {A(U)}\\ A(U)& 0 \end{pmatrix} U , U\Big\rangle_r, \quad A(U):=\Opbw{a(U;x, \xi)}\, , 
$$
is 
$$
J_c \grad H (U) = \vOpbw{- \im a (U;x, \xi)} U + R(U) U
$$
where $ R(U) $ is a real-to-real matrix of homogeneous smoothing operators in 
$ \wt\cR^{-\varrho}_p \otimes \cM_2(\C)$ for any $\varrho \geq 0$.
\end{lemma}

We now prove that  if a homogeneous 
Hamiltonian vector field $ X(U) 
=  J_c \nabla H(U) $ can be written in paradifferential form 
$$
X(U) = J_c \opbw(A(U))U + R(U)U 
$$ 
where $ A(U) $ is a matrix of symbols and $ R(U) $ is a  smoothing operator,
then $\Opbw{A(U)}  =  \Opbw{A(U)}^\top$ up to a smoothing operator.
As a consequence we may always assume, up to modifying the smoothing operator, 
that  the paradifferential operator $ \opbw(A(U)) $ is symmetric, namely that 
$ J_c \opbw(A(U))  $ is linearly Hamiltonian.

\begin{lemma}\label{HS:repre}
Let $p \in \N$, $m \in \R$ and  $\varrho \geq 0$. Let 
\be\label{XHS}
X(U) = J_c \Opbw{A(U;x,\xi)}U + R(U)U  = J_c \nabla H(U)
\ee 
be a $(p+1)$-homogeneous Hamiltonian vector field, where (cfr. \eqref{prodotto}) 
\be\label{auno}
A(U;x,\xi) = 
\begin{pmatrix} 
a(U;x,\xi) & b(U;x,\xi) \\  
\overline {b(U;x,-\xi)}&  \overline{a(U;x,-\xi)}
\end{pmatrix}  
\ee
is matrix of symbols  in $ \widetilde \Gamma_{p}^m \otimes \mathcal{M}_2(\mathbb{C})$
and $ R(U)  $ is a real-to-real
 matrix of smoothing operators in 
$ \widetilde {\cal R}_{p}^{-\varrho} \otimes \mathcal{M}_2(\mathbb{C}) $. 
Then we may write 
 \be\label{XHSA1}
X(U) = J_c \opbw( A_1(U;x,\xi) )U + R_1(U)U 
\ee 
where the matrix of paradifferential operators $\Opbw{A_1(U;x,\xi)} $ is symmetric,   
with   matrix of symbols
\be\label{aunoa} 
A_1(U;x,\xi) =  \frac12 
 \begin{pmatrix} 
a + a^\vee & b + \ov{b}\\ 
\ov{b}^\vee + b^\vee & \ov{ a+ a^\vee} 
\end{pmatrix} 
\ee
and  $ R_1(U)$ is another real-to-real  matrix of smoothing operators  in 
 $\widetilde {\cal R}_{p}^{-\varrho} \otimes \mathcal{M}_2(\mathbb{C}) $.
\end{lemma}

\begin{proof}
The linear vector field $ \di_U X(U) $, obtained linearizing 
a Hamiltonian vector field $ X(U) = J_c \nabla H(U) $, is Hamiltonian, namely 
$$ 
\di_U X(U) = J_c S (U) \quad \text{where} \quad S(U) = S(U)^\top \quad
\text{is symmetric} \, . 
$$
On the other hand, by linearizing \eqref{XHS}, 
$$
S(U)  =   \opbw(A(U)) + \opbw(\di_UA(U)[ \cdot ]) U +  J_c \breve R(U) 
$$
where $ \breve R(U) := R(U) + \di_UR(U)[ \cdot ] U $ is a 
 matrix of smoothing operators in 
$ \widetilde {\cal R}_{p}^{-\varrho} \otimes \mathcal{M}_2(\mathbb{C}) $
(see the remark after Definition \ref{Def:Maps}).
Then, since $ S(U) $ is symmetric, it results, writing for brevity $ A := A(U) $,
$$
 \opbw(A)-  \opbw(A)^\top= \big( (J_c \breve R(U))^\top 
 - \opbw(\di_UA(U)[ \cdot ]) U\big) + \big( \big[  \opbw(\di_UA(U)[ \cdot ]) U\big]^\top -J_c \breve R(U) \big). 
$$
We now apply Lemma \ref{lem:ide}  to  the spectrally localized map
 $ \opbw(A(U))-  \opbw(A(U))^\top $ which has the  form \eqref{decospesmo}
 with 
\be\label{defLR}
 L :=  \big( (J_c  \breve R(U))^\top -  \opbw(\di_UA(U)[ \cdot ]) U \big) \, , \quad 
 R := \big( \big[  \opbw(\di_UA(U)[ \cdot ]) U\big]^\top - J_c \breve R(U) \big) \, . 
 \ee
 By Lemma \ref{aggiunto}, the operator 
 $ \big[  \opbw(\di_UA(U)[ \cdot ]) U\big]^\top $ is in $  \widetilde{\mR}^{-\vr}_p \otimes \mM_2 (\C)  $,
 and therefore both $ L^\top $ and $ R $  in \eqref{defLR} are $ p$-homogeneous 
 smoothing operators in $ \widetilde{\mR}^{-\vr}_p \otimes \mM_2 (\C)  $. 
The assumptions of 
Lemma \ref{lem:ide} are satisfied, implying that 
$$ 
 \opbw(A(U))-  \opbw(A(U))^\top =: R' (U) \in  \widetilde{\mR}^{-\vr}_p  \otimes \mM_2 (\C) \, . 
 $$ 
 In conclusion we deduce 
  \eqref{XHSA1} with 
  $$
\opbw{(A_1 (U))} := \frac12 \Big( \opbw{(A(U))} + \opbw{(A(U))}^\top \Big) 
 \, , \quad R_1 (U) 
:= \frac12 J_c R' (U) + R(U ) \, , 
  $$
and  \eqref{aunoa} follows recalling \eqref{A1b}. 
 \end{proof}
 
Another  consequence of Lemma \ref{lem:ide} is the following. 

\begin{lemma}\label{spezzamento}
Let $p \in \N$, $m \in \R$ and  $\varrho \geq 0$.
Let $S(U)$ be a  matrix of spectrally localized homogeneous maps in $\wt \mS_{p} \otimes \mM_2 (\C) $ which is linearly Hamiltonian (Definition \ref{def:LH}) of the form 
\be \label{separato}
 S(U) =  J_c \Opbw{ A(U;x,\xi )} +  R(U) \, , 
\ee
where $ A(U;x,\xi) $ is a real-to-real 
matrix of symbols in $ \wt\Gamma^m_p\otimes \cM_2(\C)$
as in  \eqref{auno}, 
 and $ R(U) $ is a real-to-real matrix of  smoothing operators in $ \wt \mR^{-\vr}_{p}\otimes \mM_2(\C) $. 
 Then 
 we may write 
 $$
 S(U) = J_c\Opbw{ A_1(U;x,\xi)} +  R_1(U) \, 
$$
where  
the matrix of symbols $ A_1(U;x,\xi) $ in $   \wt \Gamma^m_{p}\otimes \mM_2(\C)$
has the form \eqref{aunoa}
 and  $R_1(U)$ is another matrix of real-to-real  smoothing operators in 
 $\wt \mR^{-\vr}_{p} \otimes \mM_2(\C) $.
 In particular the homogeneous operator 
$  J_c\Opbw{ A_1(U)}$ is linearly Hamiltonian.
\end{lemma}

\begin{proof}
It is enough to prove that the operator 
$\Opbw{ A(U)}$ is equal to   $ \Opbw{ A(U)}^\top$ up to a 
matrix of smoothing operators.
To prove this claim, recall  that $S(U)$ linearly Hamiltonian means that 
$E_c S(U)$ is symmetric, so that by \eqref{separato} one gets
$$
\Opbw{ A(U)} - \Opbw{ A(U)}^\top = -R(U)^\top E_c - E_c R(U) \, .
$$
Now, since $ S(U) $ and $\Opbw{A(U)}$ are spectrally localized maps, so is $R(U)$ in \eqref{separato}. 
By Remark \ref{rem:top.smoo} the transpose 
 $R(U)^\top$ is also a smoothing operator in $\wt \mR^{-\vr}_{p}\otimes \mM_2(\C)$, proving the claim. 
\end{proof}

\section{Construction of a Darboux symplectic corrector} 
\label{sec:Darboux}

If $ \bB(U;t) $ is a spectrally localized map   
which is linearly symplectic up to homogeneity $N$, 
then the associated nonlinear map 
$ \Phi_{\leq N}(U):= \bB_{\leq N}(U)U $, where $ \bB_{\leq N}(U) = \cP_{\leq N}(  \bB(U;t)) $, 
is {\it not} symplectic up to homogeneity $ N $. 
In  this  section we provide a systematic procedure to
construct a nearby  nonlinear map which is symplectic up to homogeneity $N$ 
according to Definition \ref{def:LSMN}. 

\begin{theorem}{\bf (Symplectic correction up to homogeneity $N$)}
\label{thm:almost}
 Let $p, N \in \N$ with $ p\leq N $. Consider a nonlinear map 
 \be \label{PsiN}
\Phi_{\leq N}(U):= \bB_{\leq N}(U)U \, ,
\ee
where 
\begin{enumerate}
\item[(i)] $\bB_{\leq N}(U)-\uno $ is a matrix of 
pluri--homogeneous spectrally localized maps in 
$ \Sigma_{p}^N \wtcS_q \otimes \mM_2 (\C) $;  
\item[(ii)] 
$ \bB_{\leq N}(U) $ is linearly symplectic  up to homogeneity $ N $ (Definition \ref{LSUTHN}). 
\end{enumerate}
Then there exists a real-to-real map 
\be\label{CiEnne}
\cC_{\leq N}(W) = W +  R_{\leq N}(W)W \quad \text{with} \quad R_{\leq N} (W)  \in 
\Sigma_p^N \wtcR_q^{-\vr} \otimes \mM_2 (\C)  \, ,  \ \mbox{ for any } \varrho \geq 0 \, , 
\ee
such that the Darboux correction
 \be \label{Dienne} 
\mD_N(U):= \big( \cC_{\leq N}\circ \Phi_{\leq N}  \big) (U) = \big( {\rm Id} +  R_{\leq N} (\Phi_{\leq N}(U)) \big)\Phi_{\leq N}(U)
\ee
   is symplectic up to homogeneity  $N$, according to Definition \ref{def:LSMN}.
\end{theorem}

\begin{remark}
The first assumption   
implies that 
 the  operator  in \eqref{GeNtra}
 is smoothing for any $ \varrho > 0 $. This fact 
and the second assumption  allow to deduce 
that the vector field  representing 
the perturbed 
symplectic $ 1$-form $ \theta_{\leq N}  $ in  \eqref{thetaB} 
is a smoothing 
perturbation of  $ E_c V $,  
see \eqref{ZB}. 
These properties are  crucial to guarantee that the vector field $Y^\tau(V)$ solving the Darboux equation \eqref{darPRIN} is smoothing (see Lemma \ref{lem.sol.dar}), which in turn implies that the Darboux corrector $\cC_{\leq N}(W)$ in \eqref{CiEnne} is a smoothing perturbation of the identity.
\end{remark}

The rest of this section is devoted to the proof of Theorem \ref{thm:almost}. 

In order to correct the nonlinear map $ \Phi_{\leq N} $ 
defined in \eqref{PsiN} we develop a perturbative  Darboux procedure to construct a nearby 
symplectic map  up to homogeneity $ N $.
The map  $\Phi_{{\leq N}} $ induces
the nonstandard symplectic 
$ 2$-form
\be\label{OmegaB}
 \Omega_{\leq N} := \Psi_{\leq N}^*  \ \Omega_c
\ee
where $\Psi_{{\leq N}} $ is the approximate inverse of $\Phi_{{\leq N}} $ defined 
by Lemma \ref{inversoapp} and $\Omega_c$ is the 
standard symplectic form  in \eqref{sfc}. 
The next lemma describes properties of the  approximate inverse $\Psi_{\leq N}$.

\begin{lemma}\label{lem:app.inv} {\bf (Approximate inverse)} The approximate inverse 
up to homogeneity $ N $ of 
the map  $ \Phi_{\leq N} (U) = \bB_{\leq N}(U)U $ defined in \eqref{PsiN} 
has the form 
\be\label{PsileqN}
\Psi_{\leq N}(V)= \bA_{\leq N}(V)V
\ee
where 

\begin{itemize}
\item[(i)]
 $ {\bf A}_{{\leq N}} (V) - \uno $ is a matrix of pluri-homogeneous spectrally localized maps in  $\Sigma_p^N \wtcS_q \otimes \mM_2 (\C) $;  
\item[(ii)]
$ {\bf A}_{{\leq N}} (V) $  is  linearly symplectic up to homogeneity $N$ (Definition \ref{LSUTHN}), 
more precisely 
\eqref{ATEA} holds.
\end{itemize}
In addition 
\be\label{sommabella2}
{\rm d}_{V} \Psi_{{\leq N}}(V)={\bf A}_{{\leq N}}(V)+  {\bf G}_{{\leq N}}(V) \, , 
\qquad  \bG_{{\leq N}}(V) \hat V := \di_V \bA_{{\leq N}}(V)[ \hat V]V \, , 
\ee
and
\begin{itemize}
\item[(iii)] 
$ \bG_{{\leq N}}(V) $ 
 is a matrix of pluri--homogeneous operators in 
 $\Sigma_p^N \widetilde \mM_{q} \otimes \mM_2 (\C) $;
\item[(iv)]  the transposed operator  
 \be\label{GeNtra}
 {\bf G}_{{\leq N}}^\top(V):= [ {\bf G}_{{\leq N}}(V)]^\top \in 
 \Sigma_p^N \wtcR^{-\varrho}_{q}  \otimes \mM_2 (\C) 
 \ee
   is a matrix of $\vr$-smoothing operators for arbitrary  $\varrho \geq 0 $ .
 \end{itemize}
 \end{lemma}

\begin{proof}
Items $(i)$  and $(ii)$ are proved in  Lemma \ref{inv.lin.simp}, and 
$(iii)$ follows by the fifth bullet after Definition \ref{smoothoperatormaps}.
Finally $(iv)$  follows applying 
Lemma \ref{aggiunto} to each spectrally localized map $ \cP_q (\bA_{\leq N}(V))$ for $ q = p, \ldots, N $ (with  $U \leadsto V$ and $V \leadsto \hat V$). 
\end{proof}

We now compute $\Omega_{{\leq N}}$. 

\begin{lemma} \label{Lem:ZB} {\bf (Non-standard symplectic form $ \Omega_{\leq N}$)} 
The symplectic $2$-form $\Omega_{{\leq N}} = \Psi_{\leq N}^*  \ \Omega_c $ in \eqref{OmegaB} is represented as 
$ \Omega_{{\leq N}} (V) =\langle E_{{\leq N}}(V) \cdot ,  \cdot \rangle_{r}  $
with  symplectic tensor  
\be \label{EB}
 E_{{\leq N}}(V) 
 =  E_c + {\bf A}_{\leq N}^\top(V)E_c {\bf G}_{\leq N} (V)+  
 {\bf G}_{\leq N}^\top(V)E_c {\bf A}_{\leq N}(V)+  {\bf G}_{\leq N}^\top(V)
 E_c {\bf G}_{\leq N}(V)+ \rS_{>N}(V)  
 \ee
where 
 \begin{itemize}
 \item[(i)]
$ {\bf G}_{{\leq N}}^\top(V)E_c {\bf A}_{{\leq N}}(V)$ and $ {\bf G}_{{\leq N}}^\top(V) E_c {\bf G}_{{\leq N}}(V)$ are matrices of pluri--homogeneous smoothing operators  in $ \Sigma_p^{2N} \wtcR^{-\varrho}_{q} \otimes \mM_2 (\C) $ for any $ \varrho \geq 0 $;  
\item[(ii)] 
 $\rS_{>N}(V):= \bA_{\leq N}^\top(V) E_c \bA_{\leq N}(V) - E_c$ is a matrix of pluri--homogeneous spectrally localized maps in 
 $ \Sigma_{N+1} \wtcS_q \otimes \mM_2 (\C) $.
\end{itemize}
Moreover 
\be\label{dthetaB=Omega}
\Omega_{{\leq N}}=\di \theta_{\leq N}  \, ,
\ee
where 
the $ 1$-form  
\be\label{thetaB}
\theta_{\leq N} := \Psi_{\leq N}^* \ \theta_c  \ , 
\qquad 
\theta_c(V)  := \frac12 \big\la E_c V, \cdot  \big\ra_r \, , 
\ee
has the form 
 \be \label{ZB}
  \theta_{{\leq N}}(V) = \frac12 \Big\langle [Z_{{\leq N}}(V)+ \rS_{>N}(V)]V, \cdot 
  \Big\rangle_{r}  \quad \text{with}
  \quad 
 Z_{{\leq N}} (V):=  E_c  + {\bf G}_{{\leq N}}^\top(V) \ E_c  \ {\bf A}_{{\leq N}} (V)  \, . 
 \ee
\end{lemma}

\begin{proof}
By \eqref{pullback.form2} we have that
$$ 
 \Omega_{{\leq N}}(V)[X,Y] = 
\big\la \di \Psi_{{\leq N}} (V)^\top \, E_c \, \di \Psi_{{\leq N}} (V) X, Y \big\ra_{r} 
$$
which, using \eqref{sommabella2} and the fact 
that $ \bA_{\leq N}(V) $ is linearly symplectic up to homogeneity  
$ N $ 
(cfr. \eqref{ATEA}), provides  formula \eqref{EB}. 
Then items ($i$)-($ii$) follow by 
\eqref{sommabella2}, \eqref{GeNtra}  and  Proposition \ref{composizioniTOTALI}.
The identity   \eqref{dthetaB=Omega} follows by \eqref{exactc} and \eqref{commu}.
Finally  
\eqref{ZB} follows similarly computing 
$ \Psi_{\leq N}^* \theta_c $ by \eqref{pullback.form1},  
$$ 
\theta_{\leq N} (V) [X] 
  = \frac12 
  \big\la \di_V \Psi_{\leq N}(V)^\top \, \,E_c  \, \Psi_{\leq N} (V) , X \big\ra \, , 
  $$
and using \eqref{PsileqN},  \eqref{sommabella2} and 
$ \bA_{\leq N}^\top(V) E_c \bA_{\leq N}(V) = E_c + \rS_{>N}(V) $. 
\end{proof}

The key step is to  implement a Darboux--type procedure to transform the symplectic form $\Omega_{{\leq N}}$ back to the standard symplectic form $\Omega_c$ up to arbitrary high 
degree of homogeneity. It turns out that the required transformation is a smoothing perturbation of the identity as claimed in Theorem \ref{thm:almost}, see
Proposition \ref{DarB}. 
This is not at all obvious, since in the expression \eqref{EB} of $E_{\leq N}(V)$ the second  operator 
${\bf A}_{\leq N}^\top(V)E_c {\bf G}_{\leq N} (V)$
is {\em not} smoothing.
However it has a nice structure that we now describe.

\begin{lemma}\label{diadg}
Let  $X(V)$ be a pluri-homogeneous vector field in $\Sigma _{p'+1} \wt \X_q$ 
for some $p' \in \N_0$.  Then 
\be \label{selfad}
\bA^\top_{\leq N}(V) E_c \bG_{\leq N} (V)[ X(V)] = 
\grad \cW(V) +  R (V)V+ M_{>N}(V)V
\ee
where 
\begin{itemize}
\item $\cW(V)$ is  $0$-form in $\Sigma_{p+p'+2} \wt \Uplambda^0_q$; 
\item $R(V)$ is  matrix of pluri--homogeneous smoothing operators in $ \Sigma_{p+p'}  \wtcR^{-\vr}_q \otimes \mM_2 (\C) $ for any $\varrho \geq 0$;
\item $M_{>N}(V)$ is a  matrix of pluri--homogeneous  operators  in $\Sigma_{N+1} \wtcM_q \otimes \mM_2 (\C)$.
\end{itemize}
\end{lemma}

\begin{proof}
For simplicity of notation we set $ \bA (V) := \bA_{\leq N}(V) $ and $  \bG (V)  := \bG_{\leq N}(V)   $.

\noindent{\sc Step 1:} {\it For any vector $W$, the linear operator
\be\label{K_W}
\bK_W(V) := \bA(V)^\top\,  E_c \, \big[ \di_V \bA(V) [W]  \big]
\ee
is symmetric up to homogeneity $N$, 
precisely 
\be \label{quasiadg}
\bK_W(V) -  \bK_W(V)^\top = \di_V \rS_{>N}(V)[ W]
\ee
 where $ \rS_{>N}(V)$ is the spectrally localized map in    $\Sigma_{N+1}
\widetilde  \mS_q \otimes \mM_2 (\C) $ of Lemma  \ref{Lem:ZB}. }\\ Indeed, differentiating the relation  $ \bA(V)^\top E_c \bA(V)=E_c+ \rS_{>N}(V)$ (see Lemma \ref{Lem:ZB} $(ii)$), 
in direction $W$, we get 
\begin{align*}
\di_V \rS_{>N}(V)[ W] & = \bA(V)^\top E_c \, 
\big[\di_V \bA(V)[ W]\big] + 
\big[\di_V \bA(V)[W] \big]^\top  E_c \bA(V)  \\
& =
\bA(V)^\top E_c \,  \big[\di_V \bA(V)[ W]  \big]-
\Big(\bA(V)^\top E_c \, \big[\di_V \bA(V)[W]\big]\Big)^\top 
\end{align*}
proving, in view of \eqref{K_W},  \eqref{quasiadg}.

\noindent {\sc Step 2:} {\it The linear operator
\be\label{KV}
\bK(V) := 
\bK_{X(V)}(V) = \bA(V)^\top E_c \big[\di_V \bA(V)[X(V)] \big]
\ee
can be decomposed as  
\be\label{decompl}
\bK(V) = S(V) + \breve R(V) +  M_{>N}(V)
\ee
where \\
\noindent 
$\bullet$ $S(V)$ is  a symmetric  matrix of spectrally localized pluri--homogeneous maps  in $\Sigma_{p+p'} \wtcS_{q} \otimes \mM_2 (\C) $; \\
\noindent
$\bullet$  $ \breve R(V)$ is 
a  symmetric matrix of  pluri--homogeneous smoothing operators in $ \Sigma_{p+p'}\wtcR^{-\vr}_{q} \otimes \mM_2 (\C) $ for any $\varrho \geq 0$ as well as its transpose;\\
\noindent
$\bullet$ 
$M_{>N}(V):= \frac12 \di_V \rS_{>N}(V)[ X(V)]$ is a  matrix of pluri--homogeneous operators in  $\Sigma_{N+1}
\widetilde  \mM_q \otimes \mM_2 (\C) $. 
}
We apply  Lemma 
 \ref{decomposta} to each component 
$\bA_q(V) := \cP_q [ \bA(V)]$, $q = p, \ldots, N$, each of which is a  map in
$\wt \cS_q \otimes \cM_2(\C)$, see  Lemma \ref{lem:app.inv} $(i)$.
Lemma  \ref{decomposta} (with $M(U)U \leadsto X(V)$)
gives the  decomposition
\be \label{decompl0}
\di_V \bA(V)[ X(V)]  = 
\sum_{q=p}^{N} q \bA_{q}  \big( X(V), V, \dots , V\big)  =S'(V)+ R'(V) \, ,  
\ee 
 where $S'(V)$ is  a  matrix of spectrally localized pluri--homogeneous maps  in $\Sigma_{p+p'} \wtcS_{q} \otimes \mM_2 (\C) $ and $ R'(V)$ is 
a  matrix of  pluri--homogeneous smoothing operators in $ \Sigma_{p+p'}\wtcR^{-\vr}_{q} \otimes \mM_2 (\C) $ for any $\varrho \geq 0$,  as well as its transpose. 
Then we obtain by \eqref{quasiadg} (with $W = X(V)$), 
\begin{align*}
\bK(V)= &\, \frac12 (\bK(V)+\bK(V)^\top)+  \frac12 \di_V \rS_{>N}(V)[X(V)]  \\
 \stackrel{\eqref{KV},\eqref{decompl0}} = &\underbrace{\frac12 \big[\bA^\top E_c S'- S'^\top E_c\bA \big](V) }_{:=S(V), \ S(V)= S(V)^\top} +\underbrace{ \frac12\big[\bA^\top E_c  R'-  R'^\top E_c\bA \big](V)}_{:=\breve R(V), \ \breve R(V)= \breve R(V)^\top } +  M_{>N}(V)  \, ,
\end{align*}
where
$M_{>N}(V)= \frac12 \di_V \rS_{>N}(V)[ X(V)]$ is
in  $\Sigma_{N+1}
\widetilde  \mM_q \otimes \mM_2 (\C) $ by Proposition \ref{composizioniTOTALI}.
 Since the maps $ S'(V)$ and $ \bA(V)$ are spectrally localized then  
 $S(V) $ belongs to $ \Sigma_{p+p'} \wtcS_q \otimes \mM_2 (\C) $ by   Proposition \ref{composizioniTOTALIs} $(ii)$   and  Lemma \ref{omonon}. 
The operator $ \breve R(V) $ is in $  \Sigma_{p+p'} \wtcR^{-\vr}_q \otimes \mM_2 (\C) $ for any $\varrho \geq 0$ by 
Proposition \ref{composizioniTOTALIs} $(i)$, Lemma \ref{omonon}  and the fact that $R'(V)$ belongs 
$  \Sigma_{p+p'} \wtcR^{-\vr}_q \otimes \mM_2 (\C) $ for any $\varrho \geq 0$
 as well as its  transpose.

\noindent {\sc Conclusion:} 
The  identity \eqref{selfad} follows by Step 2 defining 
$$
\cW(V) := \frac12 \big\la S(V) V, V \big\ra_r \, , \quad
S(V) \mbox{ in } \eqref{decompl} \ . 
$$
Indeed by  Lemma \ref{H.S.U}, $\grad \cW(V) = S(V) V + R''(V) V$ for some smoothing operator $R''(V)$ belonging to $ \wt\cR^{-\varrho}_{p+p'} \otimes \cM_2(\C)$ for any $\varrho > 0$.
Then
we get
\begin{align*}
\bA^\top_{\leq N}(V) E_c \bG_{\leq N} (V)[ X(V)]  
& \stackrel{ \eqref{selfad}, \eqref{sommabella2} }{=}
\bK(V)V \stackrel{\eqref{decompl}}{=} S(V)V + \breve R(V) V + M_{>N}(V)V \\
& = 
\grad \cW(V) -  R''(V) V + \breve R(V) V + M_{>N}(V)V  \ , 
\end{align*}
proving  
 \eqref{selfad}  with 
 $R(V) := \breve R(V) - R''(V)$ which belongs to $ \Sigma_{p+p'}  \wtcR^{-\vr}_q \otimes \mM_2 (\C) $ for any $\varrho \geq 0$.
\end{proof}

\begin{remark}\label{rem:dY}
The vector field $R(V)V$ in \eqref{selfad} depends on $ X(V)$ and its differential $ \di_V X (V)$. 
This is because  $S(V)$ depends linearly on $X(V)$
(actually it is the term   $S'(V)$ that depends linearly on $X(V)$, 
 see  \eqref{decompl0}).
Hence  the smoothing vector field $R''(V)V$ coming from the gradient $ \nabla \cW(V)$, given explicitly by $R''(V)V := L(V)^\top V$ where $L(V)W  =  \frac12 \di_V S(V)[W]V$ (see \eqref{ham.sap10}), 
 depends on the differential of $X(V)$.
\end{remark}

Now we present the main Darboux procedure. 
 
\begin{proposition}{\bf (Darboux procedure)}\label{DarB}
 There exists 
 a $\tau$--dependent pluri-homogeneous smoothing vector field  $Y^\tau(V)  $ in 
 $ \Sigma_{p+1} \wt \X^{-\varrho}_{p+1}$, for any $ \varrho \geq 0 $, 
 defined for $ \tau \in [0,1]$,  
  such that its approximate time $1$--flow
 \be\label{darbouxF}
  \mF_{\leq N}(V)= V + R_{\leq N}(V)V  \quad \text{with} \quad 
  R_{\leq N}(V) \in  \Sigma_{p}^N \wtcR^{-\vr}_q\otimes \mM_2(\C) \, , 
   \ \ \forall \varrho \geq 0 \, , 
\ee
(given by Lemma \ref{extistencetruflow}) satisfies 
\be \label{denew2B}
 \mF_{\leq N}^*\Omega_{{\leq N}} = \Omega_c + \Omega_{> N} 
\ee 
 where $\Omega_{> N}  $ is a pluri-homogeneous $ 2$-form in 
 $ \Sigma_{N+1} \wt \Uplambda_q^2 $. 
 \end{proposition} 

\begin{proof}
 We follow the famous  deformation argument  by Moser.
We define the homothety between the  symplectic $ 2$-forms 
$ \Omega_c $  and $\Omega_{\leq N} $ defined in  \eqref{OmegaB} by setting
 \be\label{omegamoser}
 \Omega^\tau := \Omega_c+\tau (\Omega_{{\leq N}}-\Omega_c) \, , \quad
  \forall \tau\in[0,1] \, .  
 \ee
Equivalently $  \Omega^\tau = \la E^\tau(V) \cdot , \cdot \ra_r $ with 
associated symplectic tensor 
\begin{align}\label{asspt}
 E^\tau (V) &  = E_c+ \tau (E_{{\leq N}}(V)-E_c) \\
 & \stackrel{ \eqref{EB}}  = E_c + \tau \bR_{{\leq N}}(V)+ \tau{\bf A}_{{\leq N}}^\top(V)\ E_c \ {\bf G}_{{\leq N}}(V) + \tau \rS_{>N}(V)\label{EtauV}
 \end{align}
  where
  
  \noindent 
 $ \bullet $ $ \bR_{{\leq N}}(V) $ is the matrix of 
 pluri--homogeneous smoothing operators
$$
 \bR_{{\leq N}}(V):=  {\bf G}_{{\leq N}}^\top(V)\ E_c  {\bf A}_{{\leq N}}(V)+  {\bf G}_{{\leq N}}^\top(V) \ E_c \ {\bf G}_{{\leq N}}(V)  
$$
belonging  to 
$ \Sigma_{p} \wtcR^{-\varrho}_{q}\otimes \mM_2 (\C) $ for any $ \varrho \geq 0 $;   
 
   \noindent 
 $ \bullet $  $\rS_{>N}(V)$ is the map in $ \Sigma_{N+1} \wtcS_q \otimes \mM_2 (\C) $ of Lemma  \ref{Lem:ZB}.

In addition, by \eqref{dthetaB=Omega} and \eqref{exactc}, we have 
\be\label{Omegadtheta}
 \Omega^\tau = \di \theta^\tau \, , 
 \quad  \theta^\tau:= \theta_c+ \tau( \theta_{\leq N}- \theta_c) 
\ee
 where $ \theta_{\leq N}$ is given in \eqref{ZB}. 
 
 We look for a $\tau$--dependent pluri-homogeneous smoothing  vector field  $Y^\tau(V) 
 $ in $ \Sigma_{p+1} \wt \X^{-\varrho}_q$, for any $ \varrho \geq 0 $, 
 such that  its approximate flow $\mathcal{F}^\tau_{\leq N}$ up to homogeneity $ N $  
 (defined by Lemma  \ref{extistencetruflow}), 
satisfies 
 \be \label{dar12}
 \frac{{\rm d}}{{\rm d}\tau} (\mF^\tau_{\leq N})^* \Omega^\tau=  \di \theta^\tau_{> N+1} \, ,  \quad \forall \tau\in[0,1] \, , 
 \ee 
for a certain $1$-form   $  \theta^\tau_{> N+1} $   in $ \Sigma_{N+2} 
 \wt \Uplambda^1_q $. 
 Then, integrating \eqref{dar12} and recalling \eqref{omegamoser}, we deduce 
   $$
   (\mF^1_{\leq N})^*\Omega_{{\leq N}} = 
   \Omega_c + \int_0^1   \di\theta^\tau_{> N+1} \, \di \tau \, ,
   $$ 
which proves \eqref{denew2B} with $\mathcal{F}_{\leq N} := \mathcal{F}^1_{\leq N}$ and 
$  
   \Omega_{> N} :=  \int_0^1 \di\theta^\tau_{> N+1} \, \di \tau $.

We now construct the vector field $Y^\tau(V)$. 
Using the definition of Lie derivative and the Cartan magic formula, we derive the chain of identities
 \begin{align}
 \notag
 \frac{{\rm d}}{{\rm d}\tau} (\mF_{\leq N}^\tau)^* \Omega^\tau 
& \stackrel{\eqref{Omegadtheta}, \eqref{commu}
}{=} \di \Big( \frac{{\rm d}}{{\rm d}\tau} (\mF_{\leq N}^\tau)^* \theta^\tau \Big)\\
&\stackrel{\eqref{deriv:pullback}}{=} \di\Big((\mF_{\leq N}^\tau)^*( \mL_{Y^\tau} \theta^\tau+ \frac{{\rm d}}{{\rm d}\tau}\theta^\tau)\Big)+ \di \tilde \theta^\tau_{>N+1} \notag \\ 
 \notag
&  \stackrel{\eqref{cartan}, \eqref{commu}}{=} (\mF^\tau_{\leq N})^*\di \big( (i_{Y^\tau}\circ {\rm d }+ {\rm d}\circ i_{Y^\tau}) \theta^\tau + 
 \theta_{{\leq N}} -\theta_c\big)+ \di \tilde \theta^\tau_{>N+1}\\
 & \stackrel{\eqref{Omegadtheta}, \eqref{d2=0}}{=} (\mF^\tau_{\leq N})^*  {\rm d}  \big( 
  i_{Y^\tau} \Omega^\tau + \theta_{{\leq N}}- \theta_c \big) +\di \tilde \theta^\tau_{>N+1}  \label{meraviglia}
 \end{align}
where $ \tilde \theta^\tau_{>N+1} $ is a $1$-form  in $ \Sigma_{N+2} \wt \Uplambda_q^1$ by Lemma \ref{lem:pullback}. 
We look for a vector field $Y^\tau(V)$ and a $0$-form  $\cW^\tau(V)$ such that 
 \be\label{darPRIN}
i_{Y^\tau} \Omega^\tau+ \theta_{{\leq N}}-\theta_c  = \breve \theta^\tau_{> N+1} + \di  \cW^\tau,  
 \ee
for some  pluri--homogeneous $1$--form $ \breve \theta^\tau_{> N+1}$ in $ \Sigma_{N+2} \wt \Uplambda^1_q $. 
If  \eqref{darPRIN} holds, then, in view of \eqref{meraviglia}, equation \eqref{dar12} is satisfied with 
$$
\theta_{> N+1}^\tau:= (\mF^\tau_{\leq N})^* \breve \theta^\tau_{>N+1}+ \tilde \theta^\tau_{>N+1} \in\Sigma_{N+2} \wt \Uplambda^1_q \, .
$$
We turn to solve equation \eqref{darPRIN}.
Using  \eqref{asspt},  \eqref{ZB},  recalling that 
$ \theta_c(V)  = \frac12 \la E_c V, \cdot  \ra_r $,
and writing $\breve \theta^\tau_{>N+1}(V)  = \frac12 \la \breve Z^\tau_{>N}(V)V, \cdot  \ra_r $,   
we first rewrite \eqref{darPRIN} as the 
equation
 \be \label{risoluta}
 E^\tau(V) Y^\tau(V)+ \frac12 \Big( 
 Z_{{\leq N}}(V)V- E_c V + \rS_{>N}(V)V  \Big)= \frac12 \breve Z^\tau_{> N}(V)V 
+ \nabla \cW^\tau(V) \ . 
 \ee

\begin{remark}
 This equation  is  linear in $Y^\tau(V)$.
 In the works  \cite{KukPer, Bam2, Cuc1,Cuc2,BaM0,BaM1} 
 the operator $E^\tau(V)$ is  a smoothing perturbation of $E_c$, so is its inverse  and the vector field $Y^\tau(V)$ is immediately a smoothing vector field.
In our case, $E^\tau(V)$ is a (possibly) unbounded perturbation of $E_c$, and its (approximate) inverse is only an $m$-operator. 
Hence, the composition of the (approximate) inverse of $E^\tau(V)$  with the  smoothing operator $Z_{\leq N}(V)  - E_c$  (see \eqref{ZB})
is only an $m$-operator,  not a smoothing one (see the bullet at pag. \pageref{McR}).
Therefore we cannot directly conclude  that $Y^\tau(V)$ is a smoothing vector field. 
We proceed differently and solve  the equation \eqref{risoluta} in homogeneity, exploiting the freedom given by the  function $\grad \cW^\tau(V)$ to remove the non-smoothing components of the equation, thanks to
 structural Lemma \ref{diadg}.
\end{remark} 

 By  \eqref{EtauV}  
  and  \eqref{ZB}, equation \eqref{risoluta} becomes 
 \begin{equation}\label{fin}
 \begin{aligned}
E_c Y^\tau(V) = &-\tfrac{1}{2} {\bf G}_{{\leq N}}^\top(V) \ E_c  \ {\bf A}_{{\leq N}} (V) V- \tau \bR_{{\leq N}}(V) Y^\tau(V)\\& - \tau {\bf A}_{{\leq N}}^\top(V)E_c{\bf G}_{{\leq N}}(V)Y^\tau(V)  + \nabla \cW^\tau(V)\\
& +  \tfrac12 \breve Z^\tau_{> N}(V)V- \tfrac12 \rS_{>N}(V)V- \tau \rS_{>N}(V) Y^\tau(V) \, . 
\end{aligned}
\end{equation}
We now solve 
 \eqref{fin} 
for a smoothing vector field $Y^\tau(V)$, a suitable function $ \cW^\tau(V)$ and a high homogeneity pluri--homogeneous map $\breve Z_{> N}^\tau(V)$  by an iterative procedure in increasing order
of homogeneity.  Note that $ {\bf G}_{{\leq N}}^\top(V) \ E_c  \ {\bf A}_{{\leq N}} (V)$ and $ \bR_{{\leq N}}(V)$  
are smoothing operators  unlike 
$ {\bf A}_{{\leq N}}^\top(V)E_0{\bf G}_{{\leq N}}(V)Y^\tau(V)$ 
that will be canceled using $ \nabla \cW^\tau(V)$, thanks to the structure property explicated in Lemma \ref{diadg}.

\begin{lemma}\label{lem.sol.dar}
Fix $ \bar N\in \N$ such that $ (\bar N+2)p\geq N+1$.
There exist

\noindent
$\bullet$ 
a pluri-homogeneous smoothing vector field $Y^\tau(V) = \sum_{a = 0}^{\bar N} Y_{(a)}^\tau(V)$, defined for any 
$ \tau \in [0,1] $,  with $Y_{(a)}^\tau(V) $ in $  \Sigma_{(a+1) p + 1}\wt \X^{-\varrho}_q$ for any $\varrho  \geq 0$, uniformly in $ \tau \in [0,1] $; \\
\noindent
$\bullet$ a pluri-homogeneous  Hamiltonian 
$\cW^\tau(V) = \sum_{a = 0}^{\bar N} 
\cW^\tau_{(a)}(V)$, defined for any 
$ \tau \in [0,1] $,  with $\cW_{(a)}^\tau(V) $ in $  \Sigma_{(a+1)p+2} \wt \Uplambda^0_q$, uniformly in $ \tau \in [0,1] $;\\
\noindent
$\bullet$ a pluri-homogeneous matrix of operators $\breve Z_{>N}^\tau(V) $, 
defined for any 
$ \tau \in [0,1] $,  in $ \Sigma_{N+1} \wt \cM_q \otimes \cM_2(\C)$, uniformly in $ \tau \in [0,1] $;

which solve equation \eqref{fin}. 
\end{lemma}

\begin{proof}
We define
\be \label{Def:Y0}
\begin{cases}
Y_{(0)}(V):= -\frac{1}{2} E_c^{-1} {\bf G}^\top_{{\leq N}}(V)\ E_c  \ {\bf A}_{{\leq N}}(V)V \\
 \cW_{(0)}^\tau(V):=0
 \end{cases}  
\ee
Note that  $Y_{(0)}(V)$ is smoothing vector field in 
 $ \Sigma_{p+1} \wt \X^{-\varrho}_{q}  $ for any $ \varrho \geq 0 $, since 
 ${\bf G}^\top_{{\leq N}}(V) \ E_c  \ {\bf A}_{{\leq N}}(V)$ are smoothing operators in $ \Sigma_{p} \wtcR^{-\varrho}_{q}\otimes \mM_2 (\C) $ for any $\varrho \geq 0$, by Lemma \ref{Lem:ZB}-($i$).
 
For $a \geq 0$, we prove the following recursive statements: {\it there exist a \\
\noindent
 {\bf (S1$)_a$}  pluri-homogeneous smoothing vector field $Y_{(a)}^\tau(V)$ belonging to $\Sigma_{(a+1)p+1}\wt \X_q^{-\vr}$ for any $\varrho \geq 0$;\\
\noindent
 {\bf (S2$)_a$}  pluri-homogeneous Hamiltonian $ \cW_{(a)}^\tau(V) $ in  $\Sigma_{{(a+1)p+2}} \wt \Uplambda^0_q $;\\
\noindent 
 {\bf (S3$)_a$}  matrix of pluri-homogeneous operators $Z_{>N,(a)}^{\tau}(V)$ in $ \Sigma_{N+1} \wtcM_q \otimes \mM_2 (\C)$; \\
 uniformly in $ \tau \in [0,1] $,  
with $ (Y_{(0)}^\tau(V), \cW_{(0)}^\tau(V),  Z_{>N,(0)}^{\tau}(V))$ defined in \eqref{Def:Y0},
satisfying,  for any $a \geq 1 $,}
 \be \label{jesima1}
\begin{aligned}
E_c Y_{(a)}^\tau(V) & =  - \tau \bR_{{\leq N}}(V) 
Y_{(a-1)}^\tau(V)\\
& \quad - \tau {\bf A}_{{\leq N}}^\top(V)E_c{\bf G}_{{\leq N}}(V)
[ Y_{(a-1)}^\tau(V)] + \nabla \cW_{(a)}^\tau(V)\\
& \quad +  Z_{>N,(a)}^{\tau}(V)V- \tau \rS_{>N}(V)Y_{(a-1)}^\tau(V)  \, . 
\end{aligned}
\ee
Given $ (Y_{(a-1)}^\tau(V), \cW_{(a-1)}^\tau(V),  Z_{>N,(a-1)}^{\tau}(V))$ we now prove
{\bf (S1$)_a$}-{\bf (S3$)_a$}.   
 Note that the first term in \eqref{jesima1} is a smoothing vector field of homogeneity $ (a+1)p+1$ while  the first term in the second line of \eqref{jesima1} has homogeneity  $ (a+1)p+1$ but it is not a smoothing vector field.
However by Lemma  \ref{diadg} we have  the decomposition 
$$
 {\bf A}_{{\leq N}}^\top(V)E_c{\bf G}_{{\leq N}}(V)[Y^\tau_{(a-1)}(V)] 
  = \grad \breve \cW^\tau_{(a-1)}(V) +\breve  R^\tau_{(a-1)}(V)V+ \breve M_{>N,(a-1)}^{\tau}(V)V \, , 
$$
where 
$ \breve  \cW^\tau_{(a-1)}(V)$ is a Hamiltonian in 
$ \Sigma_{{(a+1)p+2}} \wt \Uplambda^0_q $, 
$ \breve R^\tau_{(a-1)}(V) $ is a pluri--homogeneous  smoothing operator in $ \Sigma_{(a+1)p} \wtcR^{-\vr}_q \otimes \mM_2 (\C) $ and 
$  \breve M_{>N,(a-1)}^{\tau}(V)$ is a pluri--homogeneous operator in $\Sigma_{N+1} \wtcM_q \otimes \mM_2 (\C) $.
Then equation \eqref{jesima1} becomes 
$$
\begin{aligned}
E_c Y_{(a)}^\tau(V) & =  - \tau \bR_{{\leq N}}(V) Y_{(a-1)}^\tau(V)
- \tau \breve R^\tau_{(a-1)}(V)V  \\
& \quad + \grad \big(\cW_{(a)}^\tau (V) - \tau \breve \cW_{(a-1)}^\tau(V) \big)\\
 & \quad 
 + Z_{>N,(a)}^{ \tau}(V)V  -\tau \breve M_{>N, (a-1)}^{\tau}(V)V- \tau \rS_{>N}(V) Y_{(a-1)}^\tau(V) \, 
\end{aligned}
$$
which is solved by
$$
\begin{cases}
Y_{(a)}^\tau(V):=  -E_c^{-1}\big[  \tau \bR_{{\leq N}}(V) Y_{(a-1)}^\tau(V)+ \tau \breve R^\tau_{(a-1)}(V)V \big] \\
 {\cal W}_{(a)}^\tau(V):=\tau    \breve \cW_{(a-1)}^\tau(V) \\
  Z_{>N, (a)}^{\tau}(V)V:= \tau  \breve M_{>N, (a-1)}^{\tau}(V)V+ \tau \rS_{>N}(V)Y_{(a-1)}^\tau(V) \, 
 \end{cases} 
$$
proving  {\bf (S1$)_a$}-{\bf (S3$)_a$}.

Summing \eqref{Def:Y0} and \eqref{jesima1}  for any $a = 1, \ldots, \bar N$ we find that
$Y^\tau(V) = \sum_{a = 0}^{\bar N} Y_{(a)}^\tau(V)$ and 
$\cW^\tau(V) = \sum_{a = 0}^{\bar N} 
\cW_{(a)}^\tau (V)$ solve \eqref{fin}
with
$$
\frac12 \breve Z_{>N}^\tau(V)V
:= \frac 12 \rS_{>N}(V)V + 
\sum_{a=1}^{\bar N}  Z_{>N,(a)}^{\tau} (V) V
+
\tau \big[ \bR_{\leq N}(V)
+\bA_{\leq N}^\top(V) E_c \bG_{\leq N} (V) +  
\rS_{>N}(V)  
\big] Y_{(\bar N)}^\tau(V)
$$
which is an operator in $ \Sigma_{N+1} \wtcM_q \otimes \mM_2(\C) $, since 
$ (\bar N+2)p \geq N+ 1  $.  Lemma  \ref{lem.sol.dar} is proved.
\end{proof}
The approximate flow up to homogeneity $ N $ of  
the smoothing vector field $ Y^\tau $ defined by Lemma \ref{lem.sol.dar} solves 
\eqref{dar12}. This concludes the proof of Proposition  \ref{DarB}.
 \end{proof}

\noindent \begin{proof}[{\bf Proof of Theorem \ref{thm:almost}.}]
The map  $\mE_N :=  \Psi_{\leq N} \circ \mF_{\leq N}$, where  $ \mF_{\leq N}$ is defined in  Proposition \ref{DarB}, fulfills 
$$
 \mE_N^* \Omega_c = \mF_{\leq N}^* \Psi_{\leq N}^*  \Omega_c 
 \stackrel{ \eqref{OmegaB}}  =  \mF_{\leq N}^* \Omega_{\leq N} 
 \stackrel{\eqref{denew2B} } =  \Omega_c + \Omega_{>N} 
$$
and so $\mE_N$ is symplectic up to homogeneity $N$.
We define the map $\cC_{\leq N}$ in \eqref{CiEnne} as the  approximate inverse (given by Lemma \ref{inversoapp}) 
 of the nonlinear map $\mF_{\leq N}$ in \eqref{darbouxF}, hence it  has the claimed  form.
Since  $\Psi_{\leq N}$ is an  approximate inverse of $\Phi_{\leq N}$,
the map 
$\mD_N:= \mC_{\leq N} \circ \Phi_{\leq N}$ 
 is an approximate inverse of $\mE_N$, and so it is symplectic up to homogeneity $N$ by 
Lemma \ref{lem:ails}.
\end{proof}

\noindent \begin{proof}[{\bf Proof of Theorem \ref{thm:symplgu}.}]
We write 
the good-unknown of Alinach  \eqref{simgud} in complex variables
$ (u, \bar u)$ induced by the transformation $ \mC $  defined in \eqref{cc}, obtaining the real-to-real  
spectrally localized, linearly symplectic  map (according to Definition \ref{LS})  
$$ 
\mG_c(U):=   \mC  \Opbw{\big[\begin{smallmatrix}1 & 0 \\ -B( \mC U) & 1\end{smallmatrix} \big]} \mC^{-1},\quad 
(\eta, \psi) = \mC U \, , 
$$
where  $B(\eta, \psi) $ is the real function 
defined in \eqref{form-of-B} which, as stated in Lemma \ref{laprimapara}, belongs  
to $ \Sigma {\cal F}^\R_{K, 0,1}[r,N]$.  
Then Theorem \ref{thm:symplgu} follows by applying 
Theorem \ref{thm:almost} to the pluri-homogeneous spectrally localized map 
$ \bB_{\leq N} (U) = $ $ \mP_{\leq N}( \mG_c(U)) = \mC \opbw\big(\big[\begin{smallmatrix}1 & 0 \\ \mP_{\leq N}[-B( \mC U)] & 1\end{smallmatrix} \big]\big)\mC^{-1} $. 
\end{proof}

\part{Almost global existence of water waves}\label{part:II}

We now begin the proof of  
the almost global existence Theorem \ref{teo1}  
for solutions of the gravity-capillary water waves equations \eqref{eq:etapsi} 
with constant vorticity.

After further describing  the Hamiltonian structure of the 
water waves equations \eqref{eq:etapsi} and diagonalizing 
the linearized system at the
equilibrium, we paralinearize
the water waves equations \eqref{eq:etapsi} 
with constant vorticity, 
 written in the Zakharov-Craig-Sulem $(\eta,\psi)$ variables, which are the Hamiltonian system \eqref{HamWW} with the non--standard Poisson tensor $ J_\gamma $.
 Then we express such paralinearized  system in the   Wahl\'en  coordinates $(\eta,\zeta)$ in \eqref{Whalen}, which coincides with the Hamiltonian system in \eqref{HamWW2} in  standard Darboux form. Finally we write such paralinearized  system in the complex variable $U $  defined in \eqref{defMM-1}, i.e. \eqref{Ugrande}.
The final system \eqref{complexo} is \emph{Hamiltonian in the complex sense}, i.e. has the form \eqref{complexHS}.

\section{Paralinearization of the 
water waves equations 
with constant vorticity and its Hamiltonian structure} 
\label{sec:paraWW}

From now on we consider  \eqref{eq:etapsi} as a system
 on (a dense subspace of) the {\it homogeneous} space $ \dot L^2_r\times \dot L^2_r $, namely,  
 denoting  $ X_\gamma (\eta, \psi) $ the right hand side in \eqref{eq:etapsi}, 
 we consider  
\be\label{WWhomo}
 \partial_t (\eta, \psi) =  X_\gamma (\Pi_0^\bot \eta, \psi)  
\ee
where $\Pi_0^\perp$ is the  $ L^2 $-projector onto the space of functions with zero average. 
For simplicity of notation we shall not distinguish between \eqref{WWhomo}
and \eqref{eq:etapsi},  which are equivalent via the isometric 
isomorphism $ \Pi_0^\bot $
between $ \dot L^2(\T;\R) $ and $ L^2_0(\T;\R) $. 
System \eqref{WWhomo} is the Hamiltonian system 
as in \eqref{HamWW} 
defined on (a dense subspace of) $ \dot L^2_r \times \dot L^2_r $ generated   by 
the Hamiltonian 
$ H_\gamma ( \Pi_0^\bot \eta, \psi )  $, with $ H_\gamma  $   in \eqref{H.gamma},
computing the $\dot L^2_r$-gradients $(\grad_\eta H_\gamma, \grad_\psi H_\gamma) $
with respect to the scalar product  $\langle \cdot , \cdot \rangle_{\dot L^2_r}$ in \eqref{scpr12hom} 
and regarding the Poisson tensor $ J_\gamma $ in 
\eqref{HamWW}  as a linear operator acting in  
$ \dot L^2_r\times \dot L^2_r$. 
We shall not insist more on this detail. 
\\[1mm]
\noindent{\bf Wahl\'en variables.} 
The variables $(\eta, \psi)$
are not Darboux coordinates, since the   Poisson tensor $ J_\gamma $ in 
\eqref{HamWW}  is not the canonical  one when $\gamma \neq 0$.
Wahl\'en noted in  \cite{Wh} that, introducing the variable 
$ \zeta:= \psi -   \tfrac{\gamma}{2} \partial_x^{-1} \eta $, 
 the coordinates $(\eta, \zeta)$  are canonical coordinates. 
Precisely,  under the linear change of variables
\be \label{Whalen}
\vect{\eta}{ \psi} = \cW \vect{\eta}{\zeta}\, , \quad \cW := 
\begin{pmatrix} \uno & 0\\  \frac{\gamma}{2} \pa_x^{-1} & \uno \end{pmatrix} \, , \quad
\cW^{-1} = \begin{pmatrix} \uno & 0\\ - \frac{\gamma}{2} \pa_x^{-1} & \uno \end{pmatrix}  \, , 
\ee
the Poisson tensor  $ J_\gamma $ becomes the standard one, 
$$
\cW^{-1} J_\gamma (\cW^{-1})^\top= J \, , \quad J:= \begin{pmatrix} 0 & \uno \\ -\uno& 0\end{pmatrix} \, , 
$$
and the Hamiltonian system \eqref{HamWW} assumes the standard Darboux  form
\be \label{HamWW2}
\pa_{t}\vect{\eta}{ \zeta} = J \vect{\nabla_{\eta} \mH_\gamma( \eta, \zeta)}{\nabla_{\zeta}  \mH_\gamma( \eta, \zeta)}, \quad \mH_\gamma( \eta, \zeta):=H_\gamma \big( \eta, \zeta+ \frac{\gamma}{2} \pa_x^{-1}\eta \big) \, .
\ee
Note that the new Hamiltonian $ \mH_\gamma$ is still translation invariant
so  is its Hamiltonian vector field. 
\\[1mm]
{\bf Linearized equation at the equilibrium.} 
The linearized equations  \eqref{HamWW2} at the equilibrium $ (\eta,\zeta)=(0,0)$ are
obtained by conjugating the linearized equations \eqref{eq:etapsi} at $ (\eta,\psi)=(0,0)$, namely 
\be\label{eq:lin}
\pa_t \vect{  \eta}{  \zeta} \!\! = 	\!
\cW^{-1} \! \begin{pmatrix} 0 & \!\!\!\! G(0) \\ -(g+ \kappa D^2) & \!\!\!\! \gamma G(0) \pa_x^{-1} \end{pmatrix} \! \! \cW \! = \!
\begin{pmatrix}\frac{ \gamma}{2} G(0) \pa_x^{-1} & \!\!\! G(0) \\ -(g+ \kappa D^2 + \frac{\gamma^2}{4} G(0)D^{-2} ) &\!\!\! \frac{ \gamma}{2}  \pa_x^{-1} G(0)\end{pmatrix} \!\!\vect{ \eta}{  \zeta} 
\ee
where  $ D := \frac{1}{\ii} \pa_x $ and the   Dirichlet-Neumann operator  
$G(0)$ at the flat surface $\eta = 0$ is the Fourier multiplier 
with  symbol
 \be\label{Gxi} 
  \tG(\xi):= 
  \begin{cases} 
  \xi\tanh(\tth \xi)& 0<\tth<+\infty\\ 
 |\xi| & \tth=+\infty \, . 
 \end{cases}
 \ee 
We diagonalize 
system \eqref{eq:lin} introducing the complex variables 
\be\label{defMM-1}
\begin{aligned}
& \qquad \qquad \qquad \qquad\qquad  \vect{ u}{ {\bar u} }:= \mM^{-1} \vect{ \eta}{ \zeta} \, , \\
& \ \mM:= \frac{1}{\sqrt{2}} \begin{pmatrix} M(D) & M(D)\\ - \ii M^{-1}(D) & \ii M^{-1} (D) \end{pmatrix}, \quad
 \mM^{-1}:= \frac{1}{\sqrt{2}} \begin{pmatrix} M(D)^{-1} & \ii M(D)\\ M^{-1}(D) & -\ii M(D) \end{pmatrix} \, , 
\end{aligned}
\ee
where $ M(D) $ is  the Fourier multiplier
\be\label{defMD}
M(D):= \left( \frac{G(0)}{ g+ \kappa D^2 + \frac{\gamma^2}{4} G(0)D^{-2} } \right)^{\frac{1}{4}} \, . 
\ee
A direct computation (cfr. Section 2.2. in \cite{BFM1}), 
using the identities
\be\label{omeghino0}
M(D) \big( g+ \kappa D^2 + \frac{\gamma^2}{4} G(0) D^{-2} \big)M(D)= \omega(D)= M^{-1}(D) G(0) M^{-1}(D)  
\ee
where $\omega(D)$ is the Fourier multiplier with symbol 
\be \label{omegaxi}
 \omega(\xi):=\sqrt{ \tG(\xi)
 \Big( g+ \kappa \xi^2 + \frac{\gamma^2}{4} \frac{\tG(\xi)}{\xi^{2}} \Big)} 
 \ee
(with  $ \tG(\xi) $  defined in  \eqref{Gxi})
 shows that the variables $(u, \bar u)$ in 
 \eqref{defMM-1}
 solve the diagonal linear system 
\be \label{diaglin}
\pa_t\vect{u}{ {\bar u} }= -\ii \vOmega(D)\vect{u}{{\bar u}}, \quad \vOmega(D):=\begin{pmatrix}\Omega(D) & 0 \\ 0 & -\bar \Omega(D) \end{pmatrix}
\ee
where
\be \label{relation} 
\Omega(D):= \omega(D) + \ii \frac{ \gamma}{2} G(0) \pa_x^{-1}, \quad \bar \Omega(D):= \omega(D) - \ii \frac{ \gamma}{2} G(0) \pa_x^{-1} \, . 
\ee
The real-to-real  
system \eqref{diaglin} amounts to the scalar equation 
$$
\pa_t u = - \im \Omega(D) u \, , \quad u(x) = 
\frac{1}{\sqrt {2\pi}} \sum_{j \in \Z\setminus \{0\}} u_j \,  e^{\im j x} \ ,
$$
which, written  in Fourier basis,  decouples in infinitely many harmonic oscillators
$$
\pa_t u_j = - \im \Omega_j(\kappa) u_j \, , \quad j \in \Z \setminus \{0\} \, , 
$$
where 
\be \label{omegonej}
\Omega_j(\kappa):= \omega_j(\kappa) +  \frac{ \gamma}{2} \frac{\tG(j)}{j} \, , 
\qquad 
\omega_j(\kappa):=\sqrt{ \tG(j) \Big(g+ \kappa j^2 + \frac{\gamma^2}{4} \frac{\tG(j)}{j^{2}}\Big)} \, . 
\ee
Note that the map $ j\rightarrow  \Omega_j(\kappa)$ is not even because of the vorticity term $ \frac{ \gamma}{2} \frac{\tG(j)}{j}$ which is odd.

A fundamental property that we prove in Appendix \ref{sec:non-res} is that  
the linear frequencies $\{\Omega_j(\kappa)\}_{j \in \Z\setminus \{0\}}$  
satisfy  the non-resonance conditions of Theorem \ref{nonresfin0}.  
Thus one can think to implement a  Birkhoff normal form procedure. 
Since the water waves equations  \eqref{eq:etapsi}  are a quasilinear  system 
 we first paralinearize them.

\paragraph{Paralinearization of the water waves equations with constant vorticity.}  
We denote the horizontal and vertical components of the velocity field 
at the free interface by
\begin{align} 
\label{def:V}
& V =  V (\eta, \psi) :=  (\pa_x \Phi) (x, \eta(x)) = \psi_x - \eta_x B \,,
\\
\label{form-of-B}
& B =  B(\eta, \psi) := (\pa_y \Phi) (x, \eta(x)) =  \frac{G(\eta) \psi + \eta_x \psi_x}{ 1 + \eta_x^2} \, . 
\end{align}

\begin{lemma}{\bf (Water-waves equations in Zakharov-Craig-Sulem  
variables)}\label{laprimapara} 
Let  $N \in \N_0 $ and $ \varrho  \geq 0$. 
For any $K \in \N_0$ there exist $ s_0 ,r > 0 $ such that,
if $ (\eta, \psi) \in B^K_{s_0} (I;r)  $ solves \eqref{eq:etapsi}, then 
\begin{align}   
\partial_t \eta  & =G(0)\psi + \Opbw{-B(\eta,\psi;x)|\x|- \ii V_\gamma(\eta,\psi;x) \xi +a_0(\eta, \psi; x, \xi )}\eta+  \Opbw{ b_{-1}(\eta; x, \xi)} \psi\notag  \\
& \quad + R_1(\eta)\psi + R'_1(\eta,\psi)\eta \, , \label{eq:1n}\\
\partial_t\psi  & = -(g+\kappa D^{2})\eta +   \gamma  G(0) \pa_x^{-1}  \psi 
 +\Opbw{-\kappa {\mathtt f}(\eta;x) \x^{2}- B^2(\eta,\psi;x) 
 | \x|+c_0(\eta, \psi;x, \xi)} \eta \notag \\ 
 & \quad + 
 \Opbw{ B(\eta,\psi;x) |\x |-\ii V_\gamma(\eta,\psi;x) \x + d_0(\eta, \psi; x,\xi )} \psi   
   + R_2(\eta,\psi)\psi + R'_2(\eta,\psi)\eta \label{eq:2n}
\end{align}
where
\begin{itemize} 
\item  $V_\gamma (\eta,\psi;x) := V(\eta,\psi;x) - \gamma \eta (x) $ is a function in $ \Sigma {\cal F}^\R_{K, 0,1}[r,N]$ as well as the functions 
 $ V, B $ defined in \eqref{def:V}-\eqref{form-of-B}; 
\item  $ a_0,\, c_0,\, d_0$ are 
 symbols in $ \Sigma\Gamma^{0}_{K,0,1}[r,N]$ and 
 $b_{-1}(\eta;x,\xi) $ is a symbol  in $ \Sigma\Gamma^{-1}_{K,0,1}[r,N]$ satisfying \eqref{realetoreale};
\item  the function  $ {\mathtt f}(\eta; x) := (1+\eta_{x}^{2}(x))^{-\frac{3}{2}}-1$ belongs to $ \Sigma \mF^{\R}_{K,0,2}[ r, N]$;
\item  $R_1$, $R_{1}'$, $R_2$, $R_{2}'$ are real  smoothing  operators in 
$\Sigma\mathcal{R}^{-\varrho}_{K,0,1}[r,N]$.  
\end{itemize}
Moreover \eqref{eq:1n}--\eqref{eq:2n} are the \emph{Hamiltonian} system \eqref{HamWW}.
\end{lemma}

\begin{proof}
By Proposition 7.4 of \cite{BD},  the function 
 $B $ defined in \eqref{form-of-B} belongs to $ \Sigma {\cal F}^\R_{K, 0,1}[r,N]$, as well as  
 the function $ V $ in \eqref{def:V} and $ V_\gamma = V - \gamma \eta $.

\noindent
{\sc Paralinearization of the first equation in \eqref{eq:etapsi}.}
We  use the paralinearization of the Dirichlet-Neumann operator $ G(\eta)\psi$ proved in \cite{BD}. 
By Propositions 7.5 and 8.3 in \cite{BD}  where $ \omega := \psi- \opbw(B)\eta $ 
is the ``good unknown'' of Alinhac, using 
Propositions \ref{teoremadicomposizione} and \ref{composizioniTOTALIs}-($i$), 
the second bullet below 
\eqref{espansione2}, 
and noting that $\xi \tanh(\tth \xi) -  |\xi| \in \wt\Gamma^{-\varrho}_0$, for any $ \varrho >0 $,  
we get 
\begin{align}
G(\eta)\psi & = G(0) \big( \psi- \opbw(B) \eta\big) + \opbw\big(-\ii V \xi + \breve a_0\big) \eta + \opbw( b_{-1}) \big(\psi-  \opbw(B) \eta\big) 	\notag \\
& \quad + R'(\eta,\psi)\eta+R(\eta)\psi \notag \\
& = G(0)\psi +\Opbw{-B|\x|- \ii V \xi +  a_0' } \eta
+ \opbw( b_{-1})\psi +    R(\eta)\psi + R'(\eta,\psi)\eta \label{DNPara}
\end{align}
where $ \breve a_0,  a_0'$ are symbols in $ \Sigma \Gamma^0_{K,0,1}[r,N]$, 
$  b_{-1}$ is a symbol in $ \Gamma^{-1}_{K,0,1}[r,N]$
depending only on $ \eta $, and
$R(\eta)$, $R'(\eta, \psi)$ are smoothing operators in $\Sigma\mathcal{R}^{-\varrho}_{K,0,1}[r,N]$.

We now paralinearize the term with the vorticity. Using Lemma \ref{bony}, 
Proposition \ref{teoremadicomposizione}, the identity $ \eta_x= \Opbw{\ii \xi} \eta$ and $(i)$ of Proposition \ref{composizioniTOTALI}  we get 
\be\label{ga1}
 \eta \eta_x 
  = \opbw ( \eta ) \opbw(\ii \xi)\eta + \opbw( \eta_x) \eta + R( \eta) \eta
  = \Opbw{ \ii \eta\xi +  \tfrac12 \eta_x}\eta  + R(\eta)\eta 
\ee
where $ R(\eta) $ is a homogeneous smoothing operator in $ \widetilde \mR^{-\vr}_1$. Then 
\eqref{DNPara} and \eqref{ga1} imply \eqref{eq:1n}
with symbol $a_0:= a_0' + \frac{\gamma}{2} \eta_x $ in $ \Sigma \Gamma^0_{K,0,1}[r,N] $. Furthermore, since \eqref{eq:1n} is a real equation we may assume that $a_0$ and $b_{-1}$ satisfy \eqref{realetoreale} eventually replacing them with $\tfrac12 (a_0+\bar a_0^\vee)$ and $\tfrac12 (b_{-1}+\bar b_{-1}^\vee)$ and replacing the smoothing remainders with $ \tfrac12(R_1+\bar R_1)$ and $\tfrac12(R'_1+\bar R'_1)$ .
\\[1mm]
{\sc Paralinearization of the second equation  in \eqref{eq:etapsi}.}
By Lemma \ref{bony} and Proposition \ref{teoremadicomposizione} we get 
\be\label{interm1}
-\frac12 \psi_x^2 = - \Opbw{\ii \psi_x \xi - \tfrac12 \psi_{xx}} \psi+ R(\psi)\psi 
\ee
where  $ R(\psi) $  is a 
smoothing operator in $ \widetilde \mR^{-\vr}_1$.
Next, recalling \eqref{form-of-B}
and using Lemma \ref{bony},  we get 
\begin{align}\label{interm}
 \frac12 \frac{ (G(\eta)\psi+ \eta_x\psi_x)^2}{(1+\eta_x^2)} 
 & = \frac12 (1+\eta_x^2) B^2 \\
& = \frac12 \Opbw{(1+\eta_x^2)B} B +
\frac12\Opbw{B} \big[ (1+\eta_x^2)B \big]
+ R(\eta,\psi)\psi+ R'(\eta,\psi)\eta  \notag 
\end{align}
where $R(\eta, \psi)$, $R'(\eta, \psi)$ are smoothing operators in $\Sigma\mathcal{R}^{-\varrho}_{K,0,1}[r,N]$.
Consider the second term in the right hand side of \eqref{interm}. 
Applying Lemma \ref{bony},  Propositions \ref{teoremadicomposizione}
and \ref{composizioniTOTALI} (and since $\Opbw{B}[1] $ is a constant which we neglect because we consider \eqref{eq:etapsi} posed in homogeneous spaces) we get \be\label{Bunopiu} 
 \tfrac12 \opbw(B) \big[ (1+\eta_x^2)B \big] 
 =  \tfrac12  \Opbw{ (1+\eta_x^2)B}B +   \ii \Opbw{B^2\eta_x \xi+ \breve c_0}\eta  +R(\eta,\psi)\psi+R'(\eta,\psi)\eta 
\ee
where 
$ \breve c_0 $ is a symbol in $ \Sigma \Gamma^0_{K,0,1}[r,N]$ and 
$ R(\eta,\psi)$, $ R'(\eta,\psi) $ 
are smoothing operators in $  \Sigma \mR^{-\vr}_{K,0,1}[r,N]$.
 Then  by \eqref{interm}-\eqref{Bunopiu} and \eqref{form-of-B} we deduce that 
\be\label{intermediatepww}
\frac12 (1+\eta_x^2)B^2  = 
  \Opbw{ B} [ G(\eta)\psi+\eta_x\psi_x ]
 - \ii \Opbw{ B^2\eta_x\xi + \breve c_0} \eta+ R(\eta,\psi)\psi+ R'(\eta,\psi)\eta \, . 
\ee
In order to expand this term 
we first write 
\be\label{passint2ww}
\eta_x\psi_x =  \Opbw{ \ii \eta_x \xi - \tfrac12 \eta_{xx}} \psi  + \Opbw{ \ii \psi_x \xi 
- \tfrac12 \psi_{xx}}  \eta + R(\eta) \psi +
R'(\psi ) \eta 
\ee
where $ R(\eta), R'(\psi )  $ are smoothing homogeneous operators in 
$  \widetilde \mR^{-\vr}_1 $. 
Finally, 
using \eqref{DNPara}, \eqref{passint2ww},  Proposition \ref{composizioniTOTALI}, 
and exploiting the explicit form \eqref{def:V} of the function $ V $,  
 we conclude that  \eqref{intermediatepww} is equal to 
\begin{align}
\eqref{interm} 
& = \opbw( -B^2 |\xi|+ \breve c_0)\eta+\opbw(B|\xi|+ \ii B\eta_x\xi+ \breve d_0)\psi+ R(\eta,\psi)\psi+ R'(\eta,\psi)\eta \,  \label{interm2}
\end{align}
where $\breve c_0, \breve d_0 $ are symbols   in $ \Sigma \Gamma^0_{K,0,1}[r,N]$ and 
$R(\eta,\psi), R'(\eta,\psi) $ are smoothing operators in $  \Sigma \mR^{-\vr}_{K,0,1}[r,N]$.

Next we paralinearize the capillary term 
$$
\kappa \pa_x \Big[ \frac{\eta_x}{(1+\eta_x^2)^{1/2}}\Big] = \kappa\pa_x F(\eta_x) \, , \quad F(t):= \frac{t}{(1+t^2)^{1/2}} \, .
$$
The Bony paralinearization formula for the composition (Lemma 3.19 in \cite{BD})  and Proposition \ref{teoremadicomposizione} imply
\be\label{capilterm}
\begin{aligned}
 \pa_x  F(\eta_x)
  & = \Opbw{\ii \xi}  \Opbw{F'(\eta_x)}\eta_x+ R(\eta) \eta
  = \opbw \Big( - (1+\eta_x^2)^{-\frac{3}{2}} \xi^2+ c_{0}' \Big)\eta + R(\eta)\eta \\
  & = - D^2 \eta   - \Opbw{\mathtt f(\eta;x) \xi^2 + c_0'}\eta + R(\eta)\eta 
 \end{aligned}
\ee
where $ c_0' $ is symbol  in $ \Sigma \Gamma^0_{K,0,1}[r,N] $, 
 the function $ {\mathtt f}(\eta; x) := (1+\eta_{x}^{2}(x))^{-\frac{3}{2}}-1$ belongs to $ \Sigma \mF^{\R}_{K,0,2}[ r, N]$ 
and $R(\eta) $ is a smoothing operator in $ \Sigma \mR^{-\vr}_{K,0,1}[r,N]$. 

Next, by Lemma  \ref{bony}  and Proposition \ref{teoremadicomposizione}  we get 
\be\label{vort3} 
 \eta \psi_x 
  = \Opbw { \ii \eta\xi -\tfrac12  \eta_x} \psi +\Opbw { \psi_x } \eta +  R ( \eta) \psi+ R'(\psi) \eta
\ee
where $ R(\eta)$, $ R'(\psi) $ are homogeneous smoothing operators in 
$  \widetilde \mR^{-\vr}_1$.

Finally using \eqref{DNPara}, Propositions \ref{teoremadicomposizione} and  \ref{composizioniTOTALI}-$(i)$, and that  $ \pa_x^{-1} =  \Opbw { \frac{1}{\ii \xi}} $ we get 
\be\label{inter5}
 \pa_x^{-1} G(\eta) \psi
 = 
  G(0) \pa_x^{-1}  \psi  + \Opbw { c_{-2} } \psi + \Opbw{ d_{0}'} \eta+ R(\eta,\psi)\eta+ R'(\eta, \psi)\psi 
 \ee
 where $ c_{-2} $ is a symbol in $ \Sigma\Gamma^{-2}_{K,0,1}[r,N]$,  
the symbol  $ d_{0}' $ is in $  \Sigma \Gamma^0_{K,0,1}[ r,N]$ and $ R,R' $ are smoothing operators in
 $ \Sigma \mR^{-\vr}_{K,0,1}[r,N]$. 
 
In conclusion, collecting \eqref{interm1}, \eqref{interm2},  \eqref{capilterm},  \eqref{vort3}, \eqref{inter5}  and using the explicit form of  $ V $ in \eqref{def:V} 
we deduce the second equation \eqref{eq:2n} with symbols  
$c_0 := \breve c_0 -  \kappa  c_0' + \gamma \psi_x + \gamma d_0'$ and 
$d_0 := \frac12 \psi_{xx} + \breve d_0 - \frac{\gamma}{2}\eta_x + \gamma 
c_{-2}$ in 
$ \Sigma \Gamma^0_{K,0,1}[r,N]$. Since \eqref{eq:2n} is a real equation we may assume that $c_0$ and $d_{0}$ satisfy \eqref{realetoreale} arguing as for the first equation.
\end{proof}

\begin{remark}\label{precisa}
The symbols $ a_0,\, c_0,\, d_0$  in \eqref{eq:1n}--\eqref{eq:2n} can be explicitly computed in terms of $V$ and $B$ (e.g. see \cite{BFP}). On the other hand the symbol  $b_{-1}(\eta;x,\xi)$ is expected to be of order $ -\infty$ but in \cite{BD} 
it has been estimated as a symbol of order $-1$ only. 
\end{remark}

We write \eqref{eq:1n}-\eqref{eq:2n} as the system 
\begin{align}
\pa_t \vect{\eta}{\psi}= &\begin{pmatrix} 0 & G(0) \\ -(g+ \kappa D^2) &   \gamma G(0) \pa_x^{-1} \end{pmatrix} \vect{\eta}{\psi} \label{ParaWW} \\
& + \Opbw {\begin{bmatrix} -B |\x|- \ii V_\gamma \xi  & 0  \\  - \kappa {\mathtt f} \x^{2}- B^2 | \x| &  B |\x |-\ii V_\gamma \x\end{bmatrix} + \begin{bmatrix} a_0  & b_{-1}   \\ c_0  &  d_0 \end{bmatrix}}\vect{\eta}{\psi}+ R(\eta,\psi) \vect{\eta}{\psi}. \notag
\end{align}

\paragraph{Wahl\'en coordinates.} 
We now transform  system \eqref{ParaWW}  in the Wahl\'en coordinates $(\eta, \zeta)$ defined in \eqref{Whalen}. 

\begin{lemma}
{\bf (Water-waves equations in Wahl\'en variables)}
Let  $N \in \N_0 $ and $ \varrho  \geq 0$. 
For any $K \in \N_0$ there exist $ s_0 ,r > 0 $ such that,  
if $ (\eta, \psi) \in B^K_{s_0} (I;r)  $ solves \eqref{eq:etapsi}
then $ (\eta, \zeta) = \cW^{-1} (\eta, \psi) $ defined in \eqref{Whalen} solves 
\begin{align}
\pa_t \vect{\eta}{ \zeta} =& \begin{pmatrix}\frac{ \gamma}{2} G(0) \pa_x^{-1} & G(0) \\ -(g+ \kappa D^2 + \frac{\gamma^2}{4}G(0) D^{-2} ) & \frac{ \gamma}{2} G(0) \pa_x^{-1}\end{pmatrix} \vect{\eta}{ \zeta} \notag  \\
&+  \Opbw {\begin{bmatrix} -B^{(1)}(\eta,\zeta;x)|\x|- \ii V^{(1)}(\eta,\zeta;x)  \xi  & 0  \\ -\kappa {\mathtt f}(\eta;x) \x^{2}- [B^{(1)}(\eta,\zeta;x)]^2| \x| &  B^{(1)}(\eta,\zeta;x)|\x |-\ii V^{(1)}(\eta,\zeta;x) \x\end{bmatrix}}\vect{\eta}{\zeta} \notag \\
&+ \Opbw{\begin{bmatrix}  a_0^{(1)}(\eta,\zeta;x, \xi)  &  b_{-1}(\eta;x, \xi)  \\  c_0^{(1)}(\eta,\zeta; x, \xi) &   d^{(1)}_0(\eta,\zeta;x, \xi)\end{bmatrix}}\vect{\eta}{\zeta}+ R(\eta,\zeta) \vect{\eta}{\zeta}  \label{WWW}
\end{align}
where
\begin{itemize}
\item  $ B^{(1)}(\eta, \zeta; x):=  B( \cW (\eta, \zeta);x)$  and  $V^{(1)}(\eta, \zeta; x):=V_\gamma(\cW (\eta, \zeta) ; x) $ are functions in $\Sigma {\cal F}^\R_{K, 0,1}[r,N]$, 
$ {\mathtt f}(\eta; x)$ is the function in  $ \Sigma \mF^{\R}_{K,0,2}[ r, N]$ defined in  Lemma \ref{laprimapara};
\item  $  a_0^{(1)}(\eta, \zeta; x, \xi)$, $ c_0^{(1)}(\eta, \zeta; x, \xi)$, 
$d_0^{(1)}(\eta, \zeta; x, \xi) $ are symbols in $  \Sigma \Gamma^0_{K,0,1}[r,N]$ 
satisfying \eqref{realetoreale}, and 
the symbol $   b_{-1}(\eta; x, \xi) $   in $ \Sigma\Gamma^{-1}_{K,0,1}[r,N]$
 is defined in  Lemma \ref{laprimapara};
\item  $R(\eta,\zeta) $ 
is a matrix of real 
smoothing operators in $ \Sigma \mR^{-\vr}_{K,0,1}[r,N]\otimes \mM_2 (\C) $. 
\end{itemize}
Moreover  system \eqref{WWW} is the \emph{Hamiltonian} system \eqref{HamWW2}.
\end{lemma}

\begin{proof}
By  \eqref{Whalen} and \eqref{ParaWW}   one has 
\begin{align}
\pa_t \vect{\eta}{\zeta} & =  \cW^{-1} \begin{pmatrix} 0 & G(0) \\ -(g+ \kappa D^2) & \gamma G(0) \pa_x^{-1} \end{pmatrix} \cW \vect{\eta}{\zeta} \label{ilprimo} \\
& \ \ + \cW^{-1}  \Opbw{\begin{bmatrix} -B|\x|- \ii V_\gamma  \xi  & 0  \\ -\kappa {\mathtt f}(\eta) \x^{2}- B^2| \x| &  B|\x |-\ii V_\gamma  \x\end{bmatrix}}\cW\vect{\eta}{\zeta} \label{ilprimo2} \\
& \ \ + \cW^{-1} \Opbw {\begin{bmatrix} a_0  & b_{-1}  \\ c_0&  d_0\end{bmatrix}}\cW\vect{\eta}{\zeta}+ R(\eta,\zeta) \vect{\eta}{\zeta}  \label{ilprimo3}
\end{align}
where $R(\eta,\zeta) $ is  a matrix of smoothing operators in $ \Sigma \mR^{-\vr}_{K,0,1}[r,N]\otimes \mM_2 (\C)$.
We now compute the above conjugated operators 
applying the transformation rule
\be \label{generic}
\cW^{-1}\begin{pmatrix} \mA& \mB\\ \mC& \mD\end{pmatrix}\cW= \begin{pmatrix} \mA+ \frac{\gamma}{2} \mB \pa_x^{-1} & \mB \\ \mC - \frac{\gamma}{2} \pa_x^{-1} \mA - \frac{\gamma^2}{4} \pa_x^{-1} \mB\pa_x^{-1} + \frac{\gamma}{2} \mD  \pa_x^{-1}& \mD - \frac{\gamma}{2} \pa_x^{-1}\mB\end{pmatrix}. 
\ee
The  operator in the right hand side of \eqref{ilprimo} is given in \eqref{eq:lin}.
Then by \eqref{generic} and Proposition \ref{teoremadicomposizione}, 
\be \label{conj2} 
\begin{aligned}
\eqref{ilprimo2} =   \Opbw{\begin{bmatrix} -B|\x|- \ii V_\gamma  \xi  & 0  \\ -\kappa {\mathtt f}(\eta) \x^{2}- B^2| \x| + \breve c_{0}&  B|\x |-\ii V_\gamma \x\end{bmatrix}}+ R(\eta, \zeta) \, ,
\end{aligned}
\ee
where the symbol 
$ \breve c_{0}:=  \frac{\gamma}{2} \frac{1}{\ii\xi} \#_\vr [B|\x|+ \ii V_\gamma  \xi ]+ \frac{\gamma}{2}[B|\x |-\ii V_\gamma  \x]\#_\vr \frac{1}{\ii\xi}$ 
belongs to $ \Sigma \Gamma_{K,0,1}^0[ r,N]$ and 
$R (\eta,\zeta) $  is matrix of smoothing operators in 
$ \Sigma \mR^{-\vr}_{K,0,1}[r,N] \otimes \mM_2 (\C) $. 

Finally, by \eqref{generic} and Proposition \ref{teoremadicomposizione}, we deduce that 
\be \label{conj3} 
\eqref{ilprimo3} = \Opbw {\begin{bmatrix}  a_0'  &  b_{-1}  \\  c'_0 &   d_0'\end{bmatrix}}+ R(\eta,\zeta) \, ,  
\ee
where $a_0', c'_0, d_0' $ are symbols in $ \Sigma \Gamma_{K,0,1}^0[ r,N]$ and 
$ R(\eta,\zeta) $ are smoothing operators in $  \Sigma \mR^{-\vr}_{K,0,1}[r,N] \otimes \mM_2 (\C) $. 
In conclusion, by \eqref{eq:lin}, \eqref{conj2}, \eqref{conj3}, we deduce that 
system  \eqref{ilprimo}-\eqref{ilprimo3} has the form  \eqref{WWW}     with
symbols  $ a_0^{(1)} := a_0' $, 
$ c_0^{(1)} := \breve c_0 + c_0' $ and $ d_0^{(1)} := d_0' $ evaluated at $ (\eta, \psi) = 
\cW (\eta, \zeta)$ which belong to $ \Sigma \Gamma^0_{K,0,1}[r,N]$.
Since the Whalen transformation  is a  real map, 
we may assume that $ a_0^{(1)}, c_0^{(1)}, d_0^{(1)}  $ satisfy \eqref{realetoreale} arguing as in the previous lemma.
\end{proof}

\begin{remark}
The first two matrices of paradifferential operators in \eqref{WWW} have the linear Hamiltonian structure  \eqref{HSrealpara}-\eqref{HSrealpara1}. We do not  claim that the third matrix of paradifferential operators  in \eqref{WWW} has the linear Hamiltonian structure  \eqref{HSrealpara}-\eqref{HSrealpara1}.
Nevertheless in Lemma \ref{LemCompl} we shall recover the complex 
linear Hamiltonian structure of $ J_c \opbw(A_0^{(2)})$, up to homogeneity $ N $, thanks to the abstract Lemma \ref{HS:repre}.
\end{remark} 

\paragraph{Complex coordinates.}

We now diagonalize the linear part of the system \eqref{WWW} 
at $(\eta, \zeta) = (0,0) $ introducing the complex variables
\be\label{Ugrande}
U:= \vect{ u}{{\bar u} }= \mM^{-1} \vect{ \eta}{ {\zeta} } \, , \quad
\mM^{-1} : \dot H^{s+ \frac14} (\T,\R)
 \times \dot H^{s- \frac14} (\T,\R) \to \dot H^s_\R (\T,\C^2)  \, , \ \forall s \in \R \, , 
\ee
where $ \mM $ is the matrix of Fourier multipliers defined in \eqref{defMM-1}.

\begin{lemma}\label{LemCompl}
{\bf (Hamiltonian formulation of the water waves in complex coordinates)}
Let  $N \in \N_0 $ and $ \varrho  \geq 0$. 
For any $K \in \N_0 $  there exist $ s_0 ,r > 0 $ such that, 
if $ (\eta, \psi) \in B^K_{s_0} (I;r) $ is a solution of \eqref{WWW} then $ U$ defined in \eqref{Ugrande} solves 
\be \label{complexo}
\begin{aligned}
\pa_tU & = J_c \Opbw{A_\frac32(U;x) \omega(\xi)}U+ \frac{\gamma}{2} G(0)\pa_x^{-1}   U \\
& \quad +  J_c\Opbw{A_1(U;x,\xi)  +  A_{\frac12} (U;x,\xi)+A^{(2)}_0(U;x,\xi)} U+R(U)U 
\end{aligned}
\ee 
where $ J_c $ is the Poisson tensor  defined in \eqref{JcH.intro} and 
\begin{itemize}
\item 
$ \omega(\xi) \in \widetilde\Gamma^{\frac32}_0$ is the symbol in 
\eqref{omegaxi};
 \item $A_\frac32 (U;x) \in  \Sigma \cF^{\R}_{K,0,0}[r,N] \otimes \cM_2(\C)$ is the matrix of  real functions 
 \be\label{Adue}
 A_\frac32(U;x):=  \begin{pmatrix} f(U;x) &  1+f(U;x)\\1+f(U;x)  &f(U;x) \end{pmatrix},
 \ee
 where $f(U;x):= \frac12 \mathtt f(\cM U;x)$ belongs to $ \Sigma \mF^{\R}_{K,0,2}[r, N]$. Note that 
$J_c \opbw(A_\frac32  \omega(\xi) )$ is linearly Hamiltonian according to 
Definition \ref{def:LH}; 
 \item 
 $A_1(U;x,\xi) \in \Sigma \Gamma^{1}_{K,0,1}[r,N]\otimes \mM_2 (\C)$   
 is the  matrix of symbols
\be\label{A1}
A_1(U;x,\xi):= \begin{pmatrix}   \ii B^{(2)}(U;x) |\xi|  & -  V^{(2)}(U;x) \xi  \\   V^{(2)}(U;x) \xi &-\ii   B^{(2)}(U;x) |\xi|  \end{pmatrix}
 \ee
 where $  B^{(2)}(U;x):= B^{(1)}(\cM U; x)$ and   $V^{(2)}(U;x):=V^{(1)}(\cM U; x)$ are real functions in $ \Sigma \mF^{\R}_{K,0,1}[r, N]$. 
 Note that 
 $J_c \opbw(A_1)$ is linearly Hamiltonian;
 \item $A_{\frac12}(U;x,\xi) \in  \Sigma \Gamma^{\frac12}_{K,0,2}[r,N] \otimes \cM_2(\C)$
is the symmetric matrix of symbols 
 \be \label{A12}
 A_{\frac12}(U;x,\xi)= \frac12  \begin{pmatrix}  1 &   1  \\  1 & 1\end{pmatrix} [B^{(2)}(U;x)]^2 |\xi|M^2(\xi)  
  \ee
  where $ M(\xi)\in \widetilde \Gamma^{-\frac{1}{4}}_0$ is the symbol of the Fourier multiplier $M(D)$ in \eqref{defMD}.
Note that  $J_c \opbw(A_\frac12)$ is linearly Hamiltonian;
  \item 
 $ A^{(2)}_0(U;x, \xi)$ is a matrix of symbols in $ \Sigma \Gamma^{0}_{K,0, 1}[ r,N]\otimes \mM_2 (\C)$ and the operator $J_c \opbw(A^{(2)}_0)$ is  linearly Hamiltonian up to homogeneity $N$ according to Definition \ref{def:LHN};
 \item  $R(U) $ is a real-to-real matrix of smoothing operators in $ \Sigma \mR^{-\vr}_{K,0,1}[r,N]\otimes \mM_2 (\C) $.
\end{itemize}
  Moreover system \eqref{complexo} is \emph{Hamiltonian in the complex sense}, i.e. has the form \eqref{complexHS}. 
\end{lemma}

\begin{proof}
We begin by noting that the operator $\cM$ in the change of coordinates \eqref{Ugrande} has the form (cfr. \eqref{defMM-1})
$$
\cM = 
\breve \cM \circ \cC  , \qquad  \breve \cM := \begin{pmatrix}
M(D) & 0 \\
0 & M(D)^{-1}
\end{pmatrix}
$$
where $ M(D) $ is  the Fourier multiplier in \eqref{defMD} and $\cC$ the matrix in \eqref{cc}.
The operator $\breve \cM$  is symplectic whereas under the change of variables $\cC$ 
a real 
 Hamiltonian system in standard  Darboux form \eqref{HamWW2} assumes the standard complex form \eqref{complexHS}, see the paragraph at page \pageref{realHam.com}.
 Therefore $U$ solves a system which is Hamiltonian in the complex sense.

Since $ (\eta, \zeta) $ solves \eqref{WWW}, the complex variable $ U$ in \eqref{Ugrande} solves 
\begin{align}
\pa_t U & = \mM^{-1} \begin{pmatrix}\frac{ \gamma}{2} G(0) \pa_x^{-1} & G(0) \\ -(g+ \kappa D^2 + \frac{\gamma^2}{4} G(0) D^{-2} ) & \frac{ \gamma}{2} G(0) \pa_x^{-1}\end{pmatrix}\mM U
\label{coniucomplessa}  \\
&\ \ +  \mM^{-1} 
\Opbw{\begin{bmatrix} -B^{(1)}|\x|- \ii V^{(1)}\xi  & 0  \\ -\kappa {\mathtt f}(\eta) \x^{2}- [B^{(1)}]^2| \x| &  B^{(1)}|\x |-\ii V^{(1)} \x\end{bmatrix} } 
\mM U \label{coniucomplessa1} \\
&\ \ + \mM^{-1} 
\Opbw{\begin{bmatrix}  a^{(1)}_0  &  b_{-1}  \\  c^{(1)}_0&   d^{(1)}_0\end{bmatrix}} 
\mM U + R( U)U  \label{coniucomplessa2}
\end{align}
where  $R(U) $ is a real-to-real matrix of smoothing operators in $ \Sigma \mR^{-\vr}_{K,0,1}[r,N]\otimes \mM_2 (\C)$.

The operator in the right hand side of \eqref{coniucomplessa} is computed in 
\eqref{diaglin}-\eqref{relation}.    
In order to compute the  conjugated operators  in \eqref{coniucomplessa1}-\eqref{coniucomplessa2}, we apply the following transformation rule,   
where we denote  by $M:= M(D)$ the Fourier multiplier in \eqref{defMD} (which 
satisfies $M(D) = \bar M(D)$),  
\begin{align} \label{generic2} 
&\footnotesize   \mM^{-1} \begin{pmatrix} \mA_1& \mA_2 \\ \mA_3 & \mA_4\end{pmatrix} \mM \\
& \footnotesize  \stackrel{\eqref{defMM-1}} = 
\frac12 \begin{pmatrix} M^{-1}\mA_1 M + M \mA_4 M^{-1} + \ii M \mA_3 M - \ii M^{-1}  \mA_2 M^{-1} & M^{-1} \mA_1 M  - M \mA_4 M^{-1} + \ii M  \mA_3 M +\ii M^{-1} \mA_2 M^{-1} \\ 
M^{-1} \mA_1 M  - M \mA_4 M^{-1} - \ii M  \mA_3 M - \ii M^{-1}  \mA_2 M^{-1} &  M^{-1}  \mA_1 M  + M  \mA_4 M^{-1}  - \ii M  \mA_3 M +  \ii M^{-1}  \mA_2 M^{-1} \notag \end{pmatrix}.
\end{align}
Using \eqref{generic2} and Proposition \ref{teoremadicomposizione} we get  that
\be \label{nonsos}
\eqref{coniucomplessa1}  = \frac12 \Opbw {\begin{bmatrix} a_1& b_1 \\ \bar{b}_1^{\vee}&  \bar {a}_1^{\vee}\end{bmatrix}} + R(U) 
\ee
where   $ a_1, b_1 $ are the symbols 
\begin{align} \footnotesize
a_1 & :=  \footnotesize M^{-1} (\xi)  \#_\vr \big( -B^{(1)}|\x|- \ii V^{(1)}  \xi \big) \#_\vr M (\xi)   + M(\xi)  \#_\vr  \big( B^{(1)}|\x |-\ii V^{(1)} \x \big)\#_\vr M^{-1} (\xi)  \notag \\
& \quad + \ii M (\xi)  \#_\vr \big( 
-\kappa {\mathtt f}(\eta) \x^{2}- [B^{(1)}]^2| \x| \big) \#_\vr M  (\xi) 
\notag \\
b_1& :=  \footnotesize M^{-1}(\xi)   \#_\vr \big( 
-B^{(1)}|\x|- \ii V^{(1)}  \xi \big) \#_\vr M (\xi)   -
 M(\xi)  \#_\vr  \big( B^{(1)}|\x |-\ii V^{(1)} \x \big) \#_\vr M^{-1}(\xi)  \notag \\
& \quad + \ii M(\xi)  \#_\vr 
\big(-\kappa {\mathtt f}(\eta) \x^{2}- [B^{(1)}]^2| \x| \big)\#_\vr M(\xi)  \notag 
\end{align}
and  $R(U) $  is a real-to-real matrix of smoothing operators in 
$ \Sigma \mR^{-\vr}_{K, 0, 1}[r,N] \otimes \mM_2 (\C) $.  
Noting that 
\be\label{omega.M.asy}
\omega(\xi)- \sqrt{\kappa} |\xi|^{\frac{3}{2}} \in \widetilde \Gamma_0^{-\frac12} \, , \quad 
 M(\xi)-\kappa^{-\frac{1}{4}} |\xi|^{-\frac{1}{4}} \in  \wt \Gamma_0^{-\frac{9}{4}} \, ,
 \quad  
M^{-1}(\xi)- \kappa^{\frac{1}{4}} |\xi|^{\frac{1}{4}} \in \wt\Gamma_0^{-\frac{7}{4}} \, , 
\ee
{so that $ \kappa \xi^2 M^2 (\xi) - \omega (\xi) \in \wt\Gamma_0^{-\frac{1}{2}} $,} 
we deduce, using also the remarks after Definition \ref{def:as.ex},  that
\begin{align} 
a_1 
& =   - 2 \ii V^{(1)}  \xi- \ii {\mathtt f}(\eta) \omega(\xi) - \ii  [B^{(1)}]^2| \x|M^2(\xi) + a_0' \quad \text{with} \quad a_0' \in  \Sigma \Gamma^{0}_{K,0,1}[r,N] \, , 
\label{contacci1} \\
b_1 
& =   -2 B^{(1)} | \xi | - \ii  {\mathtt f}(\eta) \omega(\x) - \ii [B^{(1)}]^2| \x|M^2(\xi) + b_0' \quad \text{with} \quad b_0' \in  \Sigma \Gamma^{0}_{K,0,1}[r,N] \, .
\label{contacci}
\end{align}
Finally, noting that 
$ M^{-1}(\xi) \#_\vr  b_{-1}\#_\vr M^{-1} (\xi)
$ belongs to $ \Sigma \Gamma^{-\frac12}_{K,0,1}[r,N]$, 
we deduce that   
\be \label{ultimocomplesso}
\eqref{coniucomplessa2} = \opbw(A'_0) + R(U) 
\ee 
where $A'_0 $ is a real-to-real matrix of symbols in $  \Sigma \Gamma^{0}_{K,0,1}[r,N]\otimes \mM_2 (\C)$ and $R(U)$  is a real-to-real matrix of smoothing operators in
$ \Sigma \mR^{-\vr}_{K, 0, 1}[r,N] \otimes \mM_2 (\C) $. 

In conclusion, by  \eqref{diaglin}-\eqref{relation}, \eqref{nonsos}, \eqref{contacci1}, \eqref{contacci}, \eqref{ultimocomplesso}, 
computing the symbols at $ (\eta, \zeta) = \cM U $, 
we deduce that system 
\eqref{coniucomplessa}-\eqref{coniucomplessa2} has the form 
 \eqref{complexo}.  
Note that the matrices of paradifferential operators $J_c \opbw ( A_\frac32 \omega (\xi) ) $, 
$ J_c \opbw(A_1) $, $ J_c \opbw(A_{\frac12}) $ in \eqref{Adue}, \eqref{A1}, \eqref{A12} are linearly Hamiltonian according to \eqref{LHS-c}, whereas $ J_c \opbw( A^{(2)}_0) $ might not be.  Thanks to Lemma \ref{HS:repre} we 
 replace each homogeneous component of 
 $ A_\frac32 \omega (\xi)  + A_1 +  A_{\frac12} + 
  A_0^{(2)}$ with its symmetrized version, by adding another smoothing operator. Since  the symbols with positive orders are unchanged we obtain a new  operator $ J_c \opbw( A^{(2)}_0) $ (that we denote in the same way) which 
  is linearly Hamiltonian up to homogeneity $N$.  
\end{proof}

\section{Block-diagonalization and reduction to constant coefficients}\label{sec:redu}

In this section we perform several transformations  
in order to
symmetrize and reduce system \eqref{complexo} to constant coefficients up to smoothing remainders. 
In particular we will prove the following: 

\begin{proposition}\label{teoredu1}
{\bf (Reduction to constant coefficients up to smoothing operators)}
Let $N \in \N_0$ and  $\vr> 3(N+1)$. 
Then there exists $\underline K':= \underline K'(\vr)>0$ such that for any $ K\geq \underline K'$ there are $s_0 >0$, $ r > 0 $ such that 
for any solution $U \in B_{s_0,\R}^K(I;r) $ of  \eqref{complexo}, there exists a 
real-to-real invertible matrix of spectrally localized  maps $ \bB(U;t)$ such that $ \bB(U;t)-\uno \in \Sigma \mS^{\frac32 (N+1)}_{K,\underline K'-1,1}[r,N]\otimes \cM_2(\C)$ and  the following holds true: 
\\[1mm]
$(i)$ {\bf Boundedness:} $\bB(U;t)$ and $\bB(U;t)^{-1}$  
 are  non--homogeneous maps in $  \mS_{K,\underline K'-1,0}^0 [r]\otimes \mM_2 (\C) $, cfr. \eqref{piovespect}  with $m = N = 0$. \\
$(ii)$ {\bf Linear symplecticity:} The map  $ \bB(U;t) $ is linearly symplectic up to homogeneity $ N $, according to Definition \ref{LSUTHN}.
\\[1mm]
$(iii)$ {\bf Conjugation:} If $U$ solves \eqref{complexo} then  $W:= \bB(U;t)U$ solves 
\be \label{teo61}
\pa_tW = \vOpbw{\ii \tm_{\frac32}(U;t,\xi)}W 
 + R(U;t)W 
\ee
(recall notation \eqref{vecop}) where 
\be\label{m32}
\tm_{\frac32}(U;t,\xi):= -  \left[(1+\zeta(U))\omega(\xi) + \frac{\gamma}{2}\frac{\tG(\x)}{\x} + \mathtt V(U;t) \xi + \tb_{\frac12}(U;t) |\x|^{\frac12}+  \tb_0(U;t,\xi)\right]
\ee
and
\begin{itemize}
\item $\omega(\xi) \in \widetilde\Gamma^{\frac32}_0$ is the Fourier multiplier  defined in \eqref{omegaxi};
\item $\zeta(U)$ is a real function in $\Sigma \mF^{\R}_{K,0,2}[r,N]$ independent of $ x$;
\item $\mathtt V(U;t)$ is a real function in $ \Sigma\mF^{\R}_{K,1,2}[r,N]$ independent of $ x$;
\item  $\tb_{\frac12}(U;t)$ is a real function in $\Sigma \mF^{\R}_{K,2,2}[r,N]$ independent of $ x$; 
\item $\tb_0(U;t, \xi)$ is a symbol in $ \Sigma \Gamma^0_{K,\underline  K',2}[r,N]$ independent of $x$ and its imaginary part $\Im \tb_0(U;t, \xi) $ is in $ \Gamma^0_{K,\underline  K',N+1}[r]$;
\item $R(U;t)$ is a real-to-real matrix of smoothing operators in $ \Sigma \mR^{-\vr+3(N+1)}_{K,\underline  K',1}[r,N] \otimes \mM_2 (\C)$.
\end{itemize} 
\end{proposition}

\begin{remark}\label{rem:ee0}
The symbol $\tm_{\frac32}(U;t,\xi)$ in \eqref{m32} is real valued  except for the term $\tb_0(U;t,\xi)$  whose imaginary part has order 0 and homogeneity at least  $N+1$. Hence system \eqref{teo61} fulfills energy estimates in $\dot H^s(\T, \C^2)$, of the type 
 \eqref{eeN} with $N = 0$.
\end{remark}

\begin{remark}\label{rem:pro61}
One can choose $\underline K' (\varrho) \geq 3\varrho - 8(N+1)+1$.
\end{remark}

The rest of Section \ref{sec:redu} is devoted to 
the proof of Proposition \ref{teoredu1}. 
We shall use constantly the identities
\be\label{omega.der}
\omega(\xi) = \sqrt{\kappa} |\xi|^\frac32 + \wt\Gamma^{-\frac12}_0 
, 
\quad
\omega'(\xi) = \frac32 \sqrt{\kappa} |\xi|^\frac12 \textup{ sign }\xi  + \wt\Gamma^{-\frac32}_0  \ .
\ee

\subsection{A complex good unknown of Alinhac}\label{secAl1}

In this section we introduce 
a complex version of  the good unknown of Alinhac,
whose goal is to  diagonalize the  matrix of paradifferential 
operators of order $1$ in \eqref{complexo}  and remove the paradifferential  operators of order $\frac12$.
The complex good unknown that we use coincides at principal order with   
$\cM^{-1} \cG_A \cM$ where $\cG_A$ is the classical good unknown of Alinhac in \eqref{simgud} and $\cM$ is the change of variables in \eqref{defMM-1}.

\begin{lemma}\label{goodlem}
Let $N \in \N_0$ and  $\vr>0$.  Then for any $ K \in \N$ there 
are
 $s_0 >0 $, $ r > 0 $ such that for  any  solution $U \in B_{s_0,\R}^K(I;r)$ of  \eqref{complexo}, there exists  a 
real-to-real invertible matrix of spectrally localized maps  $ \mG(U)$ satisfying
$ \mG(U)- \uno  \in \Sigma \mS^0_{K,0,1}[r,N] \otimes \mM_2(\C) $ and the following holds true:
\\[1mm]
$(i)$ {\bf Boundedness:}  $ \mG(U)$  and its inverse 
 are  non--homogeneous maps in $  \mS_{K,0,0}^0 [r]\otimes \mM_2 (\C) $.
\\[1mm]
$(ii)$ {\bf Linear symplecticity:} The map  $ \mG(U) $ is linearly symplectic according to Definition \ref{LS};
\\[1mm]
$(iii)$ {\bf Conjugation:} 
If $U$ solves \eqref{complexo} then  $V_0:= \mG(U)U$ solves 
\be \label{sisV0}
\begin{aligned}
\pa_tV_0 & = J_c \Opbw{A_\frac32(U;x) \omega(\xi)  }V_0+  \frac{\gamma}{2} G(0) \pa_x^{-1} V_0 \\
& \ \ +  \vOpbw{ -\ii V^{(2)}(U;x) \xi}V_0 + J_c \Opbw{  A^{(3)}_0 (U;t,x,\xi)} V_0 + R(U)V_0
\end{aligned}
\ee 
where
\begin{itemize}
 \item the  matrix of real functions $ A_\frac32(U;x) \in \Sigma \mF^{\R}_{K,0,0}[r, N] \otimes \mM_2(\C)$ and  the real function $V^{(2)}(U;x) \in \Sigma \mF^{\R}_{K,0,1}[r,N] $ are defined in Lemma \ref{LemCompl}; 
 \item 
 $ \omega(\xi) \in  \widetilde\Gamma^{\frac32}_0$ is  the symbol defined in \eqref{omegaxi}; 
 \item The matrix of symbols  $ A^{(3)}_0(U; t, x, \xi)$    belongs to $ \Sigma \Gamma^0_{K,1,1}[r,N]
 \otimes \mM_2 (\C) $ and  the operator  $ J_c \opbw(A^{(3)}_0)$ is linearly Hamiltonian up to homogeneity $N$ according to Definition \ref{def:LHN}; 
 \item  $R(U)$ is a real-to-real matrix of smoothing operators in $ \Sigma \mR^{-\vr}_{K,0,1}[r,N]
 \otimes \mM_2 (\C) $. 
 \end{itemize}
\end{lemma}

\begin{proof}
We define $\cG(U)$ to be the real-to-real map 
\be \label{calgi}
\mG(U):= \uno - \frac{\ii}{2}  \begin{pmatrix} 1&1 \\ -1&-1\end{pmatrix} 
\Opbw{B^{(2)}(U;x) M^2(\xi)}
\ee
where $B^{(2)}(U;x) $ is the  function  in $ \Sigma \mF^{\R}_{K,0,1}[r, N] $ 
defined in  \eqref{A1} and  
$ M(\xi)\in \wt \Gamma_0^{-\frac14}$ is the symbol of 
the Fourier multiplier $M(D)$ defined in \eqref{defMD}.
Its inverse and transpose are given by
\be
\begin{aligned}
 \label{calgit}
& \mG(U)^{-1}=  \uno + \frac{\ii}{2}  \begin{pmatrix} 1&1 \\ -1&-1\end{pmatrix} \Opbw{B^{(2)}(U;x) M^2(\xi)} \, ,
\\
&  \mG(U)^{\top}=  \uno - \frac{\ii}{2}  \begin{pmatrix} 1&-1 \\ 1&-1\end{pmatrix} \Opbw{B^{(2)}(U;x)M^2(\xi)} \, . 
\end{aligned}
\ee 
By the fourth bullet  below  Definition \ref{smoothoperatormaps} the 
matrices of paradifferential operators
$\mG(U)^{\pm 1} - \uno  $   belong 
to $ \Sigma \mS_{K,0,1}^0 [r,N] \otimes \mM_2(\C) $ and
item $ (i)$ follows. Also $(ii)$ follows by a direct computation using the explicit expressions 
in \eqref{calgi} and \eqref{calgit}. 

Let us prove item $ (iii)$. 
Since $ U $ solves  \eqref{complexo} the variable $ V_0 := \mG(U)U$ solves 
\begin{align} 
\notag
\pa_tV_0  & =   \mG(U)\left[  J_c 
\Opbw{ A_{\frac32}\omega(\xi) + A_1 +A_{\frac12}  +A_0^{(2)}} + \frac{\gamma}{2} G(0) \pa_x^{-1} \right]  \mG(U)^{-1}V_0
+  (\pa_t\mG(U)) \mG(U)^{-1}V_0\\
 \label{complexo2}
& \quad +   \mG(U)R(U)\mG(U)^{-1} V_0 \, .
\end{align}
We now expand each of the above operators. 
By \eqref{calgi}, the form of $ J_c $ in \eqref{JcH.intro},  
 \eqref{calgit},   
the symbolic calculus Proposition \ref{teoremadicomposizione},
writing 
$J_c A_\frac32 \omega(\xi) = - \im \begin{psmallmatrix} 1 & 0 \\ 0 & -1 \end{psmallmatrix} \omega(\xi) - \im \begin{psmallmatrix} 1 & 1 \\ -1 & -1 \end{psmallmatrix} f(U) \omega(\xi)$ (see \eqref{Adue}), 
 and since  
 $ \begin{psmallmatrix} 1 & 1 \\
 -1 & - 1 \end{psmallmatrix}^2 = 0 $, after a lengthy computation we obtain that 
the first term in \eqref{complexo2} is   
\begin{align}
\notag
& \mG(U)   J_c   \Opbw{ A_\frac32 \, \omega(\xi)}  \mG^{-1}(U)   =    J_c \Opbw { A_\frac32 \, \omega(\xi)} \\
\notag
& \quad   + \frac12\opbw\left( \begin{bmatrix}  \omega(\xi)\#_\vr B^{(2)}M^2(\xi)-B^{(2)}M^2(\xi)\#_\vr \omega(\xi)& \omega(\xi)\#_\vr B^{(2)}M^2(\xi) + B^{(2)}M^2 (\xi) \#_\vr \omega(\xi)\\\omega(\xi)\#_\vr B^{(2)}M^2 (\xi) + B^{(2)}M^2 (\xi) \#_\vr \omega(\xi)&\omega(\xi)\#_\vr B^{(2)}M^2(\xi) -B^{(2)}M^2(\xi)\#_\vr \omega(\xi)\end{bmatrix}\right) \\
\notag
&  \quad - \frac{\ii}{2} \begin{pmatrix} 1&1\\-1&-1\end{pmatrix} \Opbw{ B^{(2)}M^2 (\xi)\#_\vr \omega(\xi)\#_\vr B^{(2)}M^2(\xi)} +R(U)   \\
\notag
&   =  
J_c \Opbw { A_\frac32 \, \omega(\xi)  } 
 + \opbw\left( \begin{bmatrix} 0&  \, B^{(2)}|\xi|\\ \ \, B^{(2)}|\xi|&0\end{bmatrix}\right)\\
& \quad  - \frac{\ii}{2} \begin{pmatrix} 1&1\\-1&-1\end{pmatrix} \Opbw{[B^{(2)}]^2 |\xi|M^2(\xi)}+  J_c\opbw(A_0) + 
 R(U) \, , \label{good1}
\end{align}
where $ A_0 $ is a matrix of symbols in $ \Sigma \Gamma^0_{K,0,1}[ r,N]\otimes \mM_2(\C)$  and 
 $R(U) $ is a matrix of smoothing operators in $ \Sigma \mR^{-\vr}_{K,0,1}[r, N]\otimes \mM_2(\C)$. In the last passage to get \eqref{good1} we also used that 
$ M^2 (\xi)\omega(\xi)= \tG(\xi) = |\xi|+ \wt \Gamma_0^{-\vr} $, for any $ \vr \geq0 $,
cfr. \eqref{omeghino0}, \eqref{Gxi}.  

Next using the explicit form \eqref{A1} of $ A_1 $ we  get, arguing similarly,  
\begin{align}
\notag
\mG(U) & J_c\opbw( A_1 ) \mG(U)^{-1} \\
\label{good2}
& =  J_c\opbw( A_1 ) + \ii \begin{pmatrix} 1& 1 \\ -1& -1\end{pmatrix} \Opbw{[B^{(2)}]^2 |\xi| M^2(\xi) } +J_c  \Opbw{A_0'} + R(U)
\end{align}
where $ A_0'$ is  a matrix of symbols in 
$ \Sigma \Gamma^0_{K,0,1}[r,N] \otimes \mM_2 (\C)$ 
and $ R(U)$  is  a matrix of smoothing operators in 
$ \Sigma \mR^{-\vr}_{K,0,1}[r,N] \otimes \mM_2(\C) $. 

Moreover, using the form \eqref{A12} of $ A_\frac12 $   we get
\be \label{good3}
 \mG(U)J_c\Opbw{ A_{\frac12}  }\mG(U)^{-1} = 
J_c  \Opbw{ A_{\frac12}  } = 
 - \frac{\ii}{2}\begin{pmatrix} 1&1\\ -1&-1\end{pmatrix} \Opbw{[B^{(2)}]^2|\xi| M^2(\xi)} \, .
\ee
Then, since $ A_0^{(2)} $ is a matrix of symbols of order zero, by  Proposition \ref{teoremadicomposizione} we have 
\be\label{good11}
  \mG(U) \left[ J_c \Opbw{A_0^{(2)}} + \frac{\gamma}{2} G(0) \pa_x^{-1} \right] \, \mG(U)^{-1} =
 \frac{\gamma}{2} G(0) \pa_x^{-1} + J_c \opbw(A_0'') + R(U) \, , 
  \ee
for a matrix of symbols $ A_0''  $ in 
$  \Sigma \Gamma^0_{K,0,1}[ r,N]\otimes \mM_2(\C)$ and smoothing operators $ R (U) $ in $  \Sigma \mR^{-\vr}_{K,0,1}[r, N]\otimes \mM_2(\C)$. 
Next by \eqref{calgi}-\eqref{calgit}
\be \label{good4}
 (\pa_t\mG(U)) \mG(U)^{-1} = 
J_c \Opbw{A_{-\frac12}} + R(U) \, , \quad  A_{-\frac12}:= \frac12   \begin{pmatrix} 1&1\\ 1&1\end{pmatrix}\opbw\left( \pa_t B^{(2)} M^2(\xi)\right) 
\ee
where, in view of the last bullets at the end of Section \ref{sec:para} and  Proposition \ref{composizioniTOTALI}-$(iv)$, $A_{-\frac12}$ is a matrix of symbols in $\Sigma \Gamma^{-\frac12}_{K,1,1}[r,N]\otimes \cM_2(\C)$ 
and $ R(U)$ is a matrix of smoothing operator in $ \Sigma \mR^{-\vr}_{K,0,1}[r,N]\otimes \mM_2(\C)$.

Finally by Proposition \ref{composizioniTOTALIs} we have that
  $\mG(U) R(U) \mG(U)^{-1}$ 
  is a matrix of smoothing operators in 
  $ \Sigma \mR^{-\vr}_{K,0,1}[r,N] \otimes \mM_2 (\C)  $, being $\mG(U)$ a spectrally localized map.

Note that, using  the expression \eqref{A1},  the sum of the terms of order 1 which are in 
\eqref{good1} and \eqref{good2} is 
\be\label{good6}
J_c \Opbw{A_1} + \opbw\left( \begin{bmatrix} 0&  \, B^{(2)}|\xi|\\ \ \, B^{(2)}|\xi|&0\end{bmatrix}\right) = \vOpbw{- \im V^{(2)}(U;x) \xi} \, .
\ee
Note also that the sum of terms which are of order $\frac12$ in 
\eqref{good1}, \eqref{good2} and \eqref{good3} equals zero.

In conclusion, by \eqref{good1},  \eqref{good2} \eqref{good3}, 
\eqref{good11}, \eqref{good4} and \eqref{good6} we obtain that system \eqref{complexo2} has the form  \eqref{sisV0}
with  $ A^{(3)}_0:= A_0 + A_0'+ A_0''  + A_{-\frac12}$   in 
$ \Sigma \Gamma^0_{K,1,1}[r,N] \otimes \mM_2 (\C) $.
Note that the paradifferential operators of positive order in \eqref{sisV0} are  linearly Hamiltonian, whereas $ J_c \opbw{(A_0^{(3)})}$ might not be. 
However   the operator in the first line of  \eqref{complexo2} 
is a spectrally localized map which is linearly Hamiltonian up to homogeneity $N$ by Lemma \ref{coniugazionehamiltoniana},  with a paradifferential structure as in 
\eqref{separato}.
Then by Lemma \ref{spezzamento} 
we 
 replace each homogeneous component of 
 $ A_\frac32 \omega (\xi)  + \begin{bsmallmatrix} 0 & - V^{(2)} \xi \\ V^{(2)} \xi & 0 \end{bsmallmatrix} + 
  A_0^{(3)}$ with its symmetrized version, by adding another smoothing operator. Since  the symbols with positive orders are unchanged we obtain a new  operator $ J_c \opbw( A^{(3)}_0) $ (that we denote in the same way) which 
  is linearly Hamiltonian up to homogeneity $N$.   
\end{proof}

 \subsection{Block-Diagonalization at highest order}
 
In this section we  diagonalize the operator 
$J_c \opbw\big(A_\frac32(U; x) \omega (\xi)\big)$ in \eqref{sisV0} 
where  $A_\frac32$ is the matrix defined in \eqref{Adue}. 
 Note that the eigenvalues of the matrix 
\begin{equation}
\label{diag1}
J_c A_\frac32(U;x)=  \ii \begin{bmatrix}
-(1 +f( U;x)) &  -f( U; x) \\
f( U; x) & 1+f( U; x)
\end{bmatrix} \, , 
\end{equation}
where $f( U;x) $ is the real function defined in \eqref{Adue}, are 
 $\pm \im \lambda( U; x)$ with 
\begin{equation}
\label{diag2}
\lambda( U; x) := \sqrt{(1+f( U;x))^2 -f( U;x)^2}  = \sqrt{1+ 2f( U;x)} \, .  
\end{equation}
Since the function $f( U;x) $ is in $ \Sigma \mF^\R_{K,0,2}[r,N]$, 
for any $  U \in B_{s_0,\R}^K(I;r)$ with $ r > 0 $ small enough it results that 
$ |f( U;x)|\leq \frac{1}{4}$,  
the function $\lambda( U;x)-1$  
belongs to $ \Sigma \mF^\R_{K,0,2}[r,N] $  and 
$$
\lambda( U;x) \geq \tfrac{\sqrt{2}}{2}  > 0 \, , \quad \forall x \in \T \, . 
$$
Actually the function $\lambda( U;x) $  is real valued also for not small $ U $, see
Remark \ref{rem:elliptic}. 

 A matrix which diagonalizes \eqref{diag1} is 
 \begin{equation} \label{diag4}
 \begin{aligned}
&  \qquad \qquad \quad F (U; x)  : = 
 \begin{pmatrix}
 h (U; x )& g(U; x ) \\
 g(U; x ) & h(U; x ) 
 \end{pmatrix} \\
&   
h := { \frac{1+f + \lambda}{\sqrt{(1+f + \lambda)^2 - f^2}} } \, , \quad 
g:=  {  \frac{-f}{\sqrt{(1+f + \lambda)^2 - f^2}} } \, .
\end{aligned}
 \end{equation}
 Note that $ F(U; x) $ is well defined since 
 $ (1+f+ \lambda)^2 - f^2  = 
(1+2f + \lambda)(1+ \lambda)  \geq 
\frac12  $. 
Moreover the matrix $ F(U; x) $ is symplectic, i.e. 
\be\label{detF=1}
\det F = h^2 - g^2 = 1 \, .
\ee
 The inverse of $F(U;x)$ is the symplectic and symmetric matrix 
  \begin{equation}
 \label{diag5}
 F(U; x )^{-1} : = 
 \begin{pmatrix}
h (U; x )& - g(U; x ) \\
-g(U; x ) & h(U; x )
 \end{pmatrix} \, . 
 \end{equation}
Moreover $F(U;x) - \id $ is a matrix of real functions in $ \Sigma\cF^\R_{K,0,1}[r,N]\otimes \mM_2(\C)$ and
\begin{equation}
\label{FAF}
F( U;x)^{-1} \,
J_c  A_\frac32( U;x)  F( U;x) 
 = 
 \begin{bmatrix} -\ii \lambda(U; x)&0\\0 &\ii \lambda(U; x)\end{bmatrix} \, ,
\end{equation}
which amounts to $(h^2+g^2) \, (1+f) + 2hg f  = \lambda$ and 
$2 h g (1+f) 
 + (h^2 + g^2)   f   = 0$.
 \begin{lemma}
 \label{diag}
 Let $N \in \N_0$ and  $\vr>0$. 
Then for any $ K \in \N$ there are $s_0 >0 $, $ r > 0 $ such that for any  solution 
$U \in B_{s_0,\R}^K(I;r)$ of  \eqref{complexo}, there exists a  real-to-real invertible matrix of spectrally localized maps   $ \Psi_1(U)$ satisfying
 $ \Psi_1(U)- \uno  \in \Sigma \mS^0_{K,0,1}[r,N] \otimes \mM_2 (\C) $ and the following holds true: 
\\[1mm]
$(i)$ {\bf Boundedness:}  
The operator $ \Psi_1(U)$ and its inverse 
 are  non--homogeneous maps in $  \mS_{K,0,0}^0 [r]\otimes \mM_2 (\C)$;
\\[1mm]
$(ii)$ {\bf Linear symplecticity:} The map  $ \Psi_1(U) $ is linearly symplectic according to Definition \ref{LS};
\\[1mm]
$(iii)$ {\bf Conjugation:} If $V_0$ solves \eqref{sisV0} then  $V_1:= \Psi_1(U)V_0$ solves the system
\begin{equation}\label{NuovoParaprod}
\begin{aligned}
\pa_t V_1& = 
\vOpbw{ -\ii \lambda(U;x) \omega(\xi)   - \ii V^{(2)}(U;x)\xi } V_1+\frac{\gamma}{2} G(0) \pa_x^{-1} V_1\\
&\quad  + J_c\Opbw{A^{(4)}_0(U;t, x, \xi)}V_1 + R(U;t)  V_1
\end{aligned}
\end{equation}
where 
\begin{itemize}
\item the function $ \lambda(U;x) \in \Sigma \mF^{\R}_{K,0,0}[r,N]$, defined in \eqref{diag2}, fulfills  $ \lambda(U;x)-1 \in \Sigma \mF^{\R}_{K,0,2}[r,N]$;
\item the Fourier multiplier   $ \omega(\xi) \in \widetilde\Gamma^{\frac32}_0$ is defined in \eqref{omegaxi}; 
\item the real function $V^{(2)}(U;x)\in \Sigma \mF^{\R}_{K,0,1}[r,N]$ is defined in Lemma \ref{LemCompl};
\item the matrix of symbols $  A^{(4)}_0(U; t,x,\xi)$  belongs to $ \Sigma \Gamma^0_{K,1,1}[r,N] \otimes \mM_2 (\C) $ and the operator
$J_c\opbw(A^{(4)}_0)$
 is linearly Hamiltonian up to homogeneity $N$;
\item $R(U;t)$ is a real-to-real matrix of smoothing operators in $ \Sigma \mR^{-\vr}_{K,1,1}[r,N] 
\otimes \mM_2 (\C)$. 
\end{itemize}
\end{lemma}

\begin{proof}
By Lemma 3.11 of \cite{BFP}, there exists a  real valued function $m(U;x) \in \Sigma \mathcal{F}^\R_{K,0,1}[r,N]  $ 
(actually 
$m(U;x):= - \log( h(U;x)  + g(U;x))$) 
such that the time 1 flow $\Psi_1(U):=  \Psi^{\tau}(U)|_{\tau=1}$  of 
$$ 
\begin{cases}
\pa_\tau\Psi^{\tau}(U)=
J_c  \, \Opbw{M(U;x)}
 \Psi^{\tau}(U)\\
\Psi^0(U)={\rm Id} \, , 
\end{cases}    \qquad 
M(U;x):=  \begin{bsmallmatrix}
-\im m(U;x) & 0 \\
0 & \im m(U;x)
\end{bsmallmatrix}\, , 
$$
fulfills 
\begin{equation}\label{expo}
\Psi_1(U) = \Opbw{F^{-1}(U;x)}+ R(U) \, , \quad \Psi_1(U)^{-1} = \Opbw{F(U;x)}+ R'(U) \, , 
\end{equation} 
where the matrix of functions $F(U;x)$ is defined in \eqref{diag4} 
 and  $R(U), R'(U) $ are matrices of smoothing operators in $ \Sigma\mathcal{R}^{-\vr}_{K,0,1}\otimes\mathcal{M}_2(\C)$.

Since  the  operator $J_c  \, \Opbw{M(U;x)}$ is linearly Hamiltonian
according to Definition \ref{def:LH}, 
 Lemma \ref{flow} guarantees that  $\Psi_1(U)$ is invertible, linearly symplectic and $\Psi_1(U)^{\pm 1}- \id$ belong to $ \Sigma \mS^0_{K,0,1}[r,N]\otimes \mM_2(\C)$.

Since $ V_0 $ solves \eqref{sisV0} then the variable $ V_1 = \Psi_1 (U) V_0 $ solves
\begin{align}
\notag
\pa_t V_1  & =  \Psi_1(U)  \,  \Big[ J_c \Opbw{
A_\frac32\omega(\xi)} + \frac{\gamma}{2} G(0) \pa_x^{-1} +\vOpbw{ -\ii V^{(2)} \xi}  +J_c \Opbw{A^{(3)}_0}   \Big] \Psi_1(U)^{-1}  \,  V_1 \\
& \ \ +   (\pa_t \Psi_1(U)) \Psi_1(U)^{-1}  \, V_1+ \Psi_1(U)R(U)   \Psi_1(U)^{-1}V_1 \ . 
\label{V1.eq}
\end{align}
Next we compute each term in \eqref{V1.eq}. We begin with 
$  J_c \,  \opbw ( 
A_\frac32\omega(\xi)   )   $.
Using \eqref{expo},  Proposition \ref{composizioniTOTALIs}-$(i)$,
Proposition \ref{composizioniTOTALI}-($i$)  and the explicit form of $A_\frac32 $ in \eqref{Adue},  one computes
 \begin{align}
\notag
  \Psi_1(U)   J_c \Opbw{
A_\frac32 \omega(\x) }   \Psi_1(U)^{-1}
&=\Opbw{F^{-1}}   J_c \,  \Opbw{
A_\frac32 \omega (\xi)} \Opbw{F} +R(U)
\\
 &= \Opbw{ 
 \begin{bmatrix}
 a_\frac32 & b_{-\frac12} \\
 \bar{b_{-\frac12}}^{\vee} & \bar{a_\frac32}^\vee
 \end{bmatrix}
 } 
+ R(U)
\label{a2b2}
 \end{align}
where $R(U) $ is a real-to-real matrix of smoothing operators in $  \Sigma \cR^{-\vr}_{K,0,1}[r,N] \otimes\mathcal{M}_2(\C) $ and 
using \eqref{diag1}, \eqref{diag4}, \eqref{diag5}, Proposition \ref{teoremadicomposizione},
we have 
\begin{equation*}
 \begin{aligned}
& a_{\frac32} := 
 -\ii \big[ h \#_\vr (1+f)\omega  \#_\vr h +
 g \#_\vr (1+f)\omega \#_\vr g
+ h \#_\vr  f\omega  \#_\vr g
+ g  \#_\vr  f  \omega \#_\vr  h \big] \\
&b_{-\frac12} := -\ii \big[
 h \#_\vr (1+f)\omega  \#_\vr g + 
g \#_\vr (1+f)\omega  \#_\vr h
 + h \#_\vr  f \omega \#_\vr h
+ g \#_\vr  f\omega  \#_\vr g \big] \, .
 \end{aligned}
 \end{equation*}
 The real-to-real structure of the symbols in \eqref{a2b2} follows also by the last bullet after Definition \ref{def:as.ex}.
By the symbolic calculus rule \eqref{espansione2},  the second and third bullets after Definition 
\ref{def:as.ex} and
\eqref{FAF}
 one has 
 \begin{align*}
& a_\frac32 = 
 -\ii 
  \big[ (h^2+g^2) \, (1+f) + 2hg f \big] \, \omega  
  + a_{-\frac12}= -\ii \lambda \omega + \breve a_{-\frac12}  \\
&b_{-\frac12} = - \im  \big[
2 h g (1+f) 
 + (h^2 + g^2)   f  \big] \omega + \breve b_{-\frac12}= \breve b_{-\frac12}
  \end{align*}
 with symbols $\breve a_{-\frac12}, \breve b_{-\frac12} $ in 
 $  \Sigma \Gamma^{-\frac12}_{K,0,1}[r,N]$.
Then we obtain
 \be \label{Ham.prima:eq}
 \Opbw{ \begin{bmatrix}
 a_\frac32 & b_{-\frac12}\\
 \bar{b_{-\frac12}}^\vee & \bar{a_\frac32}^\vee
 \end{bmatrix}} =
 \vOpbw{ -\ii \lambda \omega  } + J_c \Opbw{A_{-\frac12}} 
 \ee
 where $ A_{-\frac12}$ is a real-to-real 
 matrix of symbols in $  \Sigma \Gamma^{-\frac12}_{K,0,1}[r,N]\otimes \mM_2(\C)$.

 Proceeding similarly one finds that  
\be\label{Ham.seconda:eq}
\Psi_1(U)
 \vOpbw{
 - \ii V^{(2)}\xi  } \Psi_1(U)^{-1} =\Opbw{
\begin{bmatrix}
a_1 & b_0 \\
\bar b_0^\vee & \bar a_1^\vee
\end{bmatrix} }  + R(U)
\ee
where $R(U) $ is a real-to-real matrix of smoothing operators in $  \Sigma \cR^{-\vr}_{K,0,1}[r,N] \otimes\mathcal{M}_2(\C) $ and 
\be
\begin{aligned}\label{Ham.seconda:eq2}
a_1 & :=  h \#_\vr (-\ii V^{(2)}  \xi)    \#_\vr h - g \#_\vr(-\ii V^{(2)}  \xi )\#_\vr g 
 = (h^2 - g^2) ( -\ii V^{(2)}  \xi)  \stackrel{\eqref{detF=1}}{=}  -\ii V^{(2)}  \xi  \\
b_0 & :=  h \#_\vr ( -\ii V^{(2)}  \xi)  \#_\vr g - g \#_\vr  (-\ii V^{(2)}  \xi)  \#_\vr h 
  \in \mF^{\R}_{K,0,1}[r,N]  \, . 
\end{aligned}
\ee
In addition, using   \eqref{expo}, 
the last bullets at the end of Section \ref{sec:para}, Proposition \ref{composizioniTOTALI}-($iv$)
  we obtain that 
\begin{align}
\notag
\Psi_1(U) \left[  \frac{\gamma}{2} G(0) \pa_x^{-1} + J_c \,  \Opbw{
A_0^{(3)}} \right]\Psi_1(U)^{-1}
+ (\pa_t \Psi_1(U)) \Psi_1(U)^{-1}\\
 =  \frac{\gamma}{2} G(0) \pa_x^{-1}+ J_c \Opbw{A_0'}+ R(U;t) 
 \label{Ham.terza:eq}
\end{align}
where $A_0'$ is a real-to-real   matrix of symbols in   $\Sigma \Gamma^0_{K,1,1}[r,N]\otimes \cM_2(\C)$ and $R(U;t)$ is 
a matrix of real-to-real smoothing operators in $ \Sigma \cR^{-\vr}_{K,1,1}[r,N]\otimes \cM_2(\C)$.

Finally, by Proposition \ref{composizioniTOTALIs}, 
  $\Psi_1(U) R(U) \Psi_1(U)^{-1}$ 
  is a matrix of smoothing operators in 
  $ \Sigma \mR^{-\vr}_{K,0,1}[r,N] \otimes \mM_2 (\C)  $.

In conclusion, by \eqref{a2b2}-\eqref{Ham.prima:eq}, 
  \eqref{Ham.seconda:eq}-\eqref{Ham.seconda:eq2} and  \eqref{Ham.terza:eq} we 
  deduce  that system \eqref{V1.eq} has the form 
   \eqref{NuovoParaprod} 
   with a matrix of symbols $A_0^{(4)} := A_{-\frac12} +
    J_c \begin{bsmallmatrix}
    0  & b_0 \\ 
    b_0 &  0 
     \end{bsmallmatrix}
    +  A_0' $ in  $ \Sigma \Gamma^0_{K,1,1}[r,N]
 \otimes \mM_2 (\C) $.  
Note that the paradifferential operators of positive order in \eqref{NuovoParaprod} are  linearly Hamiltonian, whereas $J_c \opbw{(A_0^{(4)})}$ might not be. 
However  the  sum of the operators in the first line of  \eqref{V1.eq}  plus 
 $(\pa_t\Psi_1(U))\Psi_1(U)^{-1}$ 
is a spectrally localized map which is a linearly Hamiltonian operator 
up to homogeneity $N$ by Lemma \ref{coniugazionehamiltoniana},  with a paradifferential structure as in 
\eqref{separato}.
Then by Lemma \ref{spezzamento} 
  we can replace each homogeneous component of $A_0^{(4)}$ with its symmetrized version obtaining  that $ J_c \opbw{(A_0^{(4)})} $ is linearly Hamiltonian 
  up to homogeneity $N$, by adding  another smoothing operator.
\end{proof}

\begin{remark}\label{rem:elliptic}
In view of Lemmata \ref{LemCompl} and \ref{laprimapara} the function $\lambda(U;x)$ in  \eqref{diag2} is  equal to  
$  (1+\eta_x^2)^{-\frac34} $. Therefore 
the symbols $ \pm \im \lambda( U; x) \omega (\xi) $ 
are elliptic also for not small data and system 
 \eqref{NuovoParaprod} is hyperbolic  at order $\frac32$.
This is the well known fact that, in presence of capillarity,  
there is no need of the Taylor sign  condition for the local well-posedness. 
\end{remark}

\subsection{Reduction to constant coefficients of the highest order}

In this section we perform a linearly symplectic change of variable which reduces the highest order paradifferential operator 
$\textup{Op}^{{\scriptscriptstyle{\mathrm BW}}}_{\tt vec}(-\ii \lambda(U;x) \omega(\xi))$  in \eqref{NuovoParaprod} to constant coefficients.

\begin{lemma}[{\bf Reduction of the highest order}]
\label{prop.red2}
 Let $N \in \N_0$ and  $\vr>2(N+1)$. 
Then for any $ K \in \N$ there are $s_0 > 0 $, $r > 0 $ such that for 
any solution $U \in B_{s_0,\R}^K(I;r)$ of  \eqref{complexo}, there exists a real-to-real invertible matrix of spectrally localized maps  $ \Psi_2(U)$ satisfying $ \Psi_2(U)-\uno \in \Sigma \mS^{N+1}_{K,0,2}[r,N] \otimes\mathcal{M}_2(\C) $  and the following holds true:
\\[1mm]
$(i)$ {\bf Boundedness:} The linear map $\Psi_2(U)$ and its inverse 
 are  non--homogeneous maps in $  \mS_{K,0,0}^0 [r]\otimes \mM_2 (\C) $;
\\[1mm]
$(ii)$ {\bf Linear symplecticity:} The map  $ \Psi_2(U) $ is linearly symplectic according to Definition \ref{LS};
\\[1mm]
$(iii)$ {\bf Conjugation:}
If $V_1$ solves \eqref{NuovoParaprod} then  $V_2:= \Psi_2(U)V_1$ solves the system
\begin{equation}\label{coeffcost}
\begin{aligned}
\pa_t V_2 & = 
\vOpbw{ -\ii( 1+\zeta(U))\omega(\xi)-\ii V^{(3)}(U;t,x)\xi} V_2 +  \frac{\gamma}{2} G(0) \pa_x^{-1}V_2 \\
&\ \ 
+J_c\Opbw{A^{(5)}_0(U;t, x, \xi)} V_2 + R(U;t)  V_2
\end{aligned}
\end{equation}
where 
\begin{itemize}
\item $\zeta(U) $ is a $x$-independent function in $ \Sigma \cF^\R_{K,0,2}[r, N]$ and $ \omega(\xi)$ is  defined in  \eqref{omegaxi};  
\item $ V^{(3)}(U;t,x)$ is a real valued function  in $ \Sigma \cF^\R_{K,1,1}[r, N]$;
\item The matrix of symbols  $ A^{(5)}_0(U;t,x,\xi)$ belongs to  $\Sigma \Gamma^0_{K,1,1}[r, N]\otimes \cM_2(\C)$
and the paradifferential operator
$J_c\opbw(A^{(5)}_0)$ is
linearly Hamiltonian up to homogeneity $ N $; 
\item  $R(U;t) $ is a  real-to-real matrix of smoothing operators 
in $ \Sigma \cR^{-\vr+2(N+1)}_{K,1,1}[r, N] \otimes \cM_2(\C)$.
\end{itemize}

\end{lemma}
\begin{proof}
We define the map $ \Psi_2(U)$ as the time 1 flow $\Psi_2(U):= \Psi^{\tau}(U)|_{\tau=1}$ of 
$$
\pa_\tau\Psi^{\tau}(U)=
J_c  \, \Opbw{B(\tau,U; x, \xi)}
 \Psi^{\tau}(U) , \qquad 
\Psi^0(U)={\rm Id} \, ,   
$$
where
$$
 B(\tau,U; x, \xi):=
\begin{psmallmatrix}
 0 &  b(\tau, U;x,\xi )  \\  b(\tau, U;x,-\xi )  & 0  \end{psmallmatrix}, \quad b(\tau,U;x,\x):=  \frac{\beta(U;x)}{1+\tau\beta_x(U;x)}\xi \, , 
$$
and the function $\beta(U;x) $ in $  \Sigma\mF^\R_{K,0,2}[r,N]$ has to be determined. 
As $\beta(U;x)$ is real valued, the operator $J_c  \, \Opbw{ B(\tau,U; \cdot)}$ is linearly Hamiltonian. 
Thus Lemma \ref{flow}, applied with $ m=1$, guarantees that 
$ \Psi_2(U)$ is a spectrally localized map in  $\cS^0_{K,0,0}[r]\otimes \mM_2(\C)$, it is linearly symplectic and  $\Psi_2(U)^\pm - \id $ belongs to 
$  \Sigma \mS^{N+1}_{K,0,2}[r,N]\otimes \mM_2(\C)$.
Note that the  diagonal operator 
$J_c \Opbw{B} = \Opbw{ \im b(\tau, U; x, \xi) } \uno $ is a multiple of the identity and 
hence the flow $\Psi_2(U)$ acts as a scalar operator. 

Since $ V_1 $ solves \eqref{NuovoParaprod},  then the variable  $ V_2=\Psi_2(U)V_1$ solves 
 \be
 \begin{aligned}
 \label{red2.2}
\pa_t V_2 
& = \Psi_2(U)\left[\vOpbw{
 -\ii    \lambda \omega(\xi) - \ii V^{(2)} \x }  +\frac{\gamma}{2} G(0)\pa_x^{-1} +J_c\Opbw{
A_0^{(4)}
} \right] \Psi_2(U)^{-1} V_2 \\
&  +(\pa_t\Psi_2(U)) \circ\Psi_2(U)^{-1}  V_2+ \Psi_2(U)R(U;t) \Psi_2(U)^{-1} V_2 \, . 
 \end{aligned}
 \ee
 We now compute each term in \eqref{red2.2}.
 By Lemma 3.21 of \cite{BD}, 
 the diffeomorphism $ \Phi_U : x \mapsto x + \beta(U;x)$ of $ \T $
  is invertible with inverse  
 $\Phi_U^{-1} \colon y \mapsto y+ \breve \beta(U;y)$ and  
 $  \breve \beta(U;y) $ belongs to $ \Sigma\mF^\R_{K,0,2}[r,N]$.
By Theorem 3.27 of  \cite{BD}
 one has 
 \be
 \begin{aligned}
&\Psi_2(U) \,
\vOpbw{ -\ii   \,  \lambda(U;x) \omega(\xi) 
} \, \Psi_2(U)^{-1}  \label{conhigho} \\
&=
 \vOpbw{\left.\left[ -\ii  \,  
\lambda(U;y) \,  
\omega\Big(\xi \big( 1+ \breve\beta_y(U;y) \big) \Big) \right]\right|_{y = \Phi_U(x)}  
} + J_c \Opbw{A_{-\frac12}}+ R(U) 
\end{aligned}
\ee
 with a diagonal matrix of symbols $A_{-\frac12} $ in 
 $ \Sigma \Gamma_{K,0,1}^{-\frac12}[r,N]\otimes \mM_2(\C)$, 
 and a diagonal matrix of smoothing operators 
 $R(U) $ in $ \Sigma \cR^{-\vr+\frac32}_{K,0,1}[r, N] \otimes \cM_2(\C)$.
Note that
$\omega\big(\xi (1+ \breve\beta_y(U;y))\big)$
is a symbol in $\Sigma \Gamma^{\frac32}_{K,0,0}[r,N]$ by Lemma 3.23 of \cite{BD}. 
 
Now we choose $\breve \beta$ in such a way that the principal symbol 
in \eqref{conhigho}  is $x$-independent.
Since, 
by  \eqref{omega.M.asy}, 
$\omega_{-\frac12}(\xi) := \omega(\xi)-\sqrt{\kappa} |\xi|^{\frac{3}{2}}  
$ is a Fourier multiplier in $ \wt \Gamma^{-\frac12}_0 $ we get 
 \begin{align}
 \label{red2.40}
\lambda(U;y)  \omega\Big(\xi (1+ \breve\beta_y(U;y)) \Big) 
 & =  \lambda(U;y) \sqrt{\kappa} |\xi|^{\frac{3}{2}} 
 \big| 1+ \breve\beta_y(U;y) \big|^{\frac{3}{2}} 
  + \lambda(U;y) \omega_{-\frac12}\big(\xi(1+ \breve\beta_y(U;y))\big)  
 \end{align}
 and we select  $ \breve\beta(U;\cdot)$ so that
 \begin{equation}
 \label{red2.5}
\lambda(U;y)\, \big| 1+ \breve\beta_y(U;y) \big|^{\frac{3}{2}} =  1 + \zeta(U) 
\end{equation}
 with a $y$-independent function $\zeta (U) $. In order to fulfill \eqref{red2.5} we define
 the functions 
$$
\zeta(U):=  \Big(\frac{1}{2\pi}\int_{\mathbb{T}} 
\lambda(U;y)^{-\frac{2}{3}}\di y \Big)^{-\frac{3}{2}} -1 \, , \quad 
 \breve\beta(U;y):= \partial_y^{-1} \left[\Big( \frac{1+\zeta(U)}{\lambda(U;y)}
 \Big)^{\frac{2}{3}}-1\right] 
$$
which belong to $ \Sigma \cF^\R_{K,0,2}[r, N]$. 
By \eqref{red2.5} and since $ \omega(\xi)-\sqrt{\kappa} |\xi|^{\frac{3}{2}}  \in \wt \Gamma^{-\frac12}_0$, 
the expression   \eqref{red2.40} becomes 
\begin{align*}
\lambda(U;y)  \omega\Big(\xi (1+ \breve\beta_y(U;y)) \Big) 
 &  = \sqrt{\kappa}|\xi|^{\frac{3}{2}}  \, (1+\zeta(U)) + 
\lambda(U;y)  \omega_{-\frac12}\big(\xi(1+\breve\beta_y(U;y))\big)  \\
 & = \omega(\xi)  (1+\zeta(U)) + 
a_{-\frac12}
\end{align*}
where  $a_{-\frac12} $ is a real valued symbol  in 
$ \Sigma\Gamma^{-\frac12}_{K, 0, 1}[r, N]$.
Note that we used that
$
\omega_{-\frac12}\big(\xi(1+\breve\beta_y(U;y))\big) \lambda(U;y) 
- \omega_{-\frac12}(\xi) $ is a symbol in $ \Sigma\Gamma^{-\frac12}_{K, 0, 1}[r, N]$.
In conclusion \eqref{conhigho} is 
 \begin{equation}
 \label{red2.6}
 \Psi_2(U)
 \vOpbw{
 -\ii    \lambda \omega(\xi) }  \Psi_2(U)^{-1}  
= 
\vOpbw{- \im (1+\zeta(U)) \omega(\xi) }  +J_c \Opbw{A_{-\frac12}} + R(U) 
 \end{equation}
where $ A_{-\frac12} $ is a diagonal matrix of symbols 
in $  \Sigma \Gamma^{-\frac12}_{K,0,1}[r, N] \otimes \cM_2(\C)$   and 
$R(U) $ is a diagonal  matrix of smoothing operators  in $  \Sigma \cR^{-\vr+\frac32}_{K,0,1}[r, N] \otimes \cM_2(\C)$. 
 
We now compute the other terms in \eqref{red2.2}.
 Again by Theorem 3.27 of  \cite{BD} (and Lemma A.4 of \cite{BFP}) 
\be \label{red2.7}
\Psi_2(U)\,
\vOpbw{
- \ii V^{(2)} \x } 
\, \Psi_2(U)^{-1} 
=
\vOpbw{
 -\ii \breve V(U;x) \xi }  
 + R(U) 
\ee
where  $ \breve V(U;x) $ is a real function in $ \Sigma \cF^\R_{K,0,1}[r, N]$, 
and $R(U) $ is a diagonal matrix of smoothing operators   in $  \Sigma \cR^{-\vr+1}_{K,0,1}[r, N] \otimes \cM_2(\C)$. In addition, 
again by  Theorem 3.27 of  \cite{BD}
 \be \label{red2.54}
 \Psi_2(U)\, \left[ 
 \frac{\gamma}{2} G(0)\pa_x^{-1}
+ J_c \Opbw{A_0^{(4)}} 
  \right] \, \Psi_2(U)^{-1} = 
 \frac{\gamma}{2} G(0)\pa_x^{-1}+ J_c \Opbw{ A_{0}} + R(U)
 \ee
 where $ A_{0} $ is a  real-to-real  matrix of symbols in $ \Sigma \Gamma^{0}_{K,0,1}[r,N] \otimes \cM_2(\C) $ and $R(U)\in\Sigma \mR^{-\vr}_{K,0,1}[r,N] $.
Moreover, by Lemma A.5 of \cite{BFP} 
 \be\label{detipsi2}
(\pa_t\Psi_2(U)) \circ\Psi_2(U)^{-1}  = 
  \vOpbw{
-\ii g(U;t, x) \xi   
}  + R(U;t) 
\ee
 where $g(U;t,x) $ is a real  function in $ \Sigma \cF^\R_{K,1,1}[r, N]$ and  $R(U;t) $
 is a matrix of real-to-real 
 smoothing operators in $  \Sigma \cR^{-\vr+1}_{K, 1, 1}[r, N]\otimes \cM_2(\C)$. 
 Finally  $\Psi_2(U)R(U;t) \Psi_2(U)^{-1}$ in \eqref{red2.2} is a matrix of 
 real-to-real smoothing operators in $ \Sigma \mR^{-\vr+2(N+1)}_{K,1,1}[r,N] \otimes \cM_2(\C)$, by Proposition \ref{composizioniTOTALIs}.

In conclusion, 
by 
\eqref{red2.6},  \eqref{red2.7}, \eqref{red2.54},
\eqref{detipsi2}, 
 we 
  deduce  that system \eqref{red2.2} has the form 
   \eqref{coeffcost} 
   with 
   $V^{(3)}:= \breve V + g  $
   and
   $  A_0^{(5)}:=  A_{-\frac12} +  A_{0} $.
Note that the paradifferential operators of positive order in \eqref{coeffcost} are  linearly Hamiltonian, whereas $J_c \opbw{(A_0^{(5)})}$ might not be. 
However the  sum of the operators in the first line of  \eqref{red2.2}  and $(\pa_t\Psi_2(U))\Psi_2(U)^{-1}$ 
is a spectrally localized map which is linearly Hamiltonian up to homogeneity $N$ by Lemma \ref{coniugazionehamiltoniana},  with a paradifferential structure as in 
\eqref{separato}.
Then by Lemma \ref{spezzamento} 
  we can replace each homogeneous component of $A_0^{(5)}$ with its symmetrized version obtaining  that $ J_c \opbw{(A_0^{(5)})} $ is linearly Hamiltonian 
  up to homogeneity $N$, by adding  another smoothing operator.
\end{proof}

\subsection{Block-Diagonalization up to smoothing operators}

The goal of this section is to block-diagonalize system \eqref{coeffcost} up to smoothing remainders. 

\begin{lemma}\label{blockdiag}
Let $N \in \N_0$ and  $\vr>2(N+1)$. 
Then for any $ n\in \N_0$ there is $K':=K'(n) \geq 0$ (one can choose $K'=n$) such that for all $ K\geq K'+1$ there are $s_0 >0$, $r>0$ such that for any  solution 
$U \in B_{s_0,\R}^K(I;r)$ of  \eqref{complexo},  there exists a real-to-real invertible matrix of spectrally localized maps  $ \Phi_n(U)$ satisfying  $ \Phi_n(U)-\uno \in \Sigma \mS_{K,K',1}^0[r,N] \otimes\mathcal{M}_2(\C) $  and the following holds true: 
\\[1mm]
$(i)$ {\bf Boundedness:} Each  $\Phi_n(U)$ and its inverse 
 are  non--homogeneous maps in $  \mS_{K,K',0}^0 [r]\otimes \mM_2 (\C) $;
\\[1mm]
$(ii)$ {\bf Linear symplecticity:} The map  $\Phi_n(U) $ is linearly symplectic up to homogeneity $ N$ according to Definition \ref{LSUTHN};
\\[1mm]
$(iii)$ {\bf Conjugation:} 
 If $V_2$ solves \eqref{coeffcost} then  $V_{n+2}:= \Phi_n(U)V_2$ solves 
\begin{equation}
\label{pheq00n}
\begin{aligned}
\pa_t V_{n+2}
&  =  \vOpbw{-\ii \left[ (1+\zeta(U))\omega(\xi)+ \frac{\gamma}{2} \frac{\tG(\x)}{\x}+ V^{(3)}(U;t,x) \xi 
+  a_0^{(n)}(U;t, x,\xi) \right]  
} V_{n+2}  \\
& + J_c \Opbw{A_{- n}(U; t, x , \xi)} V_{n+2}  + R(U;t) V_{n+2}
\end{aligned}
\end{equation}
where   
\begin{itemize}
\item the Fourier multiplier $\omega(\xi) $ is defined in 
  \eqref{omegaxi}, the $x$-independent real function  
$\zeta(U) \in \Sigma \cF^\R_{K,0,2}[r, N]$ and 
the real function $V^{(3)}(U;t,x)\in \Sigma \cF^\R_{K,1,1}[r, N]$ are defined in Lemma \ref{prop.red2}; 
\item $a_{0}^{(n)}(U;t, x,\xi)$ is a symbol   in $\Sigma \Gamma^{0}_{K,K',1}[r, N]$ and $\Im a_{0}^{(n)}(U;t,x,\xi)$ belongs to $ \Gamma^{0}_{K,K',N+1}[r] $; 
\item The matrix of symbols $  A_{-n}(U;t,x,\xi) $ 
belongs to 
$  \Sigma \Gamma^{-n}_{K, K'+1, 1}[r, N]\otimes \cM_2(\C)$
and 
$J_c\opbw(A_{-n})$ is a linearly Hamiltonian  operator up to homogeneity $ N $;
\item $R(U;t) $ is a real-to-real matrix of smoothing operators 
 in $ \Sigma \cR^{-\vr+2(N+1)}_{K,K'+1, 1}[r, N]\otimes \cM_2(\C)$.
\end{itemize}
\end{lemma}

\begin{proof}
We prove the thesis by induction on $n\in \N_0$. \\
\noindent{\underline{Case $n  = 0$.}} It follows by \eqref{coeffcost} with 
$a_0^{(0)} := 0$, $A_0:= A_0^{(5)}$, $K' = 0$ and
$ \Phi_0(U):= \uno$. 

\noindent{\underline{Case $n \leadsto n+1$.}}
Suppose   \eqref{pheq00n} holds. We perform a transformation to push the off diagonal part of $J_c \Opbw{ A_{-n}}$ to lower order.
We write the real-to-real matrix $J_c A_{-n}$ as 
\be\label{JA_n}
J_c A_{-n} =  
J_c \begin{pmatrix}
- \im \bar b_{-n}^\vee & - \bar a_{-n}^\vee \\
- a_{-n} &  \im  b_{-n}
\end{pmatrix}
 =  \begin{bmatrix}
\ii a_{-n} & b_{-n} \\
\bar b_{-n}^\vee & -\ii\bar a_{-n}^\vee
\end{bmatrix} , \qquad 
a_{-n}, b_{-n} \in   \Sigma \Gamma^{-n}_{K, K'+1, 1}[r, N]
\ee
where, since  $J_c \Opbw{A_{-n}}$ is linearly Hamiltonian up to homogeneity $N$, by \eqref{LHS-cN} we have 
\begin{equation}
\label{pheq001}
 \Im a_{-n} \in \Gamma^{-n}_{K, K'+1, N+1}[r] \, , \qquad 
b_{-n}  -b_{-n}^\vee \in \Gamma^{-n}_{K, K'+1, N+1}[r] \, . 
\end{equation}
Denote by   $\Phi_{F^{(n)}}(U):=\Phi^\tau_{F^{(n)}}(U;t)\vert_{\tau = 1} $  the time $1$-flow 
of
\be\label{generatorF}
\begin{cases}
\pa_\tau\Phi_{F^{(n)}}^\tau (U)=
 \Opbw{F^{(n)}(U) }
\Phi_{F^{(n)}}^\tau(U)  \\
\Phi_{F^{(n)}}^0(U)={\rm Id} \, ,  
\end{cases}  
\qquad  F^{(n)}(U) :=\begin{bmatrix} 0 & f_{-n-\frac32}\\ \bar f_{-n-\frac32}^\vee & 0\end{bmatrix} 
 \ , 
\ee
where, see  \eqref{JA_n}, 
\be\label{pheq6}
 \qquad 
  f_{-n-\frac32}(U;t, x, \xi) := -  \frac{b_{-n}(U;t, x, \xi)}{ 2 \im \omega(\xi) (1+\zeta(U)) }  \in \Sigma \Gamma^{-n -\frac{3}{2}}_{K, K'+1, 1}[r, N] \, . 
\ee
By \eqref{pheq001}, \eqref{pheq6},  the symbol $f_{-n-\frac32} - f_{-n-\frac32}^\vee  
$ is in $ \Gamma^{-n-\frac{3}{2}}_{K,K'+1,N+1}[r]$ 
and therefore
$\Opbw{F^{(n)}(U)}$ is a linearly Hamiltonian  operator up to homogeneity $N$.
  Lemma
\ref{flow} implies that $\Phi_{F^{(n)}}(U)$ 
 is invertible,  
 linearly symplectic up to homogeneity $N$ and 
$\Phi_{F^{(n)}}(U)^{\pm 1} - \id $ belong to $ \Sigma \mS_{K,K'+1,1}^0[r,N] \otimes\mathcal{M}_2(\C) $.

If  $V_{n+2}$ fulfills \eqref{pheq00n},  the variable $V_{n+3}:= \Phi_{F^{(n)}}(U) V_{n+2}$ solves
\begin{align}
\pa_t  V_{n+3}
& = \Phi_{F^{(n)}}(U)  \, \left[  \vOpbw{\td^{(n)}_{\frac32}}
+ J_c \Opbw{A_{-n}}  \right]\Phi_{F^{(n)}}(U)^{-1}  V_{n+3} \label{pheq0o} \\
& \ \  + (\pa_t \Phi_{F^{(n)}}(U)) \Phi_{F^{(n)}}(U)^{-1} V_{n+3} + \Phi_{F^{(n)}}(U) R(U;t) \Phi_{F^{(n)}}(U)^{-1}  V_{n+3} \,  \label{secondariga}
\end{align}
where, to shorten notation,  we denoted
\be\label{denne}
\td^{(n)}_{\frac32}(U;t,x,\xi):= -\ii \Big( (1+\zeta(U))\omega(\xi)+ \frac{\gamma}{2} \frac{\tG(\x)}{\x}+ V^{(3)}(U;t,x) \xi +  a_0^{(n)}(U;t, x,\xi) \Big) \ . 
\ee
We first expand \eqref{pheq0o}. 
The Lie expansion formula (see e.g.  Lemma A.1 of \cite{BFP})
 says that for any operator  $M(U)$, setting 
 ${\bf \Phi}(U):= \Phi_{F^{(n)}}(U)$, 
 $\bF:= \Opbw{F^{(n)}(U)}$
 and  ${\rm Ad}_{\bF}[M]:= [\bF, M]$,  one has 
\be
\label{Lie1}
{\bf \Phi}(U) M(U)  ({\bf \Phi}(U))^{-1}  = M+\sum_{q=1}^{L}\frac{1}{q!}{\rm Ad}_{\bF}^{q}[M]
 +  \frac{1}{L!} \int_{0}^{1}  (1- \tau)^{L}{\bf \Phi}^\tau (U) {\rm Ad}_{\bF}^{L+1}[M] ({\bf \Phi}^\tau(U))^{-1} 
\di \tau  \, .
\ee
We apply this formula with $L:= L(\varrho) 
\geq (\varrho - n)/(n + \frac32)$  (in this way the integral remainder above is a smoothing operator in $\cR^{-\vr}_{K,K'+1, 1}[r, N]\otimes \cM_2(\C)$), and by 
symbolic calculus in Proposition \ref{teoremadicomposizione},
\eqref{JA_n}, \eqref{generatorF}, \eqref{pheq6} and formula \eqref{omega.der} 
we find
\begin{align}
&\eqref{pheq0o} =\vOpbw{ \td^{(n)}_\frac32 + \ii a_{-n} }  V_{n+3}\notag \\
& \ \ \ +  \Opbw{\begin{bmatrix} 0& \big[  (\bar \td^{(n)}_\frac32)^\vee - \td^{(n)}_\frac32\big] f_{-n-\frac32}+ b_{-n}\notag \\
\big[ \td^{(n)}_\frac32- (\bar \td^{(n)}_\frac32)^\vee \big]  \bar f_{-n-\frac32}^\vee + \bar{b_{-n}}^\vee & 0 \end{bmatrix}}
V_{n+3}\notag\\
&\ \ \   +J_c \Opbw{ A'_{- (n +1)} }  V_{n+3}+ R(U;t) \,  V_{n+3}\label{parastru}
\end{align}
with a real-to-real matrix of symbols 
$ A_{- (n +1)}'  $ in $  \Sigma \Gamma^{-n-1 }_{K,  K'+1, 1}[r, N]\otimes \cM_2(\C)$ and a matrix of smoothing operators $R(U;t) $ in $ \Sigma \cR^{-\vr}_{K,K'+1, 1}[r, N]\otimes \cM_2(\C)$ (we also used  Lemma 
\ref{flow} and Proposition \ref{composizioniTOTALIs} to estimate the Taylor remainder in the Lie expansion formula).
By  \eqref{denne},   
\eqref{pheq6} 
and since  $a_0^{(n)}$ is of order 0, 
we have 
\be\label{parastru10}
\big[ (\bar \td^{(n)}_\frac32)^\vee -\td^{(n)}_\frac32 \big] f_{-n-\frac32} + b_{-n} =:  b_{-n -\frac32} \in  \Sigma \Gamma^{-n-\frac{3}{2}}_{K,K'+1,1}[r,N] \, .
\ee
We pass to the first term in \eqref{secondariga}.
Using the Lie expansion (cfr. Lemma A.1 of \cite{BFP})
\begin{align}
\label{Lie2}
 (\pa_t {\bf \Phi}(U) )({\bf \Phi}(U))^{-1}   
& =\pa_t \bF + 
\sum_{q=2}^{L}\frac{1}{q!}{\rm Ad}_{\bF}^{q-1}[\pa_{t}\bF] \nonumber \\ 
& \quad +\frac{1}{L!}\int_{0}^{1}
(1- \tau )^{L} 
 {\bf \Phi}^\tau (U){\rm Ad}_{\bF}^{L}[ \pa_{t}\bF]({\bf \Phi}^\tau(U))^{-1} \di \tau \, 
\end{align}
with the same $L$ as above, the last bullets at the end of Section \ref{sec:para}, Proposition \ref{composizioniTOTALI}-$(iv)$  and \eqref{pheq6} we get 
\be \label{parastru2}
 (\pa_t \Phi_{F^{(n)}}(U)) \Phi_{F^{(n)}}(U)^{-1} = J_c \Opbw{Q(U;t,x,\xi)} + R(U;t)
\ee
with a real-to-real matrix of symbols 
$Q(U;t,x,\xi)$  in $ \Sigma\Gamma^{-n-\frac32}_{K,K'+2,1}[r,N] \otimes \mM_2(\C)$ and a matrix of smoothing operators $R(U;t) $
in $ \Sigma \cR^{-\vr}_{K,K'+2, 1}[r, N]\otimes \cM_2(\C)$. 

 Thanks to $(i)$ and $(ii)$ of Proposition \ref{composizioniTOTALIs}, the operator $\Phi_{F^{(n)}}(U) R(U;t) \Phi_{F^{(n)}}(U)^{-1}$ in \eqref{secondariga} is a smoothing operator in $ \Sigma \mR^{-\vr+ 2(N+1)}_{K,K'+1,1}[r,N]\otimes \mM_2(\C)$ as well as $R(U;t)$. In conclusion, 
 by \eqref{parastru}, \eqref{parastru10}, \eqref{parastru2}, the system in 
 \eqref{pheq0o}--\eqref{secondariga} has the form 
\begin{align}
 \label{parastru20}
\pa_t  V_{n+3} 
&=\vOpbw{\td^{(n)}_\frac32 + \ii a_{-n}   } V_{n+3}   + J_c \Opbw{ A_{- n-1}} V_{n+3}  + R(U;t) V_{n+3} 
\end{align}
 where 
the matrix of symbols $ A_{- n -1}$  in $  \Sigma \Gamma^{-n-1 }_{K,  K'+2, 1}[r, N]\otimes \cM_2(\C)$ is given by 
$ A_{- n -1} := A'_{-(n+1)} + 
J_c
\begin{bsmallmatrix} 
0 & b_{-n-\frac32}\\
\bar  b_{-n-\frac32}^\vee & 0 
\end{bsmallmatrix}
+
Q
$ and $R(U;t) $ a matrix of smoothing operators  in 
$ \Sigma \cR^{-\vr+2(N+1)}_{K,K'+2, 1}[r, N]\otimes \cM_2(\C)$.
By Lemma \ref{spezzamento},
  we  replace each homogeneous component of $A_{-n-1}$ with its symmetrized version obtaining  that $ J_c \Opbw{A_{-n-1}} $ is linearly Hamiltonian 
  up to homogeneity $N$, by adding  another smoothing operator. 
  
In conclusion, by  \eqref{denne},  system \eqref{parastru20}  has the form  \eqref{pheq00n} at step $n+1$ 
with $a_0^{(n+1)}:= a_0^{(n)}- a_{-n}$ and $K'(n+1):= K'(n)+1$. Note that the imaginary part 
$ \Im a_0^{(n+1)}$
 is in $ \Gamma^0_{K, K', N+1}[r]$ by the inductive assumption and \eqref{pheq001}.

Finally we define $\Phi_{n+1}(U):= \Phi_{F^{(n)}}(U) \circ \Phi_{n}(U)$. The claimed properties of $\Phi_n(U)$ follow by the analogous ones of 
each $\Phi_{F^{(n)}}(U)$ and Proposition \ref{composizioniTOTALIs}.
\end{proof}

For any $\vr> 2(N+1)$ we now choose in Lemma \ref{blockdiag} a number of iterative steps $n:= n_1(\vr)$ such that $ n_1\geq \vr -2(N+1)$,
so that we can incorporate $J_c \Opbw{A_{-n}(U)}$ in the smoothing remainder $R(U)$ in $ \Sigma \cR^{-\vr+2(N+1)}_{K,K'+1, 1}[r, N]\otimes \cM_2(\C)$.
 Thus,   denoting  $Z:= V_{n+2}$, we write system \eqref{pheq00n} 
 as 
\be\label{inizioZ}
\!\! \!  \pa_t Z = \vOpbw{ -\ii \left[ (1+\zeta(U))\omega(\xi) + \frac{\gamma}{2} \frac{\tG(\x)}{\x} +  V^{(3)}(U;t,x) \xi  + a_0(U;t, x,\xi) \right] } Z  + R(U;t) Z 
\ee
where the symbol $a_0(U;t, x,\xi) := a_0^{(n_1)}(U;t,x,\x)$ is given in \eqref{pheq00n} with $n=n_1(\vr)$. 
Thus $a_{0}(U;t, x,\xi)$ belongs to $\Sigma \Gamma^{0}_{K,K',1}[r, N]$ and its imaginary part 
 $\Im a_{0}(U;t,x,\xi)$    in  $ \Gamma^{0}_{K,K',N+1}[r] $ with $ K'= n_1(\vr)$.

\subsection{Reduction to constant coefficients up to smoothing operators}

The goal of this section is to reduce the  symbol in the paradifferential operator 
in \eqref{inizioZ} to an $ x $-independent one, up to smoothing operators. 
 
\begin{lemma}\label{conjcost}
Let $N \in \N_0$ and  $\vr>3(N+1)$. 
Then for any $ n\in \N_0$ there is $K'':=K''(\vr,n)>0$ (one can choose $K''=K'(n_1(\vr))+n$) such that for all $ K\geq K''+1$ there are $s_0 >0$, $r>0$ such that for any  solution 
$U \in B_{s_0,\R}^K(I;r)$ of  \eqref{complexo},   there exists a real-to-real invertible matrix of spectrally localized maps  $ \mF_n(U)$ satisfying  $ \mF_n(U)-\uno \in \Sigma \mS^{(N+1)/2}_{K,K'',1}[r,N] \otimes \cM_2(\C) $  and the following holds true:
\\[1mm]
$(i)$ {\bf Boundedness:} Each  $\mF_n(U)$ and its inverse 
 are  non--homogeneous maps in $  \mS_{K,  K'',0}^0 [r]\otimes \mM_2 (\C) $;
\\[1mm]
$(ii)$ {\bf Linear symplecticity:} The map  $\mF_n(U) $ is linearly symplectic up to homogeneity $ N$ according to Definition \ref{LSUTHN}.
\\[1mm]
$(iii)$ {\bf Conjugation:} If $ Z $ solves \eqref{inizioZ} then  
$ Z_{n}:= \mF_n(U) Z $ solves
\begin{equation}
\label{pheq00n2}
\pa_t Z_{n}  = \vOpbw{ \ii \, \tm_{\frac32}^{(n)}(U;t, \xi)+\ii a_{-\frac{n}{2}}(U;t,x,\xi)}Z_{n} + R(U;t) Z_{n}
\end{equation}
with the $x$--independent symbol 
\be \label{emmenne}
\tm_{\frac32}^{(n)}(U;t,\x):=-\left[ (1+\zeta(U))\omega(\xi)+ \frac{\gamma}{2} \frac{\tG(\xi)}{\xi} +  \mathtt V(U;t) \xi + \tb_{\frac12}(U;t)|\xi|^{\frac12} \right]+  \tb_0^{(n)}(U;t, \xi)
\ee
where
\begin{itemize}
\item  the $x$--independent function $\zeta(U) \in \Sigma \cF^\R_{K,0,2}[r, N]$ is defined in Lemma \ref{prop.red2} and  $\omega(\xi) $ 
in  \eqref{omegaxi}; 
\item  the function $\mathtt V(U;t) \in \Sigma \cF^\R_{K,1,2}[r, N]$ is 
$ x $-independent;  
\item  the function $\tb_{\frac12}(U;t) \in \Sigma\cF^\R_{K,2,2}[r, N]$ is 
$ x $-independent;  
\item the symbol $\tb_{0}^{(n)}(U;t, \xi) \in \Sigma \Gamma^{0}_{K,K'',2}[r, N]$
 is $x$--independent and its imaginary part $\Im \tb_{0}^{(n)}(U;t, \x) $ is in $ \Gamma^{0}_{K,K'',N+1}[r]$; 
\item the symbol  $a_{-\frac{n}{2}}(U;t, x,\xi)$ belongs to $\Sigma \Gamma^{-\frac{n}{2}}_{K,K''+1,1}[r, N]$ and its imaginary part $ \Im a_{-\frac{n}{2}}(U;t, x,\xi) $ is in $ \Gamma^{-\frac{n}{2}}_{K,K''+1,N+1}[r]$;
\item  $R(U;t) $ is a real-to-real matrix of smoothing operators in 
$ \Sigma \cR^{-\vr+3(N+1)}_{K,K''+1, 1}[r, N]\otimes \cM_2(\C)$.
\end{itemize}
\end{lemma}

\begin{proof}
We transform the equation \eqref{inizioZ} for the variable $Z$. 
 
\noindent{\underline{Case $n = 0$.}} 
\noindent {\bf Reduction to constant coefficients of order 1.}
 We first reduce to constant coefficients the transport term of order $1$ in \eqref{inizioZ}.
Let 
 $\Phi_{\beta_{\frac12 }}(U):=\Phi_{\beta_{\frac12 }}^\tau(U;t)\vert_{\tau = 1} $ be 
 the time $1$-flow 
of
$$
\pa_\tau \Phi_{\beta_{\frac12 }}^\tau(U)=
 \vOpbw{ -\ii \beta_{\frac12}(U;t, x) |\xi|^{\frac12}}
\Phi_{\beta_{\frac12 }}^\tau(U) , \qquad 
\Phi_{\beta_{\frac12 }}^0(U)={\rm Id} \, ,  
$$
where  $\beta_{\frac12} $ is the 
 real function  
 in $  \Sigma\mF^{\R}_{K,1,1}[r, N]$ defined by
 \be
 \begin{aligned}\label{beta12}
& \beta_\frac12(U; t, x)  := \frac{2}{3\sqrt{\kappa} (1+ \zeta(U)) } \, \pa_x^{-1} \left[  \mathtt{V}(U;t) - V^{(3)}(U;t,x) \right] \,  \\
& \mathtt{V}(U;t)  :=  \frac{1}{2\pi} \int_\T  V^{(3)}(U;t,x) \, \di x \, . 
\end{aligned}
 \ee
Note that the real $x$-independent function  
$\mathtt{V}(U;t) $ is in $  \Sigma\mF^{\R}_{K,1,2}[r, N]$ thanks to Remark \ref{rem:symbol}
(it could be also directly verified that the linear component in $U$ of the space average of $V^{(3)}$ vanishes).

By \eqref{LHS-c}
 the operator $  \vOpbw{  -\ii \beta_{\frac12} |\xi|^{\frac12}  }$
 is linearly Hamiltonian.
 By Lemma
\ref{flow}, the flow $\Phi_{\beta_{\frac12 }}(U)$ is a diagonal matrix of 
spectrally localized maps in 
$ \cS^0_{K,1,0}[r]\otimes\mathcal{M}_2(\C)$ with its inverse, 
it is linearly symplectic  and 
$\Phi_{\beta_{\frac12 }}(U)^{\pm 1} - \id $ belong to  $  \Sigma \mS_{K,1,1}^{(N+1)/2}[r,N] \otimes\mathcal{M}_2(\C) $. 

If $ Z $ solves equation \eqref{inizioZ}, then  the variable $ \breve Z :=  \Phi_{\beta_{\frac12 }}(U) Z$ satisfies 
$$
\begin{aligned}
\pa_t  \breve Z  & =   \Phi_{\beta_{\frac12}}(U)  \,  \vOpbw{  -\ii \left[ (1+\zeta(U))\omega(\xi) + \frac{\gamma}{2} \frac{\tG(\x)}{\x} +  V^{(3)} \xi  + a_0 \right]  }\Phi_{\beta_{\frac12}}(U)^{-1}\, \breve Z \\
& \quad + (\pa_t \Phi_{\beta_{\frac12}}(U))  \Phi_{\beta_{\frac12}}(U)^{-1}  \breve Z
+ \Phi_{\beta_{\frac12}}(U) R(U;t) \Phi_{\beta_{\frac12}}(U)^{-1}  \breve Z \, .
\end{aligned}
$$
 Using the Lie expansions in \eqref{Lie1}, \eqref{Lie2} with
 ${\bf \Phi}:= \Phi_{\beta_\frac12}$, $\bF:=  \vOpbw{ -\ii \beta_{\frac12} |\xi|^{\frac12}}$, 
 $L:= L(\varrho) \geq 2(\varrho + 1)$ (so the integral remainders in the Lie expansions are smoothing operators in $\cR^{-\vr}_{K,K', 1}[r, N]\otimes \cM_2(\C)$), Proposition \ref{def:as.ex} and \eqref{omega.der}
   we obtain 
\begin{align}
&\pa_t  \breve Z  = \vOpbw{ -\ii \left[ (1+\zeta(U))\omega(\xi)+ \frac{\gamma}{2} \frac{\tG(\xi)}{\xi}\right] }\,  \breve Z \notag \\
&+  \vOpbw{ -\ii \left[V^{(3)} + \frac{3}{2} \sqrt{\kappa} (\beta_{\frac12})_x (1+\zeta(U))\right]	\xi }\,  \breve Z \label{ordineone}\\
& +
 \vOpbw{-\ii b_{\frac12} (U; t, x) | \xi |^{\frac12}- \ii a_0^{(1)}(U;t, x,\xi)  } \breve Z+ R'(U;t) \breve Z +  \Phi_{\beta_{\frac12}}(U) R(U;t) \Phi_{\beta_{\frac12}}(U)^{-1}  \breve Z \notag
\end{align}
where $R'(U;t)$ belongs to $\Sigma \cR^{-\vr}_{K,K', 1}[r, N]\otimes \cM_2(\C)$ and 
\be \label{b12}
b_{\frac12}(U;t, x) :=
-\frac12  \beta_\frac12 V^{(3)}_x + (\beta_\frac12)_x V^{(3)} - \frac34 \sqrt{\kappa}(1 + \zeta(U)) \left( \frac12 \beta_\frac12 (\beta_\frac12)_{xx}+
 (\beta_\frac12)_x^2 \right)
 + \pa_t \beta_\frac12 
\ee
is a  real valued  function in $ \Sigma \cF^\R_{K,2,1}[r,N]$ 
(use also the last bullets at the end of Section \ref{sec:para} and Proposition \ref{composizioniTOTALI}-$(iv)$), 
  and we collect in  $ a_0^{(1)}(U;t, x,\xi)$ all the symbols in $ \Sigma \Gamma^0_{K,K',1}[r,N]$.  Finally 
by  Proposition \ref{composizioniTOTALIs} we  deduce that  
$ \Phi_{\beta_{\frac12}}(U) R(U;t) \Phi_{\beta_{\frac12}}(U)^{-1} $ is    
 in $ \Sigma \mR^{-\vr +  3(N+1)}_{K,K',1}[r,N] \otimes\mathcal{M}_2(\C) $.
By  \eqref{beta12} the first order term in \eqref{ordineone} is constant coefficient, namely 
$$ 
V^{(3)}(U;t, x) + \frac{3}{2} \sqrt{\kappa} (\beta_{\frac12})_x (U;t,x)(1+\zeta(U))=\mathtt{V}(U;t) \, ,  
$$
and \eqref{ordineone} reduces to 
\be \label{vudoppio0}
\pa_t  \breve Z   =   \vOpbw{  -\ii \left[ (1+ \zeta(U))\omega(\xi) + 
\frac{\gamma}{2} \frac{\tG(\xi)}{\xi}+  \mathtt{V}(U;t) \xi
+  b_{\frac12}|\xi|^{\frac12} +a_0^{(1)} 
 \right]  }  \breve Z    + R(U;t)  \breve Z \, . 
\ee
The paradifferential operators of positive order in \eqref{vudoppio0} are  linearly Hamiltonian, whereas $\textup{Op}^{\scriptscriptstyle {\rm BW}}_{\tt vec}(- \ii  a_0^{(1)} )$ might not be. 
By the usual argument, 
  we  replace each homogeneous component of $a_0^{(1)}$ with its symmetrized version obtaining  that  $\textup{Op}^{ \scriptscriptstyle {\rm BW}}_{\tt vec}( -\ii  a_0^{(1)})$   is  linearly Hamiltonian 
  up to homogeneity $N$, by adding  another smoothing operator.

\noindent {\bf Reduction to constant coefficients of order $\frac12$.}
The next step is to put to constant coefficients the symbol 
$- \ii b_{\frac12}|\xi|^{\frac12} $
in  system \eqref{vudoppio0}.
Let 
 $\Phi_{\beta_{0}}(U):=\Phi_{\beta_{0}}^\tau(U;t)\vert_{\tau = 1} $ be the time
 $1$-flow 
of
$$
\pa_\tau \Phi_{\beta_{0}}^\tau(U)=
 \vOpbw{ \ii \beta_{0}(U;t, x)\,  \sgn \, \xi}
\Phi_{\beta_{0 }}^\tau(U) , \qquad 
\Phi_{\beta_{0 }}^0(U)={\rm Id} \, ,  
$$ 
where the 
 real function $\beta_0 $ in $ \Sigma\mF^{\R}_{K,2,1}[r,N]$  is 
\be\label{b0}
 \begin{aligned}
 &  \beta_0(U;t,x)= \frac{2}{3 \sqrt{\kappa}(1+\zeta(U))} \pa_x^{-1}\left( b_{\frac12}(U;t,x) -  \tb_{\frac12}(U;t)  \right) \, , \\
 &  \tb_{\frac12}(U;t):= \frac{1}{2\pi}\int_\T b_{\frac12}(U;t,x)\, \di x \,  .
\end{aligned}
\ee
Note that the real $x$-independent function $\tb_{\frac12}(U;t) $ is in 
$ \Sigma\mF^{\R}_{K,2,2}[r,N]$ thanks to Remark \ref{rem:symbol} (it also follows by  \eqref{b12} since its linear component in $ U $ comes from $\pa_t \beta_{\frac12}$ which has zero average, see \eqref{beta12}).

By \eqref{LHS-c}, the operator $\vOpbw{\im \beta_0\, \textup{sign}\, \xi}$ is linearly Hamiltonian. 
Hence   by Lemma
\ref{flow}, $\Phi_{\beta_{0 }}(U)$ is a diagonal matrix of spectrally localized maps in 
$ \cS^0_{K,2,0}[r] \otimes\mathcal{M}_2(\C)$ with its inverse, 
it is linearly symplectic  and
$\Phi_{\beta_{0 }}(U)^{\pm 1} - \id $ belong to  $  \Sigma \mS^0_{K,2,1}[r,N] \otimes \mM_2 (\C) $. 

If $  \breve Z$ solves \eqref{vudoppio0} then the variable $ Z_0:= \Phi_{\beta_{0}}(U)  \breve Z$ solves  
\be\label{pheq300}
\begin{aligned}
\pa_t Z_0 &  = \Phi_{\beta_{0}}(U)  \,  \vOpbw{ -\ii \left[ (1+\zeta(U))\omega(\xi)+\frac{\gamma}{2} \frac{\tG(\xi)}{\xi}+\mathtt{V}(U;t) \xi \right]}
\Phi_{\beta_{0}}(U)^{-1}\, Z_0\notag \\
& \ \ + \Phi_{\beta_{0}(U)}  \,  \vOpbw{- \ii b_{\frac12}|\xi|^{\frac12} -\ii a_0^{(1)} }\Phi_{\beta_{0}}(U)^{-1}\, Z_0\notag\\
& \ \ + (\pa_t \Phi_{\beta_{0}}(U))  \Phi^{-1}_{\beta_{0}}(U) Z_0+ \Phi_{\beta_{0}}(U) R(U;t) \Phi_{\beta_{0}}(U)^{-1} Z_0 \, . 
\end{aligned}
\ee
Using the Lie expansions in \eqref{Lie1}, \eqref{Lie2} with
 $ {\bf \Phi}:= \Phi_{\beta_0} $, $ \bF:=  \vOpbw{ \ii \beta_{0} \sign (\xi) }$, 
 $ L := L(\vr) $ large enough so that the integral remainders in the Lie expansions are $\vr$-smoothing operators, and  \eqref{omega.der} 
   we obtain 
   \be
\begin{aligned}
\pa_t Z_0  & =    \vOpbw{ -\ii \left[  (1+\zeta(U))\omega(\xi)+ \frac{\gamma}{2} \frac{\tG(\xi)}{\xi} + \mathtt{V}(U;t) \xi \right]} Z_0  \\
& \ \  + \vOpbw{ \ii \left(\frac32 \sqrt{\kappa}(\beta_0)_x (1+\zeta(U)) - b_{\frac12} \right)|\xi|^{\frac12} + \ii a_0^{(2)} } \, Z_0 + R(U;t)  Z_0 \label{pheq32}
\end{aligned}
\ee
where
$ a_0^{(2)} $ is a symbol  
in $ \Sigma \Gamma^0_{K,K',1}[r,N]$ and $R(U;t) $ is a real-to-real matrix of 
smoothing operators in $  \Sigma \mR^{-\vr+ 3(N+1)}_{K,K',1}[r,N] \otimes \mM_2 (\C) $.
By  
\eqref{b0}
 the symbol of order $\frac12$  in \eqref{pheq32} is constant coefficient, namely 
  $$ 
  \frac32 \sqrt{\kappa}(\beta_0)_x (1+\zeta(U)) - b_{\frac12}
  =- \tb_\frac12(U;t) \, , 
  $$ 
and 
 system \eqref{pheq300} reduces to 
\be \label{vudoppio}
\begin{aligned}
\pa_t Z_0 &  =   \vOpbw{  -\ii \Big[ (1+ \zeta(U))\omega(\xi) + 
\frac{\gamma}{2} \frac{\tG(\xi)}{\xi}+  \mathtt{V}(U;t) \xi  
+ \tb_{\frac12}(U;t)|\xi|^{\frac12} \Big] + \im a_0^{(2)}  } Z_0 \\
& \quad + R(U;t) Z_0 \, . 
\end{aligned}
\ee 
By Lemma \ref{spezzamento},  
we  replace each homogeneous component 
of $\textup{Op}^{\scriptscriptstyle {\rm BW}}_{\tt vec}{
 ( \ii a_0^{(2)}) }$ so that it becomes  linearly Hamiltonian up to homogeneity $N$, i.e. 
 by \eqref{LHS-cN}, it results that  Im $ a_0^{(2)} $ belongs to
 $ \Sigma \Gamma^0_{K,K',N+1}[r,N] $.
  
  So far we have shown 
  that   \eqref{vudoppio} becomes \eqref{pheq00n2} with $ n = 0 $, putting 
$ a_{0} := a_0^{(2)} $
(which we consider as a symbol in $\Sigma \Gamma^0_{K, K''+1, 1}[r,N]$) 
 and $\tb_0^{(0)}:= 0$.  
   We put $\mF_0(U):= \Phi_{\beta_{0}}(U) \Phi_{\beta_{\frac12}}(U)$.\\ 
 \noindent \underline{\bf Case $n \leadsto n+1$.} The proof is by induction on $n$.
Suppose that $Z_n$ is a solution of system  \eqref{pheq00n2}.
 Let  $ \Phi_{F_n}(U): = \Phi_{F_n}^\tau(U;t)\vert_{\tau = 1} $ be the time $1$-flow  
 $$
\pa_\tau \Phi_{F_n}^\tau(U) =  \vOpbw{ \im  \beta_{-\frac{n}{2}-\frac12}(U;t,x,\xi) }  \Phi_{F_n}^\tau(U)  \, , 
\quad 
\Phi_{F_n}^0(U) = \id \, , 
 $$
where
\be\label{betan12}
\begin{aligned}
& \beta_{-\frac{n}{2}-\frac12}(U;t,x,\x):=-\frac{2\, \sgn \, \xi }{3\sqrt{\kappa}(1+\zeta(U)) | \xi|^{\frac12}} \pa_x^{-1} \left( a_{-\frac{n}{2}}(U;t,x,\xi)-\ta_{-\frac{n}{2}}(U;t) \right)   \, ,    \\
& 
\ta_{-\frac{n}{2}}(U;t):= \frac{1}{2\pi} \int_{\T} a_{-\frac{n}{2}}(U;t,x,\xi)\di x \, . 
\end{aligned}
\ee
By the inductive assumption, the symbol   $a_{-\frac{n}{2}}$ belongs to 
$ \Sigma \Gamma^{-\frac{n}{2}}_{K,K''+1,1}[r,N]$ and has imaginary part in $ \Gamma^{-\frac{n}{2}}_{K,K'' +1,N+1}[r]$.
Then the
$x$-independent symbol $\ta_{-\frac{n}{2}}$ belongs to
$  \Sigma \Gamma^{-\frac{n}{2}}_{K,K''+1,2}[r,N]$ thanks to Remark \ref{rem:symbol}
and $\textup{Im}\, \ta_{-\frac{n}{2}} \in  \Gamma^{-\frac{n}{2}}_{K,K'' +1,N+1}[r]$.
It follows that  the symbol 
$\beta_{-\frac{n}{2}-\frac12}$
belongs to  $ \Sigma \Gamma^{-\frac{n}{2}-\frac12}_{K,K''+1,1}[r,N]$ and has 
 imaginary part in $\Gamma^{-\frac{n}{2}-\frac12}_{K,K'' +1,N+1}[r]$.
 
Therefore by \eqref{LHS-cN} the operator 
 $\textup{Op}^{\scriptscriptstyle {\rm BW}}_{\tt vec}(\im \beta_{-\frac{n}{2}- \frac12})$ is linearly Hamiltonian up to homogeneity $N$.
By Lemma \ref{flow}, the flow $ \Phi_{F_n}(U) $  is invertible, 
 linearly symplectic up to homogeneity $N$ and 
$\Phi_{{F_n}}(U)^{\pm 1} - \id $ belong to 
$ \Sigma \mS_{K,K''+1,1}^0[r,N] \otimes \mM_2 (\C) $. 

If $Z_n$ solves \eqref{pheq00n2} then the variable $ Z_{n+1}:= \Phi_{F_n}(U) Z_n$ solves 
$$
\begin{aligned}
\pa_t Z_{n+1} &  = \Phi_{F_n}(U)  \,  \vOpbw{ \im \, \tm_{\frac32}^{(n)}(U;t,\xi)  + \im a_{-\frac{n}{2}} }
 \Phi_{F_n}(U)^{-1} Z_{n+1}\\
& \ \ + (\pa_t \Phi_{F_n}(U))  \Phi_{F_n}(U)^{-1} Z_n+ \Phi_{F_n}(U) R(U;t) \Phi_{F_n}(U)^{-1} Z_{n+1} \, . 
\end{aligned}
$$
Using the Lie expansions in \eqref{Lie1}, \eqref{Lie2} with
 ${\bf \Phi}:= \Phi_{F_n}$, $\bF:=  \vOpbw{ \ii \beta_{-\frac{n}{2}-\frac12}  }$ 
 with  ($L:=L(\vr)$ large enough), 
the last bullets at the end of Section \ref{sec:para} and Proposition \ref{composizioniTOTALI}-$(iv)$,
 \eqref{emmenne},  \eqref{omega.der}, 
   we obtain  that 
\begin{align}
\notag
\pa_t Z_{n+1}    & =   \vOpbw{
\ii  \,\tm_{ \frac32}^{(n)} + \ii \left[ 
\frac32 \sqrt{\kappa}(\beta_{-\frac{n}{2}-\frac12})_x  (1+\zeta(U)) | \xi|^{\frac12} \sgn\,{\xi} +a_{-\frac{n}{2}}
\right] }Z_{n+1}\\
\notag
& \quad + \vOpbw{\ii a_{-\frac{n}{2}-\frac12} }Z_{n+1} + R(U;t) Z_{n+1}\\
& \quad \stackrel{\eqref{betan12}}{=}  
 \vOpbw{
\ii \, \tm_{ \frac32}^{(n)}  + \im \ta_{-\frac{n}{2}} +\ii a_{-\frac{n}{2}-\frac12} }Z_{n+1} + R(U;t) Z_{n+1} 
\label{enneconstcoeff}
\end{align}
where we collect in $a_{-\frac{n}{2}-\frac12}$ all the symbols in $ \Sigma \Gamma^{-\frac{n}{2}-\frac12}_{K,K''+2,1}[r,N]$, 
and  $ R(U;t)$  is a smoothing operator in 
 $ \Sigma \mR^{-\vr+3(N+1)}_{K, K''+2,1}[r,N] \otimes \mM_2 (\C) $. 
By Lemma \ref{spezzamento}, we   replace each homogeneous component of 
$\textup{Op}^{\scriptscriptstyle {\rm BW}}_{\tt vec}{ ( \ii a_{-\frac{n}{2}-\frac12}) } $ so that it is linearly Hamiltonian up to homogeneity $N$; which,  
 by \eqref{LHS-cN}, is equivalent to  assume that the imaginary part 
$\Im a_{-\frac{n}{2}-\frac12}$ is  a symbol in
$ \Gamma^{-\frac{n}{2}-\frac12}_{K, K''+2,N+1}[r]$. 

System \eqref{enneconstcoeff} has the form  \eqref{pheq00n2} at step $n+1$ with
 $\tb_0^{(n+1)}:= \tb_0^{(n)}+ \ta_{-\frac{n}{2}}$
 (hence 
 $ \tm_{ \frac32}^{(n+1)} :=  \tm_{ \frac32}^{(n)} + \ta_{-\frac{n}{2}}$) and  $K''(n+1):= K''(n)+1$.

The thesis follows with $ \mF_{n+1}(U):= \Phi_{F_n}(U) \mF_n(U)$. The proof of Lemma \ref{conjcost} is complete. 
\end{proof}

\noindent {\bf Proof of Proposition \ref{teoredu1}.}
We now choose in Lemma \ref{conjcost} a number $n:=n_2(\vr) $
of iterative steps  satisfying  $ n_2(\vr) \geq 2( \vr - 3(N+1))$
so that  we incorporate $\textup{Op}^{\scriptscriptstyle {\rm BW}}_{\tt vec}{ ( \ii a_{-\frac{n}{2}}) }$ in the smoothing remainder $R(U;t)$, which belongs to $\Sigma \cR^{-\vr+3(N+1)}_{K,K''+1, 1}[r, N]\otimes \cM_2(\C)$
 with $ K'' =n_1(\vr)+n_2(\vr)\geq 3\vr - 8(N+1)$, with $n_1(\vr)$
 fixed above  \eqref{inizioZ}.
Denoting  $W:= Z_n$, system \eqref{pheq00n2} has the form  
\eqref{teo61} 
with $\tb_0(U;t,\xi):= -\tb_0^{(n_2)}(U;t,\x)$ in 
\eqref{m32}
and  taking  as $ \underline K':= K''+1 =  n_1(\vr)+n_2(\vr)+1$ 
(this proves Remark \ref{rem:pro61}). 
The variable  $ W $ can be written as $ W =  \bB(U;t)U $ where 
$$
\bB(U;t):= \mF_{n_2}(U) \circ \Phi_{n_1}(U)\circ \Psi_2(U)\circ \Psi_1(U)\circ\mG(U) 
$$
and  $\mG(U) $ is the map of Lemma \ref{goodlem}, $ \Psi_1(U) $ is the map of Lemma \ref{diag}, $ \Psi_2(U) $ is the map of Lemma \ref{prop.red2}, $ \Phi_{n_1}(U) $
 is the map of Lemma \ref{blockdiag} with number of steps
 $ n_1:=n_1(\vr)$ and 
 $ \mF_{n_2}(U) $ is the map of Lemma \ref{conjcost} with number of steps 
 $ n_2:= n_2(\vr) $. 
Since 
$$
\begin{aligned}
& \mG(U)- \uno \, , \ \ \  \Psi_1(U)- \uno  \in \Sigma \mS^0_{K,0,1}[r,N] \otimes \mM_2(\C) \, ,
\ \  \Phi_{n_1}(U)-\uno \in \Sigma \mS_{K,K',1}^0[r,N] \otimes\mathcal{M}_2(\C)  \\
& \quad \Psi_2(U)-\uno \in \Sigma \mS^{N+1}_{K,0,2}[r,N] \otimes\mathcal{M}_2(\C) \, , \quad
 \mF_{n_2}(U)-\uno \in \Sigma \mS^{(N+1)/2}_{K,K'',1}[r,N] \otimes \cM_2(\C) \, , 
\end{aligned}
$$
we deduce by Proposition \ref{composizioniTOTALIs} that $\bB(U;t)- \id $ is a 
real-to-real matrix of spectrally 
localized maps in 
$
 \Sigma \mS^{\frac32 (N+1)}_{K,\underline K '-1,1}[r,N]\otimes \mM_2(\C)$. 
 In addition  $\bB(U;t)$ 
is a spectrally localized map in 
$\cS^0_{K, \underline K'-1,0}[r,N] \otimes \mM_2(\C) $ with its inverse, 
 as each map 
$ \mF_{n_2}(U) $, $ \Phi_{n_1}(U)$, $  \Psi_2(U) $, $ \Psi_1(U)$,  $ \mG(U) $ separately.
Finally $ \bB(U;t)$ is linearly symplectic up to homogeneity $N$, being  the composition of linearly symplectic maps  up to homogeneity  $ N $. This completes the proof of 
 Proposition \ref{teoredu1}. 

\section{Hamiltonian Birkhoff normal form}\label{sec:birk}

The main result of this section is  Proposition \ref{birkfinalone} which transforms 
the water waves equations in Hamiltonian Birkhoff normal form.
This is required   to ensure that the life span of the solutions is of order $\e^{-N-1}$ with $N \in \N$. 
So 
from now on we take $N \in \N$.

In Proposition \ref{teoredu1} we have conjugated 
the  water waves  Hamiltonian system  \eqref{complexo} into \eqref{teo61}, by applying the transformation 
 $ W = \bB(U;t)U  $ which is just linearly symplectic up to homogeneity $N$. Thus the transformed system \eqref{teo61} is not Hamiltonian anymore. 
The first goal of this section is to 
construct a nearby transformation which is symplectic up to homogeneity $ N $, according to Definition \ref{def:LSMN}, thus obtaining a Hamiltonian system up to homogeneity $ N $, 
according to Definition \ref{def:ham.N}.

\subsection{Hamiltonian correction up to homogeneity $N$}

We first prove the following abstract result, which is a 
direct consequence of Theorem \ref{thm:almost}.

\begin{theorem}\label{conjham} 
Let $p, N \in \N $ with   $p \leq N$, $K, K' \in \N_0 $ with 
$K'+1 \leq K$, $r >0$. 
Let 
$ Z = \bM_0(U;t)U $ with $ \bM_0(U;t) \in \cM^0_{K,K',0}[r]\otimes \mM_2(\C) $
as in \eqref{Z.M0}. Assume 
that $ Z(t) $ solves a Hamiltonian system up to homogeneity $N$, 
according to Definition \ref{def:ham.N}. Consider 
\be\label{defPhiUBU}
\Phi(Z):= \bB(Z;t)Z 
\ee
where
\begin{itemize}
\item 
$\bB(Z;t)-\uno$ is a   matrix of spectrally localized maps in
\be \label{c1c2Mteo71}
\bB(Z;t) - \uno  \in \begin{cases}
 \Sigma\cS_{K,K',p}[r,N] \otimes \cM_2(\C) \qquad \qquad 
 \text{if} \quad  \bM_0(U;t)  = \id \, , \\
  \Sigma\cS_{K,0,p}[\breve r,N] \otimes \cM_2(\C) \, , \, \forall \breve r > 0 \quad \text{otherwise} \, . 
\end{cases}
\ee 
\item $\bB(Z;t)$ is linearly symplectic up to homogeneity $N$, 
according to Definition \ref{LSUTHN}.
\end{itemize}
Then there exists a real-to-real matrix of pluri--homogeneous smoothing operators $ R_{\leq N}( \cdot ) $ in $ \Sigma_p^N \wtcR^{-\vr}_q \otimes \mM_2 (\C) $, 
for any $\varrho >0 $,  such that the non-linear map  
$$ 
Z_+ := \big( \uno + R_{\leq N}(\Phi(Z) \big) \Phi(Z)
$$
is symplectic up to homogeneity $ N $ (Definition \ref{def:LSMN}) and thus $Z_+ $
solves a system which is Hamiltonian up to homogeneity $N$.
\end{theorem}
 
\begin{proof}
We decompose 
$ \bB(Z;t) = \bB_{\leq N}(Z) + \bB_{>N}(Z;t) $ where 
$\bB_{\leq N}(Z) := \cP_{\leq N} [\bB(Z;t)] $. 
Note that 
$ \bB_{\leq N}(Z) - \uno  $ is in 
$ \Sigma_p^N\wtcS_q \otimes \mM_2 (\C)    $
and $\bB_{>N}(Z;t) $
is in 
$ \mS_{K,K',N+1}[r] \otimes \mM_2 (\C) $. 
Since $\bB(Z;t)$ is linearly symplectic up to homogeneity $N$,  
its pluri-homogeneous component 
$\bB_{\leq N}(Z) $ is linearly symplectic up to homogeneity $N$ as well. Then 
Theorem \ref{thm:almost} applied to $\Phi_{\leq N}(Z):= \bB_{\leq N}(Z)Z$ implies the existence of pluri-homogeneous smoothing operators $ R_{\leq N}( \cdot ) $ in $  \Sigma_p^N \wtcR^{-\vr}_q \otimes \mM_2 (\C) $ for any  $  \varrho \geq 0 $, 
 such that the nonlinear map 
$ \mD_N(Z):= \big( \uno + R_{\leq N}(\Phi_{\leq N}(Z)) \big)\Phi_{\leq N}(Z) 
$
 is {\it symplectic up to homogeneity} $N$.
We then write
$$
\cD(Z;t):= \big( \uno + R_{\leq N}(\Phi(Z)) \big) \Phi(Z) = \cD_N(Z) + M_{>N}(Z;t)Z
$$ 
where, using  Proposition \ref{composizioniTOTALI}-$(ii)$ and the first bullet after Definition \ref{Def:Maps},  
$$ 
M_{>N}(Z;t) \in 
\begin{cases} 
\cM_{K,K',N+1}[r] \otimes \mM_2 (\C)  \qquad \qquad 
 \text{if} \quad  \bM_0(U;t)  = \id \, , \\
 \cM_{K,0,N+1}[\breve r] \otimes \mM_2 (\C) \, , 
 \, \forall \breve r > 0 \quad \text{otherwise} \, , 
\end{cases}
$$
showing that $\cD(Z;t)$ is symplectic up to homogeneity $N$ as well.
Then Lemma \ref{conj.ham.N} implies the thesis.
\end{proof}

The first application of Theorem \ref{conjham} is to
provide a symplectic correction of the map 
 $\Phi(U):=\bB(U;t)U$ of Proposition  \ref{teoredu1}
 and  to conjugate the 
 Hamiltonian system  \eqref{complexo}  into system 
\eqref{teo62}, which is {\it Hamiltonian up to homogeneity $ N $}.

\begin{proposition}[{\bf Hamiltonian reduction up to smoothing operators}]\label{prop:Z}
Let $N \in \N$ and  {$\vr>c(N) :=3(N+1)+ \frac32 (N+1)^3$.}
Then for any $ K\geq \underline K'$ (fixed in Proposition \ref{teoredu1}) 
there is $s_0 >0, r > 0 $, such that
for any solution $U \in B_{s_0,\R}^K(I;r) $
of  \eqref{complexo},  there exists a real-to-real matrix of pluri--homogeneous smoothing operators $R(U) $ in $ \Sigma_1^N \wtcR^{-\vr'}_q \otimes \mM_2 (\C) $ 
for any $\varrho'\geq 0$, such that defining 
\be \label{zetone}
Z_0:=  \big( \uno+ R( \Phi(U)) \big) \Phi(U),\quad \Phi(U):=\bB(U;t)U \, ,
\ee
where $\bB(U;t)$ is the 
real-to-real matrix of spectrally localized maps   defined in Proposition \ref{teoredu1},
  the following holds true:
  \\[1mm]
  {$(i)$} {\bf Symplecticity:} The non-linear map 
in \eqref{zetone}
  is symplectic up to homogeneity $N$ according to Definition \ref{def:LSMN}.
  \\[1mm]
{$(ii)$} {\bf Conjugation:}
 the variable $Z_0$ solves the {\em Hamiltonian system up to homogeneity $N$}  (cfr. Definition \ref{def:ham.N})  
\be \label{teo62}
\begin{aligned}
\pa_tZ_0 & = - \ii\vOmega(D)Z_0+  \vOpbw{-\ii (\tm_{\frac32})_{\leq N}(Z_0;\xi)- \ii (\tm_{\frac32})_{>N}(U;t,\xi) }Z_0 \\
& \quad  + R_{\leq N}(Z_0)Z_0+ R_{>N}(U;t)U 
\end{aligned}
\ee
where 
\begin{itemize}
\item $\vOmega(D)$ is the diagonal matrix of Fourier multipliers defined in \eqref{diaglin};
\item $( \tm_{\frac32})_{\leq N} (Z_0;\xi)$  is a  real valued symbol, independent of $x $,  in  
 $\Sigma_{2}^N \wt\Gamma^{\frac32}_q$;

\item $( \tm_{\frac32})_{>N} (U;t,\xi)$ is a non--homogeneous symbol, independent of $x$,  in $ \Gamma^{\frac32}_{K,\underline K',N+1}[r]$ with imaginary part $\Im (\tm_{\frac32})_{>N}(U;t,\xi) $ in $  \Gamma^{0}_{K, \underline K',N+1}[r]$;
\item  $R_{\leq N}(Z_0)$  is a real-to-real matrix of smoothing operators in
$\Sigma_{1}^N \wt\cR^{-\varrho+c(N)}_q \otimes \cM_2(\C)$;
\item $R_{>N}(U;t) $ is a real-to-real matrix of non--homogeneous smoothing operators in $ \mR^{-\vr+c(N)}_{K,\underline K',N+1}[r] \otimes \mM_2 (\C) $.
\end{itemize}
{$(iii)$} {\bf Boundedness:}
The variable $Z_0 = {\bf M}_0(U;t)U$ with $\bM_{0}(U;t) \in \mM^0_{K,\underline{K}'-1,0}[r]
 \otimes \mM_2 (\C)  $ and
 for any $s \geq s_0$,  for all $0 < r<r_0(s)$ small enough, 
for  any $U \in B_{s_0}^K(I;r)\cap C_{*\R}^{K}(I;\dot{H}^{s}(\T, \C^2))$,
there is a  constant $C:= C_{s,K}>0$ such that, for all $k=0,\dots, K-\underline K'$,  
\be \label{equivalenzaZU}
C^{-1}\| U\|_{k,s}\leq \|Z_0 \|_{k,s}\leq C \| U\|_{k,s} \, .
\ee
\end{proposition}

\begin{proof}
We  construct the symplectic corrector to the map 
 $ W:=\Phi(U)= \bB(U;t)U$ of Proposition  \ref{teoredu1} by Theorem \ref{conjham}.
Let us check its assumptions. 
By Lemma \ref{LemCompl},  the function $U$ solves the  Hamiltonian system \eqref{complexo}.
By  Proposition \ref{teoredu1}, $\bB(U;t)-\uno $ is a spectrally localized map in $  \Sigma \mS^{\frac32(N+1)}_{K,\underline K'-1,1}[r,N]\otimes \cM_2(\C)$ and $\bB(U;t) $ is linearly symplectic  
up to homogeneity $N$.
So Theorem  \ref{conjham}  (in the case $ {\bf M}_0(U;t) = \uno  $) implies the existence of 
a matrix of pluri--homogeneous smoothing operators 
$ R(W) $ in $ \Sigma_1^N \wtcR^{-\vr'}_q \otimes \mM_2 (\C) $,   for any $\varrho' \geq 0$, 
such that the variable  
\be\label{Z.W}
Z_0 :=  (\uno + R(\Phi(U)))\Phi(U)   =  (\uno+ R(W))W \  
\ee
solves a system which is Hamiltonian up to homogeneity $N$. We now prove that such system has the form \eqref{teo62}.
We will compute it by  transforming  system 
\eqref{teo61} solved by $W(t)$ under the change of variable 
$Z_0 := (\uno+ R(W))W$, exploiting that $R(W)$ is a pluri-homogeneous smoothing operator. 
We first substitute   the variable $U$ with the variable $W$  
in the  homogeneous components up to degree $N$ of both the symbols and the smoothing operator in  \eqref{teo61}. 
We first use Lemma \ref{cor.app.inv} (with $\bM_0(U;t) = \id$ 
and $p=1$) to construct an approximate inverse of $W = \Phi(U)$, getting
\be\label{W.U}
U = \Psi_{\leq N}(W) + M_{>N}(U;t) U \, , 
\qquad \Psi_{\leq N}(W) =W + \breve S_{\leq N}(W)W \, , 
\ee
where $ \breve S_{\leq N}(W) \in \Sigma_1^N \wt \cS_q^{\frac32(N+1)N} \otimes \cM_2(\C)$ and 
$M_{> N}(U;t) $ is a matrix of operators in 
$ \cM_{K,\underline K', N+1}^{\frac32 (N+1)^2}[r]\otimes \cM_2(\C)$.
Next we substitute \eqref{W.U} in the homogeneous components of order $\leq N$ 
in system \eqref{teo61} of 
$$ 
\vOpbw{ \ii \tm_{\frac32}(U;t,\xi)} = 
\vOpbw{ \ii (\tm_{\frac32})_{\leq N}(U) +  \ii (\tm_{\frac32})_{>N}(U;t,\xi)} \, , \  \  \ 
R(U;t) =   R_{\leq N}(U) + R_{>N}(U;t) \, , 
$$ 
  and substitute $W = \bB(U;t)U$ in the term 
$R_{>N}(U;t)W$. 
By  \eqref{teo61},  \eqref{m32}, Lemma \ref{sost} 
(with $Z \leadsto U$, $m'\leadsto \frac32$, $m\leadsto \frac32(N+1)^{{2}}$ and $\vr \leadsto \vr-3(N+1)$) and Proposition \ref{composizioniTOTALI} $(i)$ we obtain
\be\label{W.eq}
\begin{aligned}
\pa_t W & = - \ii\vOmega(D)W+  \vOpbw{-\ii (\wt \tm_{\frac32})_{\leq N}(W;\xi)- \ii (\wt \tm_{\frac32})_{>N}(U;t,\xi) }W \\
& \quad + \wt R_{\leq N}(W)W+ \wt R_{>N}(U;t)U
\end{aligned}
\ee
where\\
\noindent
$\bullet$   $(\wt \tm_{\frac32})_{\leq N} (W;\xi)$  is a  real valued symbol, independent of $x $, in  
 $\Sigma_{2}^N \wt\Gamma^{\frac32}_q$;
 
 \noindent
$\bullet$ $(\wt \tm_{\frac32})_{>N} (U;t,\xi)$ is a non--homogeneous symbol, independent of $x$,  in $ \Gamma^{\frac32}_{K,\underline K',N+1}[r]$ 
given by the sum of the 
old non-homogeneous symbol $\mP_{\geq N+1}(-\tm_{\frac32}(U;t,\xi))$ 
in \eqref{teo61}-\eqref{m32} and a purely real correction coming from formula \eqref{sostbony} (cfr. $a^+_{>N}$) hence its imaginary part $\Im (\wt \tm_{\frac32})_{>N}(U;t,\xi) $ is in $ \Gamma^{0}_{K,\underline K',N+1}[r]$;

\noindent
$\bullet$  $\wt R_{\leq N}(W)$  is a matrix of pluri-homogeneous smoothing operators in
$\Sigma_{1}^N \wt\cR^{-\varrho+c(N)}_q \otimes \cM_2(\C)$ with $c(N)=3(N+1)+\frac32(N+1)^3$;

\noindent 
$\bullet$ $\wt R_{>N}(U;t) $ is  a matrix of  non--homogeneous smoothing operators in $ \mR^{-\vr+c(N)}_{K,\underline K',N+1}[r] \otimes \mM_2 (\C) $.

We finally conjugate system \eqref{W.eq} 
under the change of  variable $Z_0 = W + R(W)W$ defined in \eqref{Z.W}.
 Note that system \eqref{W.eq} fulfills  Assumption {\bf (A)} at pag  \pageref{A} with $p \leadsto 1$, with $W(t)$ replacing $Z(t)$, $\bM_0(U;t)\leadsto \bB(U;t)$ and $ \vr \leadsto \vr -c(N)$.
Then we apply Lemma \ref{NormFormMap0} with the 
smoothing perturbation of the identity 
defined in \eqref{Z.W} (choosing also $ \vr':= \vr+\frac32$) and we deduce 
that $Z(t)$ satisfies system  \eqref{teo62}. 
Item $(iii)$ follows from \eqref{zetone},
the fact that  $\bB(U;t) \in   \mS_{K,\underline K'-1,0}^0 [r]\otimes \mM_2 (\C) $ (Proposition \ref{teoredu1} $(i)$), the fact that $\id + R(Z) \in \mM^0_{K,0,0}[\breve r] \otimes \mM_2 (\C)  $ for any $\breve r>0$ (by Lemma \ref{mappabonetta} and since $R(Z)$ is pluri-homogeneous) and by  Proposition \ref{composizioniTOTALI} items $(iii)$ 
(with $K' \leadsto  \underline{K'}-1$)  and $(i)$.
Finally estimate \eqref{equivalenzaZU} follows combining also \eqref{mappaBflow} and the estimate below \eqref{piovespect} for $ \bB(U;t)$ and $\bB(U;t)^{-1}$.
\end{proof}

\subsection{Super action preserving symbols and Hamiltonians} 

In this section we define the special class of ``super--action preserving" 
$\tSAP$
homogeneous symbols and Hamiltonians 
which will appear in the Birkhoff normal form reduction of the next
Section \ref{sec:BNF}.

\begin{definition}\label{def:SAPindex}
{\bf ($\tSAP$ multi-index)} 
A multi-index 
$(\alpha, \beta) \in \N_0^{\Z\setminus \{0\}} \times \N_0^{\Z\setminus \{0\}} $ 
is {\em super action preserving } if 
\begin{equation} \label{sap}
\alpha_n + \alpha_{-n} = \beta_n + \beta_{-n} \ , \qquad \forall n \in \N \, .
\end{equation}
\end{definition}
A super action preserving multi-index $(\alpha,\beta) $  satisfies $ |\alpha| = |\beta|$ where
$ |\alpha | :=  \sum_{j \in \Z\setminus\{0\}} \alpha_j  $.
If a multi-index $(\alpha, \beta) \in  \N_0^{\Z\setminus \{0\}} \times \N_0^{\Z\setminus \{0\}} $ is not super action preserving, then the set 
\be\label{ennonegotico}
\mathfrak N(\alpha, \beta) := 	\Big\{ n \in \N \, \colon \, 
\alpha_{n} + \alpha_{-n} - \beta_n  -\beta_{-n} \neq 0 \Big\}  
\ee
is not empty 
and, since $ \mathfrak N(\alpha, \beta) \subset \{ n \in \N \, : \, 
\alpha_{n} + \alpha_{-n} + \beta_n  +\beta_{-n}  \neq 0 \} $,  its cardinality satisfies
\be\label{cardNab}
 | \mathfrak N(\alpha, \beta) | \leq |\alpha + \beta| = |\alpha | + | \beta | \, . 
\ee

\begin{definition} {\bf ($\tSAP$  monomial)} \label{passaggioalpha}
Let  $p\in \N$. Given $ (\vec \jmath , \vec \sigma)= (j_a,\sigma_a)_{a=1,\dots, p} \in 
(\Z \setminus \{0\})^p \times \{ \pm \}^p$ we define the multi-index $ (\alpha,\beta) \in \N_0^{\Z \setminus \{0\}} 	\times \N_0^{\Z \setminus \{0\}}  $ with components, for any $k \in \Z \setminus \{0\}$, 
\be\label{defalbe}
\begin{aligned}
& \alpha_k(\vec \jmath, \vec \sigma):= \# \big\{ a = 1, \ldots ,p \, : \, (j_a , \sigma_a) = (k,+) \big\} \, , \\
&  \beta_k(\vec \jmath, \vec \sigma):= \# \big\{ a = 1, \ldots ,p \, : \, (j_a , \sigma_a) = (k,-) \big\} \, .
\end{aligned}
\ee
We say that a monomial of the form 
$ z_{\vec \jmath}^{\vec \sigma} = z_{j_1}^{\sigma_1}\dots z_{j_p}^{\sigma_p} $ is super-action preserving if the associated multi-index $ (\alpha,\beta)= (\alpha(\vec \jmath, \vec \sigma),\beta(\vec \jmath, \vec \sigma))$ is super-action preserving according to Definition \ref{def:SAPindex}.
\end{definition}

We now introduce the subset $ \mathfrak S_p  $ of the indexes of $\fT_p$ defined in \eqref{fTset}
 composed by super-action preserving indexes
\be\label{fSp}
\mathfrak S_p  :=  \Big\{(\vec \jmath , \vec \sigma) \in \fT_p 
\, : \, (\alpha(\vec \jmath, \vec \sigma), \beta(\vec \jmath, \vec \sigma)) \in \N_0^{\Z \setminus \{0\}} 
\times \N_0^{\Z \setminus \{0\}} \ \text{in} \ \eqref{defalbe} 
\ \text{are  super  action   preserving} \Big\} \, . 
\ee
We remark that the multi-index $(\alpha, \beta)$ associated to 
$ (\vec \jmath, \vec \sigma) \in (\Z\setminus \{ 0\} \times \{\pm\})^p$ as in \eqref{defalbe} satisfies  $|\alpha+\beta|=p$ and 
\be \label{zjalpha}
z_{\vec \jmath}^{\vec \sigma}= z^\alpha \bar z^\beta:= \prod_{j \in \Z \setminus \{0\} } z_j^{\alpha_j}{\ov{z_j}}^{\beta_j} = \prod_{n \in \N } z_n^{\alpha_n}z_{-n}^{\alpha_{-n}}
{\ov{z_n}}^{\beta_n}\ov{z_{-n}}^{\beta_{-n}}\,  \, .
\ee
It turns out
\be \label{defOmegakappa00}
\vec \sigma \cdot  \Omega_{\vec \jmath} \, (\kappa)
= 
\sigma_1 \Omega_{j_1}(\kappa)+\dots+ \sigma_p\Omega_{j_p}(\kappa)= 
 (\alpha- \beta) \cdot \vec \Omega(\kappa) = \sum_{k \in \Z  \setminus \{0\}} (\alpha_k- \beta_k)  \Omega_k(\kappa) \, ,
\ee
where we denote 
\be\label{defOmegakappa} 
\vec{\Omega}(\kappa):= \{ \Omega_j(\kappa)\}_{j\in \Z\setminus\{0\}}, \quad  \Omega_{\vec \jmath} \, (\kappa):= ( \Omega_{j_1}(\kappa), \dots, \Omega_{j_p}(\kappa) ) \, .
\ee
\begin{remark}
In view of \eqref{zjalpha} and \eqref{sap} a super action monomial has either the 
integrable  form $ |z_{j_1}|^2 \ldots |z_{j_m}|^2 $ or 
the one described in \eqref{nonintm}  (with not necessarily distinct indexes $j_1, \ldots, j_m$).
\end{remark}

\begin{remark}\label{remino}
If the monomial $ z_{\vec \jmath}^{\vec \sigma}$ is super--action preserving then, for any $j \in \Z \setminus \{0\} $,  the monomial $ z_{\vec \jmath}^{\vec \sigma}z_j \bar z_j$ is super-action preserving as well.
\end{remark}

\smallskip
For any $n \in \N$ we define the {\em super action}
\begin{equation}
\label{sa}
J_n := |z_n|^2 + |z_{-n}|^2 \, .
\end{equation}

\begin{lemma}\label{lemma:Poissonbra}
The Poisson bracket between a  monomial $z_{\vec \jmath}^{\vec \sigma} $ 
and a super-action  $ J_n $, $ n \in \N $,  defined  
 in \eqref{sa},  is 
 \be\label{poisuperaction}
\{ z_{\vec \jmath}^{\vec \sigma}, J_n \} = 
\ii \big(\beta_n + \beta_{-n} - \alpha_n - \alpha_{-n}  \big) \, z_{\vec \jmath}^{\vec \sigma} \, ,
\ee
where $ (\alpha, \beta)= (\alpha(\vec \jmath, \vec \sigma),\beta(\vec \jmath, \vec \sigma))$ is the multi-index defined in \eqref{defalbe}. 
In particular a super action preserving 
monomial $ z_{\vec \jmath}^{\vec \sigma} $
(according to Definition \ref{passaggioalpha})  Poisson commutes with any super action $ J_n $, $ n \in \N $. 
\end{lemma}

\begin{proof}
We write the monomial $ z_{\vec \jmath}^{\vec \sigma} = z^\alpha \bar z^\beta $
as in \eqref{zjalpha}. Then,  
for any $n\in \N$ and $j,k\in \Z \setminus \{0\} $, one has 
\be\label{poissonsa}
\begin{aligned}
&\pa_{z_j} (z_{\vec \jmath}^{\vec \sigma})= \alpha_j z_j^{\alpha_j-1} { \ov{z_j}^{\beta_j}  } \prod_{k \not= j } z_k^{\alpha_k} \bar{z_k}^{\beta_k}, \quad \pa_{\bar{ z_j}} 
(z_{\vec \jmath}^{\vec \sigma})= 
\beta_j z_j^{\alpha_j}  \bar{z_j}^{\beta_j-1}  \prod_{k \not= j } z_k^{\alpha_k} \bar{z_k}^{\beta_k} \, , 
\\
&\pa_{z_j} J_n= \begin{cases}
 \bar{z_j} & j=\pm n \\ 0 & j\not= \pm n \, , 
\end{cases}\quad \pa_{\bar{ z_j}} J_n= \begin{cases} z_j & j=\pm n \\ 0 & j\not= \pm n \, . 
\end{cases}
\end{aligned}
\ee
Then  by \eqref{Pois} and \eqref{poissonsa} we deduce 
\eqref{poisuperaction}.
\end{proof}

We now define a super action preserving  Hamiltonian.  

\begin{definition}\label{sapham} {\bf ($\tSAP$ Hamiltonian)}  Let $ p \in \N_0 $. 
A $(p+2)$--homogeneous super action preserving Hamiltonian $H^{(\tSAP)}_{p+2}(Z)$ is a real function of the form 
$$
H^{(\tSAP)}_{p+2}(Z)=\frac{1}{p+2} \sum_{\substack{(\vec \jmath_{p+2},\vec \sigma_{p+2})\in\mathfrak S_{p+2}} } H_{\vec \jmath_{p+2}}^{\vec \sigma_{p+2}} z_{\vec \jmath_{p+2}}^{\vec \sigma_{p+2}} 
$$
where $ \mathfrak{S}_{p+2}$ is defined as in \eqref{fSp}.
A pluri-homogeneous  super action preserving Hamiltonian is a finite sum of  
homogeneous super action preserving Hamiltonians. 
A Hamiltonian vector field is super action preserving if it is generated by a
super action preserving Hamiltonian.
\end{definition}

We now define a super action preserving symbol. 

\begin{definition}\label{sapsym} {\bf ($\tSAP$ symbol)} 
 Let $ p \in \N_0 $ and $m \in \R$. 
For $ p \geq 1 $ a  real valued,  $p$--homogeneous {\em super action preserving symbol} of order $ m  $ is a symbol $\tm_p^{({\tSAP })}(Z;\xi) $ in $ \wt\Gamma^m_p$, independent of $x$, of the form 
\be \label{symSAP}
\tm_p^{({\tSAP })}(Z;\xi)=
\sum_{({\vec{\jmath}_p},{\vec{\sigma}_p})\in \mathfrak S_p} M_{\vec{\jmath}_p}^{\vec{\sigma}_p}(\xi) z_{\vec{\jmath}_p}^{\vec{\sigma}_p } \, .
\ee 
For $ p = 0 $ we say that any  symbol in  $ \wt\Gamma^m_0 $ is 
 super action preserving.
A pluri-homogeneous  super action preserving symbol is a finite sum of  
homogeneous super action preserving symbols.
\end{definition}

\begin{remark}\label{nonsuper}
A super action preserving symbol has even degree $p$ of homogeneity.
Indeed, if $z_{\vec{\jmath}_p}^{\vec{\sigma}_p} $ is super-
action preserving then $(\alpha,\beta)$ defined in \eqref{defalbe} satisfies 
$ |\alpha| = |\beta|$ and $ p= |\alpha+\beta| = 2| \alpha|$ is even.
\end{remark}

Given a super action preserving symbol we associate a 
super action preserving  Hamiltonian according to the following lemma. 

\begin{lemma}\label{lem:ham.sap}
Let $ p \in \N_0 $, $ m \in \R $.  
If $(\tm^{({\tSAP })})_p(Z;\xi)$ is a $p$--homogeneous super action preserving symbol 
in $ \wt\Gamma^m_p $ according to Definition \ref{sapsym} then 
$$
H^{(\tSAP)}_{p+2} (Z):=\Re \Big\langle \Opbw{(\tm^{({\tSAP })})_p(Z;\xi)} z , \bar z\Big \rangle_{\dot L^2_r}
$$
is a $(p+2)$--homogeneous super action preserving Hamiltonian according to Definition \ref{sapham}.
\end{lemma}

\begin{proof}
By the expression \eqref{symSAP} and  \eqref{BW} we have 
$$
 \int_{\T}\Opbw{\tm_p^{({\tSAP })}(Z;\xi)} \Pi_0^\bot z\cdot \Pi_0^\bot \bar z\, \di x= \sum_{j\in \Z \setminus \{0\}} \sum_{(\vec{\jmath}_p, \vec{\sigma}_p) \in \mathfrak S_p} \chi_p (\vec \jmath_{p},j) M_{\vec{\jmath}_p}^{\vec{\sigma}_p}(j) z_{\vec{\jmath}_p}^{\vec{\sigma}_p} z_j \ov {z_j}
$$
where $\mathfrak S_p$ is defined in 
\eqref{fSp}.
 Then  Remark \ref{remino} implies the thesis. 
 For $ p = 0 $ the 
 Hamiltonian $ H^{(\tSAP)}_{2} (Z) $
 is a series  of integrable monomials $ z_j \ov{z_j} $. The proof of the lemma is complete. 
\end{proof}

\subsection{Birkhoff normal form reduction}\label{sec:BNF}

In this section we  finally  
transform system \eqref{teo62} into its  
Hamiltonian Birkhoff normal form,  up to homogeneity $ N $.

\begin{proposition} {\bf (Hamiltonian Birkhoff normal form)} \label{birkfinalone}
Let $N \in \N$. 
Assume that, for any value of the gravity $ g > 0 $,  
vorticity $ \gamma \in \R $ and depth $ \tth \in  (0,+\infty] $, 
the surface tension coefficient $ \kappa $ is outside the zero measure set
$ {\cal K} \subset (0,+\infty) $ defined in Theorem \ref{nonresfin0}. 

Then there exists $ \underline \varrho$ (depending on $N$) such that, for any $ \varrho \geq  \underline \varrho $, for any  $ K \geq  \underline K' (\varrho) $ (defined in Proposition \ref{teoredu1}),
there exists $ \underline s_0 > 0 $ such that, for any $ s \geq \underline s_0 $ there is $ \underline r_0:=  \underline r_0(s)>0$ such that  for all $ 0 < r < \underline r_0(s) $ small enough, and any solution $ U \in B_{\underline s_0}^K (I; r)\cap C^K_{*\R}(I; \dot H^{s}(\T;\C^2))$ of the water waves system \eqref{complexo},
there exists a 
non--linear map $ \mF_{\tnf}(Z_0)$  
such that:
\\[1mm]
$(i)$ {\bf Simpleticity:}  $ \mF_{\tnf}(Z_0)$ is symplectic up to homogeneity $N$ (Definition \ref{def:LSMN});
\\[1mm]
$(ii)$ {\bf Conjugation:} If $ Z_0 $ solves the system \eqref{teo62} then the variable $ Z:= \mF_{\tnf}(Z_0)$ solves the {\em Hamiltonian system up to homogeneity $N$}  (cfr. Definition \ref{def:ham.N})
\be\label{final:eq}
\begin{aligned}
\pa_tZ & = -\ii \vOmega(D) Z +  J_c \nabla H^{(\tSAP)}_{\frac32}(Z)+J_c \nabla H^{(\tSAP)}_{-\vr}(Z) \\
& \quad +\vOpbw{- \ii (\tm_{\frac32})_{>N}(U;t,\xi)}Z+ R_{>N}(U;t)U 
\end{aligned}
\ee
where 
\begin{itemize}
\item $H^{(\tSAP)}_\frac32(Z) $ is the super action preserving Hamiltonian 
$$
\Re \Big \langle \Opbw{ (\tm_{\frac32}^{(\tSAP)})_{\leq N}(Z;\xi)}z , \bar z\Big \rangle_{\dot L^2_r}
$$
with a pluri homogeneous super action preserving  symbol 
$ ( \tm_{\frac32}^{(\tSAP)})_{\leq N}(Z;\xi)$  in $\Sigma_2^N \wt\Gamma_q^{\frac32}$,
according  to  Definition \ref{sapsym};
\item  $J_c \nabla H^{(\tSAP)}_{-\vr}(Z)$
 is a super action preserving,  Hamiltonian, smoothing vector field in $ \Sigma_{3}^{N+1} \wt \X^{-\vr+ \underline \varrho}_q$ (see Definitions \ref{def:X} and \ref{sapham});
\item  $(\tm_{\frac32})_{>N}(U;t,\xi)$ is a non--homogeneous symbol in $ \Gamma^{\frac32}_{K,\underline K',N+1}[r]$ with imaginary part $\Im (\tm_{\frac32})_{>N}(U;t,\xi) $ in $ \Gamma^{0}_{K,\underline K',N+1}[r]$;
\item  $R_{>N}(U;t)$ is a real-to-real matrix of non--homogeneous smoothing operators 
 in $\mR^{-\vr+ \underline \varrho}_{K,\underline K',N+1}[r] \otimes\mathcal{M}_2(\C) $.
  \end{itemize}
$(iii)$ {\bf Boundedness:} there exists $C:=C_{s,K}>0$ such that   for all  $ 0\leq k\leq  K$ and  any $Z_0\in B_{\underline s_0}^K(I;r)\cap C^{K}_{*\R}(I;\dot H^{s}(\T, \C^2))$  one has 
\be\label{mappaB}
C^{-1} \| Z_0 \|_{k,s}\leq \|\mF_{\tnf}(Z_0) \|_{k,s} \leq 
C\| Z_0 \|_{k,s} \, .
\ee 
and 
 \be\label{equivalenzan}
 {C}^{-1} \| U(t) \|_{\dot H^s} \leq \| Z (t) \|_{\dot H^s} \leq C \| U(t) \|_{\dot H^s} \, , 
\quad \forall  t \in I   \, .
 \ee
\end{proposition}

\begin{proof}
We divide the proof in $N$ steps. At the $p$-th step, $ 1 \leq p \leq N $, 
we reduce the $p$--homogeneous  component of the Hamiltonian vector field which appears in the equation to its super action preserving part, up to higher homogeneity terms.
\\[1mm]
\noindent {\bf Step $1$: Elimination of the quadratic smoothing remainder in equation \eqref{teo62}.}
\\[1mm]
The $ x$-independent symbol  $(\tm_{\frac32})_{\leq N}(Z_0;\x)$ in \eqref{teo62}  belongs to $ \Sigma_2^N \wt \Gamma^{\frac32}_q$ and the only quadratic component of the vector field in \eqref{teo62} is
$R_1(Z_0)Z_0$ where 
\be \label{R1}
R_1(Z_0):= \cP_1[R_{\leq N}(Z_0)] 
 \in \wt \cR_1^{-\varrho+c(N)} \otimes \cM_2(\C) \ . 
 \ee
 Since system  \eqref{teo62} is  Hamiltonian  up to homogeneity $N$,  $R_1(Z_0)Z_0$ is a Hamiltonian vector field in $\wt \X^{-\varrho + c(N)}_2$
that we expand in Fourier coordinates 
as in \eqref{polvect}
\be\label{R1.exp}
\big(R_1(Z_0)Z_0\big)_k^\sigma= \sum_{
\substack{
(j_1, j_2, k, \sigma_1, \sigma_2, - \sigma) \in \mathfrak T_3
}
} X_{j_1,j_2,k}^{\sigma_1,\sigma_2,\sigma} 
(z_0)_{j_1}^{\sigma_1} (z_0)_{j_2}^{\sigma_2}  \,  .
\ee
In order to remove $ R_1 (Z_0) Z_0 $ from equation \eqref{teo62} 
we perform the change of variable 
 $ Z_1 = \mathtt{F}^{(1)}_{\leq N}(Z_0)$
 where $ \mathtt{F}^{(1)}_{\leq N}(Z_0)$ 
 is the time $1$-approximate flow, 
 given by Lemma \ref{extistencetruflow}, generated by  
the smoothing vector field 
 \be\label{G1}
\big(G_1(Z_0)Z_0\big)_k^\sigma= \sum_{
(j_1, j_2, k, \sigma_1, \sigma_2, - \sigma) \in \fT_3
} G_{j_1,j_2,k}^{\sigma_1,\sigma_2,\sigma} (z_0)_{j_1}^{\sigma_1} (z_0)_{j_2}^{\sigma_2} 
\ee
with  
\be\label{G1+.coeff}
G_{j_1,j_2,k}^{\sigma_1,\sigma_2,\sigma}:= 
\begin{cases} 0 &   \mbox{ if } \ 
(j_1,j_2,k,\sigma_1,\sigma_2,-\sigma) \notin \fT_3 \\
 \displaystyle 
 {\frac{ X_{j_1,j_2,k}^{\sigma_1,\sigma_2,\sigma}}{\ii\big(\sigma_1\Omega_{j_1}(\kappa)+\sigma_2 \Omega_{j_2}(\kappa)- \sigma \Omega_{k}(\kappa)\big)}} &  \mbox{ if } 
 (j_1,j_2,k,\sigma_1,\sigma_2,-\sigma) \in \fT_3 \, .
\end{cases} 
\ee

\begin{lemma}
Let $\kappa \in (0,+\infty) \setminus \cK$. 
Then the vector field 
$G_1(Z_0)Z_0$ in \eqref{G1}, \eqref{G1+.coeff} is a well defined Hamiltonian vector field in $\wt \X^{-\varrho'}_2$ with   $\vr' := \vr - c(N) - \tau$ and where  $\tau$ is defined in Theorem \ref{nonresfin0}.
\end{lemma}

\begin{proof}
We claim  that  for any $\kappa \in (0, + \infty) \setminus \cK$ 
there exist $ \tau, \nu>0$ such that 
\be\label{lowerj}
| \sigma_1\Omega_{j_1}(\kappa)+\sigma_2 \Omega_{j_2}(\kappa)- \sigma \Omega_{k}(\kappa)| > \frac{\nu}{\max\{ |j_1|,|j_2|,|k|\}^\tau} \, ,
\quad \forall (j_1,j_2,k,\sigma_1,\sigma_2,-\sigma) \in \fT_3  \, . 
\ee
Indeed,    
to  any $ (j_1,j_2,k,\sigma_1,\sigma_2,-\sigma)   $ 
we associate the multi--index 
$(\alpha,\beta)$ as in \eqref{defalbe} whose length is $|\alpha+\beta|= 3 $ and satisfies  
$
\sigma_1\Omega_{j_1}(\kappa)+\sigma_2 \Omega_{j_2}(\kappa)- \sigma \Omega_{k}(\kappa)= (\alpha-\beta)\cdot \vec \Omega(\kappa) $ by \eqref{defOmegakappa00}.
 Having length 3, by Remark \ref{nonsuper}, the multi-index $(\alpha,\beta)$ is not super--action preserving and  
therefore 
Theorem \ref{nonresfin0} implies \eqref{lowerj}. 
In view of \eqref{lowerj}  the coefficients in \eqref{G1+.coeff} are well defined.

Next we  show that $G_1(Z_0)Z_0$ is a vector field in $\wt\X^{-\vr'}_{2}$. 
As $R_1(Z_0)Z_0 $ belongs to $ \wt \X^{-\varrho + c(N)}_2$, 
 by Lemma \ref{lem:X.R} the coefficients
 $X_{j_1, j_2, k}^{ \sigma_1, \sigma_2, \sigma}$ in \eqref{R1.exp}
satisfy the 
symmetric and reality 
properties \eqref{symmetric}, \eqref{X.real} and the estimate: 
   for some  $ \mu \geq 0 $, $ C>0 $, 
\be\label{smoocara2}
 |X_{j_1, j_2, k}^{ \sigma_1, \sigma_2, \sigma} |\leq C 
\frac{\max_2\{ |j_1|, | j_2|\}^{\mu}}{\max\{ |j_1|,| j_2|\}^{\vr-c(N)}} \, , \qquad
\forall (j_1, j_2, k, \sigma_1, \sigma_2, - \sigma) \in \fT_3
\, . 
 \ee
Hence also the
coefficients $G_{j_1,j_2,k}^{\sigma_1,\sigma_2,\sigma}$
in \eqref{G1+.coeff}
 fulfill the
symmetric, reality  
properties \eqref{symmetric}, \eqref{X.real} as well as  $ X_{j_1,j_2,k}^{\sigma_1,\sigma_2,\sigma} $.
Moreover, using \eqref{smoocara2},  \eqref{lowerj}  
and the momentum relation  $ \sigma k  = \sigma_1 j_1 + \sigma_2 j_2 $, they also satisfy 
$$
| G_{j_1,j_2,k}^{\sigma_1,\sigma_2,\sigma}|\leq C \frac{\max_2\{ |j_1|, | j_2|\}^{\mu}}{\max\{ |j_1|, | j_2|\}^{\vr-c(N)-\tau}}
$$
for a new constant $C >0$ (depending on $\nu$).
Then  Lemma \ref{lem:X.R} implies that  $ G_1(Z_0)Z_0 
$ belongs to $ \wt \X^{-\varrho'}_2$ with $ \vr':= \vr-c(N)- \tau$. 

Finally we show that 
 $G_1(Z_0)Z_0$ is   Hamiltonian.
Recall that $R_1(Z_0)Z_0$ in \eqref{R1.exp} is a Hamiltonian vector field whose Hamiltonian function
 $H_{R_1}(Z_0) $ is, thanks to  Lemma \ref{carhamvec},  
\be\label{HR1}
 H_{R_1}(Z_0)= \tfrac{1}{3} 
 \!\!\! \!\!\! \!\!\! \!\!\!\!\!\!\!\!\!\!
 \sum_{(j_1, j_2, j_3, \sigma_1, \sigma_2,  \sigma_3) \in \fT_3} \!\!\!\!\!\!\!\!\!\!\!\!\!\!\!\!
[H_{R_1}]_{j_1,j_2,j_3}^{\sigma_1,\sigma_2,\sigma_3} 
 (z_0)_{j_1}^{\sigma_1} (z_0)_{j_2}^{\sigma_2} (z_0)_{j_3}^{\sigma_3} \, , \qquad
 [H_{R_1}]_{j_1,j_2,j_3}^ {\sigma_1,\sigma_2,\sigma_3} := - \ii \sigma_3 X_{j_1,j_2,j_3}^{\sigma_1,\sigma_2,-\sigma_3}  \, . 
\ee 
 Then 
  the coefficients defined for  $\sigma_1 j_1 + \sigma_2 j_2 + \sigma_3  j_3 = 0$ by 
\be\label{HG1}
[H_{G_1}]_{j_1,j_2,j_3}^ {\sigma_1,\sigma_2,\sigma_3}
:= - \ii \sigma_3  G_{j_1,j_2,j_3}^{\sigma_1,\sigma_2,-\sigma_3}
\stackrel{\eqref{G1+.coeff}, \eqref{HR1}}{=} \frac{ [H_{R_1}]_{j_1,j_2,j_3}^ {\sigma_1,\sigma_2,\sigma_3}}{\ii\big(\sigma_1\Omega_{j_1}(\kappa)+\sigma_2 \Omega_{j_2}(\kappa)+\sigma_3 \Omega_{j_3}(\kappa)\big)} 
\ee
satisfy the symmetric, reality properties  \eqref{symham},  \eqref{realham}  as well as 
the coefficients $[H_{R_1}]_{j_1,j_2,j_3}^ {\sigma_1,\sigma_2,\sigma}$.
Then Lemma \ref{carhamvec} implies that 
 $G_1(Z_0)Z_0$ is the Hamiltonian vector field generated by the Hamiltonian  $H_{G_1}$  with coefficients defined in \eqref{HG1}.
\end{proof}

We now conjugate system \eqref{teo62} 
by  the approximate time 1-flow $\mathtt{F}^{(1)}_{\leq N}(Z_0) $ generated by 
$ G_1(Z)Z $ provided  by Lemma 
\ref{extistencetruflow}, which has the form 
\be\label{mFleqN}
Z_1 :=\mathtt{F}^{(1)}_{\leq N}(Z_0)  = 
Z_0 + F_{\leq N}(Z_0) Z_0 \, , \quad F_{\leq N}(Z_0) \in \Sigma_{1}^N \wt \cR^{-\varrho'}_q \otimes  \cM_2(\C) \, .
\ee
Since $G_1(Z)Z$ is a Hamiltonian vector field, by 
Lemma 
\ref{lem:app.flow.ham}  
the approximate flow $\mathtt{F}^{(1)}_{\leq N}$ is symplectic up to homogeneity $N$.
Applying  Lemma \ref{conj.ham.N} (with $Z \leadsto Z_0$, $W \leadsto Z_1$ and $\bM_0(U;t) \in \cM_{K,\underline K'-1, 0}^0[r] \otimes \cM_2(\C)$),  we obtain that the variable $Z_1$ solves a  system which is 
 Hamiltonian up to homogeneity $N$. 
We compute it using  Lemma \ref{NormFormMap0}.
Its assumption {\bf (A)} at page \pageref{A} holds since $ Z_0  $ solves \eqref{teo62}
 (with $ a_{\leq N}  = -  (\tm_{\frac32})_{\leq N} $, $ K' = \underline{K}' $ and 
$ \varrho \leadsto \varrho - c(N)$).
  Then 
  Lemma \ref{NormFormMap0} (with $ W \leadsto Z_1 $, $ p = 1 $ and $ \vr' = \vr - c(N) - \tau $) implies that the variable $Z_1$  solves 
\be\label{Z1}
\begin{aligned}
\pa_t Z_1
& = -\ii \vOmega(D) Z_1+ \vOpbw{-\ii (\tm_{\frac32})_{\leq N}^+(Z_1;\xi)- \ii (\tm_{\frac32})_{>N}^+(U;t,\xi)}Z_1 \\
& \quad +[ R_{1 }(Z_1) + G^+_1(Z_1)]Z_1 
+ R_{\geq 2}^+(Z_1)Z_1 +  R^+_{>N}(U;t)U
\end{aligned}
\ee
where \\
\noindent 
$\bullet$
$ (\tm_{\frac32})_{\leq N}^+(Z_1;\xi)$ is a real valued symbol, independent of $x$, in $\Sigma_{2}^N \wt \Gamma_q^{\frac32}$; \\
\noindent
$\bullet$
$(\tm_{\frac32})_{>N}^+(U;t,\xi)$ is a non-homogeneous real valued symbol, independent of $x$, in $\Gamma^\frac32_{K,\underline K',N+1}[r]$ with imaginary part ${\rm Im} (\tm_{\frac32})_{>N}(U; t, \xi) $ in $  \Gamma^0_{K,\underline K', N+1}[r]$; \\
\noindent
$\bullet$
$R_1(Z_1)$ is defined in \eqref{R1}  and $G_1^+(Z_1)Z_1\in \wt \X^{-\vr'+\frac32}_{2}$ has  Fourier expansion, by \eqref{ditigi40} and \eqref{G1},    
\be\label{G1+}
(G^+_1(Z_1)Z_1)_{k}^\sigma=\!\!\! \!\!\! \!\!\! \sum_{
(j_1, j_2, k, \sigma_1, \sigma_2, - \sigma) \in \fT_3
} 
\!\!\! \!\!\! \!\!\! 
 - \ii\big(\sigma_1\Omega_{j_1}(\kappa)+\sigma_2 \Omega_{j_2}(\kappa)- \sigma \Omega_{k}(\kappa)\big)G_{j_1,j_2,k}^{\sigma_1,\sigma_2,\sigma} (z_1)_{j_1}^{\sigma_1} (z_1)_{j_2}^{\sigma_2};
\ee
\noindent
$\bullet$
$ R_{\geq 2}^+(Z_1) $ is a matrix of pluri-homogeneous smoothing operators in  $ 
\Sigma_2^N \wtcR^{-\vr + \underline \varrho(2)}_q \otimes \mM_2 (\C) $ where
\be\label{c2}
\underline \varrho(2) := c(N) + \tau + \tfrac32 \ ; 
\ee

\noindent
$\bullet$ 
$R_{>N}^+(U;t) $ is a matrix of non--homogeneous smoothing operators in $ \mR^{-\vr + \underline \varrho(2)}_{K,\underline K',N+1}[r] \otimes \mM_2 (\C) $.

By \eqref{R1.exp}, \eqref{G1+}, \eqref{G1+.coeff} we have 
\begin{equation}
\label{homologgica1}
R_{1}(Z_1)Z_1 + G_1^+(Z_1)Z_1 = 0  \ . 
\end{equation}
\noindent {\bf Step $p\geq 2$:} 
We claim the following inductive  statements hold true.  
Let $Z_0 $ solve \eqref{teo62}. Then for any $ p \geq 2 $ \\

\noindent 
{\em 
{\bf (S0$)_p$}
There is a transformation $\mF_{\leq N}^{(p-1)}(Z_0)$, symplectic up to homogeneity $ N$, fulfilling 
 $(iii)$  of Proposition \ref{birkfinalone} (with $C = 2\times 8^{p-2}$ in \eqref{mappaB}) such that the variable  $Z_{p-1}=\mF_{\leq N}^{(p-1)}(Z_0)$ has the form $Z_{p-1}= \bM_{0}^{(p-1)}(U;t)U$ with $\bM_{0}^{(p-1)}(U;t) \in \mM^0_{K,\underline{K}'-1,0}[r] 
 \otimes \cM_2(\C)  $ and
 solves the  system 
\begin{align}
\pa_t Z_{p-1} & = -\ii \vOmega(D)Z_{p-1}+J_c \nabla 
\big(H^{(\tSAP )}_{\frac32}\big)_{\leq p+1}(Z_{p-1}) +
J_c \nabla \big(H^{(\tSAP)}_{-\vr}\big)_{\leq p+1}(Z_{p-1}) \notag \\
& \ + \vOpbw{-\im (\tm_{\frac32})_{p}(Z_{p-1};\xi)-\im (\tm_{\frac32})_{\geq p+1}(Z_{p-1};\xi)}Z_{p-1}  +R_{\geq p}(Z_{p-1})Z_{p-1} \notag \\
&\ +\vOpbw{-\im (\tm_{\frac32})_{>N}(U;t,\xi)}Z_{p-1}+ R_{>N}(U;t)U \label{traspBNF}
\end{align}
where \\
\noindent
 {\bf (S1$)_p$}   $ \big(H^{(\tSAP )}_{\frac32}\big)_{\leq p+1}(Z_{p-1}) $
is the 
real valued Hamiltonian
\be\label{H.SAP.p1}
\big(H^{(\tSAP )}_{\frac32}\big)_{\leq p+1}(Z_{p-1}):= \Re \Big \langle  \Opbw{ (\tm_{\frac32}^{(\tSAP)})_{\leq {p-1}}(Z_{p-1};\xi)}z_{p-1},  \bar {z}_{p-1}\Big \rangle_{\dot L^2_r} 
\ee
with a super action preserving symbol
 $ (\tm_{\frac32}^{(\tSAP)})_{\leq {p-1}}(Z_{p-1};\xi) $ in $  \Sigma_2^{p-1} \wt \Gamma^{\frac32}_q$ (see Definition \ref{sapsym});
  its Hamiltonian vector field is given by
\be\label{ham.sap8}
J_c \nabla  
\big(H^{(\tSAP )}_{\frac32}\big)_{\leq p+1}(Z_{p-1})  =  \vOpbw{- \im (\tm_{\frac32}^{({\tSAP })})_{\leq p-1}(Z_{p-1};\xi)} Z_{p-1} + R_{\leq p-1}(Z_{p-1}) Z_{p-1}
\ee
with $R_{\leq p-1}(Z_{p-1}) \in \Sigma_2^{p-1} \wt\cR^{-\varrho'}_{q} \otimes \cM_2(\C)$ for any $\varrho' >0$ (see  Lemma \ref{lem:hamsym}).
 
 \noindent
 {\bf (S2$)_p$}  
$J_c \nabla \big(H^{(\tSAP)}_{-\vr} \big)_{\leq p+1}(Z_{p-1})$ is a   super action preserving, Hamiltonian, smoothing vector field in $\Sigma_3^{p} \wt\X_q^{-\varrho + \underline \varrho(p)}$, where
 \be\label{cp}
 \underline \varrho(1):= c(N) \ , \qquad 
 \underline \varrho(p):= \underline \varrho(p-1) + \tau + \tfrac32   \, , \ \ p \geq 2 \ ; 
 \ee

\noindent
{\bf (S3$)_p$} $(\tm_{\frac32})_{p} (Z_{p-1};\xi) $ and $(\tm_{\frac32})_{\geq p+1} (Z_{p-1};\xi)$  are  real valued symbols, independent of $x $,  respectively in $ \wt\Gamma^{\frac32}_p$ 
and $\Sigma_{p+1}^N \wt\Gamma^{\frac32}_q$;

\noindent
{\bf (S4$)_p$}  $R_{\geq p}(Z_{p-1})$  is a  smoothing operator in
$\Sigma_{p}^N \wt\cR^{-\varrho + \underline \varrho(p)}_q \otimes \cM_2(\C)$;

\noindent
{\bf (S5$)_p$} $(\tm_{\frac32})_{>N} (U;t,\xi)$ is a non--homogeneous symbol in $ \Gamma^{\frac32}_{K,\underline K',N+1}[r]$ with imaginary part $\Im (\tm_{\frac32})_{>N}(U;t,\xi) $ in $ \Gamma^{0}_{K,\underline K',N+1}[r]$;

\noindent 
{\bf (S6$)_p$} $R_{>N}(U;t) $ is a matrix of non--homogeneous smoothing operators in $ \mR^{-\vr + \underline \varrho(p)}_{K,\underline K',N+1}[r] \otimes \mM_2 (\C) $;

\noindent 
{\bf (S7$)_p$}
the system \eqref{traspBNF} is 
Hamiltonian up to homogeneity $N$.
 }
 \vspace{.5em}
 
Note that for $p = N+1$, system  \eqref{traspBNF} has the claimed form in \eqref{final:eq} with 
$ Z \equiv Z_N $,  Hamiltonians $ 
H^{(\tSAP )}_{\frac32} := \big(H^{(\tSAP )}_{\frac32}\big)_{\leq N+2}  $, 
$ H^{(\tSAP)}_{-\vr} := \big(H^{(\tSAP)}_{-\vr}\big)_{\leq N+2} $ and $ \underline \varrho:= \underline \varrho(N+1)$, thus proving Proposition \ref{birkfinalone}. We now prove the  
inductive statements  {\bf (S0$)_p$}-{\bf (S7$)_p$}. \\

 \noindent
{\bf Initialization: case $p = 2$.}  We set $\mF_{\leq N}^{(1)}:=\mathtt{F}^{(1)}_{\leq N}$ defined in
\eqref{mFleqN} which is symplectic up to homogeneity $N$.
Thanks to \eqref{mappaBflow}, the non--linear map $\mF_{\leq N}^{(1)}$ satisfies $(iii)$ of Proposition \ref{birkfinalone}. 
The  system  \eqref{Z1} with $R_1(Z_1)Z_1 + G_1^+(Z_1)Z_1 = 0$ is \eqref{traspBNF} with  Hamiltonians   $\big(H^{(\tSAP )}_{\frac32}\big)_{\leq 3} = \big(H^{(\tSAP)}_{-\vr}\big)_{\leq 3} =0 $, 
and symbols $( \tm_{\frac32})_2= \mP_2[(\tm_{\frac32})^+_{\leq N}]$, 
$( \tm_{\frac32})_{\geq 3}= \mP_{\geq 3}[(\tm_{\frac32})^+_{\leq N}]$ and
$ (\tm_{\frac32})_{>N} = (\tm_{\frac32})_{>N}^+ $. 
Furthermore 
  $Z_{1}= \bM_{0}^{(1)}(U;t)U$ with $\bM_{0}^{(1)}(U;t) \in \mM^0_{K,\underline{K}'-1,0}[r]
 \otimes \mM_2 (\C)  $ because  the map in \eqref{mFleqN} has the form  $\mathtt{F}^{(1)}_{\leq N}(Z_0)=\breve \bM_0(Z_0)Z_0$ with
$\breve \bM_0(Z_0) \in \mM^0_{K,0,0}[r] \otimes \mM_2 (\C) $ thanks to Lemma \ref{mappabonetta}, Proposition \ref{prop:Z}-$(iii)$ and Proposition \ref{composizioniTOTALI} $(iii)$ 
(with $K' \leadsto  \underline{K}'-1$).
Thus {\bf (S0$)_2$}-{\bf (S7$)_2$}
are satisfied.\\

 \noindent
 {\bf Iteration: reduction of the $p$--homogenous symbol.} Suppose {\bf (S0$)_p$}-{\bf (S7$)_p$} hold true. 
The goal of this step is to  
reduce the real valued, $ x$-independent,  $ p $-homogenous symbol  
$-\im (\tm_{\frac32})_{p}(Z_{p-1};\xi) \in \wt \Gamma^{\frac32}_p $ 
in \eqref{traspBNF}. We Fourier  expand  as in \eqref{sviFou}
\be \label{defm32p}
(\tm_{\frac32})_{p}(Z_{p-1};\xi)= \sum_{
(\vec{\jmath}_p,  \vec{\sigma}_p) \in \fT_p} \tm_{\vec{\jmath}_p}^{\vec{\sigma}_p}(\xi) \, (z_{p-1})_{\vec{\jmath}_p}^{\vec{\sigma}_p} \,  , 
\quad \bar{\tm_{\vec{\jmath}_p}^{-\vec{\sigma}_p}}(\xi)= \tm_{\vec{\jmath}_p}^{\vec{\sigma}_p}(\xi) \, . 
\ee
to its super action preserving normal form. 
We conjugate \eqref{traspBNF} under  the change of variable 
\be\label{W.bir}
W:=\Phi_p(Z_{p-1}):= \cG_{g_p}^1(Z_{p-1})Z_{p-1}
\ee
where  $ \cG_{g_p}^1(Z_{p-1}) $ is 
the time $1$-linear flow generated by $\vOpbw{\ii g_p}$
 as in \eqref{FouFlow}, 
 where $g_p$ is the Fourier multiplier 
 \be\label{bir.gp}
 g_p(Z_{p-1};\xi):= \sum\limits_{(\vec{\jmath}_p,  \vec{\sigma}_p) \in \fT_p} G_{\vec{\jmath}_p}^{\vec{\sigma}_p}(\xi)(z_{p-1})_{\vec{\jmath}_p}^{\vec{\sigma}_p} \in \wt \Gamma_p^{\frac32}
 \ee
with coefficients 
\begin{align}\label{G.coeff.op}
 G_{\vec{\jmath}_p}^{\vec{\sigma}_p}(\xi) :=\begin{cases} 0 & \mbox{ if } (\vec{\jmath}_p, \vec{\sigma}_p)\in \mathfrak{S}_p\\ 
 \displaystyle{\frac{\tm_{\vec{\jmath}_p}^{\vec{\sigma}_p}(\xi)}{-\ii \vec{\sigma}_p \cdot \Omega_{\vec{\jmath}_p}(\kappa) }} &
 \mbox{ if }  (\vec{\jmath}_p, \vec{\sigma}_p)\not\in \mathfrak{S}_p \, ,
  \end{cases} 
\end{align}
where  the super action  set  $\mathfrak{S}_{p}$ is defined in \eqref{fSp} and 
$ \Omega_{\vec{\jmath}_p}(\kappa)$ is the frequency vector in \eqref{defOmegakappa}.
\begin{lemma}
Let $\kappa \in (0, + \infty) \setminus \cK$.
The function $g_p(Z_{p-1};\xi)$ in \eqref{bir.gp}, \eqref{G.coeff.op} is a well defined, 
$ x $-independent, real valued, $p$-homogeneous symbol in $\wt \Gamma^{\frac32}_p$.
\end{lemma}

\begin{proof}
We claim  that  for any $\kappa\in (0,+\infty)\setminus \cK$
there exist $ \tau, \nu>0$ such that 
\be\label{small.div0}
\vert \vec{\sigma}_{p}\cdot  {\Omega}_{\vec{\jmath}_{p}}(\kappa) \vert > \frac{\nu}{\max(|j_1|, \ldots, |j_{p}|)^\tau} \, , \qquad 
\forall 
(\vec\jmath_{p}, \vec \sigma_{p}) \not\in \mathfrak{S}_{p} \  . 
\ee
Indeed,   
to  any $(\vec\jmath_{p},  \vec \sigma_{p})$ we associate the 
 multi--index $(\alpha,\beta)$ as in \eqref{defalbe} whose length is 
$|\alpha+\beta|=p$ and satisfies 
$
\vec{\sigma}_{p}\cdot  {\Omega}_{\vec{\jmath}_{p}}(\kappa) = (\alpha-\beta)\cdot \vec \Omega(\kappa)$ by \eqref{defOmegakappa00}.
Recalling \eqref{fSp}, 
the vector 
$(\vec\jmath_{p}, \vec \sigma_{p}) \not\in \mathfrak{S}_{p}$ 
if and only if  $(\alpha, \beta)$  is 
not super action-preserving and  therefore 
Theorem \ref{nonresfin0} implies  \eqref{small.div0}. 
Note also that, by Remark \ref{nonsuper}, 
if  $p$ is odd, there are not super-action preserving indexes, i.e. 
$ \mathfrak{S}_{p} = \emptyset $. 

In view of \eqref{small.div0}  
the coefficients in \eqref{G.coeff.op} are well defined and,  
since the coefficients $ \tm_{\vec{\jmath}_p}^{\vec \sigma_p}(\xi) $ of  the symbol $(\tm_{\frac32})_p$ fulfill
\eqref{rem:symbol.1} (with $m = \frac32$), then 
the coefficients 
$G_{\vec{\jmath}_p}^{\vec{\sigma}_p}(\xi)$ in \eqref{G.coeff.op} satisfy  \eqref{rem:symbol.1} as well  (with $\mu$ replaced by $\mu + \tau$), implying that 
the Fourier multiplier  $g_p$ in \eqref{bir.gp} belongs to $\wt\Gamma^{\frac32}_p$.
Finally  $g_p$ is real because the coefficients 
$G_{\vec{\jmath}_p}^{\vec{\sigma}_p}(\xi)$ in \eqref{G.coeff.op} satisfy \eqref{realsim} 
as $ \tm_{\vec{\jmath}_p}^{\vec{\sigma}_p}(\xi) $.
\end{proof}

By Lemma \ref{flussoconst} the flow map \eqref{W.bir} is well defined and,  
by 
\eqref{inveroGood}
for $ \| Z_{p-1} \|_{k,s_0} < r < r_0(s,K)$  small enough,  
\be\label{Psi.p.est}
2^{-1} \| Z_{p-1} \|_{k,s} \leq  \|\Phi_p(Z_{p-1}) \|_{k,s} \leq 2\| Z_{p-1} \|_{k,s} \, , 
\ \ \  \forall k = 0, \ldots, K \, .
\ee
In order to transform \eqref{traspBNF} under the change of variable 
\eqref{W.bir} we use Lemma \ref{lem:conj.fou}. 
Its assumption {\bf (A)} at page \pageref{A} holds since $ Z_{p-1}  $ solves \eqref{traspBNF}
which, in view of \eqref{ham.sap8} and 
 {\bf (S2$)_p$},  has the form \eqref{Z.conj1} 
 (with $Z\equiv Z_{p-1}$, $ a_{\leq N}\equiv -(\tm_{\frac32}^{({\tSAP })})_{\leq p-1}- (\tm_{\frac32})_{p}- (\tm_{\frac32})_{\geq p+1}$, 
 $ \varrho \leadsto  \varrho - \underline{\varrho}(p)$ and $ K'\equiv \underline{K}'$).

  Then 
 Lemma \ref{lem:conj.fou} implies that the 
variable $W $ defined in \eqref{W.bir} solves 
 \be \label{eq:W:Bir}
\begin{aligned}
\pa_t W & = -\ii \vOmega(D)W+J_c \nabla \big(H^{(\tSAP )}_{\frac32}\big)_{\leq p+1}(W) +J_c \nabla \big(H^{(\tSAP)}_{-\vr}\big)_{\leq p+1}(W)\\
& \ + \vOpbw{ - \im  [(\tm_{\frac32})_{p}(W;\xi) - g^+_p(W;\xi)] - \im (\tm_{\frac32})_{\geq p+1}^+(W;\xi)}W +R_{\geq p}^+(W)W \\
&\ +\vOpbw{- \im (\tm_{\frac32})_{>N}^+ (U;t,\xi)}W+ R_{>N}^+(U;t)U 
\end{aligned}
\ee
where\\
\noindent
$\bullet$ $ g^+_p(W;\xi) \in \wt\Gamma^{\frac32}_p$ is given in  \eqref{ditigi};\\
\noindent
$\bullet$ 
$(\tm_{\frac32})_{\geq p+1}^+ (W;\xi)$  is a  real valued symbol, independent of $x $,   in $\Sigma_{p+1}^N \wt\Gamma^{\frac32}_q$;

\noindent
$\bullet$ $(\tm_{\frac32})_{>N}^+ (U;t,\xi) 
$ is a non-homogeneous symbol independent of $x$ in $ \Gamma^{\frac32}_{K,\underline K',N+1}[r]$ with imaginary part 
${\rm Im} \, (\tm_{\frac32})_{>N}^+ (U;t,\xi) $ in $ \Gamma^{0}_{K,\underline K',N+1}[r]$;

\noindent
$\bullet$
$R^+_{\geq p}(W) $ is a matrix of pluri--homogeneous  smoothing operators in $ \Sigma_p^N \wtcR^{-\vr + \underline \varrho(p) +c(N,p)}_q \otimes \mM_2 (\C)  $ for a certain $c(N,p)\geq 0$;

\noindent
$\bullet$ $R^+_{>N}(U;t)$ is a matrix of  non--homogeneous smoothing operators in $ \mR^{-\vr + \underline \varrho(p) +c(N,p)}_{K,\underline K',N+1}[r] \otimes \mM_2 (\C) $.\\

Note that the Hamiltonian part of degree of homogeneity $\leq p$ in \eqref{eq:W:Bir} has been unchanged with respect to \eqref{traspBNF},  
thanks to the first identity in  \eqref{ditigi0} and \eqref{ditigi2}.
In view of \eqref{ditigi}, \eqref{defm32p}, \eqref{G.coeff.op},   the symbol of homogeneity $p$  in \eqref{eq:W:Bir} reduces to its super action component
$$
(\tm_{\frac32})_{p}(W;\xi)- g^+_p(W;\xi)= (\tm^{(\tSAP)}_{\frac32})_{p}(W;\xi) :=  
\sum_{(\vec{\jmath}_p,  \vec{\sigma}_p) \in \mathfrak{S}_p}  \tm_{\vec{\jmath}_p}^{\vec{\sigma}_p}(\xi) \, w_{\vec{\jmath}_p}^{\vec{\sigma}_p}  
$$
where  the super action  set  $\mathfrak{S}_p$ is defined in \eqref{fSp}, and 
then 
\eqref{eq:W:Bir} becomes  
 \be \label{eq:W:Bir2}
\begin{aligned}
\pa_t W  = &  -\ii \vOmega(D)W+J_c \nabla \big(H^{(\tSAP )}_{\frac32}\big)_{\leq p+1}(W) +J_c \nabla \big(H^{(\tSAP)}_{-\vr}\big)_{\leq p+1}(W)\\
& \ + \vOpbw{ - \im (\tm^{(\tSAP)}_{\frac32})_p(W;\xi) - \im (\tm_{\frac32})_{\geq p+1}^+(W;\xi)}W +R_{\geq p}^+(W)W \\
&\ +\vOpbw{- \im (\tm_{\frac32})_{>N}^+ (U;t,\xi)}W+ R_{>N}^+(U;t)U \, . 
\end{aligned}
\ee
We now observe that, by Lemma \ref{lem:hamsym},
\be\label{msap.p.bir}
\vOpbw{ - \im (\tm^{(\tSAP)}_{\frac32})_p(W;\xi) }W
=
J_c \nabla \big(H_{\frac32}^{(\tSAP)}\big)_{p+2}(W)+ R_p'(W)W 
\ee 
with  the Hamiltonian 
\be\label{H.SAP.p2}
\big(H_{\frac32}^{(\tSAP)}\big)_{p+2}(W) :=\Re \Big \langle  \Opbw{(\tm_{\frac32}^{({\tSAP })})_p(W;\xi)} w, \bar w\Big \rangle_{\dot L^2_r} \ , 
\ee
which is super action preserving by Lemma \ref{lem:ham.sap}, 
and  a matrix of smoothing operators  $R_p'(W) $ in  $ \wt\cR^{-\varrho'}_p \otimes \cM_2(\C)$ for any $\varrho' \geq 0$. 
Therefore \eqref{eq:W:Bir2} becomes  
 \be \label{eq:W:Bir3}
\begin{aligned}
\pa_t W = &  -\ii \vOmega(D)W+J_c \nabla \big(H^{(\tSAP )}_{\frac32}\big)_{\leq p+2}(W) +J_c \nabla  \big(H^{(\tSAP)}_{-\vr}\big)_{\leq p+1}(W)\\
& \ + \vOpbw{  - \im (\tm_{\frac32})_{\geq p+1}^+(W;\xi)}W+ [R_{\geq p}^+(W) + R_{ p}'(W)]W  \\
&\ +\vOpbw{- \im (\tm_{\frac32})_{>N}^+ (U;t,\xi)}W+ R_{>N}^+(U;t)U \, 
\end{aligned}
\ee
where (see \eqref{H.SAP.p1}, \eqref{H.SAP.p2})
\be\label{H.SAP.30}
\big(H^{(\tSAP )}_{\frac32}\big)_{\leq p+2}:= \big(H^{(\tSAP )}_{\frac32}\big)_{\leq p+1} + \big(H_{\frac32}^{(\tSAP)}\big)_{p+2} \ . 
\ee
Note that the new system \eqref{eq:W:Bir3} is not Hamiltonian up to homogeneity $N$ (unlike system \eqref{traspBNF} for $Z_{p-1}$), since the map
$\Phi_p(Z_{p-1}) = \cG^1_{g_p}(Z_{p-1}) Z_{p-1} $  in 
\eqref{W.bir} is not symplectic up to homogeneity $N$. By Lemma \ref{flussoconst} we only know  
 that $\cG^1_{g_p}(Z_{p-1})$ is linearly symplectic.
We now apply Theorem \ref{conjham} to find a  correction of  $\Phi_p(Z_{p-1})$ which is symplectic up to homogeneity $N$.
By Lemma  \ref{flussoconst}, the map $ \Phi_p(Z_{p-1})$ satisfies
the  assumptions of Theorem \ref{conjham} (with $Z$ $\leadsto$ $Z_{p-1}$,  $\bB(Z;t)$ $\leadsto$ 
$\cG_{g_p}^1(Z_{p-1})$   and using the 
inductive assumption 
$Z_{p-1} = \bM_0^{(p-1)}(U;t)U$ with $\bM_0^{(p-1)}(U;t) \in \cM_{K, \underline K'-1, 0}^0[r]\otimes \cM_2(\C)$).
Therefore  Theorem \ref{conjham}  implies the existence of 
 a matrix of pluri--homogeneous smoothing operators $R^{(p)}_{\leq N}(W) $ in $ \Sigma_p^N \wtcR_q^{-\vr-\frac32} \otimes\mathcal{M}_2(\C) $  (the thesis holds for any $\vr>0$ and we take  $\varrho$ $\leadsto$ 
 $\varrho+\frac32 $) such that the variable  
\be\label{darbouxp}
V := \cC_N^{(p)}(W):= \big( \uno + R_{\leq N}^{(p)}(W) \big)  W = 
\big( \uno + R_{\leq N}^{(p)}(\Phi_p(Z_{p-1})) \big)  \Phi_p(Z_{p-1}) 
 \ee
 is symplectic up to homogeneity $N$, thus
solves a system which is  Hamiltonian  up to homogeneity $N$.
By \eqref{darbouxp}   one has
\be\label{VMUU}
V = \breve \bM_0(U; t) U , \quad \breve \bM_0(U;t) \in \cM_{K, \underline K'-1, 0}^0[r]\otimes \cM_2(\C) 
\ee
using that 
 $ \uno + R_{\leq N}^{(p)}(W)$ belongs to $\mM^0_{K,0,0}[r] \otimes \mM_2 (\C)  $ (by Lemma \ref{mappabonetta}), 
 since $ \Phi_p(Z_{p-1})= \cG_{g_p}^1(Z_{p-1})Z_{p-1}  $ 
with $ \cG_{g_p}^1 \in \mM^0_{K,0,0}[r] \otimes \mM_2 (\C)  $ 
(Lemma \ref{flussoconst}  $(i)$), the inductive assumption $Z_{p-1} = \bM_{0}^{(p-1)}(U;t)U$ with $\bM_{0}^{(p-1)}(U;t) \in \mM^0_{K,\underline{K}'-1,0}[r]
\otimes \mM_2 (\C)  $ and Proposition \ref{composizioniTOTALI} $(iii)$.

Moreover,  regarding  $R_{\leq N}^{(p)}(W) $ as a non-homogeneous 
smoothing operator in  $ \mR^{-\varrho-\frac32 }_{K,0,p}[r] \otimes\mathcal{M}_2(\C) $ for any $ r>0$ 
(see  Lemma \ref{mappabonetta}),  estimate \eqref{mappaBflow} implies, for  $0<r< r_0(s,K)$  small, the bound
\be\label{C.p.est}
2^{-1} \| W \|_{k,s} \leq  \|\cC^{(p)}
_N(W) \|_{k,s} \leq 2\| W \|_{k,s} \, , 
\quad \forall k = 0, \ldots, K \, . 
\ee
We compute  the new system satisfied  by $V$ in \eqref{darbouxp} using Lemma \ref{NormFormMap0}.
Its assumption {\bf (A)} at page \pageref{A} holds 
(with $ K' = \underline{K}' $ and 
$ \varrho \leadsto \vr - \underline \varrho(p) -c(N,p)$)
since 
$ W $ solves \eqref{eq:W:Bir3} and \eqref{H.SAP.30}, 
\eqref{msap.p.bir}, \eqref{ham.sap8},  {\bf (S2$)_p$}.
  Then 
  Lemma \ref{NormFormMap0} 
   implies that the variable $V$  solves 
\bigskip  
  
\be \label{eq:W:Bir4}
\begin{aligned}
\pa_t  V  = &  -\ii \vOmega(D)  V+J_c \nabla \big(H^{(\tSAP )}_{\frac32}\big)_{\leq p+2}(V) +J_c \nabla \big(H^{(\tSAP)}_{-\vr}\big)_{\leq p+1}(V)\\
& \ + \vOpbw{  - \im \wt{(\tm_{\frac32})}_{\geq p+1}(V;\xi)}V + \wt R_{\geq p}(V) V  \\
&\ +\vOpbw{- \im \wt{(\tm_{\frac32})}_{>N} (U;t,\xi)}V+ \wt R_{>N}(U;t)U \, 
\end{aligned}
\ee
where  

\noindent
$\bullet$
$\wt{(\tm_{\frac32})}_{\geq p+1} (V;\xi)$  is a  real valued symbol, independent of $x $,   in $\Sigma_{p+1}^N \wt\Gamma^{\frac32}_q$;

\noindent
$\bullet$ $\wt{(\tm_{\frac32})}_{>N} (U;t,\xi) 
$ is a non-homogeneous symbol independent of $x$ in $ \Gamma^{\frac32}_{K,\underline K',N+1}[r]$ with imaginary part 
${\rm Im} \, \wt{(\tm_{\frac32})}_{>N} (U;t,\xi) $ in $ \Gamma^{0}_{K,\underline K',N+1}[r]$;

\noindent
$\bullet$
$\wt R_{\geq p}(V) $ is a matrix of pluri--homogeneous  smoothing operators in $ \Sigma_p^N \wtcR^{-\vr  + \underline \varrho(p)+c(N,p)}_q \otimes \mM_2 (\C)  $;

\noindent
$\bullet$ $\wt R_{>N}(U;t)$ is a matrix of non--homogeneous smoothing operators 
in $ \mR^{-\vr  + \underline \varrho(p)+c(N,p)}_{K,\underline K',N+1}[r] \otimes \mM_2 (\C) $.

Note that in \eqref{eq:W:Bir4} the pluri-homogeneous components up to order $p+1$ of the symbol and up to order $p-1$ of the smoothing operators are unchanged with respect to
\eqref{eq:W:Bir3}, whereas the homogeneous part of order $p$ of the smoothing remainder have been corrected by a new smoothing operator 
in $\wt\cR^{-\vr  + \underline \varrho(p)+c(N,p)}_p \otimes \cM_2(\C)$, see \eqref{ditigi22}.

Since system \eqref{eq:W:Bir4}  is Hamiltonian up to homogeneity $N$ (unlike \eqref{eq:W:Bir3}),  we have in particular that  
\be\label{ham100}
  -\ii \vOmega(D)  V+J_c \nabla  \left(  \big(H^{(\tSAP )}_{\frac32}\big)_{\leq p+2}(V) +\big(H^{(\tSAP)}_{-\vr}\big)_{\leq p+1}(V)\right)+ \cP_p[ \wt R_{\geq p}(V)] V  
\ee
is a pluri-homogeneous Hamiltonian vector field.

\noindent
{\bf Iteration: reduction of the $p$--homogeneous smoothing remainder.}
The goal of this step is 
to reduce the smoothing  homogenous vector field 
$\wt R_p(V)V:= \cP_p[ \wt R_{\geq p}(V)] V$ in 
\eqref{ham100}, which belongs to $\wt \X^{-\vr  + \underline \varrho(p)+c(N,p)}_{p+1}$, 
to its super action preserving normal form.
By \eqref{ham100} we deduce,  by difference,  that $ \wt R_p(V)V $   is Hamiltonian.
We expand  $ \wt R_p(V)V$ in Fourier coordinates 
as in \eqref{polvect}
\be\label{wtR.p}
\big(\wt R_p(V)V\big)_k^\sigma=
\!\!\!\!\!\!\!\!\!\!\!\!\!\!\!
\sum_{
(\vec{\jmath}_{p+1}, k,  \vec{\sigma}_{p+1}, - \sigma) \in \fT_{p+2}}
\!\!\!\!\!\!\!\!\!\!\!\!\!\!
 \wt X_{\vec{\jmath}_{p+1},k}^{\vec{\sigma}_{p+1},\sigma}  v_{\vec{\jmath}_{p+1}}^{\vec{\sigma}_{p+1}} \, . 
\ee
In order to reduce $\wt R_p(V)V$ to its super action preserving part we 
transform \eqref{eq:W:Bir4} under the change of variable 
$  Z_p := \mathtt{F}_{\leq N}^{(p)}(V)$
  where $\mathtt{F}_{\leq N}^{(p)}(V)$ is the time $1$-approximate flow, given by Lemma \ref{extistencetruflow}, generated by  
 the  smoothing vector field 
\be\label{Gp}
\big(G_p(V)V\big)_k^\sigma=
\!\!\!\!\!\!\!\!\!\!\!\!\!\!\!
\sum_{
(\vec{\jmath}_{p+1}, k,  \vec{\sigma}_{p+1}, - \sigma) \in \fT_{p+2}}
\!\!\!\!\!\!\!\!\!\!\!\!\!\!
 G_{\vec{\jmath}_{p+1},k}^{\vec\sigma_{p+1},\sigma} v_{\vec{\jmath}_{p+1}}^{\vec{\sigma}_{p+1}}
\ee
with 
 \begin{align}\label{Gp.coeff}
 G_{\vec{\jmath}_{p+1},k}^{\vec{\sigma}_{p+1},\sigma}  :=\begin{cases} 0 & \mbox{ if } (\vec\jmath_{p+1}, k,  \vec \sigma_{p+1}, - \sigma)\in \mathfrak{S}_{p+2}\\ 
 \displaystyle{\frac{ \wt X_{\vec{\jmath}_{p+1},k}^{\vec{\sigma}_{p+1}, \sigma} }{\ii (\vec{\sigma}_{p+1}\cdot  {\Omega}_{\vec{\jmath}_{p+1}}(\kappa)- \sigma \Omega_k(\kappa))}} &
 \mbox{ if } (\vec\jmath_{p+1}, k,  \vec \sigma_{p+1}, - \sigma) \not\in \mathfrak{S}_{p+2} \, 
  \end{cases} 
\end{align}
where  the super action  set  $\mathfrak{S}_{p+2}$ is defined in \eqref{fSp} (with $p$ replaced by $p+2$).
 
 \begin{lemma}\label{lemGpVVHam}
 Let $\kappa \in (0, + \infty) \setminus \cK$.
 The vector field $G_p(V)V$ in \eqref{Gp}, \eqref{Gp.coeff} is a well defined Hamiltonian vector field in $\wt \X_{p+1}^{-\varrho'}$
 with $\varrho' := \varrho - \underline \varrho(p) - c(N,p) - \tau$ and where  $\tau$ is defined in Theorem \ref{nonresfin0}.
 \end{lemma}
 
\begin{proof}
We claim  that  for any $\kappa \in (0, + \infty) \setminus \cK$ 
there exist $ \tau, \nu>0$ such that 
\be\label{small.div1}
\vert \vec{\sigma}_{p+1}\cdot  {\Omega}_{\vec{\jmath}_{p+1}}(\kappa)- \sigma \Omega_k(\kappa) \vert > \frac{\nu}{\max(|j_1|, \ldots, |j_{p+1}|, |k|)^\tau} \, , \qquad 
\forall 
(\vec\jmath_{p+1}, k,  \vec \sigma_{p+1}, - \sigma) \not\in \mathfrak{S}_{p+2} \  . 
\ee
Indeed,    
to  any $(\vec\jmath_{p+1}, k,  \vec \sigma_{p+1}, - \sigma)$ 
we associate the multi--index 
$(\alpha,\beta)$ as in \eqref{defalbe} whose length is $|\alpha+\beta|=p+2$ and satisfies  
$
\vec{\sigma}_{p+1}\cdot  {\Omega}_{\vec{\jmath}_{p+1}}(\kappa)- \sigma \Omega_k(\kappa) = (\alpha-\beta)\cdot \vec \Omega(\kappa)$ by \eqref{defOmegakappa00}.
Recalling \eqref{fSp}, the  vector 
$ (\vec\jmath_{p+1}, k,  \vec \sigma_{p+1}, - \sigma) \not\in \mathfrak{S}_{p+2}$ if and only if  $(\alpha, \beta)$  is 
not super action-preserving and 
therefore 
Theorem \ref{nonresfin0} implies \eqref{small.div1}. Note also that, by Remark \ref{nonsuper}, if  $p$ is odd, there are not super-action preserving indexes, i.e. $\mathfrak{S}_{p+2} = \emptyset$.
In view of \eqref{small.div1}  the coefficients in \eqref{Gp.coeff} are well defined.

Next we show that  $G_p(V)V$ is a  vector field in $\wt \X_{p+1}^{-\varrho'}$. As $\wt R_p(V)V $ belongs to $ \wt \X^{-\vr  + \underline \varrho(p)+c(N,p)}_{p+1}$,  
by Lemma \ref{lem:X.R} the coefficients 
 $ \wt X_{\vec{\jmath}_{p+1},k}^{\vec{\sigma}_{p+1},\sigma}$  in 
 \eqref{wtR.p}
satisfy the symmetric and 
reality 
 properties \eqref{symmetric}, \eqref{X.real} and the estimate: 
  for some  $ \mu \geq 0 $, $ C>0 $,
\be\label{smoocara2p}
 | \wt X_{\vec{\jmath}_{p+1},k}^{\vec{\sigma}_{p+1},\sigma}  |\leq C 
\frac{\max_2\{ |j_1|, \ldots,  | j_{p+1}|\}^{\mu}}{\max\{ |j_1|, \ldots,  | j_{p+1}| \}^{\vr  - \underline \varrho(p)-c(N,p)}} \, , \
 \quad  \forall  (\vec \jmath_{p+1}, k,  \vec \sigma_{p+1},-\sigma) \in \fT_{p+2}  \, .
 \ee
 Hence also the coefficients 
 $ G_{\vec{\jmath}_{p+1},k}^{\vec{\sigma}_{p+1},\sigma} $ 
 in \eqref{Gp.coeff} satisfy the symmetric and 
reality 
 properties \eqref{symmetric}, \eqref{X.real} as 
 $ \wt X_{\vec{\jmath}_{p+1},k}^{\vec{\sigma}_{p+1}, \sigma}  $. 
 Moreover, using  \eqref{smoocara2p},  \eqref{small.div1}  and the momentum relation  $\sigma k = \vec \sigma_{p+1} \cdot \vec \jmath_{p+1} $ they also  
satisfy 
$$
 | G_{\vec{\jmath}_{p+1},k}^{\vec{\sigma}_{p+1},\sigma}   |\leq C 
\frac{\max_2\{ |j_1|, \ldots,  | j_{p+1}|\}^{\mu}}{\max\{ |j_1|, \ldots,  | j_{p+1}| \}^{\vr  - \underline \varrho(p)-c(N,p)-\tau} }
$$
for a new constant $C >0$ (depending on $\nu$).
Then  
Lemma \ref{lem:X.R} implies that $G_p(V)V $ belongs to $ \wt \X^{-\vr'}_{p+1}$
with $ \vr':= \vr  - \underline \varrho(p)-c(N,p)-\tau$. 

Finally  we show that 
 $G_p(V)V$ is   Hamiltonian. 
Recall that $\wt R_p(V)V$ in  \eqref{wtR.p}  is a Hamiltonian vector field whose Hamiltonian function $H_{\wt R_p}(V) $ is,   thanks to Lemma \ref{carhamvec}, 
\be\label{HRp}
 H_{\wt R_p}(V)= \tfrac{1}{p+2}
\!\!\!\!\!\!\!\! 
  \sum_{(\vec{\jmath}_{p+2}, \vec{\sigma}_{p+2}) \in \fT_{p+2}} 
\!\!\!\!\!\!\!\!  
  [H_{\wt R_p}]_{\vec\jmath_{p+2}}^{\vec\sigma_{p+2}} \, v_{\vec{\jmath}_{p+2}}^{\vec{\sigma}_{p+2}}  \, , 
\qquad
[H_{\wt R_p}]_{j_1, \ldots, j_{p+2}}^{\sigma_1, \ldots, \sigma_{p+2}} :=  - \ii \sigma_{p+2}  \wt X_{\vec{\jmath}_{p+1},j_{p+2}}^{\vec{\sigma}_{p+1},-\sigma_{p+2}} \, . 
\ee
Then the coefficients defined for $   \vec \sigma_{p+2}
\cdot \vec \jmath_{p+2} = 0 $ by
\be\label{HGp}
[H_{G_p}]_{\vec \jmath_{p+2}}^ {\vec \sigma_{p+2}}
:= - \ii \sigma_{p+2}  G_{j_1, \ldots, j_{p+1}, j_{p+2}}^{\sigma_1, \ldots,\sigma_{p+1}, - \sigma_{p+2}}
\stackrel{\eqref{Gp.coeff},\eqref{HRp}}{=} \frac{ [H_{\wt R_p}]_{\vec \jmath_{p+2}}^ {\vec\sigma_{p+2}}}{\ii\big(\vec{\sigma}_{p+1}\cdot  {\Omega}_{\vec{\jmath}_{p+1}}(\kappa) + \sigma_{p+2} \Omega_{j_{p+2}}(\kappa)\big)}  
\ee
 satisfy the symmetric and reality  properties \eqref{symham}, \eqref{realham} as 
 well as the 
 coefficients $ [H_{\wt R_p}]_{\vec \jmath_{p+2}}^ {\vec\sigma_{p+2}}$.
Then Lemma \ref{carhamvec} implies that 
 $G_p(V) V$ is the Hamiltonian vector field 
generated by the Hamiltonian $H_{G_p}$  with coefficients defined in  \eqref{HGp}.
\end{proof} 

We now conjugate system \eqref{eq:W:Bir4} 
by  the approximate time 1-flow $\mathtt{F}^{(p)}_{\leq N}(V) $ generated by 
$ G_p(V)V $ provided  by Lemma 
\ref{extistencetruflow}, which has the form 
 \be\label{Zp.def}
 Z_p = 
 \mathtt{F}_{\leq N}^{(p)}(V) = V + F_{\leq N}^{(p)}(V)V \, , 
  \quad 
  F_{\leq N}^{(p)}(V) \in \Sigma_{p}^N \wt \cR^{-\varrho'}_q \otimes \cM_2(\C) \, .  
 \ee
Since $ G_p(Z)Z $ is a Hamiltonian vector field, by 
Lemma 
\ref{lem:app.flow.ham}  
the approximate flow $\mathtt{F}^{(p)}_{\leq N}$ is symplectic up to homogeneity $N$.
Applying  Lemma \ref{conj.ham.N} 
 (with $Z \leadsto V$, $W \leadsto Z_p$ and by \eqref{VMUU}), 
the variable $Z_p$ solves a  system which is 
 Hamiltonian up to homogeneity $N$. 
We compute it using
 Lemma \ref{NormFormMap0} (with $ W \leadsto Z_p $ and $ Z \leadsto V $).
Its assumption {\bf (A)} at page \pageref{A} holds 
(with $ K' = \underline{K}' $ and 
$ \varrho \leadsto \vr - \underline \varrho(p) -c(N,p)$)
since $ V $ solves \eqref{eq:W:Bir4}
and \eqref{H.SAP.30}, 
\eqref{msap.p.bir}, \eqref{ham.sap8},  {\bf (S2$)_p$}. 
Lemma \ref{NormFormMap0} implies that  the variable 
$ Z_p $ in \eqref{Zp.def} solves (see in particular \eqref{ditigi30})
 \be \label{eq:W:Bir6}
 \begin{aligned}
\pa_t  Z_p  = &  -\ii \vOmega(D)  Z_p+J_c \nabla \big(H^{(\tSAP )}_{\frac32}\big)_{\leq p+2}(Z_p) +J_c \nabla \big(H^{(\tSAP)}_{-\vr}\big)_{\leq p+1}(Z_p)\\
& \ + \vOpbw{  - \im {( \breve \tm_{\frac32})}_{\geq p+1}(Z_p;\xi)}Z_p +  
 \left[ \wt R_{p}(Z_p) + G_p^+(Z_p)\right] Z_p + R_{\geq p+1}(Z_{p}) Z_{p} \\
&\ +\vOpbw{- \im {(\breve \tm_{\frac32})}_{>N} (U;t,\xi)}Z_p+  R_{>N}(U;t)U  
\end{aligned}
\ee
where the part homogeneous up to order $p$ of the symbol and up to order $p-1$ of the smoothing operators are unchanged with respect to
\eqref{eq:W:Bir4}, whereas \\
\noindent
$\bullet$
${(\breve \tm_{\frac32})}_{\geq p+1} (V;\xi)$  is a  real valued symbol, independent of $x $,   in $\Sigma_{p+1}^N \wt\Gamma^{\frac32}_q$;

\noindent
$\bullet$ ${(\breve\tm_{\frac32})}_{>N} (U;t,\xi) 
$ is a non-homogeneous symbol independent of $x$ in $ \Gamma^{\frac32}_{K,\underline K',N+1}[r]$ with imaginary part
${\rm Im} \, \wt{(\tm_{\frac32})}_{>N} (U;t,\xi) $ in $ \Gamma^{0}_{K,\underline K',N+1}[r]$;

\noindent
$\bullet$ $G_p^+(Z_p)Z_p$ is  $p+1$-homogeneous smoothing vector field in $\wt \X^{-\vr'+\frac32}_{p+1}$ with Fourier expansion  (see  \eqref{ditigi40})
\be\label{Rp.exp}
(G_p^+(Z_p)Z_p)_k^\sigma=
\!\!\!\!\!\!\!\!\!\!\!\!\!\!\!\!\!\!\!\!\!\!\!\!
 \sum_{(\vec{\jmath}_{p+1}, k,  \vec{\sigma}_{p+1}, - \sigma)  \in \fT_{p+2}}
 \!\!\!\!\!\!\!\!\!\!\!\!\!\!\!\!\!\!\!\!\!\!\!
 -  \ii\big(\vec{\sigma}_{p+1}\cdot  \vec{\Omega}_{\vec{\jmath}_{p+1}}(\kappa)- \sigma \Omega_k(\kappa)\big)G_{\vec{\jmath}_{p+1},k}^{\vec{\sigma}_{p+1},\sigma}  (z_p)_{\vec{\jmath}_{p+1}}^{\vec{\sigma}_{p+1}} \ ;
\ee
\noindent
$\bullet$
$ R_{\geq p+1}(Z_p) $ is a matrix of pluri--homogeneous  smoothing operators in $ \Sigma_{p+1}^N \wtcR^{-\vr + \underline \varrho(p+1)}_q \otimes \mM_2 (\C)  $ where
\be\label{cp+1}
\underline \varrho(p+1) = \underline \varrho(p) + c(N,p) + \tau + \tfrac32 \ ;
\ee

\noindent
$\bullet$ $ R_{>N}(U;t)$ is a matrix of non--homogeneous smoothing operators in $ \mR^{-\vr + \underline \varrho(p+1)}_{K,\underline K',N+1}[r] \otimes \mM_2 (\C) $.\\

By \eqref{wtR.p}, \eqref{Gp.coeff}, \eqref{Rp.exp},   
the smoothing operators  of homogeneity $ p $  in \eqref{eq:W:Bir6}
reduce to 
\begin{equation}
\label{homologgica2}
\wt R_{p}(Z_p)Z_p + G_p^+(Z_p)Z_p = 
R_{p}^{(\tSAP)}(Z_p)Z_p 
\end{equation}
where $R_{p}^{(\tSAP)}(Z_p)Z_p$ is the    
vector field  
$$
 (R_{p}^{(\tSAP)}(Z_p)Z_p)_{k}^\sigma := 
 \!\!\!\!\!\!\!\!\!\!\!\!\!\!\!\!\!
\sum_{  (\vec\jmath_{p+1}, k,  \vec \sigma_{p+1}, - \sigma) \in \mathfrak S_{p+2}
} 
\!\!\!\!\!\!\!\!\!\!\!\!\!\!\!\!\!\!
\wt X_{\vec{\jmath}_{p+1},k}^{\vec{\sigma}_{p+1}, \sigma}  \, (z_p)_{\vec{\jmath}_{p+1}}^{\vec{\sigma}_{p+1}} 
\in \wt \X_{p+1}^{-\vr  + \underline \varrho(p)+c(N,p)} \, ,
$$
that,  in view of \eqref{HRp}, 
is 
 generated by
the super action preserving Hamiltonian
(cfr. Definition \ref{sapham})
$$
 H_{\wt R_p}^{(\tSAP)}(Z_p):= \frac{1}{p+2} \sum_{ (\vec\jmath_{p+2}, \vec\sigma_{p+2}) \in \mathfrak{S}_{p+2}
} [H_{\wt R_p}]_{\vec\jmath_{p+2}}^{\vec\sigma_{p+2}} \, (z_p)_{\vec{\jmath}_{p+2}}^{\vec{\sigma}_{p+2}}  \, .
$$
By \eqref{homologgica2} and  
 $  R_{p}^{(\tSAP)}(Z_p)Z_p =  J_c\nabla   H_{\wt R_p}^{(\tSAP)} (Z_p) $, 
system \eqref{eq:W:Bir6}
becomes
 \be \label{eq:W:Bir7}
 \begin{aligned}
\pa_t  Z_p   & =   -\ii \vOmega(D)  Z_p+J_c \nabla  \big(H^{(\tSAP )}_{\frac32}\big)_{\leq p+2}(Z_p) +J_c \nabla \big(H^{(\tSAP)}_{-\vr}\big)_{\leq p+2}(Z_p)\\
& \  + \vOpbw{  - \im {( \tm_{\frac32})}_{p+1}(Z_p;\xi)
 - \im {( \tm_{\frac32})}_{\geq p+2}(Z_p;\xi)}Z_p
+ R_{\geq p+1}(Z_{p}) Z_{p} \\
&\ +\vOpbw{- \im {(\tm_{\frac32})}_{>N} (U;t,\xi)}Z_p+  R_{>N}(U;t)U \, 
\end{aligned}
\ee
with $\big(H^{(\tSAP )}_{\frac32}\big)_{\leq p+2}$ defined in \eqref{H.SAP.30} (see also \eqref{msap.p.bir}-\eqref{H.SAP.p2}), 
\be\label{hasmop+2}
\big(H^{(\tSAP)}_{-\vr}\big)_{\leq p+2}(Z_p) :=
\big(H^{(\tSAP)}_{-\vr}\big)_{\leq p+1}(Z_p) + H_{\wt R_p}^{(\tSAP)}(Z_p) \,  , 
\ee
and
$x$-independent real symbols
$$
\begin{aligned}
& {( \tm_{\frac32})}_{p+1}(Z_p;\xi) := \cP_{p+1}\big[{( \breve \tm_{\frac32})}_{\geq p+1}(Z_p;\xi) \big] \in \wt \Gamma^{\frac32}_{p+1} \\
& {( \tm_{\frac32})}_{\geq p+2}(Z_p;\xi):= \cP_{\geq p+2} \big[{(\breve  \tm_{\frac32})}_{\geq p+1}(Z_p;\xi) \big] \in \Sigma_{p+2}^N \wt \Gamma^{\frac32}_q \, . 
\end{aligned}
$$
System  \eqref{eq:W:Bir7} has therefore 
the form \eqref{traspBNF} at the step $ p +1 $ 
 with  
\be\label{Zp:comp}
Z_p:=\cF_{\leq N}^{(p)}(Z_0) := 
\mathtt{F}_{\leq N}^{(p)} \circ  \cC_N^{(p)}\circ \Phi_p \circ\cF_{\leq N}^{(p-1)}(Z_0) \, , 
\ee
see \eqref{W.bir}, \eqref{darbouxp}, \eqref{Zp.def}. 
The map 
$\cF_{\leq N}^{(p)} $ is symplectic up to homogeneity $ N $  
as  $ \cF_{\leq N}^{(p-1)} $, because  $ \cC_N^{(p)}\circ \Phi_p$ 
is symplectic up to homogeneity $ N $ (cfr. \eqref{darbouxp}) 
as well as  the   time $1$-approximate flow 
$ \mathtt{F}_{\leq N}^{(p)} $ 
 generated by the smoothing Hamiltonian vector field $G_p(V)V$
 (cfr. Lemma \ref{lemGpVVHam}) 
by  Lemma \ref{lem:app.flow.ham}. 
In addition  the map
$\cF_{\leq N}^{(p)}(Z_0)$ satisfies 
\eqref{mappaB} (with $p$-dependent constants) because of the inductive assumption,  \eqref{Psi.p.est}, \eqref{C.p.est} and 
\eqref{mappaBflow}.
 Furthermore 
  $Z_{p}= \bM_{0}^{(p)}(U;t)U$ with $\bM_{0}^{(p)}(U;t) \in \mM^0_{K,\underline{K}'-1,0}[r]
  \otimes \mM_2 (\C)  $.
This follows from \eqref{Zp:comp} using that 
$\uno + {F}_{\leq N}^{(p)}$ and $ \uno + R_{\leq N}^{(p)}$ belong to $\mM^0_{K,0,0}[r] \otimes \mM_2 (\C)  $
 (recall identities \eqref{Zp.def}, \eqref{darbouxp} and use 
 Lemma \ref{mappabonetta}),  
 since $ \Phi_p(Z_{p-1})= \cG_{g_p}^1(Z_{p-1})Z_{p-1}  $ 
with $ \cG_{g_p}^1 \in \mM^0_{K,0,0}[r] \otimes \mM_2 (\C)  $ 
(Lemma \ref{flussoconst}  $(i)$), the inductive assumption $\cF_{\leq N}^{(p-1)}(Z_0) = \bM_{0}^{(p-1)}(U;t)U$ with $\bM_{0}^{(p-1)}(U;t) \in \mM^0_{K,\underline{K}'-1,0}[r]
\otimes \mM_2 (\C)  $ and Proposition \ref{composizioniTOTALI} $(iii)$.
The proof of 
{\bf (S0$)_{p+1}$} is complete.

System \eqref{eq:W:Bir7} satisfies  also 
 {\bf (S1$)_{p+1}$}--{\bf (S6$)_{p+1}$} with $\underline{\varrho}(p+1)$ defined in \eqref{cp+1}  by \eqref{H.SAP.30}, \eqref{msap.p.bir}-\eqref{H.SAP.p2} and   
\eqref{hasmop+2}.
    {\bf (S7$)_{p+1}$} follows by Lemma \ref{conj.ham.N} 
because $ \cF_{\leq N}^{(p)}$ is symplectic up to
homogeneity $N$. 
This concludes the proof of the inductive step.

Finally \eqref{equivalenzan} follows by \eqref{mappaB} and \eqref{equivalenzaZU} (where $ C $ denote different constants). 
\end{proof}

\begin{remark}\label{rem:Nwave}
{\bf (Integrability of fourth and six order Hamiltonian Birkhoff normal form)}
The $ \tSAP $ Hamiltonian monomials (Definition \ref{sapham})
of degree $ 4 $ are integrable. Indeed, a-priori 
they are either the integrable ones 
 $ |z_{j_1}|^2  |z_{j_2}|^2  $ or (i) 
$ |z_{j_1}|^2  z_{j_2} \overline{z_{-j_2}} $ or (ii) $
z_{j_1} \overline{z_{-j_1} } z_{j_2} \overline{z_{-j_2}}  $. The momentum condition
implies in case (i)  that  $ j_2 = 0 $, which is not allowed. In case (ii) 
it yields
$ j_1 + j_2 = 0 $ and so $ z_{j_1} \overline{z_{-j_1} } z_{j_2} \overline{z_{-j_2}}  = z_{j_1} \overline{z_{j_2} } z_{j_2} \overline{z_{j_1}}  $ is integrable. 
The $ \tSAP $ Hamiltonian monomials of degree $ 6 $ may contain the not integrable
monomials 
$$ 
(i) \ |z_{j_1}|^2 z_{j_2} \overline{z_{-j_2}}  z_{j_3} \overline{z_{-j_3}} \, , 
\quad (ii) \  |z_{j_1}|^2 |z_{j_2}|^2  z_{j_3} \overline{z_{-j_3}} 
\, , 
\quad (iii) \ z_{j_1} \, \overline{z_{-j_1}}\,
z_{j_2}\, \overline{z_{-j_2}} \, z_{j_3} \, \overline{z_{-j_3}} \, . 
$$
By momentum conservation, a monomial of the form ($i$) is  integrable 
(as in case (ii) above) and
a monomial of the form ($ii$) has $ j_3 = 0 $ thus it is not allowed.  
The monomials  ($iii$)  
turn out to be, for $ \gamma \neq 0 $,  $ \tth = + \infty $, 
Birkhoff non-resonant, namely 
\be\label{lowinfh}
\begin{aligned}
& \big|\Omega_{j_1}(\kappa)- \Omega_{-j_{1}}(\kappa) + 
\Omega_{j_2}(\kappa)- \Omega_{-j_{2}}(\kappa) +
\Omega_{j_3}(\kappa)- \Omega_{-j_{3}}(\kappa) \big| \\
& =
\gamma  \big| \sign (j_1) + \sign (j_2) +  \sign (j_3) \big|  \geq \gamma \, .
\end{aligned} 
\ee
Therefore they might  be eliminated, obtaining an integrable  
normal form Hamiltonian at the degree $  6 $. The same holds in finite depth
exploiting also the momentum restriction $ j_1 + j_2 + j_3 = 0   $ and that  
$
\gamma  \big( \tanh ( \tth j_1) + \tanh ( \tth j_2) + \tanh ( \tth j_3)
\big)  \neq 0 
$ 
by the concavity of $ y \mapsto \tanh ( \tth y) $ for $ y >  0 $. Note that for 
$ |j_1|, |j_2|, |j_3| \geq M $ large enough we have a uniform lower bound 
as in \eqref{lowinfh}. 
In conclusion the fourth and six order Hamiltonian 
Birkhoff normal form of the water waves equations \eqref{eq:etapsi} is integrable.
\end{remark}

\section{Energy estimate and proof of Theorem \ref{teo1}}\label{sec:EE}

The  Hamiltonian equation 
\be \label{troncata}
\pa_tZ= - \im \vOmega(D)Z + J_c \nabla H^{(\tSAP)}_{\frac32}(Z)+J_c \nabla  H^{(\tSAP)}_{-\vr}(Z)  \, , 
\ee
obtained by \eqref{final:eq} neglecting 
the symbol and the smoothing operator of homogeneity larger than $ N $, preserves the  Sobolev norms. Equation \eqref{troncata} can be also written as the Hamiltonian PDE
\be\label{SAPequazione}
\pa_tZ= J_c \nabla  H^{(\tSAP)}(Z)   
\ee 
where $H^{(\tSAP)}(Z)$ is the super--action preserving Hamiltonian (cfr. Definition \ref{sapham})
\be 
H^{(\tSAP)}(Z):=H^{(2)}(Z)+H^{(\tSAP)}_{\frac32}(Z)+H^{(\tSAP)}_{-\vr}(Z) \, ,
\quad H^{(2)}(Z):= \sum_{j\in \Z\setminus\{0\}} \Omega_j(\kappa)|z_j|^2 \, .
\ee
Actually the following more precise result holds.

\begin{lemma}\label{lemmatronco}
{\bf (Super-actions)}
The super--actions  $J_n(Z)= |z_n|^2+ |z_{-n}|^2$, defined in \eqref{sa}, for any $  n\in \N $, 
are  prime integrals of the Hamiltonian equation
\eqref{troncata}.
In particular the Sobolev norm 
$$ 
\| Z\|_{\dot H^s}^2= \sum_{ j\in \Z \setminus \{0\}} | j|^{2s} |z_j|^2 = \sum_{n\in \N} n^{2s} J_n(Z)
$$ 
is constant along the flow of \eqref{troncata}. 
\end{lemma}
\begin{proof}
By \eqref{SAPequazione}, \eqref{Pois} and Lemma \ref{lemma:Poissonbra} we have 
$$
 \frac{\di}{\di t} J_n(Z)
= \di_Z J_n(Z)\left[  J_c \nabla H^{(\tSAP)}(Z)\right] =  \{ J_n(Z),  H^{(\tSAP)}(Z)\} = 0
$$
since the Hamiltonian $H^{(\tSAP)}(Z)$ contains  only super action preserving monomials.
\end{proof}

By the previous lemma,  
in order to derive an energy estimate for the solutions of 
\eqref{final:eq}, and thus for  a solution $ U $ of  \eqref{complexo},   
we have to estimate the non-homogeneous term  in \eqref{final:eq}.
We need  the following lemma.

 \begin{lemma}\label{equiv-tempo}
Let $ K \in \N $.
Then there is $ s_0 > 0 $ such that, 
for any $ s \geq s_0 $, for all $ 0 < r \leq \bar r(s) $ small, 
 if  $U$  belongs to $B_{s_0}^0(I;r)\cap C_{*\R}^{0}(I;\dot{H}^{s}(\T, \C^2))$
 and solves \eqref{complexo}, then $U$ belongs to $C_{*\R}^{K}(I;\dot{H}^{s}(\T, \C^2))$ and  $\forall \, 0\leq k \leq K$
 there is a constant $ \bar{C}_{s,k}\geq 1$  such that
\begin{equation}\label{aslong1}
\|\pa_{t}^{k}U(t)\|_{\dot{H}^{s- \frac32 k}}\leq \bar{C}_{s,k} \|U(t)\|_{\dot{H}^{s}} \, 
, \quad \forall t \in I
\, .
\end{equation}
In particular the norm 
$\|U (t) \|_{K,s}$ defined in \eqref{Knorm} is equivalent to the norm $\|U(t) \|_{\dot{H}^{s}}$.
\end{lemma}

\begin{proof}
We argue by induction proving that for any $0\leq k\leq K$, there are $s_0, \bar r_k>0$ such that if $U \in  B_{s_0}^0(I;\bar r_k)\cap C_{*\R}^{0}(I;\dot{H}^{s}(\T, \C^2))$ solves \eqref{complexo},  then  $U\in C^k_{*\R}(I;\dot H^s(\T;\C^2))$ with estimate \eqref{aslong1}.
For $k=0$ the estimate \eqref{aslong1} is trivial. 
Then assume \eqref{aslong1} holds true for $0,\dots, k-1\leq K-1$.  
Next we  write \eqref{complexo} as
$
\pa_t U = \Opbw{A(U;x;\xi)}U + R(U)U
$
where, by Lemma \ref{LemCompl}, the smoothing operator $R(U)$ is in 
$ \Sigma\mathcal{R}^{-\varrho}_{k-1,0,1}[r,N]\otimes\mathcal{M}_2(\C) $  and 
the matrix of symbols
$$
A(U;x,\xi):=J_c A_\frac32(U;x) \omega(\xi)+ \frac{\gamma}{2} \frac{\tG(\xi)}{\ii \xi} \mathds{1} +
 J_c \big[ A_1(U;x,\xi)  +  A_{\frac12} (U;x,\xi)+A^{(2)}_0(U;x,\xi)] 
 $$ 
belongs to $ \Sigma\Gamma_{k-1,0,0}^{\frac32}[r,N] \otimes\mathcal{M}_2(\C) $ and so,
 by the fourth bullet after Definition \ref{smoothoperatormaps}, 
   $\Opbw{ A(U;x,\xi)}$ belongs to $ \Sigma\mS^{\frac32}_{k-1,0,0}[r,N]\otimes \cM_2(\C)$.
Let $s_0:= s_0(k)>0$ given in Definitions \ref{Def:Maps}, \ref{smoothoperatormaps}. By the inductive estimate \eqref{aslong1} up to $k-1$, we have that 
$  U\in B_{s_0}^{k-1}(I, \bar C_{s_0,k}' r) $ with 
$  \bar C_{s_0,k}':= \sum_{j=0}^{k-1} \bar C_{s_0,j} $. 
   Then, for any $ s\geq s_0$, there is $\bar r_k:=\bar r(s,k)>0$ such that if $0< r< \bar r_k$ 
the operator $\Opbw{ A(U;x,\xi)}$ fulfills estimate below \eqref{piovespect} (with $K' =0$, $m = \frac32$), and $R(U)$ estimate below
\eqref{piove} (with $K'=0$, $m=0$) so that
\begin{align}
\|\pa_t^k U\|_{\dot{H}^{s-\frac32k}} &\leq \| \pa_t^{k-1}(\Opbw{ A(U;x,\xi)}U) \|_{\dot{H}^{s-\frac32(k-1)-\frac32}}+ \| \pa_t^{k-1}(R(U)U)\|_{\dot{H}^{s-\frac32(k-1)-\frac32}}\notag\\
&\leq C_{s,k}\|U\|_{k-1,s} = C_{s,k} \sum_{ j=0}^{ k-1} \| \pa_t^{j} U\|_{\dot H^{s-\frac32j}}\leq \bar C_{s,k} \|U\|_{\dot H^s}\label{induttivaboh} 
\end{align}
by the inductive hypothesis \eqref{aslong1} for $j=0, \dots, k-1$ and setting $ \bar C_{s,k}:= C_{s,k} \sum_{j=0}^{k-1}\bar C_{s,j}$. This proves
\eqref{aslong1} at step $ k $. 
 We finally fix  $ \bar r(s):= \min(\bar r_1, \dots,\bar r_K)$ proving the lemma.
\end{proof}

\subsection*{Proof of Theorem \ref{teo1}} 
We deal only with the case $N  \in \N$, since the cubic energy estimate 
\eqref{eeN}
in case $N = 0$ follows directly from Proposition \ref{teoredu1} (see also  Remark \ref{rem:ee0}), yielding the local  time of existence of order $\e^{-1}$.

So from now on we consider $N \in \N$. For any value of the gravity $ g > 0 $, depth $ \tth \in (0,+\infty] $ and vorticity $ \gamma \in \R $, 
let $ {\cal K} \subset (0,+\infty) $ be the zero measure set 
defined in Theorem \ref{nonresfin0}. 
Assume that the surface tension coefficient $ \kappa $ belongs  to the complementary set $ (0,+\infty) \setminus {\cal K} $.  
 Let $\underline \varrho>0 $ be the constant given by Proposition \ref{birkfinalone}.  
 
{\it \begin{itemize} 
\item {\bf Choice of the parameters:}
From now on we fix in Proposition \ref{birkfinalone} the parameters
\be\label{parametri1}
  \vr:= \underline \varrho \, ,\quad
 K:= \underline K'(\vr) \in \N\, ,
\ee
where $\underline K'(\vr)$ is defined in Proposition \ref{teoredu1}.
Thus Proposition \ref{birkfinalone} provides
$ \underline s_0 > 0 $ and for any
\be \label{parametri2}
s\geq \underline s_0 \,  \quad \text{we fix} \quad  0<r< \min( \underline r_0(s), \bar r(s)) \, , 
\ee
where $ \underline r_0(s) > 0  $  is given by  Proposition \ref{birkfinalone} and $ \bar r(s)>0$ is given by Lemma \ref{equiv-tempo}. Let 
 $U(t) $ be a solution of system  \eqref{complexo} in  $B_{\underline s_0}^K(I;r)\cap C_{*\R}^{K}(I;\dot{H}^{s}(\T, \C^2))$. \end{itemize}}
 
 In order to prove Theorem \ref{teo1} we have to provide energy estimates of the 
 non-homogenous ``vector field"  in \eqref{final:eq}:
\be\label{XVFN}
X_{\geq N+2} (U,Z) := 
\vect{X^+_{\geq N+2} (U,Z) }{ 
\ov {X^+_{\geq N+2} (U,Z)}} := 
\vOpbw{- \ii (\tm_{\frac32})_{>N}(U;t,\xi)}Z+ R_{>N}(U;t)U \, . 
\ee
The following lemma  holds since the imaginary part of the $x$--independent symbol $(\tm_{\frac32})_{>N}(U;t,\xi)$ has order zero and because, with the choice of $ \vr=\underline \varrho$ in \eqref{parametri1}, the remainder $R_{>N}(U;t)$ in \eqref{XVFN} belongs to $\mR^{0}_{K,\underline{K}',N+1}[r] \otimes\mathcal{M}_2(\C) $.
\begin{lemma}\label{ene2}
{\bf (Energy estimate)}
The  non-homogenous vector field  $X_{\geq N+2} (U,Z)$ defined in \eqref{XVFN}, where $(\tm_{\frac32})_{>N}$ and  $R_{>N}(U;t)$ are defined in Proposition \ref{birkfinalone} with parameters given in \eqref{parametri1}, \eqref{parametri2}, 
satisfies, for any $ t \in I $, the energy estimate
\be\label{stimFF}
\Re \int_\T     |D|^{s} \,  X_{\geq N+2}^+ (U,Z) \cdot 
\overline{|D|^s z} \, \di x   \leq C_{s,K} \| Z \|_{\dot H^s}^{N+3} \, .
\ee
\end{lemma}

\begin{proof}
Since $ (\tm_{\frac32})_{>N}(U;t,\xi)$ is $x$--independent 
and has imaginary part $ \Im (\tm_{\frac32})_{>N}$ in $ \Gamma_{K,\underline{K}',N+1}^0[r]$
(cfr. Proposition \ref{birkfinalone})
then  $\opbw\big(-\ii ( \tm_{\frac32})_{>N} \big)$ 
commutes with the derivatives 
$ |D|^s$ and, by \eqref{piovespect} (with $k=0$, $m=0$, $s=0$ and $K=\underline K'$), 
\be\label{noveundici0}
\begin{aligned}
 &\Re \int_\T |D|^{s}\Opbw{-\ii ( \tm_{\frac32})_{>N}(U;t,\xi)}z \cdot |D|^{s} \ov z \, \di x\\
 &=\Re  \int_\T  \Opbw{ \Im ( \tm_{\frac32})_{>N}(U;t,\xi)  } |D|^{s} z \cdot |D|^{s}\bar z  \, \di x  \lesssim \| U\|_{\underline{K}',s_0}^{N+1} \| Z \|_{\dot H^s}^2 \\
 &  \stackrel{ \text{Lemma \ref{equiv-tempo}}, \eqref{equivalenzan}} {\lesssim_{s,K}}
 \| Z \|_{\dot H^s}^{N+3} \, . 
 \end{aligned}
\ee
We consider now the contribution coming from the
smoothing operator $R_{>N}(U;t)$ in $\mR^{0}_{K,\underline{K}',N+1}[r] \otimes\mathcal{M}_2(\C) $.
Using \eqref{piove} (with $k=0$, $m=0$), Lemma \ref{equiv-tempo} and \eqref{equivalenzan}  we get  
\be \label{novedieci0}
 \Re \int_\T |D|^{s}\big(R_{>N}(U;t)U\big)^+ \cdot |D|^{s} \ov z \, \di x \lesssim_{s}  
\| U\|_{\underline{K}',s_0}^{N+1} \| U\|_{\underline{K}',s} \| Z \|_{\dot H^s}  \lesssim_{s,K} \| Z \|_{\dot H^s}^{N+3} \, . 
\ee
The estimate \eqref{stimFF} follows by \eqref{noveundici0} and \eqref{novedieci0}. 
\end{proof}

We  now  prove the following key bootstrap result. 
By time reversibility we may, without loss of generality, consider  only positive times $t>0$.

\begin{proposition}{\bf (Bootstrap)} \label{proapriori}
For any $ s\geq \underline s_0$  there exist $c_0:= c_0(s,K)\in (0,1)$,  such that for any solution  $U(t)\in C^K_{*\R}(I; \dot H^{s}(\T;\C^2))$ of \eqref{complexo} fulfilling, for some $0< \e_1 < \min (\underline r_0(s), \bar r(s)) $,
\be\label{apriorih}
{\| U(0)\|}_{\dot H^s} \leq c_0 \e_1 \, , \qquad \sup_{t\in[0,T]} 
\sum_{k=0}^K {\| \pa_t^k U(t) \|}_{\dot H^{s- \frac32 k}} \leq \e_1 \, ,\qquad T \leq
 c_0 \e_1^{-N-1} \,,
\ee 
 then we have the improved bound
\be\label{aprioric}
\sup_{t\in[0,T]} \sum_{k=0}^K {\| \pa_t^k U(t) \|}_{\dot H^{s- \frac32 k}} \leq \dfrac{\e_1}{2} \, .
\ee 
\end{proposition}

\begin{proof}
By \eqref{apriorih} we have  that $U \in B^K_s(I;\e_1)$ with $I:= [-T, T]$. 
By Proposition \ref{birkfinalone} (applied with  $r\leadsto \e_1<\underline r_0(s)$), the variable 
$$
Z(t):=\mF_{\tnf}(Z_0(t)),\quad Z_0(t):=  \big( \uno+ R( \bB(U(t);t)U(t)) \big) \bB(U(t);t)U(t) \, ,
 $$
where $\bB(U;t)$ is defined in Proposition \ref{teoredu1} and the smoothing operator $R( \cdot ) $  is defined in \eqref{zetone},
 solves \eqref{final:eq} and has a Sobolev norm equivalent to that of $U(t)$, see \eqref{equivalenzan}. 
 Lemmata  \ref{lemmatronco} and \ref{ene2} (using also $\varepsilon_1 < \min( \underline r_0(s), \bar r(s))$) imply that the solution $Z(t)$ of system \eqref{final:eq} satisfies  the energy estimate
\be\label{stimaEW}
\frac{\di}{\di t} \| Z(t)\|_{\dot H^s}^2 \leq C_{s,K} \| Z \|_{\dot H^s}^{N+3}
\ee
and therefore
for all $ 0 \leq t \leq T $, by \eqref{equivalenzan}, 
\begin{align}
\notag
{\|U(t)\|}_{\dot{H}^{s}}^{2} \leq C_{s,K}  {\|Z(t)\|}^{2}_{\dot{H}^{s}}  
&\stackrel{\eqref{stimaEW}}{\leq}  C_{s,K} \left(  {\| Z(0) \|}^{2}_{\dot{H}^s} 
+  \int_0^t {\| Z(\tau)\|}^{N+3}_{\dot{H}^s} \, \di \tau \right) \\
\label{eeN}
& \stackrel{\eqref{equivalenzan}} \leq C_{s,K}'{\|U(0)\|}^{2}_{\dot{H}^{s}} + C_{s,K}' \int_{0}^{t} \|U(\tau)\|^{N+3}_{\dot{H}^{s}} \, \di \tau \, . 
\end{align}
Then, by the a priori assumption \eqref{apriorih}
we deduce that, for all $0\leq t\leq T\leq  c_0\e_1^{-(N+1)}$,
\be
\label{aprioriEne12}
{\|U(t)\|}^{2}_{\dot{H}^{s}} \leq C_{s,K}' c^{2}_0 \e_1^{2} + C_{s,K}' \, {T} \, \e_1^{N+3} 
\leq  \e_1^{2}(C_{s,K}' c_0^2+C_{s,K}' c_0) \, .
\ee
The desired conclusion \eqref{aprioric}
on the norms $C^k_t  H^{s-\frac32k}_x$
follows by Lemma \ref{equiv-tempo}, \eqref{aprioriEne12},   
choosing $c_0$  small enough depending on $s$ and $ K$. 
\end{proof}

\noindent
{\bf Proof of Theorem \ref{teo1} concluded.} 
\\[1mm]
{\em Step 1: Local existence.} 
By the  local existence theory, there exist 
$ r_{{\rm loc}}, s_{{\rm loc}} >0$ such that for any $s \geq  s_{{\rm loc}}$ there are $T_{\rm loc}>0, C_{\rm loc}\geq1$ such that, any initial datum $(\eta_0, \psi_0)\in H_0^{s+\frac{1}{4}}(\mathbb{T})\times {\dot H}^{s-\frac{1}{4}}$ with 
$$
\|\eta_0\|_{H_0^{s+\frac14}}+\|\psi_0\|_{\dot{H}^{s-\frac14}}\leq r_{{\rm loc}}
$$
 there exists a unique classical solution $(\eta(t), \psi(t))$ in 
$C^0\bigl([-T_{\rm loc},T_{\rm loc}],H_0^{s+\frac14}(\mathbb{T})\times {\dot H}^{s-\frac14}(\mathbb{T})\bigr)$ of \eqref{eq:etapsi} satisfying the initial condition $\eta(0) = \eta_0, \psi(0)= \psi_0$ and 
\begin{equation}\label{smallness_loc}
\sup_{t\in [-T_{\rm loc},T_{\rm loc}]}\big( 
 \|\eta(t)\|_{H_0^{s+\frac14}}+\|\psi(t)\|_{\dot{H}^{s-\frac14}}\big)\leq C_{\rm loc} 
 \big( \| \eta_0 \|_{H^{s+\frac14}_0}+  \|\psi_0\|_{{\dot H}^{s-\frac14}} \big) \, . 
\end{equation}

\begin{remark}
The local existence result can be derived  arguing as in  \cite{FI1,BMM} (which have  the very same  paradifferential functional setting developed in this paper). 
One could also extend the proof of  \cite{ABZduke} to the case of constant vorticity.
\end{remark}
\noindent
{\em Step 2: Complex variables.}
System \eqref{complexo} and the water waves equations \eqref{eq:etapsi} are equivalent under the linear change of variables 
\be \label{etainU}
U= \vect{u}{\bar u}
=  \mM^{-1} \circ \cW^{-1} \vect{\eta}{\psi}, \quad u 
= \frac{1}{\sqrt 2} M(D)^{-1} \eta + \frac{\im}{\sqrt 2} M(D) \zeta \, , 
\quad \zeta:= \psi - \frac{\gamma}{2} \pa_x^{-1} \eta \, , 
\ee 
defined in \eqref{Whalen}, \eqref{defMM-1}.
In view of \eqref{etainU} we have the equivalence of the norms: for some $ \tC:= \tC(\mathtt{h},\gamma, \kappa)\geq 1$, any 
$ t$ 
\be\label{improved.est} 
\tC^{-1} \|  U(t)\|_{\dot H^{s}}\leq \norm{ \eta(t)}_{H_0^{s+\frac14 }} + \norm{ \psi(t)}_{\dot H^{s-\frac14 }} \leq \tC \|  U(t)\|_{\dot H^{s}} \, .
\ee
{\em Step 3: Bootstrap argument.}  
Consider the local solution of \eqref{eq:etapsi} with initial datum $(\eta_0,\psi_0)$ satisfying 
\be\label{datoinizialepiccolo}
\norm{ \eta_0}_{H_0^{s+\frac14}} + \norm{ \psi_0}_{\dot H^{s-\frac14 }} \leq \e \leq \e_0 
\ee
where 
\be \label{sceltafinale}
 \e_0:= \min\Big(\frac{\underline r_0(s)} {\breve C_{s,K}},  \frac{\bar r(s)}{\breve C_{s,K}},  r_{{\rm loc}}\Big) \, , \quad 
 \breve C_{s,K}:= \max \Big( \frac{\tC}{c_0(s,K)}, K \bar C_{s,K} \tC C_{{\rm loc}} \Big) \, . 
\ee
By \eqref{improved.est}, \eqref{datoinizialepiccolo}, \eqref{sceltafinale}, 
Lemma \ref{equiv-tempo} (with $ r\leadsto \tC C_{{\rm loc}} \varepsilon_0 < \bar r(s)$) and 
\eqref{smallness_loc} we deduce that 
\be\label{conditobene}
\| U(0)\|_{\dot H^s} \leq \tC \e, \quad \| \pa_t^k U (t) \|_{\dot H^{s-\frac32k}} \leq \bar {C}_{s,K} \tC C_{{\rm loc}} \e \, , \quad \forall  0<t< T_{{\rm loc}} \, , \  \forall k=0,\dots ,K \, .
\ee
By \eqref{conditobene} and 
\eqref{sceltafinale},  the solution $ U(t)$ of \eqref{complexo}  satisfies, for any $ 0<\e<\e_0 $,  the smallness condition 
$$
{\| U(0)\|}_{\dot H^s} \leq c_0 \e_1 \, , \qquad \sup_{t\in[0,T_{\rm loc}]} 
\sum_{k=0}^K {\| \pa_t^k U(t) \|}_{\dot H^{s- \frac32 k}} \leq \e_1  \quad
\text{with} \quad
\varepsilon_1:=  \breve C_{s,K} \e \, , 
$$
which is \eqref{apriorih}  with  $ T = T_{{\rm loc}} $.
Proposition \ref{proapriori} and a  standard bootstrap argument imply that  
the maximal time   $ T_{\rm max} \geq T_{{\rm loc}} $ 
of existence of the solution $ U(t) $   is larger than $ T_\e := c_0 \e_1^{-N-1}$ and 
$ \| U (t) \|_{\dot H^{s}} \leq  \e_1 $ for any $ 0 \leq t < T_\e  $.   
By \eqref{etainU} this proves that the solution
$ (\eta, \psi )(t)  = \cW {\cal M} U (t) $  of the water waves equations \eqref{eq:etapsi} satisfies
\eqref{timeexi} and \eqref{boundsolN} 
with $c:= c_0 \breve C_{s,K}^{-N-1}$ and $C:= \tC\breve C_{s,K}$.

\appendix

\section{Conjugation lemmata}

We collect in this Appendix important results   used along the paper about how paradifferential equations are conjugated under the flow of an unbounded Fourier multiplier
(Lemma \ref{lem:conj.fou})
and an approximate flow generated by a smoothing vector field (Lemma \ref{NormFormMap0}).   
\smallskip

The following result about the approximate inverse is a consequence of Lemma \ref{inversoapp}. 

\begin{lemma}\label{cor.app.inv}
Assume \eqref{Z.M0} with $U $ in $ B^K_{s_0,\R}(I;r)$, 
$K, K' \in \N_0$ with $K' \leq K$, $r >0$.
 Let $p, N \in \N $, $p \leq N$,  $m \geq 0$, and consider 
\be\label{W.change}
W := \Phi(Z)  = Z + S(Z;t)Z 
\ee
where 
\be \label{c1c2M}
 S(Z;t) \in \begin{cases}
 \Sigma\cS^m_{K,K',p}[r,N] \otimes \cM_2(\C) \qquad \qquad 
 \text{if} \quad  \bM_0(U;t)  = \id \, , \\
  \Sigma\cS^m_{K,0,p}[\breve r,N] \otimes \cM_2(\C) \, , \, \forall \breve r > 0 \quad \text{otherwise} \, . 
\end{cases}
\ee
Then we may write 
\be\label{W.change2}
Z = \Psi_{\leq N}(W) + M_{>N}(U;t)U  \quad 
\text{with} \quad \Psi_{\leq N}(W) := W +  \breve S_{\leq N}(W)W \, , 
\ee
where 
\begin{itemize}
\item 
$ \breve S_{\leq N}(W) $ is a matrix of pluri-homogeneous spectrally localized maps 
in $  \Sigma_p^N \wt \cS^{m(N-p +1)}_q \otimes \cM_2(\C)$; 
 \item 
$M_{>N}(U;t)  $ is a matrix of  non-homogeneous operators in 
 $\cM^{m(N-p +2)}_{K, K',N+1}[r] \otimes \cM_2(\C) $.
\end{itemize}
\end{lemma}

\begin{proof}
By   Lemma \ref{inversoapp} there exists 
an approximate inverse up to homogeneity $N$ of
 the pluri-homogeneous nonlinear map  (obtained by \eqref{W.change})
$$
\Phi_{\leq N}(Z):= Z + \cP_{\leq N}[S(Z;t)] Z
$$ 
  having the form 
  $$
  \Psi_{\leq N}(W) = W + \breve S_{\leq N}(W)W \quad \text{with} \quad
   \breve S_{\leq N}(W) = 
   \sum_{q=p}^N \breve S_q (W) \, ,  \quad  
  \breve S_q (W) \in  \wt \cS^{m(N-p +1)}_q \otimes \cM_2(\C) 
 \, . 
  $$ 
Applying $ \Psi_{\leq N}$ to \eqref{W.change}, writing $ 
\Phi(Z) = \Phi_{\leq N}(Z) +S_{>N}(Z;t)Z $ 
with  
$S_{>N}(Z;t) := S(Z;t)- \cP_{\leq N}[S(Z;t)] $,
we get 
    \begin{align}
    \notag
 &   \Psi_{\leq N}(W)     =  \Psi_{\leq N}\big(\Phi_{\leq N}(Z) +S_{>N}(Z;t)Z \big) \\
 \notag
 & = \Psi_{\leq N}\big(\Phi_{\leq N}(Z) \big)+
\underbrace{\int_0^1 \di_W \Psi_{\leq N}\left( \Phi_{\leq N}(Z) + \tau S_{>N}(Z;t)Z\right) [S_{>N}(Z;t)Z] \di \tau}_{ =: S''_{>N}(Z;t)Z}  \\
& \stackrel{\eqref{invmappe}}{ = }
 Z + S'_{>N}(Z)Z +  S''_{>N}(Z;t)Z   \label{ZWinlem}
     \end{align}
  where  $S'_{>N}(Z)\in  \Sigma_{N+1}  \wtcS_q^{m(N-p+2)}\otimes \mM_2 (\C) $
  by Lemma  \ref{inversoapp}-($i$) 
  and,  according to  \eqref{c1c2M}, by Proposition \ref{composizioniTOTALIs}-($v$)
\be\label{S''N}
  S''_{>N}(Z;t) 
  \in 
  \begin{cases}
   \cS^{m(N-p +2)}_{K,K',N+1}[r] \otimes \cM_2(\C)  \qquad \qquad  \ \ 
 \text{if} \quad  \bM_0(U;t)  = \id \, , \\
   \cS^{m(N-p +2)}_{K,0,N+1}[\breve r] \otimes \cM_2(\C) \, , \, \forall \breve r > 0
   \quad \text{otherwise} 
   \, . 
  \end{cases}
\ee
Finally we substitute $Z= \bM_0(U;t)U$ 
where $ \bM_0(U;t) $ is in $ \cM^0_{K,K',0}[r]\otimes \mM_2(\C) $
(cfr. \eqref{Z.M0}) in the non--homogeneous term  $S'_{>N}(Z)Z $ and $ S''_{>N}(Z;t)Z$ 
in \eqref{ZWinlem}-\eqref{S''N}  and using $(iii)$ and $(i)$ Proposition \ref{composizioniTOTALI}  we 
deduce \eqref{W.change2} and that 
$M_{>N}(U;t) \in \cM^{m(N-p +2)}_{K, K',N+1}[r] \otimes \cM_2(\C) $.
\end{proof}

We provide a lemma concerning how  paradifferential and smoothing operators change by substituting in the `internal" variables a close to the identity map.

\begin{lemma} \label{sost}
Assume $ W= \bM_0(U;t) U$ with $ \bM_0(U;t) \in \mM_{K,K',0}^0[r] \otimes \mM_2(\C)$, $U \in B_{s_0, \R}^{K}(I;r)$ for some $r,s_0>0$ and  $ 0\leq K'\leq K$.
 Let $p, N\in \N$ with $p\leq N$, $m\in \R$  and 
consider a nonlinear map 
\be\label{Z.W1} 
 Z=\Psi_{\leq N}(W)+M_{>N}(U;t)U \quad \text{where} \quad  \Psi_{\leq N}(W)= W+ M_{\leq N}(W)W 
\ee
 with  
 
 $\bullet$
 $  M_{\leq N}(W)$ is a matrix of pluri-homogeneous $m$-operators   in $\Sigma_{p}^N\wt \mM_q^{m} \otimes \mM_2(\C)$;
 
 $\bullet$
 $M_{>N}(U;t)$ is a matrix of non homogeneous $m$-operators in  $\mM_{K,K',N+1}^{m}[r] \otimes \mM_2(\C) $.
 
 \noindent Then 
\begin{itemize}
\item[$(i)$] ({\bf Symbols}) 
if $ a_{\leq N}(Z;\xi)$ is a pluri-homogeneous real-valued symbol, independent of $x$, in $ \Sigma_{2}^N \wt \Gamma_q^{m'}$, $m'\in \R$, then 
\be
\begin{aligned} 
&\Opbw{a_{\leq N}(Z ;\xi)}_{|Z=\Psi_{\leq N}(W)+M_{>N}(U;t)U}\\
\label{sostbony}&= \Opbw{a^+_{\leq N}(W;\xi) + a^+_{>N}(U;t,\xi)} + R^+_{\leq N}(W) +R_{>N}^+(U;t) \, ,
\end{aligned}
\ee
where\\
\noindent
$\bullet$ 
$a^+_{\leq N}(W;\xi)$ is a pluri-homogeneous real-valued symbol independent of $x$ in  $\Sigma_{2}^N \wt \Gamma^{m'}_q $ such that 
\be \label{invariatosym}
\mP_{\leq p+1}(a^+_{\leq N}(W;\xi))=\mP_{\leq p+1}(a_{\leq N}(W;\xi)) \, ;
\ee
\noindent
$\bullet$
$a^+_{>N}(U;t,\xi)$ is a non-homogeneous real valued symbol independent of $x$ in $\Gamma^{m'}_{K,K',N+1}[r]$;\\
\noindent
$\bullet$  $ R^+_{\leq N}(W)$ is a pluri-homogeneous smoothing operator in  $ \Sigma_{p+2}^N \wtcR^{-\vr}_q \otimes \mM_2(\C) $ for any $\varrho \geq0$;\\
\noindent
$\bullet$ $R_{>N}^+(U;t)$  is a non-homogeneous smoothing operator  in $\mR^{-\vr}_{K,K',N+1}[r]  $ for any $\varrho \geq 0$.
\item[$(ii)$]({\bf Smoothing operators}) If $ R_{\leq N}(Z)$ is a pluri-homogeneous smoothing operator in $ \Sigma_{1}^N \wtcR^{-\vr}_q $, for some $\vr\geq0$,  then 
\be\label{sostsmoo}
\begin{aligned}
(&R_{\leq N}(Z)Z)_{|Z=\Psi_{\leq N}(W)+M_{>N}(U;t)U}= R^+_{\leq N}(W)W + R_{>N}^+(U;t)U \, ,\\
&R_{\leq N}(Z)_{|Z=\Psi_{\leq N}(W)+M_{>N}(U;t)U}= R'_{\leq N}(W) + R_{>N}'(U;t)
\end{aligned}
\ee
where \\
$\bullet$ $R^+_{\leq N}(W)$ and $R'_{\leq N}(W)$ are pluri-homogeneous smoothing operators in $ \Sigma_{1}^N \wtcR^{-\vr+(N+1)\max(m,0)}_q $ such that 
\be \label{invariatosmo}
\mP_{\leq p}(R^+_{\leq N}(W))=\mP_{\leq p}(R'_{\leq N}(W))=\mP_{\leq p}(R_{\leq N}(W)) \, ;
\ee
$\bullet$ $R_{>N}^+(U;t)$ and $R_{>N}'(U;t)$ are non-homogeneous smoothing operators  in $\mR^{-\vr+(N+1)\max(m,0)}_{K,K',N+1}[r]$.  
\end{itemize}
If  $\vr\leq  (N+1)\max(m,0)$ we regard  $ R^+_{\leq N}$ and $R_{>N}^+$ as  operators in $ \Sigma_1^N \wtcM^{-\vr+ (N+1)\max(m,0)}_q$ and respectively $\mM^{-\vr+(N+1)\max(m,0)}_{K,K',N+1}[r]$.
\end{lemma}

\begin{remark}
The previous lemma is stated for $x$ independent symbols (since it is used in this case) 
but it holds also for a general symbol. 
\end{remark}

\begin{proof} {\bf Proof of $(i)$:} We expand by multilinearity the operator in \eqref{sostbony}.
We denote the homogeneous components of $M_{\leq N}(W)$ in \eqref{Z.W1} as 
$ M_\ell(W) := \mP_\ell( M_{\leq N}(W))$ for $\ell=p, \dots, N$ and $M_0(W):=\uno$.
 Note that 
 $$ 
 (M_{\ell}(W)W)_{|W=\bM_0(U;t)U}=M^{(\bM_0)}_{\ell}(U;t)U
 $$ 
 where  
 $ M^{(\bM_0)}_{\ell}(U;t):=M_{\ell}(\bM_0(U;t)U)\bM_0(U;t)$ 
 belongs to $ \mM^m_{K,K',\ell}[r]\otimes \cM_2(\C)$ for  $ \ell\in \{0,p,\dots, N\} $
 thanks to $(i)$ and $(ii)$ of 
 Proposition \ref{composizioniTOTALI}. 
 Then by multilinearity decompose the operator in \eqref{sostbony} as 
\begin{align}
&\Opbw{a_{\leq N}(Z;\x)}_{|Z=\Psi_{\leq N}(W) + M_{>N}(U;t)U}\notag\\
&= \sum_{\ta=2}^N \sum_{q=2}^N\sum_{\substack{ \ell_1,\dots, \ell_q\in \{ 0, p, \dots , N\}\\\ell_1+\dots+\ell_q+q=\ta} }\Opbw{a_q(Z_1, \dots,Z_q; \xi)}_{|Z_1=M_{\ell_1}(W)W, \dots, Z_q= M_{\ell_q}(W)W}\label{omosomma}\\
&+ \sum_{q=2}^N\sum_{\substack{\ell_1,\dots, \ell_q \in \{0,p,\dots N\}\\\ell_1+\dots+\ell_q+q\geq N+1} }\Opbw{a_q(Z_1, \dots, Z_q; \xi)}_{|Z_1= M^{(\bM_0)}_{\ell_1}(U;t)U, \dots, Z_q=M^{(\bM_0)}_{\ell_q}(U;t)U}\label{nonomosomma}\\
 &+ \sum_{q=2}^N \, \sum_{n = 0}^{q-1} \,  
\!\!\sum_{\ell_1,\dots, \ell_n \in \{0,p,\dots N\} } 
\binom{q}{n} \, \Opbw{a_q(Z_1, \dots,Z_q;  \xi)}_{\big|
 \begin{smallmatrix} 
&Z_1= M^{(\bM_0)}_{\ell_1}(U;t)U, \dots, Z_n=M^{(\bM_0)}_{\ell_n}(U;t)U \\
&Z_{n+1}= \dots = Z_q= M_{>N}(U;t)U
\end{smallmatrix}
}.\label{nonomosomma2}
\end{align}
By $(iv)$ of  Proposition \ref{composizioniTOTALI} we have that
\be\label{omosomma10}
\eqref{omosomma} = \Opbw{a_{\leq N}^+(W)} + R^+_{\leq N}(W)
\quad \text{with} \quad 
R_{\leq N}^+(W)\in \Sigma_2^N \wtcR^{-\vr}_q 	\, ,    \  \forall \varrho \geq 0 \, , 
\ee
where $a_{\leq N}^+(W)= \sum_{\ta=2}^N a^+_{\ta}(W;\xi)$ and 
\be
 \quad a^+_{\ta}(W;\xi):= \sum_{q=2}^N\sum_{\substack{ \ell_1,\dots, \ell_q\in \{ 0, p, \dots , N\}\\\ell_1+\dots+\ell_q+q=\ta} }a_q(M_{\ell_1}(W)W, \dots, M_{\ell_q}(W)W; \xi)\label{appiu}
\ee
belongs to $\wt \Gamma^{m'}_\ta$.
 For $ \ta \leq p+1$ the sum in  \eqref{omosomma} and 
 \eqref{appiu} reduces to the indices $ q=\ta$, $ \ell_1=\dots=\ell_\ta=0$. As a consequence  $a^+_{\ta}(W;\xi)= a_{\ta}(W;\xi)$ for $\ta=2,\dots,p+1$, proving \eqref{invariatosym}.
For the same reason the remainder $R^+_{\leq N}(W)$ in \eqref{omosomma10} actually belongs to $\Sigma_{p+2}^N \wtcR^{-\vr}_q$.

Now we consider the non-homogeneous terms which arise from lines \eqref{nonomosomma} and \eqref{nonomosomma2}. Thanks to $(iv)$ of Proposition  \ref{composizioniTOTALI} we have that 
$$
\eqref{nonomosomma}+ \eqref{nonomosomma2}= \Opbw{a^+_{>N}(U;t,\xi)}+ R^+_{>N}(U;t)
$$
where $R_{>N}^+(U;t)$ is a
 smoothing operator in $\mR^{-\vr}_{K,K',N+1}[r]$ for any $\varrho\geq 0 $,  and 
$a^+_{>N}(U;t,\xi)$
is a non--homogeneous symbol  in $ \Gamma^{m'}_{K,K',N+1}[r]$ which is real valued and $x$-independent as well as $a_{\leq N}$.

\noindent {\bf Proof of $(ii)$:}  Proceeding in similarly to  $(i)$ we expand the left hand side of  \eqref{sostsmoo} as 

 \begin{align}
&(R_{\leq N}(Z)Z)_{|Z=\Psi_{\leq N}(W)+M_{>N}(U;t)U}\notag\\
&= \sum_{\ta=1}^N \sum_{q=1}^N\sum_{\substack{ \ell_1,\dots, \ell_{q+1}\in \{ 0, p, \dots , N\}\\\ell_1+\dots+\ell_{q+1}+q=\ta} }R_q(M_{\ell_1}(W)W, \dots, M_{\ell_q}(W)W)M_{\ell_{q+1}}(W)W\label{omosommasmo}\\
&+ \sum_{q=1}^N\sum_{\substack{\ell_1,\dots, \ell_q \in \{0,p,\dots N\}\\\ell_1+\dots+\ell_q+q\geq N+1} }R_q(M^{(\bM_0)}_{\ell_1}(U;t)U, \dots, M^{(\bM_0)}_{\ell_q}(U;t)U)M^{(\bM_0)}_{\ell_{q+1}}(U;t)U\label{nonomosommasmo}\\
 &+ \sum_{q=1}^N \, \sum_{n = 0}^{q-1} \,  
\sum_{\ell_1,\dots, \ell_{n+1} \in \{0,p,\dots N\} } 
\binom{q}{n} \, R_q( Z_1, \ldots, Z_q)M^{(\bM_0)}_{\ell_{n+1}}(U;t)U_{\big|
 \begin{smallmatrix} 
&Z_1= M^{(\bM_0)}_{\ell_1}(U;t)U, \dots, Z_n=M^{(\bM_0)}_{\ell_n}(U;t)U \\
&Z_{n+1}= \dots = Z_q= M_{>N}(U;t)U
\end{smallmatrix}
}\\
 &+ \sum_{q=1}^N \,  \sum_{n = 0}^{q-1} \,  
\sum_{\ell_1,\dots, \ell_n \in \{0,p,\dots N\} } 
\binom{q}{n} \, R_q( Z_1, \ldots, Z_q)M_{>N}(U;t)U_{\big|
 \begin{smallmatrix} 
&Z_1= M^{(\bM_0)}_{\ell_1}(U;t)U, \dots, Z_n=M^{(\bM_0)}_{\ell_n}(U;t)U \\
&Z_{n+1}= \dots = Z_q= M_{>N}(U;t)U
\end{smallmatrix}
} .
\label{nonomosommasmo2}
\end{align}
Thanks to $(ii)$ of Proposition \ref{composizioniTOTALI} (with $m\leadsto -\vr$, $m_\ell \leadsto m$), the term in \eqref{omosommasmo} can be written as $R^+_{\leq N}(W)W$ where $ R^+_{\leq N}(W)$ is a pluri--homogeneous smoothing operator in $ \Sigma_1^N \wtcR^{-\vr+ (N+1)\max(m,0)}_q$, moreover $ R^+_{\leq N}(W)= \sum_{\ta=1}^N R_\ta^+(W)$ with 
\be\label{appiusmo}
R_\ta^+(W):=\sum_{q=1}^N\sum_{\substack{ \ell_1,\dots, \ell_{q+1}\in \{ 0, p, \dots , N\}\\\ell_1+\dots+\ell_{q+1}+q=\ta} }R_q(M_{\ell_1}(W)W, \dots, M_{\ell_q}(W)W)M_{\ell_{q+1}}(W) \, .
\ee
For $ \ta \leq p$ the sum in  \eqref{appiusmo} reduces to the  indices $ q=\ta$, $ \ell_1=\dots=\ell_{\ta+1}=0$. As a consequence   $R^+_{\ta}(W)= R_{\ta}(W)$ for $\ta=1,\dots,p$, proving \eqref{invariatosmo}.
 Applying again Proposition \ref{composizioniTOTALI} $(ii)$  we get that the terms in \eqref{nonomosommasmo}--\eqref{nonomosommasmo2} can be written as $ R_{>N}^+(U;t)U$ where $  R_{>N}^+(U;t)$ is a non--homogenous smoothing operator in $ \mR^{-\vr+(N+1)\max(m,0)}_{K,K',N+1}[r]$. This concludes the proof of the first identity in \eqref{sostsmoo}. The second one follows with the same analysis, without the need of substitute the last variable. 
\end{proof}

\paragraph{Conjugation lemmata.}
The following 
conjugation Lemmata \ref{lem:conj.fou} and \ref{NormFormMap0} 
are used in the nonlinear Hamiltonian Birkhoff normal form reduction performed in Section \ref{sec:birk}.

\smallskip

The following hypothesis shall be assumed in both Lemmata \ref{lem:conj.fou} and \ref{NormFormMap0}:
\\[1mm]
{\em {\bf Assumption (A)}: \label{A} \,
Assume  $Z := \bM_0(U;t)U $ where $\bM_0(U;t) \in  \cM^0_{K,K',0}[r] \otimes \mM_2 (\C)$, $U \in B^K_{s_0, \R}(I;r)$ for some $r,s_0>0$ and $0\leq K' \leq K$. 
 Let  $ N \in \N$  and assume that $Z$ solves the system
\begin{align}\label{Z.conj1}
\pa_t Z = - \ii \vOmega(D) Z +\vOpbw{ \ii a_{\leq N}(Z;\x)+ \ii a_{>N}(U;t,\xi)}Z + R_{\leq N}(Z)Z+ R_{>N}(U;t)U  
\end{align}
where $\vOmega(D) $ is the diagonal matrix of Fourier multiplier operators defined in \eqref{diaglin} and 
\begin{itemize}
\item $a_{\leq N} (Z;\x) $ is a real valued   pluri-homogenous 
symbol, independent of $x$, in $ \Sigma_2^N \wt \Gamma^{\frac32}_q$; 
\item 
$a_{>N} (U;t,\xi) $ is a non-homogenous symbol, independent of $x$, in $ \Gamma^{\frac32}_{K,K',N+1}[r]$ with imaginary part 
${\rm Im} \, a_{>N} (U;t,\xi) $ in $  \Gamma^{0}_{K,K',N+1}[r]$;
\item 
 $R_{\leq N} (Z)  $ is a real-to-real matrix of pluri-homogeneous  smoothing operators in $  \Sigma_1^N \wtcR^{-\vr}_q
\otimes \mM_2 (\C)  $; 
\item 
$R_{>N} (U;t) $ is a real-to-real  matrix of
non-homogeneous smoothing operators 
in $ \mR^{-\vr}_{K,K',N+1}[r] \otimes \mM_2 (\C) $.
\end{itemize}
}

\vspace{.5em}

We also write system \eqref{Z.conj1} in the form  
\be \label{Z.conj2}
\pa_t Z = -\ii \vOmega(D)Z+ M_{\leq N}(Z) Z+ M_{>N}(U;t)U 
\ee
where    
$M_{\leq N}(Z) $ are $ \frac32 $-operators in 
$ \Sigma_{1}^N \wtcM_q^{\frac32} \otimes  \mM_2 (\C) $ and  $M_{>N}(U;t)
$ is in $ \mM^{\frac32}_{K,K',N+1}[r] \otimes \mM_2 (\C) $ by the fourth remark below Definition \ref{smoothoperatormaps}.

\begin{lemma}[{\bf Conjugation under the flow of a Fourier multiplier}]\label{lem:conj.fou}
Assume {\bf (A)} at page \pageref{A}.
 Let $ g_p(Z;\xi)$ be a $p$--homogeneous real symbol independent of $x$ in $ \wt \Gamma^{\frac32}_p$,  $p \geq 2$, 
that we expand as 
\be \label{symtens}
g_p(Z;\x) = 
\sum_{
(\vec{\jmath}_p,  \vec{\sigma}_p) \in \fT_p
} G_{\vec{\jmath}_p}^{\vec{\sigma}_p}(\xi)  z_{\vec{\jmath}_p}^{\vec{\sigma}_p} \, , \quad \bar{G_{\vec{\jmath}_p}^{-\vec{\sigma}_p}}(\xi)= G_{\vec{\jmath}_p}^{\vec{\sigma}_p}(\xi)\in \C  
\ee
and denote by $ \cG_{g_p}(Z):=\cG_{g_p}^1(Z) $  the time $1$-flow defined in  \eqref{FouFlow} generated by $\vOpbw{\im g_p(Z;\xi)}$.
If $Z(t)$ solves system \eqref{Z.conj1}, then the variable 
\be \label{new:foumu}
W:= \cG_{g_p} (Z) Z
\ee 
solves the system 
\be\label{Z.conj10}
\pa_t W =   - \ii \vOmega(D) W +\vOpbw{ \ii a_{\leq N}^+(W;\x)+ \ii a_{>N}^+(U;t,\xi)}W  + 
 R_{\leq N}^+(W)W+ R^+_{>N}(U;t)U  
\ee
where 
\begin{itemize}
\item  $ a^+_{\leq N} (W;\x) $ is a real valued pluri-homogenous symbol, independent of $x$,  in  $\Sigma_2^N \wt \Gamma^{\frac32}_q $, with components
\be\label{ditigi0}
\begin{aligned}
& \cP_{\leq p-1} [a^+_{\leq N}(W;\xi)] =  \cP_{\leq p-1}[a_{\leq N}(W;\xi)] \, , \\
& \cP_{ p} \left[a^+_{\leq N}(W;\xi)\right]  = \cP_{ p} \left[a_{\leq N}(W;\xi) \right]
  + g^+_p(W;\xi)
  \end{aligned}
\ee
where $g^+_p(W;\xi) \in  \wt \Gamma^{\frac32}_p $ is the real, $x$-independent symbol 
\be \label{ditigi}
g^+_p(W;\xi):=\sum_{(\vec{\jmath}_p,  \vec{\sigma}_p) \in \fT_p}-\ii \big(\vec{\sigma}_{p}\cdot \Omega_{\vec{\jmath}_p}(\kappa) \big) G_{\vec{\jmath}_p}^{\vec{\sigma}_p}(\xi)  w_{\vec{\jmath}_p}^{\vec{\sigma}_p} ; \ee 
\item $a^+_{>N} (U;t,\xi) 
$ is a non-homogeneous symbol, independent of $x$, in $ \Gamma^{\frac32}_{K,K',N+1}[r]$ with imaginary part  ${\rm Im} \, a_{>N}^+ (U;t,\xi) $ 
belonging to $ \Gamma^{0}_{K,K',N+1}[r]$;
\item
$R^+_{\leq N}(W) $ is a real-to-real matrix of pluri--homogeneous  smoothing operators in $ \Sigma_1^N \wtcR^{-\vr + c(N,p)}_q \otimes \mM_2 (\C)  $ for some $c(N,p) >0$ (depending only on $N,p$) and fulfilling 
\be\label{ditigi2}
\cP_{\leq p} [R^+_{\leq N}(W)] =  \cP_{\leq p}[R_{\leq N}(W)] \, ;
\ee
\item 
$R^+_{>N}(U;t)$ is a real-to-real matrix of non--homogeneous smoothing operators in $ \mR^{-\vr+c(N,p)}_{K,K',N+1}[r] \otimes \mM_2 (\C) $.
\end{itemize}
\end{lemma}

\begin{proof}
Since $Z(t)$ solves \eqref{Z.conj1} then differentiating \eqref{new:foumu} we get 
\begin{align}
\pa_t W  & =   \cG_{g_p}(Z)\big[-\ii  \vOmega(D)+\vOpbw{\ii a_{\leq N}(Z;\x)+ \ii a_{>N}(U;t,\xi)}\big] \cG_{g_p}(Z)^{-1}W\label{paraspazio}\\
& \  \ + \cG_{g_p}(Z) \, R_{\leq N}(Z) Z + \cG_{g_p}(Z) R_{>N}(U;t)U\label{smoospazio}\\
& \ \ + (\pa_t \cG_{g_p}(Z)) \cG_{g_p}(Z)^{-1}W\label{paratempo}.
\end{align}
We now compute \eqref{paraspazio}--\eqref{paratempo} separately. As  $\cG_{g_p}(Z)$ is the time $1$-flow of a Fourier multiplier it commutes with every Fourier multiplier, and   \eqref{paraspazio} is equal to 
\be\label{primaparaspazio}
 \eqref{paraspazio}=-\ii  \vOmega(D)W+\vOpbw{\ii a_{\leq N}(Z;\x)+ \ii a_{>N}(U;t,\xi)}W \, . 
\ee
Now we write the symbol $a_{\leq N}(Z;\xi)$ in  terms  of $W$. 
By Lemma \ref{flussoconst}  $(iii)$ we have that 
  $ \cG_{g_p}(Z)- \uno $ is a matrix of spectrally localized maps in $  \Sigma\cS^{\frac32(N+1)}_{ K,0,p}[ \breve r,N] \otimes
  \cM_2(\C)$ for any $ \breve r > 0 $. 
  By Assumption {\bf (A)} we have $Z=\bM_0(U;t)U$
  with $ \bM_0(U;t) \in  \cM^0_{K,K',0}[r] \otimes \mM_2 (\C) $.
   Lemma \ref{cor.app.inv} (with $ S(Z;t)= \cG_{g_p}(Z)-\uno $) provides 
an approximate inverse of $W = \cG_{g_p}(Z)Z$ of the form 
\be \label{ZWin}
Z=  \Psi_{\leq N}(W)+  M_{>N}(U;t)U, \quad
\Psi_{\leq N}(W):= W  + \breve S_{\leq N}(W)W
\ee
where $\breve S_{\leq N}(W) $ is a  matrix of spectrally localized maps in  
$ \Sigma_{p}^N \wt\cS_q^{\frac32 (N+1) (N-p+1)} \otimes\cM_2(\C)  $ and   $M_{>N}(U;t) 
$ is  in 
$ \mM^{\frac32 (N+1) (N-p+2)}_{K,K',N+1}[r] \otimes \cM_2(\C) $.
The map \eqref{ZWin} has the form \eqref{Z.W1}, so by  Lemma \ref{sost} $(i)$  
(with $ m' = 3/2$) we obtain
\begin{align}
& \vOpbw{\ii a_{\leq N}(Z;\xi)}_{|Z= \Psi_{\leq N}(W)+  M_{>N}(U;t)U}\notag \\
&= \vOpbw{\ii a'_{\leq N}(W;\xi) +\ii a'_{>N}(U; t, \xi)} + R'_{\leq N}(W)  + R_{>N}'(U;t) \ ,\label{conj.fou10}
\end{align}
where \\
\noindent
$\bullet$ $a'_{\leq N}(W;\xi) $ is a real valued pluri-homogenous symbol independent of $x$ in 
$ \Sigma_{2}^N \wt \Gamma^\frac32_q$ with 
\be\label{preciaa}
\mP_{\leq p+1} (a'_{\leq N}(W;\xi)) = \mP_{\leq p+1} (a_{\leq N}(W;\xi)) \,  ; 
\ee
\noindent
$\bullet$ $a'_{>N}(U;t,\xi) $ is a non-homogenous real valued 
symbol independent of $x$ in $ \Gamma_{K,K',N+1}^\frac32[r]$; \\ 
\noindent
$\bullet$ $R'_{\leq N}(W) $ 
is a real-to-real matrix of pluri-homogeneous smoothing operators in $  \Sigma_{p+2}^N \wtcR^{-\vr}_q \otimes \cM_2(\C) $
 for any $ \varrho \geq 0 $;\\
\noindent
$\bullet$ $R'_{>N}(U;t) $ is a real-to-real matrix of non-homogeneous smoothing operators in $\cR^{-\varrho}_{K, K', N+1}[r] \otimes \cM_2(\C)$ for any $ \varrho \geq 0 $.

We now consider  the terms in  \eqref{smoospazio}.
Since $ \cG_{g_p}(Z)-\uno $ belongs to $ \Sigma \mS^{\frac32(N+1)}_{K,0,p}[r,N]
\otimes \mM_2(\C) $  by  Lemma  \ref{flussoconst}, 
 Proposition \ref{composizioniTOTALIs}-$(i)$ implies that 
  \be \label{nuovissima}
 \cG_{g_p}(Z) R_{\leq N}(Z)Z=R'_{\leq N}(Z)Z +R'_{> N}(Z;t)Z 
 \ee
 where $R'_{\leq N}(Z) $ is  in $ \Sigma_{1}^N \wtcR^{-\vr+\frac32(N+1)}_q \otimes \mM_2(\C) $ with
\be\label{lem.conj.R} 
 \cP_{\leq p}[R_{\leq N}'(Z)] = \cP_{\leq p}[R_{\leq N}(Z)] 
 \ee
 and 
 $R'_{>N}(Z;t)$ is 
 in $ \mR^{-\vr+\frac32(N+1)}_{K,0,N+1}[\breve r] \otimes \cM_2(\C) $ for any $ \breve r>0$.
 We now substitute $Z= \Psi_{\leq N}(W)+ M_{>N}(U;t)U$ (cfr. \eqref{ZWin}) in the homogeneous components of \eqref{nuovissima} and $ Z=\bM_0(U;t)U$ in the non-homogeneous ones of \eqref{nuovissima} and \eqref{smoospazio}. 
 So using Lemma \ref{sost} (with $ m\leadsto \tfrac32(N+1)(N-p+2)$, $\vr\leadsto \vr- \tfrac32(N+1)$), 
  $(i)$ and $(iii)$  of  Proposition \ref{composizioniTOTALI}, $(i)$ and $(iv)$ of Proposition \ref{composizioniTOTALIs} we get 
 \begin{align}
 \eqref{smoospazio} & = (R'_{\leq N}(Z)Z)_{| Z=\Psi_{\leq N}(W)+M_{>N}(U;t)U} +R'_{> N}(\bM_0(U;t)U;t)\bM_0(U;t)U \notag \\
 & \quad +\cG_{g_p}(\bM_0(U;t)U) R_{>N}(U;t)U 
 =R''_{\leq N}(W)W +R''_{> N}(U;t)U \label{conj.fou11}\, , 
 \end{align}
 where\\
 \noindent
 $\bullet$ $ R_{\leq N}''(W)$ is a pluri-homogeneous smoothing operator in 
 $\Sigma_1^N \wt\cR_q^{-\varrho + c(N,p) } \otimes \cM_2(\C)$ with  $c(N,p) = \frac32 (N+1) + \frac32 (N+1)^2 (N-p+2)$ and 
\be
\cP_{\leq p}[R_{\leq N}''(W)]\stackrel{\eqref{invariatosmo}}{=} \cP_{\leq p}[R'_{\leq N}(W)] \stackrel{\eqref{lem.conj.R}}{=}\cP_{\leq p}[R_{\leq N}(W)]  \ ;\label{chainR}
\ee
 \noindent
 $\bullet$ $ R_{>N}''(U;t)$ is a non--homogeneous smoothing operator in 
 $\cR^{-\varrho + c(N,p) }_{K, K', N+1}[r] \otimes \cM_2(\C)$.\\
 We finally consider the last term \eqref{paratempo}. Using that $ \cG_{g_p} (Z)$ commutes with every Fourier multiplier we get 
 \be \label{patgp}
 \eqref{paratempo}= \vOpbw{ \ii \pa_t g_p(Z;\xi)}W \, .
 \ee
So, using \eqref{Z.conj2}, \eqref{symtens} and the identity $( -\ii \vOmega(D)Z)_{j}^\sigma= -\ii \sigma  \Omega_j(\kappa) z_j^\sigma$, we obtain 
   \begin{align}
  & \pa_t g_p(Z;\xi)    = \ \sum_{a=1}^p   g_p(\overbrace{Z,\dots, Z}^{a-\text{times}}, -\ii \vOmega(D)Z ,Z, \dots , Z;\xi)+ p\,   g_p(M_{\leq N}(Z)Z,Z,\dots, Z;\xi)\label{ftdtgp}\\
   & 	\qquad\qquad \qquad+p\,  g_p(M_{>N}(U;t)U,Z,\dots, Z;\xi)_{|Z=\bM_0(U;t)U}\notag \\
   &= g^+_p(Z;\xi) +  p\,   g_p(\mP_{\leq N-p}(M_{\leq N})(Z)Z,Z,\dots, Z;\xi)\notag\\
   &\quad + p\, g_p(\mP_{> N-p}(M_{\leq N})(Z)Z,Z,\dots, Z;\xi)_{|Z=\bM_0(U;t)U}+p\,  g_p(M_{>N}(U;t)U,Z,\dots, Z;\xi)_{|Z=\bM_0(U;t)U}\notag\\
    &=g^+_p(Z;\xi)+g'_{\geq p+1}(Z;\xi)+ g'_{> N}(U;t,\xi)   \notag 
   \end{align}
where $g^+_p$ is the real valued symbol in $ \wt \Gamma^{\frac32}_p$ in \eqref{ditigi}, the real valued pluri-homogeneous symbol $g'_{\geq p+1}$ is in $ \Sigma_{p+1}^N \wt \Gamma^{\frac32}_q$ thanks to $(iv)$ of Proposition \ref{composizioniTOTALI} and
the real valued  
non-homogeneous symbol $g'_{> N}$ is in $\Gamma^{\frac32}_{K,K',N+1}[r]$ using also $(ii)$ of Proposition \ref{composizioniTOTALI}.
Then by \eqref{patgp}, \eqref{ftdtgp} and using the second part of $(iv)$ of  Proposition \ref{composizioniTOTALI}
 we obtain 
\be
\eqref{paratempo}
    =  \vOpbw{\ii g^+_p(Z;\xi)+\ii g'_{\geq p+1}(Z;\xi)+ \ii g'_{> N}(U;t,\xi)} W 
   + R'_{\geq p+1}(Z) W
+  
 R'_{> N}(U;t)U  
\ee
where 
$ R'_{\geq p+1} $   is a matrix of pluri-homogeneous smoothing operators in $ \Sigma_{p+1}^N \wtcR^{-\vr}_q
\otimes \cM_2(\C) $ and
$R'_{>N}(U;t) $ belongs to $ \cR^{-\varrho}_{K,K', N+1}[r] \otimes \cM_2(\C)$.
Then we substitute the variable $Z$ in \eqref{paratempo} using \eqref{ZWin} and Lemma \ref{sost}, to obtain 
\be 
 \eqref{paratempo}= \vOpbw{\ii  g^+_p(W;\xi)+ \ii g^+_{\geq p+1}(W;\xi)+\ii  g^+_{>N}(U;t,\xi)}W+ R'''_{\leq N}(W)W+ R'''_{>N}(U;t)U \label{conj.fou12}
\ee
where \\
\noindent
$\bullet$ $g^+_p(W;\xi)$ is the homogeneous symbol in \eqref{ditigi};\\
\noindent
$\bullet$ $g^+_{\geq (p+1)}(W;\xi) $ is a pluri-homogeneous real valued symbol in $ \Sigma_{p +1}^N \wt \Gamma^\frac32_q$;\\
\noindent
$\bullet$  $g^+_{>N}(U;t,\xi) $ is a non-homogeneous real valued symbol in $ \Gamma^\frac32_{K,K',N+1}[r]$;\\
\noindent
$\bullet$ $R'''_{\leq N}(W) $ is a matrix of pluri-homogeneous operators in $ \Sigma_{p+1}^N \wtcR^{-\vr + c(N,p)}_q \otimes \cM_2(\C) $; 

\noindent
$\bullet$  $ R'''_{>N}(U;t) $ is a matrix of smoothing operators in
$ \mR^{-\vr+c(N,p)}_{K,K',N+1}[r] \otimes \cM_2(\C) $.

In conclusion, combining \eqref{primaparaspazio}, \eqref{conj.fou10}, \eqref{chainR} and \eqref{conj.fou12} we deduce  \eqref{Z.conj10} 
with $a^+_{\leq N}= a'_{\leq N}+ g_{p}^+ + g_{\geq p+1}^+ $ and $a^+_{>N}:= a_{>N}+ a'_{>N}+ g^+_{>N}$ which  has imaginary part equal to the one of $ a_{>N}$ belonging to $ \Gamma^{0}_{K,K',N+1}[r]$. 
Moreover \eqref{ditigi0} follows by \eqref{preciaa} and \eqref{ditigi2} by  \eqref{conj.fou11}.
\end{proof}

The following lemma describes how a system is conjugated under 
 a smoothing perturbation of the identity.

\begin{lemma}[{\bf Conjugation under a smoothing perturbation of the identity}]\label{NormFormMap0}
Assume {\bf (A)} at page \pageref{A}. 
Let $F_{\leq N}(Z) $ be a real-to-real matrix of pluri-homogeneous smoothing operators 
in $  \Sigma_p^N \wt \cR_q^{-\varrho'} \otimes \cM_2(\C) $ for some $ \varrho' \geq 0 $.
If $Z(t)$ solves \eqref{Z.conj1} 
 then the variable 
\be\label{W.Z1} 
W:= \mF_{\leq N}(Z):= Z + F_{\leq N}(Z)Z
\ee
solves 
\be\label{normalformsmoo}
\pa_tW  = -\ii \vOmega(D) W +\vOpbw{ \ii a^+_{\leq N}(W;\x)+ \ii a^+_{>N}(U;t,\xi)}W + R^+_{\leq N}(W)W+ R^+_{>N}(U;t)U
\ee
where 
\begin{itemize}
\item  
$ a^+_{\leq N} (W;\x) $ is a real valued pluri-homogenous symbol, independent of $x$,  in  $\Sigma_2^N \wt \Gamma^{\frac32}_q $, with components 
\be\label{ditigi1}
\cP_{\leq p+1} [a^+_{\leq N}(W;\xi)] =  \cP_{\leq p+1}[a_{\leq N}(W;\xi)] \, ;
\ee
\item  $a^+_{>N} (U;t,\xi) 
$ is a non-homogeneous symbol, independent of $x$, in $ \Gamma^{\frac32}_{K,K',N+1}[r]$ with imaginary part 
${\rm Im} \, a_{>N}^+ (U;t,\xi) $ belonging to $ \Gamma^{0}_{K,K',N+1}[r]$;
\item 
$R^+_{\leq N}(W) $ is a real-to-real matrix of 
pluri--homogeneous  smoothing operators in 
$ \Sigma_1^N \wtcR^{-\vr_*}_q \otimes \mM_2 (\C)  $,  
$\varrho_* := \min{(\varrho, \varrho'-\frac32)}$ ($\vr\geq0$ is the smoothing order in Assumption {\bf (A)} at page \pageref{A}), with components 
\be\label{ditigi20}
\cP_{\leq p-1} [R^+_{\leq N}(W)] =  \cP_{\leq p-1}[R_{\leq N}(W)] \ , 
\ee
and, denoting $F_p(W) := \cP_p(F_{\leq N}(W))$ in $ \wt \cR^{-\varrho'}_p \otimes \cM_2(\C)$, one has 
\be\label{ditigi22}
\cP_{ p} [R^+_{\leq N}(W)]  = \cP_p[R_{\leq N}(W)] +\di_W \big(F_{p}(W)W\big)[ -\ii  \vOmega (D) ] + \ii   \vOmega (D)F_{p}(W);
\ee
\item 
$R^+_{>N}(U;t)$ is a real-to-real matrix of non--homogeneous smoothing operators in $ \mR^{-\vr_*}_{K,K',N+1}[r] \otimes \mM_2 (\C) $.
\end{itemize}
In addition, if  $\cF_{\leq N}(Z)$ in \eqref{W.Z1} is the   approximate time $1$-flow  (given by Lemma \ref{extistencetruflow})  of a vector field 
$G_p(Z)Z$, where $G_p(Z)\in \wt \mR^{-\vr'}_p \otimes \mM_2 (\C)$ has Fourier expansion 
\be \label{gipi}
\big(G_p(Z)Z\big)^\sigma_k=
\!\!\!\!\!\!\!\!\!\!\!\!\!\!\!\!\!\!
\sum_{
(\vec {\jmath}_{p+1}, k,  \vec{\sigma}_{p+1}, - \sigma) \in \fT_{p+2}
} 
\!\!\!\!\!\!\!\!\!\!\!\!\!\!\!\!\!\!
G_{\vec \jmath_{p+1},k}^{\vec{\sigma}_{p+1},\sigma}  z_{\vec{\jmath}_{p+1}}^{\vec{\sigma}_{p+1}} \, ,
\ee
then \eqref{ditigi22} reduces to 
\be\label{ditigi30}
\cP_{ p} [R^+_{\leq N}(W)] =  \cP_{p}[R_{\leq N}(W)]  + G_p^+(W)
\ee
where $G_p^+(W)\in  \wt \cR^{-\varrho'+\frac32}_p \otimes \cM_2(\C)$  is the smoothing operator with Fourier expansion 
\be\label{ditigi40}
\begin{aligned}
&(G_p^+(W)W)^\sigma_k= 
\!\!\!\!\!\!\!\!\!\!\!\!\!\!\!\!\!\!
\sum_{
(\vec {\jmath}_{p+1}, k,  \vec{\sigma}_{p+1}, - \sigma) \in \fT_{p+2}
} 
\!\!\!\!\!\!\!\!\!\!\!\!\!\!\!\!\!\!
   -\ii \big(  \vec{\sigma}_{p+1}\cdot  \Omega_{\vec{\jmath}_{p+1}}(\kappa)- \sigma \Omega_{k}(\kappa)\big)G_{ \vec{\jmath}_{p+1},k}^{\vec{\sigma}_{p+1},\sigma} \, w_{\vec{\jmath}_{p+1}}^{\vec{\sigma}_{p+1}} \,.
\end{aligned}
\ee
\end{lemma}

\begin{proof}
Since $Z(t)$ solves \eqref{Z.conj1} then differentiating \eqref{W.Z1} we get
\begin{align}
\label{W.Z2}
& \pa_t W  =   - \ii \vOmega(D) Z +\vOpbw{ \ii a_{\leq N}(Z;\x)+ \ii a_{>N}(U;t,\xi)}Z \\
\label{W.Z3}
& \qquad \quad + R_{\leq N}(Z)Z +  \di_Z (F_{\leq N}(Z)Z)[- \im \vOmega(D) Z ] \\
\label{W.Z4}
& \qquad  \quad + \di_Z (F_{\leq N}(Z)Z) [\vOpbw{ \ii a_{\leq N}(Z;\x)}Z] 
+ \di_Z (F_{\leq N}(Z)Z) [R_{\leq N}(Z)Z] 
\\
\label{W.Z5}
&+  R_{>N}(U;t)U +  \di_Z (F_{\leq N}(Z)Z)\vOpbw{  \ii a_{>N}(U;t,\xi)}Z 
+  \di_Z (F_{\leq N}(Z)Z)[ R_{>N}(U;t)U] \, . 
\end{align}
Note that, 
by the first remark below Definition \ref{Def:Maps}, $\di_Z (F_{\leq N}(Z)Z)$ are pluri-homogeneous  smoothing operators in $\Sigma_{p}^N \wtcR^{-\vr'}_q\otimes \mM_2(\C)$.
We proceed to analyze the various lines.
In  \eqref{W.Z2}  we substitute $Z = W - F_{\leq N}(Z)Z$ and use Proposition \ref{composizioniTOTALIs} $(i)$ to get
\be
\begin{aligned}
\eqref{W.Z2} = &
- \ii \vOmega(D) W +\vOpbw{ \ii a_{\leq N}(Z;\x)+ \ii a_{>N}(U;t,\xi)}W
\\
& + \im  \vOmega(D) F_p(Z)Z + R_{\geq p+1}(Z)Z + R_{>N}(U;t)U
\end{aligned}
\label{W.Z5.1}
\ee
with 
smoothing operators 
$R_{\geq p+1 }(Z) $ in $ \Sigma_{p+1}^N\wt \cR_q^{-\varrho' + \frac32} \otimes \cM_2(\C)$
and 
$R_{> N}'(U;t)  $ in $ \cR^{-\varrho'+\frac32}_{K,K',N+1}[r] \otimes \cM_2(\C) $.
Note that to obtain \eqref{W.Z5.1} we also substituted  $Z = \bM_0(U;t)U$
 with $\bM_0(U;t) \in \cM^0_{K,K',0}[r] \otimes \cM_2(\C)$ in the smoothing operators of homogeneity $\geq N+1$ 
using also $(i)$--$(ii)$ of Proposition \ref{composizioniTOTALI}.
From now on we will do this consistently.

We consider now
lines \eqref{W.Z3},  \eqref{W.Z4}. Using Proposition 
\ref{composizioniTOTALI} $(i)$--$(ii)$  and  Proposition 
\ref{composizioniTOTALIs} $(i)$ 
 we get 
\be 
\begin{aligned}
\eqref{W.Z3} + \eqref{W.Z4}  & =  
 R_{\leq p-1}(Z)Z 
+ R_p(Z)Z +  \di_Z \big(F_{p}(Z)Z\big)[ -\ii  \vOmega (D) Z]
\\
&  \quad + R_{\geq p+1}'(Z)Z 
+ R_{> N+1}'(U;t) U 
\end{aligned}
\label{W.Z5.345}
\ee
with 
$$
R_{\leq p-1}(Z) := \cP_{\leq p-1}[R_{\leq N}(Z)] \, , \quad 
R_{ p}(Z) := \cP_{p}[R_{\leq N}(Z)] \, ,
$$
and 
smoothing operators  
$R_{\geq p+1}'(Z) $ in $ \Sigma_{p+1}^N \wt \cR^{-\varrho_*}_q \otimes \cM_2(\C)$ and 
$R_{> N+1}'(U;t)  $ in $ \cR^{-\varrho_*}_{K,K',N+1}[r] \otimes \cM_2(\C) $, where $\varrho_*:= \min(\varrho, \varrho' - \frac32)$.

Next consider \eqref{W.Z5}. Substituting $Z = \bM_0(U;t)U$ and using  Proposition 
\ref{composizioniTOTALI} $(i)$--$(ii)$ we get
\be
\eqref{W.Z5} = R_{>N}''(U;t) U \label{W.Z7}  \quad \text{with} \quad 
R_{> N}''(U;t)  \in \cR^{-\varrho_*}_{K,K',N+1}[r] \otimes \cM_2(\C) \, . 
\ee
Collecting \eqref{W.Z5.1}, \eqref{W.Z5.345} and \eqref{W.Z7} we have obtained that
\eqref{W.Z2}-\eqref{W.Z5} is the system
\begin{align}
\notag
\pa_t W  = &  - \ii \vOmega(D) W +\vOpbw{ \ii a_{\leq N}(Z;\x)+ \ii a_{>N}(U;t,\xi)}W \\
\notag
&+ R_{\leq p-1}(Z)Z  + R_{p}(Z)   Z
+\di_Z \big(F_{p}(Z)Z\big)[ -\ii  \vOmega (D) Z]
+ \im  \vOmega(D) F_p(Z)Z  \\
& + R_{\geq p+1}'''(Z)Z + R_{>N}'''(U;t)U \label{W.Z9}
\end{align}
with  smoothing operators 
$R_{\geq p+1}'''(Z)  $ in $ \Sigma_{p+1}^N \wt \cR^{-\varrho_*}_q \otimes \cM_2(\C)$ and 
  $R_{> N+1}'''(U;t)  $ in $  \cR^{-\varrho_*}_{K,K',N+1}[r] \otimes \cM_2(\C) $. 
Finally we replace the 
variable $Z$ with the variable $W$ in \eqref{W.Z9} by means of an approximate inverse
of  $ W = \mF_{\leq N}(Z)$ in \eqref{W.Z1}. 
Lemma \ref{inversoapp} implies the existence of 
an approximate inverse 
$$
\Phi_{\leq N}(W) :=  W  + \breve F_{\leq N}(W) W \, , 
\quad  
\breve F_{\leq N}(W) \in  \Sigma_{p}^N \wt \cR^{-\varrho'}_{q} \otimes \cM_2(\C) \,  
$$ 
of the map $ \mF_{\leq N}(Z)$ in \eqref{W.Z1}. 
 Then, applying $\Phi_{\leq N}$ to  \eqref{W.Z1},  we get $Z = \Phi_{\leq N}(W)+ \wt R_{>N}(Z)Z$ where $\wt R_{>N}(Z)$ belongs to  $ \Sigma_{N+1} \wtcR^{-\vr'}_q \otimes \mM_2(\C)$, 
 and substituting  $Z= \bM_0(U;t)U$ in the pluri-homogeneous high--homogeneity term  $\wt R_{>N}(Z)Z$ and using $(ii)$ of Proposition \ref{composizioniTOTALI} we get 
\be\label{W.Z10}
Z = \Phi_{\leq N}(W) + \breve R_{>N}(U;t)U  \quad \text{where}\quad \breve R_{> N}(U;t)\in \cR^{-\varrho'}_{K,K',N+1}[r] \otimes \cM_2(\C) \, .
\ee
Finally we substitute  \eqref{W.Z10} in \eqref{W.Z9}
and, using Lemma \ref{sost} $(ii)$, we deduce 
\begin{align*}
\pa_t W  = &  - \ii \vOmega(D) W +\vOpbw{ \ii a_{\leq N}^+(W;\x)+ \ii a_{>N}^+(U;t,\xi)}W \\
&+ R_{\leq p-1}(W)W  + R_{p}(W)   W
+\di_W \big(F_{p}(W)W \big)[ -\ii  \vOmega (D) W]
+ \im  \vOmega(D) F_p(W)W  \\
& + R_{\geq p+1}^+(W)W + R_{>N}^+(U;t)U
\end{align*}
which gives \eqref{normalformsmoo} and the properties below. 
Note that
$a_{>N}^+(U;t,\xi)$ is given by the old non-homogeneous symbol $a_{>N}(U;t,\xi)$ and a purely real correction coming from formula \eqref{sostbony}. Hence  the imaginary part 
${\rm Im }\, a_{>N}^+(U;t,\xi) ={\rm Im }\, a_{>N}(U;t,\xi) $ belongs to  $ \Gamma^0_{K,K',N+1}[r]$.

Let us prove the  last part of the lemma. By \eqref{flowexp} and since $G_p(Z)$ is $\tau$-independent,   $F_p(Z)= G_p(Z)$. Then the correction term in 
\eqref{ditigi22}  is  $G_p^+(W):=\di_W \big(G_{p}(W)W\big)[ -\ii  \vOmega (D) ] + \ii   \vOmega (D)G_{p}(W)$, which has the Fourier expansion  \eqref{ditigi40} by \eqref{gipi} and  
the identity $( -\ii \vOmega(D)Z)_{j}^\sigma= -\ii \sigma  \Omega_j(\kappa) z_j^\sigma$.
 Note that the smoothing properties of $G_p^+(W)W$ can also be  directly verified by the characterization of Lemma  \ref{carasmoofou}.
\end{proof}

\section{Non--resonance conditions} \label{sec:non-res}

The goal of this section is to prove that the linear 
frequencies  $ \vec{\Omega}(\kappa) $, defined in \eqref{defOmegakappa}
and \eqref{omegonej}, 
satisfy, for any value of the  
gravity $ g $, vorticity $ \gamma $ and depth $  \tth $, 
the following non-resonance properties, except a zero measure set of 
surface tension coefficients $ \kappa  $.  
 
 \begin{theorem}\label{nonresfin0} 
{\bf (Non-resonance)} Let $ M \in \N $. For any $ g >0 $, $ \tth \in (0,+\infty]$ and $ \gamma \in \R $, there exists a zero measure set $ {\mathcal K} \subset (0,+\infty) $ 
such that, for any compact interval $[\kappa_1,\kappa_2] \subset (0,+\infty) $
there is  $ \tau > 0 $ and, for any $ \kappa \in [\kappa_1,\kappa_2] \setminus {\cal K} $
the following holds: there is a positive constant $ \nu > 0 $ such that  
for any multi-index  
$ (\alpha, \beta) \in \N_0^{\Z\setminus\{0\}}\times \N_0^{\Z\setminus\{0\}}$ 
of length  $|\alpha + \beta| \leq M$,  
which is not super action preserving 
(cfr. Definition \ref{def:SAPindex}),  it results
\begin{equation}
\label{verad0}
|\vec{\Omega}(\kappa) \cdot(\alpha - \beta)| > \frac{\nu }{\Big(\max\limits_{j\in \textup{supp}(\alpha \cup \beta)}  |j|\Big)^\tau}  
\end{equation}
where $\textup{supp}(\alpha \cup \beta):=\{ j\in \Z\setminus\{0\}\colon  \alpha_{j} +\beta_{j}\not=0\} $.
\end{theorem}

Theorem \ref{nonresfin0}  extends 
Proposition 8.1 in \cite{BD}, which is valid only in the irrotational case 
$ \gamma = 0 $ and in finite depth. 
Theorem \ref{nonresfin0}  follows by the next result
where we fix a 
compact interval of surface tension  coefficients putting $\cK:= \cap_{\nu >0} \cK_{\nu}$.

\begin{proposition}\label{nonresfin}
Let  $M \in \N$ and 
fix   a compact interval $ \cI := [\kappa_1,\kappa_2]$ with $ 0 < \kappa_1 < \kappa_2 $. 
Then there exist $\nu_M, \tau,  \delta >0 $  such that for any $\nu \in (0, \nu_M)$, 
there is a set $\cK_{\nu} \subset \cI$ of measure $O(\nu^\delta)$ 
such that for any  $\kappa \in \cI \setminus \cK_{\nu} $ the following holds: 
for any multi-index $ (\alpha, \beta) \in \N_0^{\Z\setminus\{0\}}\times \N_0^{\Z\setminus\{0\}}$ 
of length  $|\alpha + \beta| \leq M $,    
which is not super action preserving (cfr. Definition \ref{def:SAPindex}),  
one has
\begin{equation}
\label{veradis}
|\vec{\Omega}(\kappa) \cdot(\alpha - \beta)| > \frac{\nu}{\Big(\max\limits_{j\in \textup{supp}(\alpha \cup \beta)}  |j|\Big)^\tau} \, . 
\end{equation}
\end{proposition}
The proof makes use of  
Delort-Szeftel Theorem 5.1 of \cite{DelS} about measure estimates 
for sublevels of subanalytic functions, whose statement is the following.

\begin{theorem}[Delort-Szeftel]
\label{DS}
Let $X$ be a closed ball $\bar{B(0,r_0)} \subset \R^d$ 
and $Y$ a compact interval of $\R$.
Let $f\colon X \times Y \to \R$ be a continuous subanalytic function, $\rho: X \to \R$ a real analytic function, $\rho \not\equiv 0$. Assume
\begin{itemize}
\item[(H1)] $f$ is real analytic on $\{ x \in X \colon \rho(x) \neq 0 \} \times Y$;
\item[(H2)] for all $x_0\in X$ with $\rho(x_0) \neq 0$, the equation $f(x_0,y) =0$ has only finitely many solutions $y \in Y$. 
\end{itemize}
Then there are $N_0 \in \N$, $\alpha_0> 0 $, $ \delta > 0 $ , $ C  > 0 $, such that for any $\alpha \in (0, \alpha_0]$, any $ N \in \N$, $N \geq N_0$, any $x$ with $\rho(x) \neq 0$,
\begin{equation*}
\meas \big\{ y \in Y \colon \ \ \ |f(x,y)|\leq \alpha |\rho(x)|^N \big\} \leq C \alpha^\delta \, |\rho(x)|^{N\delta} \, . 
\end{equation*}
\end{theorem}
We shall first prove Proposition \ref{nonresfin} for  deep water,  
in 
Section \ref{deepNR}, and then, for any 
finite depth,   in Section \ref{FNR}.

\subsection{Deep-water case}\label{deepNR}

In the deep water case $ \mathtt h=+\infty $, 
by \eqref{omegonej} and \eqref{Gxi}, 
the linear frequencies are 
\begin{equation}
\label{omega_j_inf_dep}
\Omega_j(\kappa) = \omega_j(\kappa) +  \frac{\gamma}{2}\, \sgn(j) \, , \quad \omega_j(\kappa)=\sqrt{(\kappa j^2 +g) |j| 
+ \frac{\gamma^2}{4} } \, . 
\end{equation}
In this case Proposition \ref{nonresfin} is a consequence of the following  result.  

\begin{proposition}
\label{p:nrg}
Let  $\cI = [\kappa_1,\kappa_2] $ and  
consider two  integers $\tA, \tM\in \N$. 
Then there exist $\alpha_0, \tau , \delta >0$ (depending on $\tA, \tM$) 
such that for any $\alpha \in (0, \alpha_0)$, there is a 
set $\cK_\alpha \subset \cI$ of measure $O(\alpha^\delta)$,   
such that  for any $\kappa \in \cI \setminus\cK_\alpha$ the following holds: 
for any 
$ 1 \leq n_1 <  \ldots < n_\tA $,  
any 
$\vec c:=( c_0, c_1, \ldots, c_\tA) \in\Z\times (\Z \setminus\{0\})^{\tA}$, with  
$| \vec c \, |_\infty := \max_{a=0,\ldots, \tA} |c_a| \leq \tM $,  one has
\begin{equation}
\label{nrg}
\Big| {\sum_{a=1}^\tA c_a \, \omega_{n_a}(\kappa) + \frac{\gamma}{2}c_0} \Big|  > \frac{\alpha }{(\sum_{a=1}^\tA n_a)^{\tau}}.
\end{equation}
\end{proposition}
Before proving Proposition \ref{p:nrg} we deduce as a corollary Proposition \ref{nonresfin}
when $ \mathtt h = + \infty $. 

\begin{proof}[{\bf Proof of Proposition \ref{nonresfin}}]
For any multi-index $ (\alpha, \beta) \in \N_0^{\Z\setminus\{0\}}\times \N_0^{\Z\setminus\{0\}} $ with length  $|\alpha + \beta| \leq M$, 
using that $ \omega_j(\kappa)$ in \eqref{omega_j_inf_dep} is even in $ j $, 
 we get
\begin{align}
\vec{\Omega}(\kappa) \cdot(\alpha - \beta)  & = \sum_{j \in \Z \setminus \{0\}} \omega_j( \kappa)(\alpha_j - \beta_j)  + \sum_{j \in \Z \setminus \{0\}} \frac{\gamma}{2} \sign(j) (\alpha_j - \beta_j) \notag \\
& = \sum_{n >0} \omega_n(\kappa)(\alpha_n + \alpha_{-n} - \beta_n - \beta_{-n}) +\frac{\gamma}{2}\sum_{n >0} (\alpha_n - \alpha_{-n} - \beta_n + \beta_{-n}) \notag \\
& = \sum_{n \in \mathfrak N(\alpha, \beta)} \omega_n(\kappa)(\alpha_n + \alpha_{-n} - \beta_n - \beta_{-n}) +\frac{\gamma}{2} c_0  \label{eq:nrdeep}
\end{align}
where 
 $ \mathfrak N(\alpha, \beta) $ is the set 
defined in \eqref{ennonegotico} and 
$ c_0:=    \sum_{n >0} \alpha_n - \alpha_{-n} - \beta_n + \beta_{-n} \in \Z $. 
Since $ (\alpha, \beta ) $ is not super action preserving (cfr. Definition \ref{def:SAPindex})
then   $\mathfrak N(\alpha, \beta)  $ is not empty. 
By \eqref{cardNab} the  cardinality 
$ \tA := |\mathfrak N(\alpha, \beta) | $ satisfies $ 1 \leq \tA \leq M $. 
Denoting  by $1 \leq n_1 <  \ldots < n_{\tA} $ 
the distinct elements of $\mathfrak N(\alpha, \beta)$, and the integer numbers
$$
c_a:= \alpha_{n_a} +\alpha_{-n_a} - \beta_{n_a} - \beta_{-n_a} \in \Z \setminus \{0\} \, , \quad
\forall  a = 1, \ldots, \tA \, , 
$$
we deduce by \eqref{eq:nrdeep}
that
$$
\vec{\Omega}(\kappa) \cdot(\alpha - \beta)
 = \sum_{a=1}^\tA \omega_{n_a}(\kappa) c_a + \frac{\gamma}{2} c_0 \, . 
$$
By the definition of $ \mathfrak{N}(\alpha,\beta)$, each  integer
$c_a \neq 0 $, $a=1, \ldots, \tA$,   and  $|c_a| \leq |\alpha| + |\beta| \leq M $.
Similarly $|c_0|\leq M$. 
Applying Proposition \ref{p:nrg} with $ \tM = M $ we deduce \eqref{veradis} with  $ \nu:= \frac{\alpha}{M^\tau}$.
\end{proof}
The rest of the section is devoted to the proof of Proposition \ref{p:nrg}.

\begin{proof}[{\bf Proof of Proposition \ref{p:nrg}}]
For any  $\vec n := (n_1, \ldots, n_\tA) \in \N^\tA $  with
$ 1 \leq n_1 <  \ldots < n_\tA $,  we denote  
\begin{equation}\label{x(n)}
x_0(\vec n) := \frac{1}{\sum_{a =1}^\tA n_a} \, , \qquad 
x_a(\vec n):= x_0(\vec n) \sqrt{ n_a} \, , \quad  \forall \,  a=1,\ldots, \tA \, . 
\end{equation}
Clearly
\begin{equation}
\label{x.estn}
0 < \tfrac{1}{\sum_{a =1}^\tA n_a} \leq 
x_a(\vec n)  \leq 1 \ , \qquad \forall\,   a = 0, \ldots, \tA 	\, . 
\end{equation}
If \eqref{nrg} holds, then multiplying it by $ x_0(\vec n)^3$, one gets that the inequalities 
\begin{equation}
\label{nrn}
 \Big|{\sum_{a=1}^\tA c_a \, 
 \sqrt{\kappa x_a^6  + g x_a^2 x_0^4 + \frac{\gamma^2}{4} x_0^6  }  
 + \frac{\gamma}{2}c_0 x_0^3 } \Big| \geq \alpha \, x_0^{\tau+3}  
\end{equation} 
hold at any $x_a = x_a(\vec n) $, $ a = 0, \ldots, \tA $, defined in \eqref{x(n)}.
This suggests to define  the function
\begin{equation} \label{lambdag}
  \lambda(y, x_0, \kappa) := \sqrt{\kappa y^6  + g y^2 x_0^4 + \frac{\gamma^2}{4} x_0^6 } \, ,
  \end{equation}
and, for $  \vec c =( c_0, c_1, \ldots, c_\tA) \in\Z\times (\Z \setminus\{0\})^{\tA} $,
\be\label{effecci}
f_{\vec c} \, \colon [-1, 1]^{\tA+1} \times \cI \to \R  \, , \quad 
f_{\vec c} \, (x,\kappa) := \sum_{a=1}^\tA c_a \, \lambda(x_a, x_0,   \kappa)  + \frac{\gamma}{2} c_0 x_0^3  
\ee
where $ x := (x_0, x_1, \ldots, x_\tA ) $. 

We estimate the sublevels of  $\kappa \mapsto f_{\vec c}(x,\kappa)$ using Theorem \ref{DS}. Let us verify its assumptions. The set  $X := [-1,1]^{\tA+1}$  is a   closed ball in $\R^{\tA+1}$.
The function $f_{\vec c}: X\times \cI \to \R$ is continuous and subanalytic.
Then we  define the non-zero real analytic function 
\be\label{defrho1}
\rho(x) := \prod_{a=0}^{\tA}
 x_a \, \prod_{1 \leq a<b \leq \tA }(x_a^2 - x_b^2) \, .
 \ee
We observe that $ \rho(x) $ evaluated  at  
$ x(\vec n) := (x_0(\vec n), \ldots,  x_\tA(\vec n) ) $, 
defined in \eqref{x(n)}, satisfies 
\begin{equation} \label{rho.estg}
\Big(\sum_{a=1}^\tA n_a \Big)^{- \tau_1 }
\leq  |{\rho(x(\vec n))} | \lesssim_\tA 
\Big(
\sum_{a=1}^\tA n_a\Big)^{-1}
\end{equation}
with $ \tau_1 :=  \tA  + 1  + 2\binom{\tA}{2}  $, 
as follows  by \eqref{x.estn}, \eqref{defrho1} 
and the assumption that the $n_a$'s are all distinct, thus 
$ |n_a - n_b| \geq 1 $, for any $ a \neq b $. 

We show now that  the assumptions (H1) and (H2)  of Theorem \ref{DS} hold true. 
 
 \noindent \underline{Verification of (H1).} 
If   $\rho(  x) \neq 0$ then, by \eqref{defrho1},  
\begin{equation} \label{rho_condg}
x_a \neq 0 \, , \  \forall \, 0 \leq a \leq \tA \, , 
\qquad \mbox{and} \qquad x_a^2 \neq x_b^2  \, ,  \  \forall \, 1 \leq a < b \leq \tA \, .
\end{equation} 
In particular on the set $\{ x \in X \colon \rho(x) \neq 0 \} \times Y$ the function 
$  \lambda(x_a,  x_0, \kappa) $ in \eqref{lambdag} in real analytic and thus the function 
 $f_{\vec c} (x,  \kappa) $ in \eqref{effecci} is real analytic. \\
  \noindent \underline{Verification of (H2).} The fact that, for any $ x \in X $ 
  such that $ \rho(x) \neq 0$, the {\it analytic} function $\kappa \mapsto f_{\vec c}(x,\kappa)$ 
possesses only a finite number of zeros on the interval 
$ \cI $, is 
a consequence  of the next lemma.  

\begin{lemma}
For any $ x \in X $ 
  such that $ \rho(x) \neq 0$, the 
  function $ \kappa \mapsto f_{\vec c}(x,\kappa)$ is not identically zero in $\cI$.
\end{lemma}

\begin{proof}  
We argue by contradiction, assuming that there 
  exists $  x= (  x_a)_{0 \leq a \leq \tA}\in X$ with $\rho(x) \neq 0$
   such that $f_{\vec c}(  x, \kappa) = 0 $ for any  $\kappa$ in the interval $\cI $.
Then
the function $ \kappa \mapsto f_{\vec c}(  x,\kappa) $ 
is identically zero also on the larger domain 
of analyticity $(-\frac{\gamma^2}{4} x_0^6, + \infty) $.
Note that $ x_0^2 > 0 $ because $\rho(x) \neq 0 $, cfr.  \eqref{rho_condg}.   
 In particular, for any  $ l \in \N $,  all the derivatives 
 $ \partial_\kappa^l f_{\vec{c}}(  x, \kappa) \equiv 0 $ are zero 
 in the interval   $ (-\frac{\gamma^2}{4} x_0^6, + \infty)$.   
 
 Now we compute such derivatives at $ \kappa = 0 $ by differentiating \eqref{effecci}. 
The derivatives of the function $ \lambda(y, x_0, \kappa) $ defined in \eqref{lambdag} are
given by, for suitable constants $ C_l \neq 0 $, 
  \begin{equation*}
  \partial_\kappa^l  \lambda(y, x_0, \kappa)  = C_l y^{6l} \, \lambda(y, x_0, \kappa)^{1-2l} \, , \quad
  \forall l \in \N \, .
  \end{equation*}
Thus we obtain
  \begin{equation}\label{formadermu}
 \partial_\kappa^l  \lambda(y, x_0, \kappa) \vert_{\kappa = 0}
 = C_l \, \mu(y,x_0)^l \, \lambda(y, x_0, 0) \qquad 
 \text{where} \qquad 
 \mu(y,x_0):= \frac{y^6}{g y^2 x_0^4 + \frac{\gamma^2}{4}x_0^6} \, ,
  \end{equation}
and, recalling  \eqref{effecci},  
$$
\partial_\kappa^l f_{\vec c}(x,\kappa) \vert_{\kappa = 0} 
= C_l \sum_{a=1}^\tA c_a \mu(x_a,x_0)^l  \, \lambda(x_a, x_0, 0) \, , \quad 
\forall l \in \N \, .
 $$
As a consequence, the conditions 
 $\partial_\kappa^l f_{\vec c} 	\,  (  x,\kappa) \vert_{\kappa = 0} = 0$ for any $l=1, \ldots, \tA $,  imply that 
\be\label{Axc0}
 A(  x) \vec {{\mathfrak c}} = 0  
\ee
where  $ A(x) $ is the $ \tA \times \tA $-matrix 
$$
A( x):= \begin{pmatrix}
 \mu(x_1, x_0)  \lambda(  x_1, x_0, 0) &  
 \cdots & \mu(x_\tA, x_0) \lambda(  x_\tA, x_0,  0) \\
 \mu(x_1, x_0)^2 \lambda(  x_1,  x_0, 0) & 
 \cdots & \mu(x_\tA, x_0)^2 \lambda(  x_\tA,  x_0, 0) \\
 \vdots 
 & \ddots & \vdots  \\
 \mu(x_1, x_0)^{\tA} \lambda(  x_1, x_0, 0) & 
  \ldots  & \mu(x_\tA, x_0)^\tA \lambda(  x_\tA, x_0, 0)  
 \end{pmatrix}  \quad \text{and} \quad
{\vec {\mathfrak c}} := \begin{bmatrix} c_1\\ \vdots\\  c_\tA\end{bmatrix}  
\in  (\Z \setminus\{0\})^{\tA} \, . 
 $$ 
Since the vector $ \vec {{\mathfrak c}} \neq 0 $, we deduce by \eqref{Axc0} that 
the matrix  $ A ( x ) $ has zero determinant. 
On the other hand, by the multi-linearity of the determinant,   
\begin{align}
\det A(  x) & = \prod_{a=1}^\tA \mu(x_a, x_0) \lambda(  x_a, x_0, 0) 
\det \begin{pmatrix}
 1  &   
 \cdots & 1 \\
 \mu(x_1, x_0)  & 
 \cdots & \mu(x_\tA, x_0) \\
 \vdots & 
 \ddots & \vdots  \\
 \mu(x_1, x_0)^{\tA-1} & 
 \ldots  & \mu(x_\tA, x_0)^{\tA-1}  
 \end{pmatrix}  \notag \\
 & = \prod_{a=1}^\tA \mu(x_a, x_0) \lambda(  x_a, x_0, 0)  \, \prod_{1 \leq a< b\leq \tA} (\mu(x_a, x_0) - \mu(x_b, x_0))  \label{detA}
\end{align}
by  a Vandermonde determinant. 
The condition  $\rho(  x) \neq 0$ 
implies, by \eqref{rho_condg},  \eqref{formadermu} and \eqref{lambdag}, that
 $$ 
 \prod_{a=1}^\tA \mu(x_a, x_0) \lambda(x_a, x_0, 0)  \neq 0 \, , 
 $$  
 and,  
in view of \eqref{detA}, the determinant 
 $\det A(x) = 0$ if and only if $\mu(x_a,x_0) = \mu(x_b,x_0)$ for some $1 \leq a< b \leq \tA$.
Since the  function $y \to \mu(y,x_0)$ in  \eqref{formadermu} is even and strictly monotone on the two intervals $ ( 0,+\infty)$ and $(-\infty, 0)$, it follows that 
$$
\mu(x_a,x_0) = \mu(x_b,x_0) \quad \Rightarrow \quad |x_a| = |x_b| \, .
$$
This contradicts $\rho(x) \neq 0$, see \eqref{rho_condg}. The lemma is proved. 
\end{proof}

We have  verified assumptions (H1) and (H2) of Theorem \ref{DS}.
We thus conclude that 
there are $N_0 \in \N$, $\alpha_0, \delta, C  >0$, such that for any $\alpha \in (0, \alpha_0]$, any $ N \in \N$, $N \geq N_0$, any $x \in X$ with $\rho(x) \neq 0$,
\begin{equation} \label{meas2g}
\meas \big\{ \kappa \in \cI \colon \ \ \ |f_{\vec c}(x,\kappa)|\leq \alpha |\rho(x)|^N \big\} \leq C \alpha^\delta \, |\rho(x)|^{N\delta} \, . 
\end{equation}
For any $\vec n = (n_1, \ldots, n_\tA) \in \N^\tA $ with $ 1 \leq n_1 < \ldots < n_\tA $,  we consider the set
\be\label{badni}
\cB_{\vec c,  \vec n}(\alpha, N) := \big\lbrace
 \kappa \in \cI \colon \ \ \ |f_{\vec c}(x(\vec n),\kappa)|\leq \alpha |\rho(x(\vec n))|^N 
\big\rbrace
\ee
where $x(\vec n) = (x_a(\vec n))_{a=1, \ldots,\tA} \in X $ is defined in \eqref{x(n)}.
By \eqref{rho.estg} we get  
$\rho(x(\vec n)) \neq 0 $. Then \eqref{meas2g} yields
\begin{equation}
 \label{measGg}
 \meas \ \cB_{\vec c,  \vec n}(\alpha, N) \leq C \alpha^\delta |\rho(x(\vec n))|^{N\delta} . 
 \end{equation} 
Consider  the set 
\be\label{defKNA1}
\cK(\alpha, N) :=  
\bigcup_{
              \substack{ \vec{n}=(n_1, \ldots, n_\tA) \in \N^{\tA},\, 1 \leq n_1 < \ldots < n_\tA   \\ 
                       \vec c \in \Z\times (\Z \setminus \{0\})^{\tA}, \, |\vec c \, |_{\infty} \leq \tM } 
         }  
  \cB_{\vec c, \vec n}(\alpha, N)  \subset \cI  \, .
\ee
By \eqref{measGg} and  \eqref{rho.estg} it results 
 \begin{equation} \label{measK1g}
 \meas \ \cK(\alpha, N)  \leq C(\tA, \tM) \alpha^\delta \sum_{n_1, \ldots, n_\tA  \in \N} 
\frac{1}{(\sum_{a=1}^\tA n_a)^{N\delta}}  \leq C'(\tA, \tM) \alpha^\delta 
 \end{equation}
 for some finite constant $ C'(\tA, \tM) < + \infty $,   provided 
$  N \delta > \tA $. We fix 
$$ 
\underline N := [ \tA \delta^{-1}] + 1  \qquad \text{and} \qquad 
{\cal K}_\alpha := {\cal K} (\alpha, \underline N) 
 $$ 
 whose measure satisfies 
$ |{\cal K}_\alpha| \lesssim_{\tA, \tM} \alpha^\delta $ by \eqref{measK1g}.  
For any $\kappa \in \cI \setminus {\cal K}_\alpha $, for any 
$ \vec n = (n_1, \ldots, n_\tA) $ with  $ 1 \leq n_1 <  \ldots < n_\tA $, 
 for any  $ \vec c \in \Z\times (\Z \setminus \{0\})^{\tA} $ with 
 $ |\vec c \, |_{\infty} \leq \tM $, 
 one has, by 
 \eqref{defKNA1} and \eqref{badni}, 
\begin{equation}
\label{f1g}
|{f_{\vec c}(x(\vec n), \kappa)}| 
>  \alpha |\rho(x(\vec n))|^{\underline N} \stackrel{\eqref{rho.estg}}{\geq } 
\frac{ \alpha }{\left(\sum_{a=1}^\tA n_a\right)^{  
\tau_1 \underline{N} }} \, . 
\end{equation}
Recalling  the definition of $ f_{\vec c}$ in \eqref{effecci}, \eqref{lambdag} and $ x_0(\vec n)$ in \eqref{x(n)}, 
the lower bound \eqref{f1g} 
implies  \eqref{nrg} with $ \tau := \tau_1 \underline{N} - 3 $, cfr. \eqref{nrn}. 
\end{proof}

\subsection{Finite depth case}\label{FNR}

We consider now the finite depth case $0 <  \tth <+\infty$ where the frequencies are, by  \eqref{omegonej} and \eqref{Gxi}, 
\begin{equation} \label{omega_j_fin_dep}
\Omega_j(\kappa) = \omega_j(\kappa) +  \frac{\gamma}{2}\, \tanh(\tth j) \, , \quad \omega_j(\kappa)=\sqrt{j\tanh(\tth j)\Big(\kappa j^2 +g
+ \frac{\gamma^2}{4}\frac{\tanh(\tth j)}{j}\Big)  } \, .
\end{equation}
In this case Proposition \ref{nonresfin} is a consequence of the following  result.   
\begin{proposition}
\label{p:nrg2}
Let  $\cI = [\kappa_1,\kappa_2]$ and 
consider $\tA, \tM\in \N$ and $ \tB \in \N_0 $. 
Then there exist $ \alpha_0, \tau , \delta > 0 $ (depending on $\tA,\tB, \tM$) such that for any $\alpha \in (0, \alpha_0)$, there is a set $\cK_\alpha \subset \cI$ 
of measure $O(\alpha^\delta)$,  
 such that  for any $\kappa \in \cI \setminus\cK_\alpha$ the following holds:
for any $ 1 \leq n_1 \leq \ldots \leq n_\tA $
any $ \vec m= (m_1,\dots, m_\tB) \in \N^\tB$,  any 
$\vec c:=(c_1, \ldots, c_\tA) \in  (\Z \setminus\{0\})^{\tA}$ with  
$|\vec c \, |_\infty \leq \tM$  and 
$ \vec d=(d_1,\dots, d_\tB)\in (\Z\setminus\{0\})^\tB$ with $|\vec d \, |_{\infty}\leq \tM$, 
one has
\begin{equation}
\label{nrg2}
\Big| \sum_{a=1}^\tA c_a \, \omega_{n_a}(\kappa) 
+ \frac{\gamma}{2}\sum_{b=1}^\tB d_b \tanh(\tth  m_b) \Big|  
> \frac{\alpha }{(\sum_{a=1}^\tA n_a+\sum_{b=1}^\tB m_b)^{\tau}} \, .
\end{equation}
If $ \tB = 0 $, by definition, the sums in \eqref{nrg2} in the index 
$ b $   are empty and the vectors $ \vec m, \vec d $ are not present. 
\end{proposition}

Before proving Proposition \ref{p:nrg2} we deduce  
Proposition \ref{nonresfin}
in the case of finite depth $ \tth $. 

\begin{proof}[{\bf Proof of Proposition \ref{nonresfin}}]
Let $ (\alpha, \beta) \in \N_0^{\Z\setminus\{0\}}\times \N_0^{\Z\setminus\{0\}} $ be a 
multi-index  with length  $ | \alpha + \beta | \leq M $,
which is not super action preserving (cfr. Definition \ref{def:SAPindex}).  
Denote 
by $ 1 \leq n_1 < \ldots <  n_{\tA}$ the elements 
of $\mathfrak N(\alpha, \beta) $
defined in \eqref{ennonegotico}, and let $\tA := |\mathfrak N(\alpha, \beta) | \geq 1 $.
We also consider the set 
$ \mathfrak{M}(\alpha,\beta):=\{ n\in \N \colon\alpha_n - \alpha_{-n} - \beta_n + \beta_{-n}\not=0\} $, which could be empty.
We denote by $ \tB:= | \mathfrak{M}(\alpha,\beta)|$ its cardinality, which could be zero,  
and   $ 1 \leq m_1 < \dots < m_\tB $ its distinct elements, if any.  
 By  \eqref{cardNab} it results $1 \leq \tA \leq M $ and $ 0 \leq \tB \leq M$ (arguing as for \eqref{cardNab}).
By \eqref{omega_j_fin_dep} and  
since  $ \omega_j(\kappa)$ is even in $ j $,  
we get 
\begin{align}
\vec{\Omega}(\kappa) \cdot(\alpha - \beta)  & = \sum_{j \in \Z \setminus \{0\}} \omega_j( \kappa)(\alpha_j - \beta_j)  + \sum_{j \in \Z \setminus \{0\}} \frac{\gamma}{2} \tanh( \tth j) (\alpha_j - \beta_j) \notag \\
& = \sum_{n >0} \omega_n(\kappa)(\alpha_n + \alpha_{-n} - \beta_n - \beta_{-n}) +\frac{\gamma}{2}\sum_{n >0}\tanh( \tth n) (\alpha_n - \alpha_{-n} - \beta_n + \beta_{-n}) \notag \\
& = \sum_{n \in \mathfrak{N}(\alpha,\beta)} 
\omega_n(\kappa)(\alpha_n + \alpha_{-n} - \beta_n - \beta_{-n}) + 
\frac{\gamma}{2}
\sum_{n \in \mathfrak{M}(\alpha,\beta)}\tanh( \tth n) (\alpha_n - \alpha_{-n} - \beta_n + \beta_{-n}) \notag \\ 
& =  \sum_{a=1}^\tA \omega_{n_a}(\kappa) c_a 
+ \frac{\gamma}{2} \sum_{b=1}^\tB d_b\tanh(\tth m_b) \label{eq:a1M}
\end{align}
having defined
$$
c_a := \alpha_{n_a} +\alpha_{-n_a} - \beta_{n_a} - \beta_{-n_a} \, , \qquad 
d_b :=    \alpha_{m_b} - \alpha_{-m_{b}} - \beta_{m_{b}} + \beta_{-m_{b}} \, .
$$
By the definition of $ \mathfrak{N}(\alpha,\beta)$, 
each $c_a \in \Z\setminus \{0\}$
and  $|c_a| \leq |\alpha| + |\beta| \leq M $ for any  $a=1, \ldots, \tA $. 
If $ \mathfrak{M}(\alpha,\beta) $ is empty then $ \tB = 0 $, and the second  sum in 
\eqref{eq:a1M} in the index $ b $ is not present. On the other hand, if  
$ \tB \geq 1 $, 
 by the definition of 
$ \mathfrak{M}(\alpha,\beta)$, each $ d_b\in \Z\setminus \{0\}$ and  
$|d_b|\leq |\alpha| + |\beta| \leq M$ for any $ b=1,\dots, \tB $. 
Applying in both cases 
Proposition \ref{p:nrg2} with $ \tM = M $ we deduce \eqref{veradis} 
with $ \nu:= \frac{\alpha}{(2M)^\tau}$.
\end{proof}

\begin{proof}[{\bf Proof of Proposition \ref{p:nrg2}}]
We write the proof in the case $ \tB \geq 1 $. 
In the case $ \tB = 0 $ the same argument works.
For any $ \vec n := (n_1, \ldots, n_\tA) \in \N^\tA $
with $ 1 \leq n_1 < \ldots < n_\tA $ 
and $ \vec m := (m_1,\dots, m_\tB) \in  \N^\tB $ we define  
\begin{equation}
\label{x(n)t(n)}
\begin{aligned}
&x_0(\vec n,\vec m ) := \frac{1}{\sum_{a =1}^\tA n_a+\sum_{b =1}^\tB m_b} \,, 
\qquad x_a(\vec n,\vec m):= x_0(\vec n,\vec m) \sqrt{ n_a} \, , \ \forall \,  a=1,\ldots, \tA \, \\ 
&t_a(\vec n):= \sqrt{\tanh(\tth n_a)} \, ,   \ \ \forall \, a=1,\ldots, \tA \, , \quad
t_{\tA+b}(\vec m):= \sqrt{\tanh( \tth m_b)} \, , \ \ \forall \,  b=1,\dots,\tB \, .
\end{aligned}
\end{equation}
 Clearly
\begin{equation}
\label{xt.estn2}
\begin{aligned}
0 < \tfrac{1}{\sum_{a =1}^\tA n_a+\sum_{b =1}^\tB m_b} \leq x_a(\vec n,\vec m )  \leq 1 \, , \qquad &\sqrt{\tanh(\tth )}
\leq t_a(\vec n)  \leq 1 \, ,  \qquad \forall a = 0, \ldots, \tA \, ,  \\ 
& \sqrt{\tanh(\tth)}\leq t_{\tA+b}(\vec m)  \leq 1 \, , \quad \forall \, b= 1, \ldots, \tB \, . 
\end{aligned}
\end{equation}
If \eqref{nrg2} holds, then multiplying it by $x_0(\vec n,\vec m)^3$, 
one gets, recalling \eqref{omega_j_fin_dep},  that the inequalities 
\begin{equation}\label{nrn1}
 \Big| {\sum_{a=1}^\tA c_at_a \, \sqrt{\kappa x_a^6  + g x_a^2 x_0^4 + \frac{\gamma^2}{4} t_a^2 x_0^6  }  + \frac{\gamma}{2}\sum_{b=1}^\tB d_bt_{\tA+b}^2 x_0^3 } \Big| \geq \alpha \, x_0^{\tau+3}  
\end{equation}
 hold at any $x_0=x_0(\vec n, \vec m )$, $x_a = x_a(\vec n,\vec m )$, $t_a=t_a(\vec n)$, $a=1,\dots , \tA$ and $ t_{\tA+b}=t_{\tA+b}(\vec m)$, $ b=1, \dots, \tB$, defined in \eqref{x(n)t(n)}.
This suggests to define the function
\begin{equation} \label{lambdagh}
  \lambda(y, s, x_0, \kappa) := 
  \sqrt{\kappa y^6  + g y^2 x_0^4 + \frac{\gamma^2}{4} s^2x_0^6 } \, , 
  \end{equation}
and, for $ \vec c:=(c_1, \ldots, c_\tA)  \in  (\Z \setminus\{0\})^{\tA} $
 and  $\vec d = (d_1,\dots, d_\tB)  \in (\Z\setminus\{0\})^\tB$,  
\be\label{effeccid}
f_{\vec c,\vec d} \, \colon [-1, 1]^{2\tA+\tB+1} \times \cI \to \R \, , \quad
f_{\vec c,\vec d} \, (x,t,\kappa) := \sum_{a=1}^\tA c_a t_a \, \lambda(x_a, t_a, x_0, \kappa)  + \frac{\gamma}{2} \sum_{b=1}^\tB d_b t_{\tA+b}^2 x_0^3 \, ,  
\ee
with variables 
$ x = (x_0, \ldots, x_{\tA})$ and $ t = (t_1, \ldots, t_{\tA+\tB}) $.  

We estimate the sublevels of  $\kappa \mapsto f_{\vec c,\vec d} \, (x,t, \kappa)$ using Theorem \ref{DS}.
The set  $X := [-1,1]^{2\tA+\tB+1}$  is a   closed ball of $\R^{2\tA+\tB+1}$.
The function $f_{\vec c, \vec d}: X\times \cI \to \R$ is continuous and subanalytic.
Then we define the non-zero real analytic function 
\be\label{defrhoh}
\rho(x,t) := x_0\prod_{a=1}^{\tA}
 x_a t_a \, \prod_{1 \leq a<b \leq \tA } 
 \Big[ \big(g x_a^2 x_0^4 + \frac{\gamma^2}{4}t_a^2 x_0^6 \big)x_b^6- \big(g x_b^2 x_0^4 + \frac{\gamma^2}{4}t_b^2 x_0^6 \big)x_a^6 \Big] \, . 
\ee
We observe the following lemma. 
\begin{lemma}
There exist positive constants $ c(\tA) := c(\tA,g,\gamma,h)$, 
$ C(\tA) := C(\tA,g,\gamma) > 0 $ such that, for any 
\be\label{defxtnm}
x(\vec n, \vec m ) := (x_a(\vec n, \vec m))_{a=0, \ldots,\tA} \, , \quad 
 t(\vec n, \vec m):= (t_1(\vec n), \dots, t_\tA(\vec n), t_{\tA+1}(\vec m), \dots,
  t_{\tA+\tB}(\vec m)) \, , 
\ee
defined by \eqref{x(n)t(n)}, it results 
\begin{equation}
\label{rho.estgh}
c(\tA){\Big(\sum_{a=1}^\tA n_a+\sum_{b=1}^\tB m_b\Big)^{-\tau_2}}
\leq  |{\rho(x(\vec n,\vec m ), t(\vec n,\vec m))} |
\leq C(\tA) \Big(\sum_{a=1}^\tA n_a+\sum_{b=1}^\tB m_b \Big)^{-1}
\end{equation}
with $ \tau_2 := \tA +1 + 12\binom{\tA}{2} $. 
\end{lemma}

\begin{proof}
The upper bound \eqref{rho.estgh}  directly follows by  \eqref{xt.estn2}. 
The lower bound \eqref{rho.estgh} follows by  \eqref{xt.estn2} and
the fact that, since $ 1 \leq n_1 < \ldots < n_\tA  $ are all distinct, 
\begin{align*}
& \Big|(g x_a^2 x_0^4 + \frac{\gamma^2}{4}t_a^2 x_0^6)x_b^6- (g x_b^2 x_0^4 + \frac{\gamma^2}{4}t_b^2 x_0^6))x_a^6 \Big| \notag \\
& = 
x_0^{12}(\vec n, \vec m) (g n_a + \frac{\gamma^2}{4}\tanh(\tth n_a) )(g n_b + \frac{\gamma^2}{4}\tanh(\tth n_b))
\Big|\frac{n_a^3}{g n_a + 
\frac{\gamma^2}{4}\tanh(\tth n_a)}-\frac{n_b^3}{g n_b + \frac{\gamma^2}{4}\tanh(\tth n_b)}\Big|
\notag \\
&\geq x_0(\vec n, \vec m)^{12} \, g^2 \,  
\min_{y\geq 1}\Big(\frac{\di}{\di y} \frac{y^3}{g y + \frac{\gamma^2}{4}\tanh(\tth y)}\Big)  | n_a- n_b| \geq c \, x_0(\vec n, \vec m)^{12} 
\end{align*}
for some constant $ c > 0 $, having used  
in the last passage that,  for any $ y\geq 1$, $ g>0$, $ \gamma \in \R$ and $ \tth>0 $, 
\begin{align*}
\frac{\di}{\di y} \frac{y^3}{g y + \frac{\gamma^2}{4}\tanh(\tth y)}& = 
\frac{4 y^2 (8 g y + 3 \gamma^2 \tanh(\tth y) -  \gamma^2 \tth y \, {\rm sech}^2(\tth y) )}{(4 g y + \gamma^2 \tanh(\tth y))^2}\\
\geq & \frac{2 y^2}{ 16 g^2 y^2 + \gamma^4}\big(8 g+ \gamma^2( 3  \tanh(\tth y)- \tth y \, {\rm sech}^2(\tth y))\big)
\geq  \frac{2}{16 g^2  + \gamma^4}  8 g
\end{align*}
using that 
the function $ 3  \tanh(y)- y \, {\rm sech}^2( y)$ is 
 positive for  $ y > 0 $.
\end{proof}

We show now that  the assumptions (H1) and (H2)  of Theorem \ref{DS} hold true. 
 
 \noindent \underline{Verification of (H1).} 
 By \eqref{defrhoh}, if $\rho(x,t) \neq 0$ then, by \eqref{defrhoh},  
\be\label{condrhod0h}
 x_a \neq 0 \, , \ \forall  a=0, \ldots, \tA \,   , \quad  t_a\not=0 \, ,  \ \forall  a=1, \ldots, \tA \, . 
\ee 
 In particular on the set $\{ (x,t) \in X \colon \rho(x,t) \neq 0 \} \times Y$ the function $f_{\vec c,\vec d}$ in \eqref{effeccid} is real analytic. 
 \\
 \noindent \underline{Verification of (H2).} 
For any $  (x,t) $ such that 
   $\rho(x,t) \neq 0$, 
   the analytic function $\kappa \mapsto f_{\vec c,\vec d} \, (x,t,\kappa)$ 
   possesses only a finite number of zeros as a  
  consequence of the next lemma. 
   
   \begin{lemma}
   For any $ (x,t)$ 
  such that $ \rho(x,t) \neq 0$,
   the analytic
     function $\kappa \mapsto f_{\vec c,\vec d} \, (x,t,\kappa)$ is not identically zero in $\cI$.
   \end{lemma}
   
   \begin{proof} 
  Assume by contradiction that there exists 
  $$ 
   (x,t)\in X \, , \  x=(x_a)_{0 \leq a \leq \tA}\, , \ 
  t= (t_a)_{0 \leq a \leq \tA+\tB} \quad \text{with} \quad \rho(x,t) \neq 0 \quad
  \text{such \ that} \quad 
  f_{\vec c, \vec d} \, (  x,t, \kappa) = 0  
  $$
for any  $\kappa$ in the interval $\cI $.
Then, by analyticity, the function $\kappa \mapsto f_{\vec c, \vec d} \, (x, t, \kappa)$  is identically zero also on the larger interval $(-\delta, + \infty)$ where 
$ \delta:=\min_{1\leq a\leq \tA}t_a^2\frac{\gamma^2}{4} x_0^6$.  
Note that $\delta >0$ as $t_a^2$, $ x_0 ^6 > 0 $ by \eqref{condrhod0h}. 
  In particular, for any $ l \in \N $, all the derivatives 
  $\partial_\kappa^l f_{\vec{c}, \vec d} \, (x, t, \kappa) \equiv 0 $ are zero in the interval  
  $\kappa \in (-\delta, + \infty)$.
  
  We now compute  such derivatives at $\kappa = 0$ differentiating \eqref{effeccid}. 
The derivatives of 
the function $\lambda(y, s, x_0, \kappa)$ defined in \eqref{lambdagh}
are given by, for suitable constants $ C_l \neq 0 $,  
  \begin{equation*}
  \partial_\kappa^l  \lambda(y,s, x_0, \kappa)  = C_l y^{6l} \, \lambda(y, s, x_0, \kappa)^{1-2l} \, , 
\quad \forall l \in \N \, .
  \end{equation*}
Thus we obtain
\be\label{defmus}
 \partial_\kappa^l  \lambda(y,s, x_0, \kappa) \vert_{\kappa = 0}
 = C_l \, \mu(y,s,x_0)^l \, \lambda(y, s, x_0, 0)  \quad 
 \text{where} \quad 
 \mu(y,s, x_0):= \frac{y^6}{g y^2 x_0^4 + \frac{\gamma^2}{4}s^2x_0^6} \, ,
\ee
and, 
recalling \eqref{effeccid}, 
$$
\partial_\kappa^l f_{\vec c, \vec d } \, (x,t,\kappa) \vert_{\kappa = 0}
 = C_l \sum_{a=1}^\tA c_a t_a 
 \mu(x_a,t_a,x_0)^l  \, \lambda(x_a,t_a, x_0, 0) \, , \quad \forall l \in \N \, .
$$
As a consequence, the conditions 
 $\partial_\kappa^l f_{\vec c, \vec d} \, (  x,t,\kappa) \vert_{\kappa = 0} = 0$ for any 
 $l=1, \ldots, \tA$ imply that 
 \begin{equation}
 \label{van1gh}
 A(x,t) \vec r = 0
 \end{equation}
where $ A(x,t) $ is the $ \tA \times \tA $ matrix 
$$
A(x,t):= \begin{pmatrix}
 \mu(x_1,t_1, x_0)  \lambda(x_1,t_1,  x_0, 0) & 
 \cdots & \mu(x_\tA,t_\tA, x_0) \lambda(  x_\tA,t_\tA,  x_0, 0) \\
 \mu(x_1,t_1, x_0)^2 \lambda(  x_1,t_1,  x_0,0) & 
 \cdots & \mu(x_\tA,t_\tA, x_0)^2 \lambda(  x_\tA,t_\tA, x_0,0) \\
 \vdots & 
 \ddots & \vdots  \\
 \mu(x_1,t_1, x_0)^{\tA} \lambda(  x_1,t_1, x_0, 0) &  
 \ldots  & \mu(x_\tA,t_\tA, x_0)^\tA \lambda(x_\tA,t_\tA, x_0, 0)  
 \end{pmatrix} 
 $$
 and $\vec r $ is the vector
  $$
  \vec r:= \begin{bmatrix}r_1\\ \vdots\\  r_\tA\end{bmatrix} \, ,
  \quad r_a:= c_a t_a \in \Z \setminus \{0\} \, ,  \ \forall \, a=1,\dots, \tA \, , 
$$
because by assumption each $ c_a \neq 0 $ and \eqref{condrhod0h} holds. 
Since $\vec r\neq 0 $, we deduce by \eqref{van1gh} 
that the matrix  $A(x,t)$ has zero determinant. On the other hand, by the 
multilinearity of the determinant,  
\begin{align}
\det A(x,t)&  = \prod_{a=1}^\tA \mu(x_a,t_a,x_0) \lambda(  x_a,t_a,x_0, 0) 
\det \begin{pmatrix}
 1  & 
 \cdots & 1 \\
 \mu(x_1,t_1,x_0)  &  
 \cdots & \mu(x_\tA,t_\tA,x_0) \\
 \vdots & 
  \ddots & \vdots  \\
 \mu(x_1,t_1,x_0)^{\tA-1} & 
  \ldots  & \mu(x_\tA,t_\tA,x_0)^{\tA-1}  
 \end{pmatrix} \notag  \\
 & = \prod_{a=1}^\tA \mu(x_a, t_a,x_0) \lambda(  x_a, t_a,x_0, 0)  \, \prod_{1 \leq a< b\leq \tA} (\mu(x_a, t_a,x_0) - \mu(x_b,t_b,x_0)) \, . \label{svidet2h}
\end{align}
The condition $ \rho(x,t) \neq 0$ implies, by \eqref{condrhod0h}, \eqref{defmus} and 
\eqref{lambdagh}, 
that 
$$ 
\prod_{a=1}^\tA \mu(x_a,t_a,x_0) \lambda(x_a,t_a,x_0,0)  \neq 0 \, , 
$$
 and, in view of
\eqref{svidet2h}, the determinant  
 $\det A(x,t) = 0 $ if only if $\mu(x_a,t_a,x_0) = \mu(x_b,t_b,x_0)$ 
 for some $1 \leq a< b \leq \tA$.
By the definition of the function $(y,s,x_0) \to \mu(y,s,x_0)$ in \eqref{defmus} it follows that 
$$
\mu(x_a,t_a,x_0) = \mu(x_b,t_b,x_0) \quad \Rightarrow  \quad 
\big( g x_a^2 x_0^4 + \frac{\gamma^2}{4}t_a^2 x_0^6 \big) x_b^6 - 
\big( g x_b^2 x_0^4 + \frac{\gamma^2}{4}t_b^2 x_0^6 \big)x_a^6=0 \, . 
$$
In view of \eqref{defrhoh} this contradicts $\rho(x,t) \neq 0$.
\end{proof}

We have  verified assumptions (H1) and (H2) of Theorem \ref{DS}.
We thus conclude that 
there are $N_0 \in \N$, $\alpha_0, \delta, C  >0$, such that for any $\alpha \in (0, \alpha_0]$, any $ N \in \N$, $N \geq N_0$, any $(x,t) \in X$ with $\rho(x,t) \neq 0$,
\begin{equation}
\label{meas2gh}
\meas \big\{ \kappa \in \cI \colon \ \ \ |f_{\vec c,\vec d} \, (x,t, \kappa)|\leq \alpha |\rho(x,t)|^N 
\big\} \leq C \alpha^\delta \, |\rho(x,t)|^{N\delta} \, . 
\end{equation}
For $\vec n = (n_1, \ldots, n_\tA) \in \N^\tA$ with $ 1 \leq n_1 < \ldots < n_\tA $ and 
 $\vec m= (m_1, \ldots, m_\tB) \in \N^\tB$ we consider the set
\be\label{defBnmh}
\cB_{\vec c,\vec d ,\vec m ,\vec n}(\alpha, N) := \left\lbrace
 \kappa \in \cI \colon \ \ \ |f_{\vec c, \vec d} \, (x(\vec n,\vec m ),t(\vec n, \vec m),\kappa)|\leq \alpha |\rho(x(\vec n,\vec m),t(\vec n, \vec m))|^N 
\right\rbrace
\ee
where $x(\vec n, \vec m ) $
and  $ t(\vec n, \vec m) $
are defined in \eqref{defxtnm}. 
By \eqref{rho.estgh} 
we deduce that  $\rho(x(\vec n, \vec m),t(\vec n,\vec m)) \neq 0$, and 
\eqref{meas2gh} implies that 
\begin{equation}
 \label{measGgh}
 \meas \, \cB_{\vec c,\vec d ,\vec m ,\vec n}(\alpha, N) \leq C \alpha^\delta |\rho(x(\vec n,\vec m), t(\vec n, \vec m))|^{N\delta} \, . 
 \end{equation} 
Consider the set 
\begin{equation}\label{defKhf}
\cK(\alpha, N) :=  \bigcup_{
              \substack{ \vec{n}=(n_1, \ldots, n_\tA) \in \N^{\tA},\, 1 \leq n_1 < \ldots < n_{\tA}  \\ 
              \vec m= ( m_1, \dots m_\tB)  \in \N^{\tB}\\
                        \vec c \in (\Z \setminus \{0\})^{\tA}, \, |\vec c|_{\infty} \leq \tM \\
                        \vec d \in (\Z \setminus \{0\})^{\tB}, |\vec d|_{\infty} \leq \tM}}
\cB_{\vec c,\vec d ,\vec m ,\vec n}(\alpha, N) \subset \cI  \, .
 \end{equation}
By \eqref{measGgh} and 
    \eqref{rho.estgh} one gets
 \begin{equation} \label{measK1gh}
 \meas \,
\cK(\alpha, N)  \leq 
C(\tA, \tM) \alpha^\delta \sum_{\substack{n_1, \ldots, n_\tA \in \N \\ m_1, \dots, m_\tB \in \N }} 
\frac{1}{(\sum_{a=1}^\tA n_a+\sum_{b=1}^\tB m_b)^{N\delta}}  \leq C'(\tA, \tM) \alpha^\delta 
 \end{equation}
 for some finite constant $ C'(\tA, \tM) < + \infty $, provided 
$ N \delta > \tA+\tB $. 
 We fix 
 $$ 
 \underline{N} := [ (\tA + \tB) \delta^{-1} ] + 1 \quad \text{and define}  
 \quad
  {\cal K}_\alpha := \cK(\alpha, \underline{N})  
  $$  
  whose measure satisfies
 $ | {\cal K}_\alpha| = O(\alpha^\delta ) $, in view of \eqref{measK1gh}.
  
In conclusion, 
for any $\kappa \in \cI \setminus \cK_ \alpha$, for any  
$ \vec n = (n_1, \ldots, n_\tA) $ with 
 $ 1 \leq n_1 < \ldots < n_\tA  $ and $ \vec m \in \N^\tB $,
 any 
$  \vec c \in (\Z \setminus \{0\})^{\tA} $ with $ |\vec c \, |_{\infty} \leq \tM $ 
   and $  \vec d \, \in (\Z \setminus \{0\})^{\tB}$ with $ |\vec d \, |_{\infty} \leq \tM $,   
 it results,
 by \eqref{defKhf}  and \eqref{defBnmh}, that
\be\label{f1gh}
|{f_{\vec c, \vec d} \, (x(\vec n, \vec m ),t(\vec n, \vec m),  \kappa)}| 
 >  \alpha |\rho(x(\vec n, \vec m), t(\vec n, \vec m))|^{\underline{N}}  \stackrel{\eqref{rho.estgh}}{\geq } 
\frac{c(\tA) \alpha }{\left(\sum_{a=1}^\tA n_a+\sum_{b=1}^\tB m_b\right)^{\tau_2 \underline{N} }} \, . 
\ee
Recalling the definition of 
$ f_{\vec c,\vec d} $ in \eqref{effeccid} and $ x_0(\vec n,\vec m) $ in \eqref{x(n)t(n)}, 
the lower bound \eqref{f1gh} 
implies \eqref{nrg2} with $ \tau := \tau_2  \underline{N} - 3 $ (cfr. \eqref{nrn1}) 
and re-denoting $ \alpha c(\tA) \leadsto \alpha $.
\end{proof}

 Supported by PRIN 2020XB3EFL, {\it Hamiltonian and Dispersive PDEs}.

\end{document}